\documentclass{amsart}
\usepackage{amsfonts,amssymb,mathtools,mathabx}

\usepackage[usenames,dvipsnames]{xcolor}

\usepackage{tikz}
\usetikzlibrary{arrows, shapes, positioning, calc, intersections}
\usetikzlibrary{decorations.pathreplacing, decorations.markings}

\usepackage{relsize}					
\usepackage{exscale}					

\usepackage{geometry, enumerate, bbm, bm}

\usepackage[font=small, format=plain, 
					labelfont=bf, up, textfont=it, up]{caption}

\usepackage{float}
\usepackage[section]{placeins}
\usepackage{hyperref}
\hypersetup{
    unicode=false,          		
    pdftoolbar=true,        		
    pdfmenubar=true,        	
    pdffitwindow=false,     		
    pdfstartview={FitH},    		
   pdftitle={Global well-posedness and soliton resolution for the derivative nonlinear Schrodinger equation},    					
   	pdfauthor={Robert Jenkins, Jiaqi Liu, Peter Perry, Catherine Sulem},     	
    linktocpage=true,			
    pdfnewwindow=true,      	
    colorlinks=true,       			
    linkcolor=blue,          		
    citecolor=PineGreen,    	
    filecolor=magenta,      		
    urlcolor=cyan           		
}
    
\theoremstyle{plain}
\newtheorem{theorem}{Theorem}[section]
\newtheorem{corollary}[theorem]{Corollary}
\newtheorem{lemma}[theorem]{Lemma}
\newtheorem{proposition}[theorem]{Proposition}

\theoremstyle{definition}
\newtheorem{definition}[theorem]{Definition}
\newtheorem{problem}[theorem]{Problem}
\newtheorem{RHP}[theorem]{Riemann-Hilbert Problem}
\newtheorem{DBAR}[theorem]{$\dbar$-Problem}

\theoremstyle{remark}
\newtheorem{remark}[theorem]{Remark}

\numberwithin{figure}{section}
\numberwithin{equation}{section}

\allowdisplaybreaks

\DeclareMathOperator{\ad}{ad}
\DeclareMathOperator{\real}{Re}
\DeclareMathOperator{\dist}{dist}
\DeclareMathOperator{\imag}{Im}
\DeclareMathOperator{\sgn}{sgn}
\DeclareMathOperator{\sech}{sech}

\DeclareMathOperator{\Res}{Res}
\DeclareMathOperator{\res}{res}

\begin{document}

\title[Global Well-Posedness and Soliton Resolution for DNLS]{Global Well-posedness and soliton resolution for the Derivative Nonlinear Schr\"{o}dinger equation}
\author{Robert Jenkins}
\author{Jiaqi Liu}
\author{Peter Perry}
\author{Catherine Sulem}
\address[Jenkins]{Department of Mathematics, University of Arizona, Tucson, Arizona 85721--0089}
\address[Liu]{Department of Mathematics, University of Kentucky, Lexington, Kentucky 40506--0027}
\address[Perry]{ Department of Mathematics, University of Kentucky, Lexington, Kentucky 40506--0027}
\address[Sulem]{Department of Mathematics, University of Toronto, Toronto, Ontario M5S 2E4, Canada }
\thanks{P. Perry supported in part by a Simons Research and Travel Grant.}
\thanks {C. Sulem supported in part by NSERC Grant 46179-13}
\date{\today}
\begin{abstract}
We study the Derivative Nonlinear Schr\"odinger equation for  general initial conditions 
in weighted Sobolev spaces that can support bright solitons (but excluding spectral 
singularities).  We prove global well-posedness  and give a full description of  the long-
time behavior of the solutions in the form of a finite sum of  localized solitons and a 
dispersive component.  At leading order and in space-time
cones,  the solution has the form   of a multi-soliton whose parameters  are slightly 
modified  from their initial  values by  solitons-solitons and solitons-radiation
 interactions.  Our analysis  provides an explicit expression 
for the correction dispersive term.  
We use  the nonlinear steepest descent method  of Deift and Zhou \cite{DZ03}  
revisited by the $\overline{\partial}$-analysis of Dieng-McLaughlin \cite{DM08} and complemented by the recent 
work of Borghese-Jenkins-McLaughlin \cite{BJM16} on soliton resolution for the focusing nonlinear Schr\"{o}dinger equation. 
\end{abstract}

\maketitle

\tableofcontents 

%
%
\begingroup
\let\clearpage\relax

%
%

%
%
%

\newcommand{\newtext}[1]{{\color{blue}{#1}}}

%
%

\newcommand{\rarr}{\rightarrow}
\newcommand{\darr}{\downarrow}
\newcommand{\uarr}{\uparrow}

\newcommand{\sig}{\sigma_3}

\newcommand{\dotarg}{\, \cdot \, }

\newcommand{\eps}{\varepsilon}
\newcommand{\lam}{\lambda}
\newcommand{\Lam}{\Lambda}

\newcommand{\dee}{\partial}
\newcommand{\dbar}{\overline{\partial}}

\newcommand{\vect}[1]{\boldsymbol{\mathbf{#1}}}

\newcommand{\dint}{\displaystyle{\int}}

%
%

\newcommand{\resp}{resp.\@}
\newcommand{\ifff}{if and only if }
\newcommand{\ie}{i.e.}

%
%

\newcommand{\lp}{\left(}
\newcommand{\rp}{\right)}
\newcommand{\lb}{\left[}
\newcommand{\rb}{\right]}
\newcommand{\lw}{\left<}
\newcommand{\rw}{\right>}

\newcommand{\one}{\bm{1}}

%
%


\newcommand{\C}{\mathbb{C}}
\newcommand{\R}{\mathbb{R}}
\newcommand{\N}{\mathbb{N}}


\newcommand{\calB}{\mathcal{B}}
\newcommand{\calC}{\mathcal{C}}
\newcommand{\calD}{\mathcal{D}}
\newcommand{\calE}{\mathcal{E}}
\newcommand{\calF}{\mathcal{F}}
\newcommand{\calG}{\mathcal{G}}
\newcommand{\calI}{\mathcal{I}}
\newcommand{\calK}{\mathcal{K}}
\newcommand{\calL}{\mathcal{L}}
\newcommand{\calM}{\mathcal{M}}
\newcommand{\calN}{\mathcal{N}}
\newcommand{\calR}{\mathcal{R}}
\newcommand{\calS}{\mathcal{S}}
\newcommand{\calU}{\mathcal{U}}
\newcommand{\calZ}{\mathcal{Z}}


\newcommand{\ba}{\breve{a}}
\newcommand{\bb}{\breve{b}}

\newcommand{\br}{\breve{r}}
\newcommand{\bs}{\breve{s}}

\newcommand{\balpha}{\breve{\alpha}}
\newcommand{\bbeta}{\breve{\beta}}
\newcommand{\brho}{\breve{\rho}}
\newcommand{\bgamma}{\breve{\gamma}}

\newcommand{\bC}{\breve{C}}
\newcommand{\bN}{\breve{N}}


\newcommand{\bfe}{\mathbf{e}}
\newcommand{\bff}{\mathbf{f}}
\newcommand{\bfN}{\mathbf{N}}
\newcommand{\bfM}{\mathbf{M}}


\newcommand{\qbar}{\overline{q}}
\newcommand{\rbar}{\overline{r}}
\newcommand{\wbar}{\overline{w}}
\newcommand{\zbar}{\overline{z}}


\newcommand{\etabar}{\overline{\eta}}
\newcommand{\lambar}{\overline{\lambda}}
\newcommand{\lambdabar}{\overline{\lambda}}
\newcommand{\rhobar}{\overline{\rho}}
\newcommand{\btheta}{\overline{\vartheta}}
\newcommand{\zetabar}{\overline{\zeta}}


\newcommand{\hatphi}{\widehat{\phi}}

%
%
%

\newcommand{\bigo}[1]{\mathcal{O} \left( #1 \right)}
\newcommand{\littleo}[2][ ]{ {o}_{#1} \left( #2 \right) }

\newcommand{\bigO}[2][ ]{\mathcal{O}_{#1} \left( {#2} \right)}
\newcommand{\norm}[2][ ]{\left\| {#2} \right\|_{#1}}

%
%

\newcommand{\medcup}{{\mathsmaller{\bigcup}}}

%
%

\newcommand{\upmat}[1]
{
	\begin{pmatrix}	
	0	&	#1	\\
	0	&	0
	\end{pmatrix}
}
\newcommand{\lowmat}[1]
{
	\begin{pmatrix}
	0	&	0	\\
	#1	&	0
	\end{pmatrix}
}
\newcommand{\upunitmat}[1]
{
	\begin{pmatrix}
	1	&	#1	\\
	0	&	1
	\end{pmatrix}
}
\newcommand{\lowunitmat}[1]
{
	\begin{pmatrix}
	1	&	0	\\
	#1	&	1
	\end{pmatrix}
}

%
%

\newcommand{\La}{ {\mathcal{L} } }
\newcommand{\B}{ {\mathcal{B } } }

\newcommand{\qsol}{q_{\mathrm{sol}}}
\newcommand{\usol}{u_{\mathrm{sol}}}
\newcommand{\tqsol}{\widetilde{q}_{\mathrm{sol}}}

\newcommand{\poles}{\Lambda}
\newcommand{\indicator}{\chi_{_\poles}}
\newcommand{\poledist}{ {\mathrm{d}_{\poles} } }
\newcommand{\coeff}{\mathcal{C}}
\newcommand{\data}{\sigma_d}

\newcommand{\NN}{ {N} }
\newcommand{\nn}{ {n} }
\newcommand{\nk}[2][]{ {{n}_{#1}^{(#2)}} }

\newcommand{\mk}[1]{ m^{(#1)} }
\newcommand{\vk}[1]{ {v^{(#1)}} }
\newcommand{\Sk}[1]{ {\Sigma^{(#1)} } }
\newcommand{\Wk}[2][]{ {W_{#1}^{(#2)} } }

\newcommand{\sgnt}{\eta}
\newcommand{\pospoles}{\Delta_{\xi,\sgnt}^+}
\newcommand{\negpoles}{\Delta_{\xi,\sgnt}^-}
\newcommand{\posnegpoles}{\Delta_{\xi,\sgnt}^\pm}
\newcommand{\posint}{ I^{+}_{\xi,\sgnt}} 
\newcommand{\negint}{ I^{-}_{\xi,\sgnt}} 						

\newcommand{\Nrhp}[1][]{ {{\mathcal{N}}_{#1}^{\mathsc{rhp}} } }
\newcommand{\Nsol}[1][]{ {{\mathcal{N}}_{\!\!#1}^{\mathrm{sol}} } }
\newcommand{\NsolAlt}[1][]{ {{\mathcal{N}}_{#1}^{\negpoles} } }
\newcommand{\NPC}[1][]{ {{\mathcal{N}}_{#1}^{\mathsc{pc}} } }

\newcommand{\Uxi}{ {\mathcal{U}_\xi} }

\newcommand{\error}{ {{\mathcal{E}}} }

\newcommand{\Ske}[1][]{ { \Sigma^{(\error)}_{#1} }  }
\newcommand{\vke}{ \vk{\error} }

\newcommand{\mout}{M^{(\textrm{out})}}
\newcommand{\mxi}{M^{(\xi)}}
\newcommand{\mPC}{M^{(\textsc{pc})}}

\newcommand{\sol}{\mathrm{sol}}

%
%

\newcommand{\spacing}[2]
{ 
{	
	\mathrlap{#1}
	\hphantom{#2} 
}
}

%
%

\newcommand{\pd}[3][ ]{\frac{\partial^{#1} #2}{\partial #3^{#1} } }
\newcommand{\od}[3][ ]{\frac{\mathrm{d}^{#1} #2}{\mathrm{d} #3^{#1} } }
\newcommand{\vd}[3][ ]{\frac{ \delta^{#1} #2}{ \delta #3^{#1}} } 

%
%

\renewcommand{\Re}{\mathop{ \mathrm{Re}}\nolimits}
\renewcommand{\Im}{\mathop{ \mathrm{Im} }\nolimits}
\newcommand{\im}{\mathrm{i} }

%
%

\newcommand{\triu}[2][1]{\begin{pmatrix} #1 & #2 \\ 0 & #1 \end{pmatrix}}
\newcommand{\tril}[2][1]{\begin{pmatrix} #1 & 0 \\ #2 & #1 \end{pmatrix}}
\newcommand{\diag}[2]{\begin{pmatrix} #1 & 0 \\ 0 & #2 \end{pmatrix}}
\newcommand{\offdiag}[2]{\begin{pmatrix} 0 & #1 \\ #2 & 0 \end{pmatrix}}

\newcommand{\striu}[2][1]{\begin{psmallmatrix} #1 & #2 \\ 0 & #1 \end{psmallmatrix}}
\newcommand{\stril}[2][1]{\begin{psmallmatrix} #1 & 0 \\ #2 & #1 \end{psmallmatrix}}
\newcommand{\sdiag}[2]{\begin{psmallmatrix*}[c] #1 & 0 \\ 0 & #2 \end{psmallmatrix*}}
\newcommand{\soffdiag}[2]{\begin{psmallmatrix*}[r] 0 & #1 \\ #2 & 0 \end{psmallmatrix*}}
\newcommand{\stwomat}[4]{\begin{psmallmatrix*} #1 & #2 \\ #3 & #4 \end{psmallmatrix*}}

%
%

\newcommand{\mathsc}[1]{ {\text{\normalfont\scshape#1}} }

%
%

\newcommand{\oldnorm}[2][]
{
	 {\left\|  #2 \right\|_{#1} } 
}

\newcommand{\twomat}[4]
{
\begin{pmatrix}
	#1	&	#2	\\
	#3	&	#4
\end{pmatrix}
}

\newcommand{\Twomat}[4]
{
	\left(
		\begin{array}{ccc}
			#1	&&	#2	\\
			\\
			#3	 &&	#4
		\end{array}
	\right)
}

\newcommand{\twovec}[2]
{
	\left(
		\begin{array}{c}
			#1		\\
			#2
		\end{array}
	\right)
}

\newcommand{\Twovec}[2]
{
	\left(
		\begin{array}{c}
			#1		\\
			\\
			#2
		\end{array}
	\right)
}

%
%

\tikzset{->-/.style={decoration={
  markings,
  mark=at position .55 with {\arrow{triangle 45}} },postaction={decorate}}
}

%
%
\newcommand{\FigPhaseA}{
\begin{tikzpicture}[scale=0.7]
\path[fill=pink,opacity=0.5]	(0,0) rectangle(4,4);
\path[fill=pink,opacity=0.5]    (-4,-4) rectangle(0,0);
\path[fill=cyan,opacity=0.5]		(-4,0) rectangle (0,4);
\path[fill=cyan,opacity=0.5]		(0,-4) rectangle (4,0);
\draw[fill] (0,0) circle[radius=0.075];
\draw[thick,->,>=stealth] 	(0,0) -- (2,0);
\draw	[thick]    	(2,0) -- (4,0);
\draw[thick,->,>=stealth]	(-4,0) -- (-2,0);
\draw[thick]		(-2,0) -- (0,0);
\draw[thin,dashed]	(0,4) -- (0,-4);
\node at (2,2) 	{$e^{2it\theta} \gg 1 $};
\node at (-2,-2)  {$e^{2it\theta} \gg 1 $};
\node at (-2,2)	{$e^{2it\theta} \ll 1 $};
\node at (2,-2)	{$e^{2it\theta} \ll 1 $};
\node[above] at (0,4)	{$ \sgnt = +1 $};
\node[below] at (2,-0.05)		{$\posint$};
\node[below] at (-2,-0.05)		{$\negint$};
\node[below right] at (0,0)		{$\xi$};
\end{tikzpicture}
}

%
%

\newcommand{\FigPhaseB}{
\begin{tikzpicture}[scale=0.7]
\path[fill=cyan,opacity=0.5]	(0,0) rectangle(4,4);
\path[fill=cyan,opacity=0.5]    (-4,-4) rectangle(0,0);
\path[fill=pink,opacity=0.5]		(-4,0) rectangle (0,4);
\path[fill=pink,opacity=0.5]		(0,-4) rectangle (4,0);
\draw[fill] (0,0) circle[radius=0.075];
\draw[thick,->,>=stealth] 	(0,0) -- (2,0);
\draw	[thick]    	(2,0) -- (4,0);
\draw[thick,->,>=stealth]	(-4,0) -- (-2,0);
\draw[thick]		(-2,0) -- (0,0);
\draw[thin,dashed]	(0,4) -- (0,-4);
\node at (2,2) 	{$e^{2it\theta} \ll 1 $};
\node at (-2,-2)  {$e^{2it\theta} \ll 1 $};
\node at (-2,2)	{$e^{2it\theta} \gg 1 $};
\node at (2,-2)	{$e^{2it\theta} \gg 1 $};
\node[above] at (0,4)	{$ \sgnt = -1 $};
\node[below] at (-2,-0.05)		{$\posint$};
\node[below] at (2,-0.05)		{$\negint$};
\node[below right] at (0,0)		{$\xi$};
\end{tikzpicture}
}

%
%

\newcommand{\posDBARcontours}{
\resizebox{0.45\textwidth}{!}{
\begin{tikzpicture}
\path [fill=gray!20] (0,0) -- (-4,4) -- (-4.5,4) -- (-4.5,-4) -- (-4,-4) -- (0,0);
\path [fill=gray!20] (0,0) -- (4,4) -- (4.5,4) -- (4.5,-4) -- (4,-4) -- (0,0);
%
\draw [help lines] (-4.5,0) -- (4.5,0);
\draw [thick][->-] (-4,4) -- (0,0);
\draw [thick][->-] (0,0) -- (4,4);
\draw [thick][->-] (0,0) -- (4,-4);
%
\foreach \pos in { 
		(-1.5,1.4), (-1.5,-1.4), (-3.5,2.5), 
		(-3.5,-2.5), (3,2), (3,-2)}
\draw[color=white, fill=white] \pos circle [radius=.2];
%
\node[left] at (3,3) {$\Sigma_{1}\,$};
\node[right] at (-3,3) {$\Sigma_{2}$};
\node[right] at (-3,-3) {$\,\Sigma_{3}$};
\node[left] at (3,-3) {$\Sigma_{4}\,$};
\draw[fill] (0,0) circle [radius=0.025];
\node[below] at (0,0) {$\xi$};
\node[above] at (0,4) {$\sgnt = +1$};
%
\node at (1,.4) {$\Omega_{1}$};
\node at (0,1.08) {$\Omega_{2}$};
\node at (-1,.4) {$\Omega_{3}$};
\node at (-1,-.4) {$\Omega_{4}$};
\node at (0,-1.08) {$\Omega_{5}$};
\node at (1,-.4) {$\Omega_{6}$};
%
\node[left] at (4.5,0.8) {$\tril{-R_1 e^{-2it \theta}} $ };
\node[right] at (-4.5,0.8) {$\triu{ -R_3 e^{2it \theta}} $ };
\node[right] at (-4.5,-0.8) {$\tril{ R_4 e^{-2it \theta}} $ };
\node[left] at (4.5,-0.8) {$\triu{ R_6 e^{2it \theta }} $};
\node at (0,2.5) {$\diag{1}{1}$};
\node at (0,-2.5) {$\diag{1}{1}$};
\end{tikzpicture}
}
}

%
%

\newcommand{\negDBARcontours}{
\resizebox{0.45\textwidth}{!}{
\begin{tikzpicture}
\path [fill=gray!20] (0,0) -- (-4,4) -- (-4.5,4) -- (-4.5,-4) -- (-4,-4) -- (0,0);
\path [fill=gray!20] (0,0) -- (4,4) -- (4.5,4) -- (4.5,-4) -- (4,-4) -- (0,0);
%
\draw [help lines] (-4.5,0) -- (4.5,0);
\draw [thick][->-] (-4,4) -- (0,0) ;
\draw [thick][->-](-4,-4)--(0,0);
\draw [thick][->-] (0,0) -- (4,4);
\draw [thick][->-] (0,0) -- (4,-4);
%
\foreach \pos in { 
		(-1.5,1.4), (-1.5,-1.4), (-3.5,2.5), 
		(-3.5,-2.5), (3,2), (3,-2)}
\draw[color=white, fill=white] \pos circle [radius=.2];
%
\node[left] at (3,3) {$\Sigma_{2}\,$};
\node[right] at (-3,3) {$\Sigma_{1}$};
\node[right] at (-3,-3) {$\,\Sigma_{4}$};
\node[left] at (3,-3) {$\Sigma_{3}\,$};
\draw[fill] (0,0) circle [radius=0.025];
\node[below] at (0,0) {$\xi$};
\node[above] at (0,4) {$\sgnt = -1$};
%
\node at (1,.4) {$\Omega_{3}$};
\node at (0,1.08) {$\Omega_{2}$};
\node at (-1,.4) {$\Omega_{1}$};
\node at (-1,-.4) {$\Omega_{6}$};
\node at (0,-1.08) {$\Omega_{5}$};
\node at (1,-.4) {$\Omega_{4}$};
%
\node[left] at (4.5,0.8) {$\triu {-R_3 e^{2it \theta}} $ };
\node[right] at (-4.5,0.8) {$\tril{ -R_1 e^{-2it \theta}} $ };
\node[right] at (-4.5,-0.8) {$\triu{ R_6 e^{2it \theta}} $ };
\node[left] at (4.5,-0.8) {$\tril{ R_4 e^{-2it \theta }} $};
\node at (0,2.5) {$\diag{1}{1}$};
\node at (0,-2.5) {$\diag{1}{1}$};
\end{tikzpicture}
}
}

%
%

\newcommand{\FigPCjumps}{
\resizebox{0.45\textwidth}{!}{
\begin{tikzpicture}
\draw [help lines] (-4,0) -- (4,0);

\draw[->-] (-3,3) -- (0,0);
\draw[->-] (-3,-3) -- (0,0);
\draw[->-] (0,0) -- (3,3);
\draw[->-] (0,0) -- (3,-3);

\node[left] at (2,2) {$\Sigma_{1}\,$};
\node[right] at (-2,2) {$\Sigma_{2}$};
\node[right] at (-2,-2) {$\,\Sigma_{3}$};
\node[left] at (2,-2) {$\Sigma_{4}\,$};
\draw[fill] (0,0) circle [radius=0.025];
\node[below] at (0,0) {$\xi$};
\node[above] at (0,3) {$\sgnt = 1$};
%
\node at (1,.4) {$\Omega_{1}$};
\node at (0,1.08) {$\Omega_{2}$};
\node at (-1,.4) {$\Omega_{3}$};
\node at (-1,-.4) {$\Omega_{4}$};
\node at (0,-1.08) {$\Omega_{5}$};
\node at (1,-.4) {$\Omega_{6}$};
%
\node[right] at (2.0,1.5) {$\tril[1] { s_\xi }$}; 
\node[right] at (2.0,-1.5) {$\triu[1]{ r_\xi } $}; 
\node[left] at (-2.0,1.5) {$\triu[1]{ \dfrac{r_\xi}{1+r_\xi s_\xi} } $}; 
\node[left] at (-2.0,-1.5) {$\tril[1] { \dfrac{s_\xi}{1+r_\xi s_\xi} } $}; 
%
\end{tikzpicture}
}}

%
%
\newcommand{\FigPCjumpsB}{
\resizebox{0.45\textwidth}{!}{
\begin{tikzpicture}
\draw [help lines] (-4,0) -- (4,0);

\draw[->-] (-3,3) -- (0,0)  ;
\draw[->-] (-3,-3) -- (0,0);
\draw[->-] (0,0) -- (3,3);
\draw[->-] (0,0) -- (3,-3);

\node[left] at (2,2) {$\Sigma_{2}\,$};
\node[right] at (-2,2) {$\Sigma_{1}$};
\node[right] at (-2,-2) {$\,\Sigma_{4}$};
\node[left] at (2,-2) {$\Sigma_{3}\,$};
\draw[fill] (0,0) circle [radius=0.025];
\node[below] at (0,0) {$\xi$};
\node[above] at (0,3) {$\sgnt = -1$};
%
\node at (1,.4) {$\Omega_{3}$};
\node at (0,1.08) {$\Omega_{2}$};
\node at (-1,.4) {$\Omega_{1}$};
\node at (-1,-.4) {$\Omega_{6}$};
\node at (0,-1.08) {$\Omega_{5}$};
\node at (1,-.4) {$\Omega_{4}$};
%
\node[right] at (2.0,1.5) {$\triu[1]{ \dfrac{r_\xi}{1+r_\xi s_\xi} }$}; 
\node[right] at (2.0,-1.5) {$\tril[1] { \dfrac{s_\xi}{1+r_\xi s_\xi}} $}; 
\node[left] at (-2.0,1.5) {$\tril[1] { s_\xi }$}; 
\node[left] at (-2.0,-1.5) {$\triu[1]{ r_\xi }$}; 
%
\end{tikzpicture}
}}

%
%

%
%

\newcommand{\solfig}{
\hspace*{\stretch{1}}
\begin{tikzpicture}[scale=0.7]						
\coordinate (shift) at (-40:1.5);
\coordinate (x1) at (0.7,0);
\coordinate (x2) at (2.1,0);
\coordinate (topleft) at ($(x1)+(95:6)$);
\coordinate (bottomleft) at ($(x1)+(-120:6)$);
\coordinate (topright) at ($(x2)+(60:6)$);
\coordinate (bottomright) at ($(x2)+(-85:6)$);
\coordinate (C) at ($ (x1)!.5!(x2)+(77.5:.5)$);

\begin{scope}
  \clip (-4,-4) rectangle (4,4);
  \path[name path=top] (-4,4) -- (4,4);
  \path[name path=bottom] (-4,-4) -- (4,-4);
  \path[name path=right] (4,4) -- (4,-4);  	
  \path [fill=gray!15] (x1) -- (topleft) -- (topright) -- 
    (x2) -- (bottomright) -- (bottomleft) -- (x1);
  \draw[thick,name path=leftcone] (bottomleft) -- (x1) -- (topleft);
  \draw[thick,name path=rightcone] (bottomright) -- (x2) -- (topright);
    \draw [help lines, name path=axis][->] (-4,0) -- (3.3,0)
    	node[label=0:$x$] {};
    \draw [help lines][->] (-0.5,-4) -- (-0.5,3.0)
    	node[label=90:$t$] {};
    \path [name intersections={of= leftcone and top, by=TL}];
    \path [name intersections={of= rightcone and right, by=TR}];
    \path [name intersections={of= leftcone and bottom, by=BL}];
    \path [name intersections={of= rightcone and bottom, by=BR}];
\end{scope}
    \node [fill=black, inner sep=1.5pt, circle, label=-70:$x_2$] at (x2) {};
    \node [fill=black, inner sep=1.5pt, circle, label=-177:$x_1$] at (x1) {};
	\node [label=180:${x-v_1 t = x_1}$] at (TL) {};
	\node [label=90:${x-v_2 t = x_2}$] at (TR) {};
    \node [label=170:${x-v_2 t = x_1}$] at (BL) {};
    \node [label=10:${x-v_1 t = x_2}$] at (BR) {};
    \node at (C) {$\mathcal{S}$};
\end{tikzpicture}
\hspace*{\stretch{1}}
\begin{tikzpicture}[scale=0.7]			
	\coordinate (v1) at (0.75,0);
	\coordinate (v2) at (-2,0);
	\path [fill=gray!15] ($(v1)+(0,-1)$) -- ($(v1)+(0,7)$) 
	--($(v2)+(0,7)$) -- ($(v2)+(0,-1)$) --  ($(v1)+(0,-1)$);
	\draw[thick] ($(v1)+(0,-.4)$) -- ($(v1)+(0,7)$);
	\draw[thick] ($(v2)+(0,-.4)$) -- ($(v2)+(0,7)$);
	\draw[help lines][->] (-4,0) -- (4,0) 
		node[label=0:$\Re \lam$] {};
	\node [label=${-v_1/4}$] at (.75,-1.2) {};
	\node [label=${-v_2/4}$] at (-2,-1.2) {};
\node [fill=black, inner sep = 1pt, circle, label=-90:$\lambda_1$] at (3,6) 			{};
\node [fill=black, inner sep = 1pt, circle, label=-90:$\lambda_2$] at (1.5,5) 		{};
\node [fill=black, inner sep = 1pt, circle, label=-90:$\lambda_3$] at (0,5.5) 		{};
\node [fill=black, inner sep = 1pt, circle, label=-90:$\lambda_5$] at (-3.3,4.4) 	{};
\node [fill=black, inner sep = 1pt, circle, label=-90:$\lambda_8$] at (-0.8,0.8) 	{};
\node [fill=black, inner sep = 1pt, circle, label=-90:$\lambda_4$] at (-2.5,6.5) 	{};
\node [fill=black, inner sep = 1pt, circle, label=-90:$\lambda_6$] at (-1.6,3) 		{};
\node [fill=black, inner sep = 1pt, circle, label=-90:$\lambda_9$] at (2,1) 			{};
\node [fill=black, inner sep = 1pt, circle, label=-90:$\lambda_7$] at (-3.8,1.1) 	{};
\node [fill=black, inner sep = 1pt, circle, label=-90:$\lambda_{10}$] at (3.8,2.3) 	{};
\end{tikzpicture}
\hspace*{\stretch{1}}
}

%
%
 
\newcommand{\NRHP}{N^{\mathrm{RHP}}}
\newcommand{\diagmat}[2]
{
	\begin{pmatrix}
		#1	&	0	\\
		0	&	#2
	\end{pmatrix}
}
\newcommand{\sigeps}{\begin{pmatrix} 0 & 1 \\ \eps & 0 \end{pmatrix}}
\section{Introduction}

In this paper will prove global well-posedness and soliton resolution for the derivative nonlinear Schr\"{o}dinger equation
\begin{gather}
\label{DNLS1}
i u_t + u_{xx} - i \varepsilon (|u|^2 u)_x =0\\
\label{data1}
u(x,t=0) = u_0
\end{gather}
for initial data in a dense and open subset  of the function space $H^{2,2}(\R)$ which contains $0$ and also initial data of arbitrarily large $L^2$ norm. As we explain below, this set is spectrally determined and includes initial data with at most finitely many soliton components and no algebraic solitons.  Here
$\eps = \pm 1$ and 
$H^{2,2}(\R)$ is the completion of $C_0^\infty(\R)$ in the norm
$$ \norm[H^{2,2}]{u} = \left( \norm[L^2]{(1+(\dotarg))^2 u(\dotarg)}^2 + \norm[L^2]{u''}^2 \right)^{1/2}. $$

It is known that the Cauchy problem is locally well-posed in $H^{1/2}(\R)$ (see, for example, Takaoka \cite{Takaoka99}) and globally well-posed for small data \cite{HO92,Wu14}. Precisely, for any $u_0\in H^{1/2}(\R)$ such that $\|u_0\|_{L^2} < \sqrt{4\pi}$, 
there exists a unique solution $u\in C(\R, H^{1/2}(\R))$ \cite{GuoWu17}.  A key structural property of DNLS  discovered by Kaup and Newell \cite{KN78} is that it is integrable by inverse scattering: that is, there is a linear spectral problem with $u(x,t)$ as potential whose spectral data (consisting of a reflection coefficient, describing the continuous spectrum of 
the linear problem,  together with eigenvalues and norming constants, describing the discrete spectrum of the linear problem) evolve linearly under the flow. This linearizing transformation, together with an inverse defined via a Riemann-Hilbert problem defined by the spectral data, gives a method to integrate the equation explicitly. 

In recent works \cite{LPS15,LPS16} (referred to as Papers I and II), we used the inverse scattering tools to prove
global existence and  long-time behavior of solutions to DNLS
for initial conditions in the weighted Sobolev space $H^{2,2}(\R)$,  restricting  initial conditions to those that  do not support solitons.  We proved that the amplitude of the solution
 decays like the solution of the linear problem, namely $|t|^{-1/2}$ as $|t| \to \infty$, and the phase behaves like the phase of the free dynamics modified by  a logarithmic correction. These results are analogous to those of Deift and Zhou \cite{DZ03} for the defocussing NLS, and improve earlier results of Kitaev-Vartanian \cite{KV97} on DNLS. 
The asymptotic state is fully described in  terms of the scattering data associated to  the initial condition.  
Using a different approach, Pelinovsky and Shimabukuro  proved global existence  to the DNLS equation for initial conditions that do not support solitons in \cite{PS17}
and recently complemented  their study allowing a finite number of discrete eigenvalues \cite{SSP17}.

Here we will prove global well-posedness and solution resolution for the DNLS equation with  initial data in an open and dense subset of $H^{2,2}(\R)$ (excluding spectral singularities, a notion that we will define precisely  later). Soliton resolution refers to the property that the solution decomposes into the sum of a finite number of separated solitons and a radiative part as $|t| \to \infty$.  The solitons parameters 
are slightly modulated, due to the soliton-soliton and soliton-dispersion interactions. We fully describe the dispersive part which contains two components,  one coming from the continuous  spectrum and another one from the interaction of the discrete and continuous spectrum. This decomposition is a central feature
in nonlinear wave dynamics and has been the object of many theoretical and numerical studies. It has been established in many {\em{perturbative}} contexts,   that is when the initial condition is
close to a soliton or a multi-soliton. In non-perturbative cases,  this property was proved rigorously for KdV  \cite{ES83},  mKdV \cite{S86} and  for the focusing NLS equation
\cite{BJM16} using the inverse scattering approach.
The last result has been conjectured for a long time \cite{ZS72} but rigorously proved only recently.

The soliton resolution conjecture is at the heart of current studies in nonlinear waves and extends to solutions that blow up in finite time. 
 In the context of  non\newtext{-}integrable equations, Tao \cite{Tao08} considered the NLS equation with potential in high dimension ($d\ge 11$) and proved the existence of a global  attractor, assuming radial symmetry.
A recent work by Duykaerts, Jia, Kenig and Merle \cite{DJKM16}  concerns  the focusing energy critical wave equation for which they prove that any bounded solution  asymptotically behaves like  a finite sum of modulated solitons, a regular component in the finite time blow up case or a free radiation in the global case, plus a residue term that vanishes asymptotically in the energy space as time approaches the maximal time of existence  (see  also \cite{DKM13,DKM15} for other cases,   radial and non-radial, in various dimensions).

\subsection{Global well-posedness}

Equation \eqref{DNLS1} is gauge-equivalent to the equation
\begin{gather}
	\label{DNLS2}
	iq_t + q_{xx} + i \eps q^2 \bar q_x + \frac{1}{2} |q|^4 q = 0, \\
	\label{data}
	q(x,t=0) = q_0(x)
\end{gather}
via the gauge transformation
\begin{equation}
\label{G}
\calG (u)(x) = \exp\left(- i\eps \int_{-\infty}^x |u(y)|^2 \, dy \right) u(x).
\end{equation}
This nonlinear, invertible mapping is an isometry of $L^2(\R)$, maps soliton solutions to soliton solutions, and maps dense open sets to dense open sets in weighted Sobolev spaces. Because of this, global well-posedness for \eqref{DNLS2} on an open and dense set $U$ in $H^{2,2}(\R)$ containing data of arbitrary $L^2$-norm implies 
global well-posedness of \eqref{DNLS1} on a subset $\calG^{-1}(U)$ of $H^{2,2}(\R)$ with the same properties. For this reason, we will prove the global well-posedness result for \eqref{DNLS2}.

Our analysis exploits the discovery of Kaup and Newell  that \eqref{DNLS2} generates an isospectral flow for the 
linear spectral problem 
\begin{equation}
\label{LS}
\Psi_x =	 -i\zeta^2 \sigma_3 \Psi + \zeta Q \Psi + P \Psi
\end{equation}
where
$$ 
\sigma_3 = \diagmat{1}{-1}, 
\quad
Q(x) = 
\begin{pmatrix}
	0	&	q(x)	\\
	\eps \overline{q(x)}	& 	0
\end{pmatrix},
\quad
P(x) = \frac{i\eps}{2} 
\begin{pmatrix}
-|q(x)|^2 & 	0	\\
0	&	|q(x)|^2
\end{pmatrix},
$$
and the unknown $\psi$ is a $2 \times 2$ matrix-valued function of $x$. This spectral problem defines a map $\calR$ from $q \in H^{2,2}(\R)$ to spectral data that we will describe, and has an inverse $\calI$ defined by a Riemann-Hilbert problem which recovers the potential $q(x)$. Moreover, the spectral data for a solution $q(t) = q(x,t)$ of \eqref{DNLS2} obey a linear law of evolution. Thus the solution operator $\calM$ for the Cauchy problem \eqref{DNLS2} is given by
\begin{equation}
\label{sol.op}
\calM (q_0,t) = \left( \calI \circ\Phi_t \circ \calR \right) q_0 
\end{equation}
where $\Phi_t$ is the linear evolution on spectral data. To state our results we first describe the set $U$ and the maps $\calR$,  $\Phi_t$, and $\calI$ in greater detail.

The direct scattering map $\calR$ maps $q \in H^{2,2}(\R)$ into spectral data defined by special solutions of \eqref{LS}. To describe the data, first
note that, if $q=0$, all solutions of \eqref{LS} are matrix multiples of $e^{-i\zeta^2 x \sigma_3}$ and
therefore bounded provided $\zeta \in \Sigma$, where
$$ \Sigma=\left\{ \zeta \in \C: \imag \zeta^2 = 0\right\}. $$
It is natural to set $\Psi(x,\zeta) = M(x,\zeta) e^{-i\zeta^2 x \sigma_3}$ and look for
bounded solutions of 
\begin{equation}
\label{LS.m}
\frac{d}{dx} M(x,\zeta) = -i\zeta^2 \ad \sigma_3 (M)+ \zeta Q(x) M + P(x)M, \quad
\ad \sigma_3 (A) \coloneqq [\sigma_3,A].
\end{equation}
If $q \in L^1(\R) \cap L^2(\R)$, a perturbation argument shows that \eqref{LS.m}  admits
bounded solutions for  $\zeta \in \Sigma$. Indeed, there exist 
unique solutions $M^\pm(x,\zeta)$ of \eqref{LS.m} with $\lim_{x \rarr \infty} M^\pm(x,\zeta) = \begin{pmatrix} 1 & 0 \\ 0 & 1 \end{pmatrix}$. The functions 
$$\Psi^\pm(x,\zeta) = M^\pm(x,\zeta) e^{-i\zeta^2 x \sigma_3}$$ are the \emph{Jost solutions} for \eqref{LS}.

The Jost functions define data associated with the continuous spectrum of the problem \eqref{LS} in the following way.  If $\Psi$ solves \eqref{LS} then $\det \Psi(x,\zeta)$ is independent of $x$, and,  if $\Psi_1$ and $\Psi_2$ solve \eqref{LS} for given $\zeta$,
then $\Psi_2 = \Psi_1 A$ for a constant matrix $A$. 
Thus, for $\zeta \in \Sigma$, the Jost solutions obey 
\begin{equation}
\label{trans}
\Psi^+(x,\zeta) = \Psi^-(x,\zeta) T(\zeta), \quad
T(\zeta) = 
\begin{pmatrix}
a(\zeta)	&	\bb(\zeta)	\\
b(\zeta)	&	\ba(\zeta)
\end{pmatrix}.
\end{equation}
Since $\det \Psi^\pm = 1$ it follows that
\begin{equation}
\label{det.ab}
a(\zeta) \ba(\zeta) - b(\zeta) \bb(\zeta) = 1.
\end{equation}
The functions $a$, $\ba$, $b$ and $\bb$ obey the symmetries
\begin{equation}
\label{sym.ab}
a(-\zeta) 	= a(\zeta), 		\quad 	\ba(\zeta) 	= \overline{a(\zetabar)}, \quad
b(-\zeta)	= - b(\zeta) 	\quad	\bb(\zeta) 	= \eps \overline{b(\zetabar)}.
\end{equation}
as follows from \eqref{trans}, the unicity of $\Psi^\pm$,  and the fact that the maps
\begin{equation}
\label{maps.Psi}
\Psi(x,\zeta) \mapsto \sigma_3 \Psi(x,-\zeta) \sigma_3, \quad
\Psi(x,\zeta) \mapsto \sigma_\eps \overline{\Psi(x,\zetabar)} \sigma_\eps^{-1}
\end{equation}
preserve the solution space of \eqref{LS}, where
\begin{equation}
\label{sigeps.def}
\sigma_\eps = \sigeps.
\end{equation}
The functions $a$ and $\ba$ are bounded continuous functions
with
\begin{equation}
\label{asy.a}
\lim_{|\zeta| \rarr \infty} a(\zeta) = \lim_{|\zeta| \rarr \infty} \ba(\zeta) = 1.
\end{equation}
Write $\Psi^+ = (\Psi_1^+, \Psi_2^+)$ where $\Psi_1^+$ and $\Psi_2^+$ are column vectors,
and similarly for $\Psi^-$. Let 
$$ \Omega^\pm = \{z \in \C: \pm \imag z^2 > 0 \}.$$
The columns $\Psi_1^+$ and $\Psi_2^-$ admit analytic continuations to
$\Omega^+$, while the columns $\Psi_1^-$ and $\Psi_2^+$ admit analytic continuations
to $\Omega^-$.  From the formulas
\begin{equation}
\label{a.wronski}
 \ba(\zeta) = 
\begin{vmatrix}
\Psi_{11}^+(x,\zeta)	&	\Psi_{12}^-(x,\zeta)	\\
\Psi_{21}^+(x,\zeta)	&	\Psi_{22}^-(x,\zeta)
\end{vmatrix},
\quad
a(\zeta)=
\begin{vmatrix}
\Psi_{11}^-(x,\zeta)	&	\Psi_{12}^+(x,\zeta)	\\
\Psi_{21}^-(x,\zeta)	&	\Psi_{22}^+(x,\zeta)
\end{vmatrix}.
\end{equation}
(which are easy consequences of \eqref{trans}), it follows that the  functions $\ba$ and $a$ admit analytic continuations respectively to $\Omega^+$ and $\Omega^-$. Zeros of $\ba$ and $a$ correspond to soliton components in the initial data $q$.
More precisely, zeros of $\ba$ and $a$ on $\Sigma$ are spectral singularities and correspond to algebraic solitons, while zeros of $\ba$ in $\Omega^+$ and of $a$ in $\Omega^-$ correspond to bright solitons. 

Owing to \eqref{sym.ab},
the zeros of $\ba$ and $a$ come in ``quartets'' $(\zeta, - \zeta, \zetabar, -\zetabar)$. Thus if $\Omega^{++} = \{ z \in \Omega^+: \imag \zeta > 0\}$, the set 
$$\calZ^{++} = \{ \zeta \in \Omega^{++}: \ba(\zeta) = 0 \}$$ 
uniquely determines the zeros of $a$ and $\ba$ in $\Omega^+ \cup \Omega^-$. We denote
$$ \calZ = \medcup_{\, \zeta \in \calZ^{++}} \{ \zeta, -\zeta, \zetabar, -\zetabar\}.$$
If $\ba$ has a simple zero at $\zeta$,
the \emph{norming constant}  $c_\zeta$ is given by
\begin{equation}
\label{czeta}
c_\zeta = b_\zeta/\ba'(\zeta)
\end{equation} 
where 
\begin{equation}
\label{czeta.bdef}
 \begin{bmatrix}
	\Psi_{11}^-(x,\zeta)	\\
	\Psi_{21}^-(x,\zeta)
\end{bmatrix}
=
b_\zeta
\begin{bmatrix}
	\Psi_{12}^+(x,\zeta)	\\
	\Psi_{22}^+(x,\zeta)
\end{bmatrix}.
\end{equation}

\begin{definition}
\label{def:U}
We denote by $U$ the set of $q \in H^{2,2}(\R)$ so that $\ba$ has at most finitely many simple zeros in $\Omega^{++}$ and no zeros on $\Sigma$. We denote by $U_N$ the subset of $U$ consisting of potentials so that $\ba$ has exactly $N$ simple zeros. 
\end{definition}

The set $U$ is a disjoint union of open sets $U_N$ where $N=0,1,2,\ldots$ counts the number of zeros of $\ba$
in $\Omega^{++}$. 
The set $U_0$ contains a neighborhood of $0$ since, for the zero potential, $T(\zeta)$ is the identity matrix, and $T$
is a continuous function of $q \in L^1(\R) \cap L^2(\R)$. Excluded from $U$ are potentials with spectral singularities or infinitely many zeros of $\ba$. By mimicking a construction of Zhou (see \cite[Example 3.3.16]{Zhou89a}), it is easy to see that $H^{2,2}(\R)\setminus U$ contains potentials $q$ belonging to $\calS(\R)$.  The following subsets of $U_N$ will play an important role in describing continuity properties of the direct scattering map.

\begin{definition}
\label{def:U.bounded}
We will call a subset $U'$ of $U_N$ \emph{bounded} if
there is are strictly positive constants $c$ and  $C$ so that $\norm[H^{2,2}]{q} \leq C$ ,
$\imag \zeta_i^2 > c$ for all zeros $\zeta$ of $\ba$  and $\inf_{\zeta \in \Sigma} |\ba(\zeta)| \geq c$ for all $q \in U'$.
\end{definition}

We now describe the spectral data $\left(\rho, \{ \lam_j, C_j \}_{j=1}^N \right)$ associated to $q \in U_N$. By the symmetries \eqref{sym.ab}, the function
\begin{equation}
\label{ab->rho}
\rho(\lam) = \zeta^{-1}  \bb(\zeta)/a(\zeta), \quad \lam = \zeta^2
\end{equation}
is a well-defined function on $\R$, and obeys the nonlinear constraint
\begin{equation}
\label{rho.con}
1- \eps \lam |\rho(\lam)|^2 > 0.
\end{equation}
as follows from \eqref{det.ab} and the boundedness of $a$ and $\ba$.  We denote by $S$ the set of all $\rho \in H^{2,2}(\R)$ obeying
\eqref{rho.con}. 
The remaining data $\{ \lam_j, C_j \}$ are associated with eigenvalues of the problem \eqref{LS}, i.e., 
zeros of $\ba$. They are the image of $\{ \zeta_j, c_j \}$ under the quadratic transformation
\begin{equation}
\label{zeta->lambda}
\lam_j = \zeta_j^2, \quad C_j = 2c_j, \quad 1 \leq j \leq N.
\end{equation}
Motivated by these observations, we define sets $V$ and $V_N$, and the direct scattering map $\calR$,  as follows.

\begin{definition}
\label{def:V} We denote by $V$ the 
the disjoint union
$$ V = \medcup_{N=0}^\infty \, V_N$$ 
where $ V_0 = S $ and, for $N \geq 1$,
$$ V_N = S \times (\C^+ \times \C^\times)^N. $$
\end{definition}

\begin{definition}
\label{def:R}
The \emph{direct scattering map} is the map
\begin{align*}
 \calR :  U_N  	&\rarr   		S \times \left(  \C^+ \times \C^\times \right)^N\\
q					&\mapsto  \left( \rho, \{ \lam_j,  C_j \}_{j=1}^N \right) 
\end{align*}
for each $N$.  
\end{definition}

Let 
\begin{equation}
\label{distances}
d_\Lambda = \frac{1}{2} \inf_{\lam \neq \mu \in \Lambda} |\lam-\mu|.
\end{equation}
Note that, since $\Lambda$ is invariant under complex conjugation, $|\Im \lam_j| \geq d_\Lambda$ for each $j$.

\begin{definition}
\label{def:V.bounded}
We call a subset $V'$ of $V_N$ \emph{bounded} if there is a constant $C$ with
$$\norm[H^{2,2}]{\rho} + \sup_{1 \leq j \leq N} |C_j| + \sup_{1 \leq j \leq N} |\lam_j| \leq C$$ and a constant $c>0$ so that 
$d_\Lambda \geq c$
for all data $\left( \rho, \{ \lam_j, C_j \}_{j=1}^N\right)$ in $V'$. 
\end{definition}

Our first
result is:

\begin{theorem}
\label{thm:R}
The set $U$ is open and dense in $H^{2,2}(\R)$, contains a neighborhood of $0$, and contains $q$ with arbitrary $L^2$ norm. Moreover, the direct scattering map 
$$ \calR: U \rarr V $$
maps bounded subsets of $U_N$ into bounded subsets of $V_N$ for each $N$,  and is
uniformly Lipschitz continuous on bounded subsets of $U_N$. 
\end{theorem}

Theorem \ref{thm:R} is a direct consequence of Theorems \ref{thm:R.lip}, Theorem \ref{thm:generic}, and Proposition \ref{prop:empty} which asserts the existence of $q \in U_0$ with arbitrarily large $L^2$-norm.

\begin{remark}
One-soliton solutions corresponding to eigenvalue $\lam = |\lam| e^{i\phi} \in \C^+$
have $L^2$ norm $\sqrt{4(\pi - \phi)}$ if $\eps=1$ and $\sqrt{4\phi}$ if $\eps=-1$. In the limit $\phi \darr 0$ ($\eps=+1$) or $\phi \uarr \pi$ ($\eps=-1$), the pole approaches the real axis and the bright soliton becomes an algebraic soliton. As we will compute, $N$-soliton solutions separate into a train of one-soliton solutions so that an $N$-soliton solution may have $L^2$-norm in $(0,N\sqrt{4\pi})$. It follows that the set $U_N$ contains 
elements of $L^2$ norm arbitrarily close to $0$.
\end{remark}

It follows from the Lax representation for \eqref{DNLS2} that the spectral data for a solution
$q(x,t)$ of \eqref{DNLS2} obey the linear evolution
\begin{equation}
\label{sd.ev.zeta}
\begin{gathered}
\dot{a}(\zeta,t)=\dot{\ba}(\zeta)=0, \\
 \dot{b}(\zeta,t) = -4i\zeta^4 b(\zeta,t), \,\,\, \dot{\bb}(\zeta,t) = 4i  \zeta^4 \bb(\zeta,t), \quad \dot{c_j}(t) = -4i\zeta_j^4 c_j(t)
\end{gathered}
\end{equation}
so that
\begin{equation}
\label{sd.ev}
\dot{\rho}(\lam,t) = -4i \lam^2 \rho(\lam,t), \quad	\dot{\lam_j} = 0, \quad \dot{C_j} = -4i \lam_j^2 C_j.
\end{equation}

\begin{definition}
\label{flowmap}
For each nonnegative integer $N$, $t \in \R$, $\rho \in S$, and $\{ \lam_j, C_j \}_{j=1}^N \in \left(\C^+ \times \C^\times\right)^N$, the linear evolution $\Phi_t: V_N \rarr V_N$ is given by
$$ \Phi_t \left(\rho, \{ \lam_j, C_j \}_{j=1}^N \right) = \left( e^{-4i(\dotarg)^2 t}  \rho(\dotarg), \{ \lam_j, C_j e^{-4i\lam_j^2 t }\}_{j=1}^N \right). $$
\end{definition}

It is easy to see that the map $\Phi_t$ preserves $V_N$ for each $N$ and is jointly continuous in $t$ and the data $\left(\rho, \{ \lam_j, C_j \}_{j=1}^N \right)$.

We now describe  the inverse scattering map $\calI$
which recovers $q(x,t)$ from the time-evolved spectral data.  The inverse scattering map is defined by a Riemann-Hilbert problem which we will first describe in the $\zeta$ variables (Problem \ref{RHP1}) and then in the $\lam$ variables (Problem \ref{RHP2}). 

To describe the Riemann-Hilbert Problem in the $\zeta$ variables, we recall the \emph{Beals-Coifman solutions} of \eqref{LS.m}.  Denote by $M_1^\pm$ the first column of the normalized Jost solutions $M^\pm$, and similarly denote by $M_2^\pm$ the second column of $M^\pm$. It can be shown that the matrix-valued function
$$ M(x,z) = 	\begin{cases}
						\begin{bmatrix}
							M_1^+(x,z)	&	\dfrac{M_2^-(x,z)}{\ba(z)}
						 \end{bmatrix}
						 &	 z \in \Omega^+ \setminus \calZ\\
						 \\
						 \begin{bmatrix}
						 	\dfrac{M_{1}^-(x,z)}{a(z)}	&	M_2^+(x,z)
						 \end{bmatrix}
						 & z \in \Omega^- \setminus \calZ
					\end{cases}
$$
defines a meromorphic function from $\C \setminus \left(\Sigma \cup \calZ \right)$ to $SL(2,\C)$ with 
$M(x,z) \rarr \begin{pmatrix} 1 & 0 \\ 0 & 1 \end{pmatrix}$ as $|z| \rarr \infty$. As a function of $x$, (1) $M(x,z)$ solves \eqref{LS.m}, (2) $M(x,z) \rarr \begin{pmatrix} 1 & 0 \\ 0 & 1 \end{pmatrix}$ as $x \rarr \infty$ and (3) $M(x,z)$ is bounded as $x \rarr -\infty$. These three properties uniquely characterize $M(x,z)$. The function $M$ has continuous boundary values on $\Sigma$ and satisfies the following Riemann-Hilbert problem. To state it, we first recall that the 
exponential of the linear operator $\ad \sigma_3$ acts on $2 \times 2$ matrices as
$$ e^{it\ad\sigma_3} 
	\begin{pmatrix} 
		a& b \\ c & d 
	\end{pmatrix} =
	\begin{pmatrix}
		 a & e^{2it} b \\ 
		 e^{-2it}c & d 
	\end{pmatrix}
$$
and is an automorphism. Next, we define
\begin{align}
\label{r}
r(\zeta) 		&=	\bb(\zeta)/a(\zeta)\\
\label{r.to.br}
\br(\zeta)	&=	\eps \overline{r(\zetabar)}
\end{align}
and impose the conditions
\begin{equation}
\label{rr}
r(-\zeta) = -r(\zeta), \quad 1 - r\br \geq c > 0.
\end{equation}
Note that, under the evolution \eqref{sd.ev.zeta}, we have $\dot{r} =-4i\zeta^4 r$ and $\dot{\br} = 4i\zeta^4 \br$.
Finally, let $\varphi(x,t,\zeta)$ be phase function
\begin{equation}
\label{phase.zeta}
\Theta(x,t,\zeta) = -\left(\zeta^2 \frac{x}{t} + 2 \zeta^4 \right). 
\end{equation}

\begin{RHP}
\label{RHP1}
Given $x,t \in \R$, $r \in H^1(\Sigma)$ obeying \eqref{rr}, and $\{ \zeta_j, c_j \}_{j=1}^N \in (\Omega^{++} \times \C^\times )^N$, 
find a matrix-valued function $M(x,t,z) : \C \setminus \left(\Sigma \cup \calZ \right) \rarr SL(2,\C)$
with the following properties:
\begin{itemize}
\item[(i)]		$M(x,t,z) = \begin{pmatrix} 1 & 0 \\ 0 & 1 \end{pmatrix}+ \bigO{\dfrac{1}{z}}$ as $|z| \rarr \infty$
\item[(ii)]	$M(x,t,z)$ has continuous boundary values $M_\pm(x,t,\zeta)$  on $\Sigma$, taken
				as $z \rarr \zeta \in \Sigma$ from $\Omega^\pm$,
				which obey the jump relation
				$$ M_+(x,t,\zeta) = M_-(x,t,\zeta) e^{it\Theta \ad \sigma_3} v(\zeta), 
					\quad	
					v(\zeta) 	=	\begin{pmatrix}
											1-r(\zeta)\br(\zeta)
												&	r(\zeta)	\\[10pt]
											-\br(\zeta)								
												&	1	
									\end{pmatrix}
				$$
\item[(iii)]		$M(x,t,z)$ has simple poles at $\zeta \in \calZ$ with
				$$ \Res_{z = \zeta} M(x,t,z) = \lim_{z \rarr \zeta} M(x,t,z) e^{it\Theta \ad \sigma_3} v(\zeta)$$
				where, for $\zeta \in \calZ^{++}$,
				$$ v(\pm\zeta) = \lowmat{c_\zeta }   , \quad 
					v(\pm \zetabar) = \upmat{\overline{c_\zeta} }$$
\end{itemize}
\end{RHP}

It can be shown that $M(x,t,z)$ solves \eqref{LS.m} as a function of $x$. The reconstruction formula
\begin{equation}
\label{q.zeta}
q(x,t) = \lim_{z \rarr \infty} 2iz M_{12}(x,t,z) 
\end{equation}
is an easy consequence of the large-$z$ expansion for $M(x,t,z)$ and \eqref{LS.m}.

As in the direct problem, the symmetries of the scattering data allow for a reduction, this time to a Riemann-Hilbert problem with contour $\R$ which is the image of $\Sigma$ under the map $\zeta \mapsto \zeta^2$. Recall \eqref{ab->rho} and \eqref{zeta->lambda}, and denote by $\Lam^+$ (resp.\ $\Lam$) the image of $\calZ^{++}$ 
(resp.\ $\calZ$) under the map $\zeta \mapsto \zeta^2$. It follows from the definition that $\Lam = \Lam^+ \cup \overline{\Lam^+}$. Under
the change of variables, the phase function \eqref{phase.zeta} becomes
\begin{equation}
\label{phase.lambda}
\theta(x,t,\lam) = - \left(\lambda \frac{x}{t} + 2 \lambda^2\right).
\end{equation}

\begin{RHP}
\label{RHP2}
Given $x,t \in \R$, $\rho \in S$ and $\{ \lam_j, C_j \}_{j=1}^N$ in $(\C^{++} \times \C^\times)^N$, find a matrix-valued function $N(x,t,z) : \C \setminus (\R \cup \Lam) \rarr SL(2,\C)$ with the following properties:
\begin{itemize}
\item[(i)]		$N_{22}(x,t,z) = \overline{N_{11}(x,t,\zbar)}$, $N_{21}(x,t,z) = \eps \overline{N_{12}(x,t,\zbar)}$, 
\item[(ii)]	$N(x,t,z) = \begin{pmatrix}
							1	&	0	\\
							q*	&	1
						\end{pmatrix}
					+ \bigO{\dfrac{1}{z}}
		 $ as $|z| \rarr \infty$, 
\item[(iii)]	$N$ has continuous boundary values $N_\pm$ on $\R$ and
				$$ N_+(x,\lam) = N_-(x,\lam) e^{it\theta \ad \sigma_3} J(\lam),
				\quad 
				J(\lam) = \begin{pmatrix}
								1-\lam |\rho(\lam)|^2 	&	 \rho(\lam)	\\
								-\eps \lam \overline{\rho(\lam)}	&	1
							\end{pmatrix}
				$$
\item[(iv)] 	For each $\lam \in \Lam$, 
				$$ \Res_{z = \lam} N(x,t,z) = \lim_{z \rarr \lam} N(x,t,z) e^{it\theta \ad \sigma_3} J(\lam) $$
				where for each $\lam \in \Lam^+$
				$$ J(\lam)  = \lowmat{\lam C_\lam}, \quad J(\overline{\lam}) = \upmat{\overline{C_\lam}}. $$
\end{itemize}
\end{RHP}

\begin{remark}
\label{rem:RHP2}
Although the symmetry condition (i) uniquely determines $q^*$, it is equivalent, and more effective for analysis, to consider the 
Riemann-Hilbert problem 
for the row vector $n(x,t,z)=(N_{11}(x,t,z),N_{12}(x,t,z))$ obeying the asymptotic condition
$$ n(x,t,z) = (1,0) + \bigO{\frac{1}{z}},$$
the jump relations (ii) and the residue condition (iii).
\end{remark}

We may recover $q$ from the reconstruction formula (compare \eqref{q.zeta})
\begin{equation}
\label{q.lam}
q(x,t) = \lim_{z \rarr \infty} 2iz N_{12}(x,t,z).
\end{equation}

To define the inverse scattering map and state its mapping properties, we temporarily set $t=0$ in Problem \ref{RHP2}.
\begin{definition}
\label{def:I}
For each $N$, the \emph{inverse scattering map} is the map 
\begin{align*}
\calI: V_N 	&\rarr U_N		\\
\left(\rho, \left\{ \lam_j, C_j\right\}_{j=1}^N \right)	&	\rarr	q
\end{align*}
defined by Problem \ref{RHP2} (with $t=0$) and \eqref{q.lam}.
\end{definition}

Our next result is:

\begin{theorem}
\label{thm:I}
The inverse scattering map $\calI : V \rarr U$ takes bounded subsets of $V_N$ to bounded subsets of $U_N$ for each $N$, and is uniformly Lipschitz continuous on bounded subsets of $V_N$. Moreover, $\calR \circ \calI$ is the identity map on $V$ and $\calI \circ \calR$ is the identity map on $U$.
\end{theorem}

This result is proved as Theorem \ref{thm:Lip.I} in section \ref{sec:Lip}.

The composition $\calI \circ \Phi_t$ is given by RHP \ref{RHP2} and the reconstruction formula \eqref{q.lam}.
In his thesis, Lee proved that the right-hand side of \eqref{sol.op} is the correct solution operator
for \eqref{DNLS2} for $q_0 \in U \, \cap \, \calS(\R)$. As a consequence of Lee's result, Theorem \ref{thm:R}, and Theorem
\ref{thm:I}, we conclude:

\begin{theorem}
\label{thm:GWP}
Let $U$ be the set of $q_0$ given in Definition \ref{def:U}.
\begin{itemize}
\item[(i)]		The Cauchy problem for \eqref{DNLS2} has a unique global
				solution for initial data $q_0 \in U$, 
\item[(ii)]	the solution map $\calM: (q_0,t) \rarr q(x,t)$ is a continuous 
				map from $U \times [0,T]$ to $H^{2,2}(\R)$ for any $T>0$, and
\item[(iii)]	For any $T>0$, 
				$$
					\sup_{t \in [0,T]} \norm[H^{2,2}]{\calM(q_1,t) - \calM(q_2,t)} 
					\lesssim \norm[H^{2,2}]{q_1-q_2}
				$$
				with constant uniform in $q_1,q_2$ in a bounded subset of $U$ and $t \in [0,T]$.
\end{itemize}
\end{theorem}

Together with the local well-posedness result of Takaoka \cite{Takaoka99}, Theorem \ref{thm:GWP} establishes global well-posedness for the DNLS in an open and dense subset of $H^{2,2}(\R)$. 

\subsection{Soliton Resolution}

To deduce large-time asymptotics and soliton resolution for \eqref{DNLS1}, we will first obtain large-time asymptotics for \eqref{DNLS2} and then compute the asymptotics of the phase in \eqref{G} in terms of spectral data. These two results together will yield large-time asymptotics and soliton resolution for \eqref{DNLS1}.

First, we compute long-time asymptotic behavior and prove solution resolution for \eqref{DNLS2} with initial data $q_0 \in U$ using the Deift-Zhou method of steepest descent applied to Problem \ref{RHP2}. Note that the phase function $\theta$ in \eqref{phase.lambda}   has a single critical point at $\xi = -x/4t$.
From the evolution \eqref{sd.ev}, Problem \ref{RHP2}, and Remark \ref{rem:RHP2}, it suffices to solve the Riemann-Hilbert problem for the row vector-valued function $n(x,t,z) = (N_{11}(x,t,z),N_{12}(x,t,z))$. 

We will choose intervals $[v_1,v_2]$ of velocities and $[x_1,x_2]$ of initial positions and compute the asymptotic behavior of $q(x,t)$ in space-time regions of the form
\begin{equation}
\label{S.cone}
\calS(v_1,v_2,x_1,x_2)	=	
\left\{ (x,t) : x=x_0 + vt \text{ for } v \in [v_1,v_2], \, \, x_0 \in [x_1,x_2] \right\}.
\end{equation}
In what follows, we set
\begin{equation}
\label{LamI}
\Lambda(I) = \{ \lam \in \Lambda: \real (\lam) \in I \}, \quad N(I) = |\Lam(I)|. 
\end{equation}
Solitons in $\Lambda([-v_2/4, -v_1/4])$ should be `visible' in these asymptotics, but remaining solitons will 
move either too slowly or too fast to be seen in the moving window.

To state our result, we need the following notation. For $\xi \in \R$ and $\eta \in \{ -1 , +1 \}$  (here $\eta = \sgn t$), let
\begin{align}
I_{\xi,\eta}^-	&=	\left\{ \lam \in \C:		\imag \lam = 0, \quad -\infty < \eta \Re \lam \leq \eta \xi \right\}, \\
I_{\xi,\eta}^+	&=	\left\{ \lam \in \C:		\imag \lam = 0, \quad \eta \xi < \eta \Re \lam < \infty \right\}.
\end{align}  
For $x_1 \leq x_2$ and $v_1 \leq v_2$, let $\calS(v_1,v_2,x_1,x_2)$ be the subset of $\R^2$ given by

\begin{figure}[H]
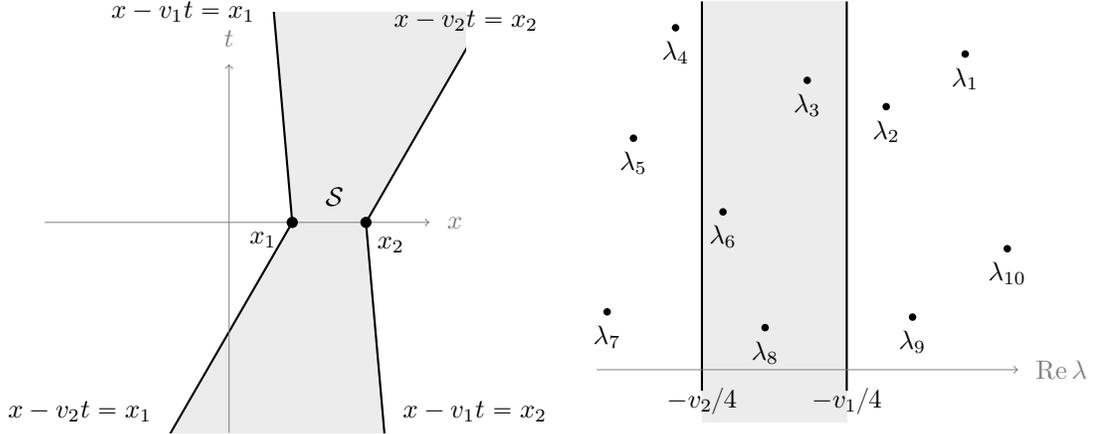

\solfig
\caption{
Given initial data $q_0(x)$ which generates scattering data 
$\left\{ \rho, \{ \lam_k, C_k \}_{k=1}^N \right\}$, then asymptotically as $|t| \to \infty$ inside the space-time cone $\mathcal{S}(v_1,v_2,x_1,x_2)$ (shaded on left) the solution $q(x,t)$ of \eqref{DNLS2} approaches an $N(I)$-soliton $\qsol(x,t)$ corresponding to the discrete spectra in $\poles(I)$ 
(shaded region on right) and connection coefficients $\widehat{C}_k$ which are modulated by the soliton-soliton and soliton-radiation interactions as described in Theorem~\ref{thm:long-time}.
\label{fig:light.cone}
}
\end{figure}
We will prove:

\begin{theorem}
\label{thm:long-time}
Suppose that $q_0 \in U_N$ with scattering data $\left( \rho, \{ \lam_k,C_k \}_{k=1}^N \right)$. Fix $x_1, x_2, v_1, v_2 \in \R$ with $v_1 < v_2$, and let $$I = [-v_2/4,-v_1/4].$$
Denote by  $\qsol(x,t;\mathcal{D}_I)$ the soliton solution of \eqref{DNLS2} with modified discrete scattering data
$$ \left\{ (\lam_k, \widehat{C_k}): \lam_k \in \Lambda(I) \right\}$$
where
$$
\widehat{C_k} =  C_k \, 
	\prod_{\Re \lam_j \in I_{\xi,\eta}^- \setminus I}
	\left(	\frac{\lam_k - \lam_j}{\lam_k - \overline{\lam_j}} \right)^2 
	\exp \left(  
					\frac{i}{\pi} \int_{I_{\xi,\eta}^-} 
						\frac{\log\left( 1 - \eps \lam |\rho(\lam)|^2 \right)}{\lam-\lam_k}  \, d\lam
			\right)
$$
Then, as $|t| \rarr \infty$ in the cone $\calS(v_1,v_2,x_1,x_2)$, we have
\begin{equation}
\label{q.long-time}
q(x,t) \underset{t \rarr \infty}{\sim}  \qsol(x,t;\mathcal{D}_I) + t^{-1/2} f(x,t) + \bigO{t^{-3/4}}. 
\end{equation}
Here $f(x,t)$ is given by
$$
f(x,t)=  2^{-1/2} 
	\left[ 
		A_{12}(\xi,\eta) \calN_{11}^{\mathrm{sol}}(\xi;x,t)^2 + 
		\eps \xi \overline{A_{12}(\xi,\eta)} \calN_{12}^{\mathrm{sol}}(\xi;x,t)^2 
	\right]
$$
where the constants $A_{12}(\xi,\eta)$ are given by \eqref{Axi} and $\calN^{\mathrm{sol}}$ solves Problem \ref{outermodel}.
\end{theorem}

Equation \eqref{q.long-time} expresses soliton resolution in the following sense. First, as described below, the function $q_\sol(x,t)$ is generically asymptotic to a superposition of one-soliton solutions.
Second, the term at order $t^{-1/2}$ represents a dispersive contribution; in the no-soliton case, i.e., if $v_1,v_2$ are chosen in Theorem~\ref{thm:long-time} such that $N(I) = 0$, 
$\Nsol \equiv I$, and $\qsol \equiv 0$, so using \eqref{Axi} the asymptotic solution reduces to
\begin{equation}
	\begin{gathered}
	   q(x,t) = \frac{1}{\sqrt{2|t|}} \frac{\kappa(\xi)}{ \xi} e^{i\alpha_\pm(\xi) } 
	   e^{i x^2/(4t) \mp \kappa(\xi) \log|8t| }
	   + \bigo{t^{-3/4}}, \quad t \to \pm \infty \\
	   \alpha_\pm(\xi) = \frac{\pi}{4} - \arg(-\eps \xi \overline{\rho(\xi)})
	   \pm \arg \Gamma(i\kappa(\xi)) 
	   \pm \frac{1}{\pi} \int_{\mp \infty}^\xi \log|\xi - \lam| 
	      \mathrm{d}_\lam \log(1-\eps \lam |\rho(\lam)|^2).
	\end{gathered}
\end{equation}
which agrees with the dispersive asymptotics determined, for example, in \cite{LPS16}. 

Previously, Kitaev-Vartanian computed similar asymptotics in the 
no-soliton sector \cite{KV97} and the finite-soliton sector \cite{KV99}, making more stringent assumptions on regularity together with a smallness assumption on the reflection coefficient that we do not require.

Theorem~\ref{thm:long-time} implies, as a special case, the asymptotic separation of the solution $q(x,t)$ into a sum of one-solitons whenever the $\lam_k \in \poles^+$ have distinct real parts. 
If the initial data $q_0$ generates scattering data $\{ \rho, \{\lam_k, C_k\}_{k=1}^N \}$ with $\lam_k = \eta_k + i \tau_k$ then
applying Theorem~\ref{thm:long-time} repeatedly to sets $\mathcal{S}_k$ each of which contains a single soliton speed $v=-4\eta_k$ one finds that the solution of \eqref{DNLS2}-\eqref{data} satisfies 
\begin{equation}\label{sol.res.generic}
	q(x,t) = \sum_{k=1}^N 
		q_{\textrm{sol}}(x,t; \lam_k, x_k^\pm, \alpha_k^\pm) + \bigo{|t|^{-1/2}}
		\qquad
		t \to \pm \infty	
\end{equation}		
where, setting $u = \real \lambda$,
\begin{multline}
	q_{\textrm{sol}}(x,t;\lam,x_0,\alpha_0) =  \\
	\varphi(x - x_0 +4 u t , \lam)
		\exp i \left\{ 4|\lam|^2 t - 2u(x + 4u t) 
		  - \frac{\eps}{4} \int_{-\infty}^{x- x_0 + 4 u t} \varphi(\eta)^2 d\eta
		  - \alpha_0 \right\}
\end{multline}
is the form of a general one-soliton solution of \eqref{DNLS2}. 
Here $\alpha_0$, $x_0$, and $\varphi$ are given as in \eqref{1sol}.
The asymptotic phases
are 
\begin{gather}
	x^\pm_k = 
		\frac{1}{4\tau_k} \log \left| \frac{\lam_k C_k^2}{4 \tau_k^2} \right|
		+ \frac{1}{2\tau_k} 
		\sum_{\substack{\lam_j \in \poles^+ \\ \mathclap{\pm (\eta_k - \eta_j) > 0}}} 
		  \log \left| \frac{ \lam_k - \lam_j}{\lam_k - \overline{ \lam_j} } \right|
		  \pm \frac{1}{2\pi} \int_{\eta_k}^{\mp\infty}
		    \frac{ \log(1-\eps s | \rho(s)|^2) }{(s-\eta_k)^2+\tau_k^2} ds \\
	\alpha^\pm_k = 
		\arg \lp i \lam_k C_k \rp 
		+ \sum_{\substack{\lam_j \in \poles^+ \\ \mathclap{\pm (\eta_k - \eta_j) > 0}}}
		    \arg \lp \frac{ \lam_k - \lam_j}{\lam_k - \overline{ \lam_j} } \rp
		\mp \frac{1}{\pi} \int_{\eta_k}^{\mp \infty}    
		    \frac{ (s - \eta_k) \log(1-\eps s | \rho(s)|^2) }{(s-\eta_k)^2+\tau_k^2} ds
		 \mod{2\pi}
\end{gather}
so that the total phase shifts of the $k$th soliton, as it interacts both with the other solitons and the radiation component, are
\begin{align}
	x_k^+ - x_k^- &= 
		\frac{1}{2\tau_k} 
		\sum_{\mathclap{j \neq k }} \sgn(\eta_k - \eta_j) 
		  \log \left| \frac{ \lam_k - \lam_j}{\lam_k - \overline{ \lam_j} } \right|
		  + \frac{1}{2\pi} \int_{-\infty}^{\infty}
		    \frac{ \sgn(s-\eta_k) \log(1-\eps s | \rho(s)|^2) }{(s-\eta_k)^2+\tau_k^2} ds \\
	\alpha_k^+ - \alpha_k^- &= 
		\sum_{\mathclap{j \neq k}}
		    \sgn(\eta_k - \eta_j) 
		    \arg \lp \frac{ \lam_k - \lam_j}{\lam_k - \overline{ \lam_j} } \rp
		- \frac{1}{\pi} \int_{-\infty}^{\infty}    
		    \frac{ |s - \eta_k| \log(1-\eps s | \rho(s)|^2) }{(s-\eta_k)^2+\tau_k^2} ds
	 \mod{2\pi}	    
\end{align} 

In the non-generic case in which two or more $\lam_j \in \poles^+$ have the same real part one still observes a form of soliton resolution akin to \eqref{sol.res.generic}. In this case, the one-solitons in \eqref{sol.res.generic} corresponding to spectral values with the same real part coalesce to form higher-order solitons called breathers; these breathers are localized traveling waves whose amplitudes exhibit quasi-periodic oscillations in the frame of the traveling wave.

To obtain a similar asymptotic formula for \eqref{DNLS1}, we use the gauge transformation \eqref{G} to write 
$$
u(x,t) = 
	\left[ \calG^{-1}	
			\circ
				\calM
			\circ\calG
	\right]	(u_0)  (x).
$$ 
where $\calM$ is given by \eqref{sol.op}.
We obtain an asymptotic formula for $u(x,t)$ in terms of spectral data for $q_0 = \calG (u_0)$ in Proposition \ref{prop:u.gauge.expansion}, which plays a key result and introduces some complications in the asymptotic formulas for small $\xi$.   Recall Definition \ref{def:U} and the fact that $\calG$ maps dense open subsets to dense open sets in $H^{2,2}(\R)$,
and recall the space-time region \eqref{S.cone}. If $u_0 \in \calG^{-1}(U)$, then $q_0 = \calG(u_0)$ has 
no spectral singularities and the scattering coefficient $\ba$ for $q_0$ has at most finitely many zeros. 

\begin{theorem}
\label{thm:long-time-gauge}
Suppose that $u_0 \in \calG^{-1}(U)$,  let $q_0 = \calG u_0$, and let $\calR(q_0) = \left\{ \rho, \{ \lam_k,C_k \}_{k=1}^N \right\}$. Fix $v_1,v_2,x_1,x_2$ as in Theorem \ref{thm:long-time}, let $I=[-v_2/4,-v_1/4]$,  let $\xi = -x/(4t)$, and let $\eta = \pm 1$ for $\pm t > 0$. 
Let $u_\sol(x,t) = \calG^{-1} \left( q_\sol(\cdot,t) \right)(x)$. 
Fix $M > 0$.
The solution $u(x,t)$ of \eqref{DNLS1} has the following asymptotics as $|t| \rarr \infty$ in the cone
$\calS(v_1,v_2,x_1,x_2)$.
\begin{enumerate}
\item[(i)]	 For $|\xi| \geq M t^{-1/8}$,
$$ 
u(x,t) = \left[ 
					u_\sol(x,t;\calD(I)) + t^{-1/2} g(x,t) + \bigO{t^{-3/4}} 
			\right]
			e^{i\alpha_0(\xi,\eta)} 
$$
\item[(ii)]	For $|\xi| \leq M t^{-1/8}$,
$$ 
u(x,t) = F(\xi,t,\eta) \left[ u_\sol(x,t;\calD(I)) + t^{-1/2}\widetilde{g}(x,t) + \bigO{t^{-3/4}} \right] e^{i\alpha_0(\xi,\eta)} $$
\end{enumerate}
In the above formulas, 
\begin{align*}
\alpha_0(\xi,\eta)	&= 	-4 \sum_{\real \lam_k \in I^+_{\xi,\eta} \setminus I} \arg \lam_k
	- \frac{1}{\pi} \int_{I_{\xi,\eta}^+} \frac{\log\left(1-\eps \lam|\rho(\lam)|^2 \right)}{\lam} \, d \lam \\
g(x,t)		&=		\left[	
								A_{12}(\xi,\eta)
								\calN_{11}^{\sol}(\xi;x,t)^2 + 
								\eps \xi \overline{\calN^{\sol}(\xi;x,t)}^2  \right.	\\
			&\quad	\left.
								+ 2\eps q_\sol(x,t;\calD(I)) \, 
									\real
										\left(
												A_{12}(\xi,\eta) \calN_{11}^\sol(\xi;x,t) 
												\overline{\calN_{12}^\sol(\xi;x,t)}
										\right)
						\right]  \\
			&\quad	\times	e^{i\eps \int_{-\infty}^x |q_\sol(y,t;\calD(I))|^2 \, dy}\\
\widetilde{g}(x,t) &= g(x,t)  \\ 
	&\quad + (1-G(\xi,t,\sgnt)) \overline{A_{12}(\xi,\sgnt)} \qsol(y,t;\mathcal{D(I)}) 
	\exp \lp \ \  4i \sum\limits_{\mathclap{\Re \lam_k \in I \cup \negint}} \arg \lam_k \rp 
	\int_x^\infty \usol(y,t;\mathcal{D}(I)) dy ,
\intertext{where, setting $p=e^{i\pi/4} |8t\xi^2|^{1/2}$,}
F(\xi,t,\eta)		&=	\left[ e^{p^2/4} p^{-i\eta \kappa(\xi)} D_{i\eta \kappa(\xi)}(p) \right]^{-2}\\
G(\xi,t,\eta)	&=	\frac{p D_{i\eta\kappa(\xi)-1}(p)}{D_{i\eta\kappa(\xi)}(p)}.
\end{align*}
\end{theorem}

\begin{remark}
Note that in Theorem \ref{thm:long-time}, we have used  that the $N$-soliton solution    $\qsol(x,t;\mathcal{D}_\xi)$ and  the reduced $N(I)$- soliton solution
$\qsol(x,t;\mathcal{D}_I)$ with respective scattering data 
 $\mathcal{D}_\xi = \{(\lam_k,\widetilde{C}_k) \}_{k=1}^N$
 and 
$\mathcal{D}_I = \{(\lam_k, \widehat{C}_k ), \   \widehat{C}_k = \widetilde{C}_k   \smashoperator[l]{ \prod_{
	\substack{\lam_j \in \poles \setminus \poles(I) \\ \Re \lam_j \in  \negint}
	}}
	\lp \frac{ \lam_k - \lam_j}{ \lam_k - \overline{\lam_j} } \rp^2 ) \,:\, \lam_k \in \poles(I) \}$
	are exponentially close at $|t|\to \infty$ (see Proposition \ref{prop:sol separation}). However, this is not the case for their respective gauge factor for which
\[
Ê Ê Ê\exp \Big (i\eps \int_{-\infty}^xÊ |\qsol(x,t; \mathcal{D}_\xi)|^2 dxÊ \Big)Ê =
Ê Ê Ê\exp \Big( i\eps \int_{-\infty}^xÊ |\qsol(x,t; \mathcal{D}_I)|^2 dxÊ \Big)   \exp \big( -4i \sum_{\Re \lam_k \in I^- \setminus I} \arg \lam_k 
  \big).
\]
This explains the presence of the first term in the right-hand side of the formula for $\alpha_0$ in Theorem \ref{thm:long-time-gauge}.	
\end{remark}

We close this introduction by sketching the contents of this paper. After a review of the Beals-Coifman approach to Riemann-Hilbert problems in section \ref{sec:prelim}, we consider the direct scattering map in section \ref{sec:direct}. Section \ref{sec:direct} should be read in concert with section 3 of \cite{LPS16}; the main 
new results of this section are the proof that the set $U$ from Definition \ref{def:U} is open and dense (see Theorem \ref{thm:generic} and Propositions \ref{prop:direct.generic} and Proposition \ref{prop:direct.open}) and that, restricted to the set $U$, the maps from $q$ to the discrete scattering data are Lipschitz continuous (Proposition \ref{prop:direct.discrete.lip}. Together with the results from section 3 of \cite{LPS16}, these results imply Theorem \ref{thm:R}. All of these results are based on analysis of the Volterra integral equations for the normalized Jost solutions $N^\pm$ in the $\lam$ variables (see \eqref{direct.n1}--\eqref{direct.n2}).

Section \ref{sec:inverse} contains the proof of Theorem \ref{thm:I}. We first prove that there exists a unique solution to Problem \ref{RHP1} (section \ref{sec:RHP1}) using a uniqueness theorem of Zhou \cite{Zhou89} together with the Beals-Coifman theory. Next, we use the results of section \ref{sec:RHP1} and a change of variables to prove existence and uniqueness of solutions to Problem \ref{RHP2} in 
section \ref{sec:RHP2}. Finally, in section \ref{sec:Lip},  we study the Beals-Coifman 
integral equations for Problem \ref{RHP2} to prove Lipschitz continuity of the inverse map, 
proving Theorem \ref{thm:I}.

We give the proof of Theorem  \ref{thm:long-time} in sections \ref{sec:deform},  \ref{sec:models}, and \ref{sec:largetime}, using the steepest descent method of Deift and Zhou \cite{DZ93}, the later approach of Dieng-McLaughlin \cite{DM08}, and the very recent work of Borghese, Jenkins and McLaughlin \cite{BJM16} on the focusing cubic NLS which shows how to treat a problem with discrete as well as continuous spectral data. Following \cite{BJM16} we reduce Problem \ref{RHP2} to an `outer' model which describes the asymptotic behavior of solitons, and in `inner' model which computes the contributions due to the interactions of solitons and radiation. 

In section \ref{sec:deform} we begin with the row-vector RHP (see Remark \ref{RHP2}) and deform to an RHP on a new contour $\Sigma^{(2)}$ whose jump matrices approach the identity exponentially fast away from the critical point $\xi$. In section \ref{sec:conj},
we  conjugate the row-vector RHP for $n=(n_1,n_2)$  to a form suitable for lensing, i.e., deforming the contour $\R$ about the critical point $\xi$ so that the jump matrix of the new Riemann-Hilbert problem approaches the identity exponentially fast away from the critical point $\xi$. The solution of the conjugated RHP, Problem \ref{rhp.n1}, is  the unknown $n^{(1)}(x,t,z)=n(x,t,z) \delta^{-\sigma_3}$ where $\delta$ solves a scalar RHP with contour $(-\infty,\xi)$ and has asymptotics $1 + \bigO{1/z}$ as $|z| \rarr \infty$.  The jump matrix for Problem \ref{rhp.n1} has a `good factorization for lensing. In order to deform the contour, we extend the scattering data in the jump matrix of Problem \ref{rhp.n1} into the complex plane to define a new unknown $n^{(2)}$, solving a mixed $\dbar$-Riemann-Hilbert problem, Problem \ref{rhp.n2}, described in section \ref{sec:extensions}. The extension introduces non-analyticity of $n^{(2)}$ which is solved away at a later step. The unknown $n^{(2)}$ coincides with $n^{(1)}$ in the sectors $\Omega_2$ and $\Omega_5$ (see figure \ref{fig:n2def}) and is piecewise analytic in the sectors $\Omega_1$--$\Omega_6$ with no jumps across the real axis. 

In section \ref{sec:models} we construct a solution $\calN^{\mathrm{RHP}}$ 
of the Riemann-Hilbert problem, Problem \ref{rhp.Nrhp}, corresponding to Problem \ref{rhp.n2}. Thus the function $n^{(3)} = n^{(2)}\left(\calN^{\mathrm{RHP}}\right)^{-1}$ obeys a pure $\dbar$-problem, Problem \ref{dbar.n3}. Because the original RHP contains both `continuous' and `discrete' data, the solution $\calN^{\mathrm{RHP}}$ consists of
an `outer' model for the soliton components (see section \ref{sec:outer model}, Problem \ref{outermodel} and Proposition \ref{outer.soliton}) and an `inner' model for the stationary phase point (see section \ref{sec:local model},  Problem \ref{rhp.localmodel} and Proposition \ref{prop:PCmodel.est}).  The outer and inner models are used to build a parametrix for 
Problem \ref{rhp.Nrhp} in section \ref{sec:RHP.exist}. The `gluing' of parametrices is carried out by solving a small-norm Riemann-Hilbert problem, Problem \ref{rhp.E}.  The $\dbar$ problem for $n^{(3)}$ is solved in section \ref{sec:dbar} and is shown to have asymptotic behavior $$n(z) = \begin{pmatrix} 1 & 0 \\ 0 & 1 \end{pmatrix}+ \bigO{t^{-3/4}},$$ 
Thus, (suppressing the $(x,t)$ dependence for brevity)
$$ n(z) = n^{(3)}(z) \mathcal{N}^{\mathrm{RHP}}(z) \mathcal{R}^{(2)}(z)^{-1} \delta(z)^{\sigma_3}.$$
The leading contribution to $q(x,t)$ in \eqref{q.lam} comes from the explicitly computable model factor $\calN^{\mathrm{RHP}}$ owing to the asymptotics of $n^{(3)}$, the fact that $\mathcal{R}^{(2)}$ is the identity in sectors $\Omega_2$ and $\Omega_5$ (and we can take $z \rarr \infty$ in either sector), and the diagonal matrix $\delta^{-\sigma_3}$ does not change $n_2$ at order $\bigO{1/z}$.

With the solution of Problem \ref{RHP2} in hand, we give the proof of Theorems \ref{thm:long-time} and 
\ref{thm:long-time-gauge} in sections \ref{subsec:sol-q} and \ref{subsec:sol-u}, respectively.  To prove Theorem \ref{thm:long-time-gauge}, we work out an asymptotic formula for the phase factor in the gauge transformation \eqref{G} in terms of spectral data, Proposition \ref{prop:u.gauge.expansion}. This in 
turn relies on a weak Plancherel formula, Lemma \ref{lem.Plancherel}. 

Appendix \ref{app:BC} collects the Beals-Coifman integral equations for Riemann-Hilbert problems \ref{RHP1} and \ref{RHP2}. We give the jump matrices and residue relations for the `left' normalized version of Problem \ref{RHP2} in Appendix \ref{app:left}. We construct $N$-soliton solutions $q_\sol$ for \eqref{DNLS2} in Appendix \ref{app:solitons}. Finally, in Appendix \ref{app:empty}, we construct initial data for \eqref{DNLS2} with arbitrarily large $L^2$ norm and empty discrete spectrum, showing that \eqref{DNLS2} (and hence, by gauge transformation, \eqref{DNLS1}) admits global solutions with arbitrarily large $L^2$ norm.

\subsection*{Acknowledgements}

PAP and JL would like to thank the University of Toronto for hospitality during part of the time that this work was done. Conversely, RJ and CS would like to thank the department of Mathematics at the University of Kentucky for their hospitality during  part of the time that this work was done.  The authors are grateful to Peter Miller for numerous helpful discussions and particularly for his suggestion to use the ideas of Tovbis and Venakides to construct soliton-free initial data of large $L^2$-norm.


\subsection*{Notation Index}

$~$
\vskip  0.5cm

\subsubsection*{General}

\begin{itemize}

\item[$\Omega^\pm$]	 	Set of $z \in \C$ with $\pm \imag z^2  > 0$
\item[$\Omega^{++}$]		First quadrant in $\C$

\item[$\calF$]		Fourier transform $\left(\calF f\right)(\lam) = \dfrac{1}{\pi}\dint_{-\infty}^\infty e^{-2i\lam x} f(x) \, dx$
\smallskip

\item[$\calF^{-1}$]	Inverse Fourier transform 
						$\left(\calF^{-1}g\right)(x) = \dfrac{1}{\pi} \dint_{-\infty}^\infty e^{2i\lam x} g(\lam) \, d\lam$
						
\smallskip

\item[$C$]						Cauchy integral over a contour 
\item[$C^\pm$]				Cauchy projectors 
						
\smallskip

\item[$\sigma_2$]			Pauli matrix 
									$\begin{pmatrix} 0 & 1 \\ 1 & 0 \end{pmatrix}$
									
\smallskip
\item[$\sigma_3$]			Pauli matrix
									$\begin{pmatrix} 1 & 0 \\ 0 & -1 \end{pmatrix}$

\smallskip

\item[$\sigma_\eps$]  		The matrix 
									$\begin{pmatrix} 0 & 1 \\ \eps & 0 \end{pmatrix}$
\end{itemize}

\subsubsection*{Maps and Spaces}

\begin{itemize}
\item[$\calG$]				Gauge transformation \eqref{G} linking \eqref{DNLS1} and \eqref{DNLS2}
\item[$\calR$]					Direct scattering map (Definition \ref{def:R})
\item[$\calI$]					Inverse scattering map (Definition \ref{def:I})
\item[$\Phi_t$]				Flow on scattering data defined by \eqref{sd.ev}
\item[$\calM$]				Solution operator $\calI \circ \Phi_t \circ \calR$ for \eqref{DNLS2}
\item[$U$]						Spectrally determined, open and dense subset of $H^{2,2}
									(\R)$	(Definition \ref{def:U}) and domain of $\calR$
\item[$V$]						Image of $U$ under $\calR$ (Definition \ref{def:V}) and 
									domain of $\calI$
\item[$U_N$]					$N$-soliton sector of $U$
\item[$V_N$]					Image of $U_N$ under $\calR$
\item[$S$]						Set of $\rho \in H^{2,2}(\R)$ with 
									$1- \eps \lam|\rho(\lam)|^2 > 0$
\item[$X$]						Function space (see \eqref{X})
\item[$X^\sharp$]			Function space (see \eqref{Xsharp})
\item[$Y$]						Set of $\rho \in L^{2,5/4}(\R)$ with $\rho' \in L^{2,3/4}(\R)$
									and $1- \eps \lam|\rho(\lam)|^2 > 0$  (See
									\eqref{Y})
\end{itemize}

\subsubsection*{Scattering Solutions,  Scattering Data}

\begin{itemize}
			
\item[$\Psi^\pm$]			Jost solutions of \eqref{LS}	with
									$\lim_{x \rarr \pm \infty} \Psi^\pm(x,\zeta) e^{i\zeta^2 x \sigma_3} = \begin{pmatrix} 1 & 0 \\ 0 & 1 \end{pmatrix}$
\item[$M^\pm$]				Normalized Jost solutions $M(x,z) = \Psi^\pm(x,\zeta) e^{i\zeta^2 x \sigma_3}$
\item[$a,b,\ba,\bb$]			Scattering coefficients in $\zeta$ variables, \eqref{trans} 
\item[$r.\br$]					Reflection coefficient in $\zeta$ variables, 
									\eqref{r}--\eqref{r.to.br}
\item[$\alpha,\beta$]		Scattering coefficients in $\lam$ variables, 
									\eqref{alpha.beta.def}
\item[$\rho$]					Scattering data in $\lam$ variable, \eqref{ab->rho}
\item[$\zeta_j, c_j$] 		Zero of $\ba$ and associated norming constant 
									(see  \eqref{czeta})
\item[$\lam_j, C_j$] 		Zero of $\balpha$ and associated norming constant 
									(see \eqref{zeta->lambda}	
\item[$\calZ$]					Set of zeros of $a$ and $\ba$
\item[$\calZ^{++}$]			$\calZ \cap \Omega^{++}$
\item[$\Lambda$]			Image of $\calZ$ under  $z \mapsto z^2$ and set of zeros of 
									$\alpha$ and $\balpha$
\item[$\Lambda^{++}$] 	Image of $\calZ^{++}$ under $z \mapsto z^2$

\end{itemize}

\subsubsection*{Beals-Coifman Integral Equations}
\begin{itemize}
\item[$\calC_w$]			Beals-Coifman integral operator (see \eqref{BC.Cw}) associated to RHP with factorized jump matrix (see \eqref{RHP.model.v})
\item[$\mu$]					Matrix-valued solution of Beals-Coifman 
									integral equation for  Problem \ref{RHP1}; see
									\eqref{RHP1.11.AI}--\eqref{RHP1.12.AI.disc}
\item[$\nu$]					Row vector-valued solution of 
									Beals-Coifman integral equation
									for Problem \ref{RHP2}; see 
									\eqref{RHP2.11}--\eqref{RHP2.12.disc}
\item[$\nu^\flat$]				Solution of iterated Beals-Coifman equation for 
									Problem \ref{RHP2} (see \eqref{nuflat} and \eqref{nuflat.int})
									belonging to $X$
\item[$\nu^\sharp$]			$\nu^\flat - \mathbf{e}$, element of $X^\sharp$ obeying \eqref{nusharp.int}
\item[${\nu}_0,\nu^*$]		See \eqref{nuflat.to.nu0}
\end{itemize}

\subsubsection*{Riemann-Hilbert Problems}

\begin{itemize}
\item[$\Sigma$]				Oriented contour $\imag z^2 = 0$ that divides $\C$ 
									into $\Omega^+$ and $\Omega^-$
\item[$\Sigma^{(2)}$]		Oriented contour $\real z^2 = 0$ (see \eqref{Sigma_k})

\item[$M_\pm$]				Boundary values of Beals-Coifman solutions 
									of \eqref{LS.m} on $\Sigma$ as 
									$\pm \imag \zeta^2 \darr 0$
\item[$\Theta$]  			Phase function \eqref{phase.zeta} ($\zeta$ variables)
\item[$\theta$]				Phase function \eqref{phase.lambda} ($\lam$ variables)
\item[$\xi$]						Unique critical point of $\theta$, $\xi=-x/4t$
\item[$\calN^{\mathrm{sol}}$]
									`Outer model' (Problem \ref{outermodel}) which
									describes solitons (see Proposition \ref{outer.soliton})
\item[$\calN^{\mathrm{PC}}$]
									`Inner model' (Problem \ref{rhp.localmodel}) that
									describes radiation (see Proposition 
									\ref{prop:PCmodel.est})
\item[$\calN^{\mathrm{RHP}}$] 
									Solution to Problem \ref{rhp.Nrhp} (see \eqref{error def})
\item[$\calE$]					Gluing error for gluing $\calN^{\mathrm{sol}}$ and
									$\calN^{\mathrm{PC}}$ to form $\calN^{\mathrm{RHP}}$ (see Problem \ref{rhp.E},
									Lemma \ref{lem:Err}, and Proposition \ref{prop:E})
\item[$n$]						Row vector-valued solution to Problem \ref{RHP2}
\item[$n^{(1}$, $n^{(2)}$, $n^{(3)}$]  
									Respective solutions to 
									Riemann-Hilbert Problem  \ref{rhp.n1}, 
									$\dbar-$Riemann-Hilbert Problem \ref{rhp.n2}, 
									and $\dbar$-problem \ref{dbar.n3}
\end{itemize}

\newpage
%
%

\section{Preliminaries}
\label{sec:prelim}
 
We recall several important results of the Beals-Coifman theory of Riemann-Hilbert problems. References include the paper of Beals-Coifman \cite{BC84}, the paper of Zhou \cite{Zhou89} and the recent monograph on Riemann-Hilbert problems by Trogdon and Olver \cite{TO16}.

Let  $\Gamma$ be a \emph{complete} contour in $\C$, i.e., an oriented contour that divides $\C\setminus \Gamma$ into two regions $\Omega^+$ and $\Omega^-$ so that  $\Omega^+ \cap \Omega^- = \emptyset$,   $\Omega^+$ lies to the left of $\Gamma$, and $\Omega^-$ lies to the right of $\Gamma$. We allow for the possibility of self-intersections (which occur in Problems \ref{RHP1} and \ref{RHP1c} below).

The Cauchy transform of a function $f \in L^2(\Gamma)$ is given by 
$$ (Cf)(z) = \frac{1}{2\pi i} 
	\int_\Gamma 
		\frac{f(s)}{s-z} \, ds.
$$
We will sometimes write $C_\Gamma$ for $C$ to emphasize the dependence on the contour $\Gamma$.
If $f \in H^1(\Gamma)$ then $(Cf)(z)$ has H\"{o}lder 
continuous boundary values $C^\pm f$ on $\Gamma$ where
$$ (C^\pm f)(\zeta) = \lim_{\substack{z \rarr \zeta, \\ z \in \Omega^\pm}} (Cf)(z). $$
The operators $C^\pm$ are bounded operators on $L^2(\Gamma)$ with 
\begin{equation}
\label{Cpm}C^+ - C^- = I,
\end{equation} 
where $I$ denotes the identity operator.  For contours with finite self-intersections, see Trogdon-Olver \S 2.5 for a proof of the $L^2$-boundedness result under hypotheses that include our Problems \ref{RHP1} and \ref{RHP1c}.

It is worth noting that, if $\Gamma$ is the union of two disjoint contours $\Gamma_1$ and $\Gamma_2$ and $h \in H^1(\Gamma)$, then
\begin{equation}
\label{Cpm.disjoint}
\left(C^\pm h\right)(\zeta)=
	\begin{cases}
		\left(C^\pm_{\Gamma_1} h \right)(\zeta)
			+ \left(C_{\Gamma_2} h\right)(\zeta),
			&	\zeta \in \Gamma_1, \\
		\\
		\left(C^\pm_{\Gamma_2} h\right)(\zeta)
			+ \left(C_{\Gamma_1} h\right)(\zeta),
			&	\zeta \in \Gamma_2.
	\end{cases}
\end{equation}

We consider the following Riemann-Hilbert problem for a matrix-valued function $M(z)$ on $\C \setminus \Gamma$. We write $g_\pm \in \dee C(L^2)$ if $(g_+,g_-)$ are a pair of functions in $L^2(\Gamma)$ with $g_\pm = C^\pm h$ for a fixed function $h \in L^2(\Gamma)$.  For simplicity of notation, we omit the parametric dependence of the data $v$ and the solution $M$ on $x,t \in \R$ that occurs in Problems \ref{RHP1} and \ref{RHP2}.

\begin{RHP}
\label{RHP.model}
Given a $2\times 2$ matrix-valued function $v$ on $\Gamma$ with 
$v, v^{-1} \in L^\infty(\Gamma)$, 
find a function
$$ M(\dotarg): \C \setminus \Gamma \rarr SL(2,\C) $$
with $M(\dotarg)$ which has continuous boundary values 
$$ M_\pm(\zeta) = \lim_{ \substack{ z \rarr \zeta, \\ z \in \Omega^\pm}}
	M(z) $$
for $\zeta \in \Gamma$ satisfying the jump relation $M_+(\zeta) = M_-(\zeta) v(\zeta)$, and 
$$M_\pm(\dotarg) - I \in \dee C(L^2). $$
\end{RHP}

We'll refer to such a Riemann-Hilbert problem as \emph{normalized} owing to the asymptotic condition $M^\pm(\dotarg)-I \in \dee C(L^2)$. 

Given a normalized Riemann-Hilbert problem on $\Gamma$ with jump matrix 
\begin{equation}
\label{RHP.model.v}
v=(I-w_-)^{-1} (I+w_+),
\end{equation}
$w_\pm \in L^\infty(\Gamma) \cap L^2(\Gamma)$, we may introduce a new matrix-valued function $\mu$ defined in terms of the putative boundary values $M_\pm$ by
\begin{equation}
\label{BC.mu}
\mu = M_+(I+w_+)^{-1} = M_-(I-w_-)^{-1}, \quad \mu - I \in L^2(\Gamma)
\end{equation}
It is easy to see that
$$M_+ - M_- = \mu(w^-+w^+)$$
so that 
\begin{equation}
\label{mu-to-M}
M(z) = I + \int_\Gamma \frac{\mu(s)(w^-(s) + w^+(s))}{s-z} \, \dfrac{ds}{2\pi i}.
\end{equation}
Thus, in order to find $M(z)$ on $\C \setminus \Gamma$, it suffices to compute $\mu$.
Taking the boundary value as $\imag(z) \rarr 0$ in $\Omega^+$
we recover that
$$M_+(\zeta) = I + C^+
		\left[ \mu(\dotarg)
			\left(w^-(\dotarg) + w^+(\dotarg)\right)
		\right](\zeta)
$$
A short computation using  \eqref{Cpm} shows that $\mu$ solves the \emph{Beals-Coifman integral equation}
\begin{equation}
\label{BC.int}
\mu = I + \calC_w \mu 
\end{equation}
where 
\begin{equation}
\label{BC.Cw}
\calC_w h = C^+(hw_-) + C^-(hw_+).
\end{equation}

A basic result of the Beals-Coifman theory is:

\begin{theorem}
\label{thm:BC}
There exists a unique solution $M$ to the  Riemann-Hilbert problem \ref{RHP.model} if and only if there exists a unique solution $\mu$ with $\mu-I \in L^2(\Gamma)$ to the  Beals-Coifman integral equation
\eqref{BC.int}.
\end{theorem}

For proof and discussion see for example Zhou \cite[Proposition 3.3]{Zhou89} and \cite[Section 2.7]{TO16}.

A sufficient condition for the  operator $\calC_w$ to be Fredholm is that $w^\pm \in H^1(\Gamma)$ (see Trogdon and Olver \cite[Theorem 2.59]{TO16}). The merit of this result is that existence and uniqueness for Problem \ref{RHP.model} is then equivalent to a standard problem in the Fredholm theory of integral equations. 

A second important result of the theory concerns problems with residues such as Problems \ref{RHP1} and \ref{RHP2}.  Associated to the residues are discrete data $\{ \zeta_j, c_j\}_{j=1}^{2N}$, $\zeta_j \in \C \setminus \Gamma$ and $c_j \in \C \setminus {0}$,  which describe the location of residues and the leading term of the Laurent series. We assume a symmetry $\zeta_{n+j} = \overline{\zeta_j}$ and $c_{j+n} = \overline{c_j}$, and we restrict to the case of simple poles, which will suffice for our application. We denote by $\calZ$ the finite set $\{ \zeta_j \}_{j=1}^{2N}$. 

\begin{RHP}
\label{RHP.model.disc}
Find an analytic function 
$$M(\dotarg):\C \setminus \left(\Gamma \cup \calZ\right) \rarr SL(2,\C)$$ 
so that:
\begin{itemize}
\item[(i)]		$M(\dotarg)$ has continuous boundary values 
$$ M_\pm(\zeta) = \lim_{z \rarr \zeta, z \in \Omega^\pm}
	M(z) $$
for $\zeta \in \Gamma$ satisfying the jump relation $M_+(\zeta) = M_-(\zeta) v(\zeta)$ and 
$$M_\pm(\dotarg) - I \in \dee C(L^2). $$
\item[(ii)]	For each $\zeta_j \in \calZ$,  
$$ 
\Res_{z=\zeta_j} M(z) = \lim_{z \rarr \zeta_j} M(z) v_j
$$
where
$$ v_j = 	\begin{pmatrix}
					0		&	0	\\
					c_j	&	0
				\end{pmatrix},
	\quad 1  \leq j \leq N, 
	\qquad
	v_j =		\begin{pmatrix}
					0	& c_j	\\
					0	&	0
				\end{pmatrix},
	\quad N+1 \leq j \leq 2N.
$$
\end{itemize}
\end{RHP}

The residue condition means that for $1 \leq j \leq N$, the first column of 
$M(z)$ has a pole given by $c_j$ times the value of the second column, while, for $N+1 \leq j \leq 2N$, the second column of $M(z)$ has a residue given by $c_j$ times the value of the first column. 

Problem \ref{RHP.model.disc} is equivalent to a Riemann-Hilbert problem with no discrete data but having an augmented contour 
$$ \Gamma' = \Gamma \cup \{ \gamma_j \}_{j=1}^{2N} $$
where each $\gamma_j$ is a simple closed curve in $\C \setminus \Gamma$ surrounding $\zeta_j$ and no other element of $\calZ$. The new curves are given an orientation consistent with the orientation of the original contour $\Gamma$ and so $\Gamma'$ also divides $\C \setminus \Gamma'$ into 
two disjoint sets ${\Omega'}^+$ and ${\Omega'}^-$.  The additional jump matrices are constructed from the discrete data.  In what follows, $D_j$ is the interior of $\gamma_j$ (see Figures \ref{fig:RHP1.c} and \ref{fig:RHP2}), and $C^\pm$ are the Cauchy projectors for the contour $\Gamma'$. As before, we say that $g_\pm \in \dee C (L^2)$ if
$g_\pm = C^\pm h$ for a fixed function $h \in L^2(\Gamma')$. 

\begin{RHP}
\label{RHP.model.aug}
Find an analytic function 
$$M: \C \setminus {\Gamma}': \rarr SL(2,\C)$$ 
with continuous  boundary values 
$$M_\pm(\zeta) = \lim_{\substack{z \rarr \zeta, \\z \in {\Omega'}^\pm}} M(z),$$
so that $M_\pm(\dotarg) - I \in \dee C_{{\Gamma}'}(L^2)$ and
$ M_+(\zeta) = M_-(\zeta) v'(\zeta) $
for $\zeta \in {\Gamma}'$, where
$$
v'(\zeta) = 	\begin{cases}
					v(\zeta)	,	&	\zeta \in \Gamma,\\
					\\
					\begin{pmatrix}
						1										&	0	\\
						\dfrac{c_i}{\zeta - \zeta_j}	&	1
					\end{pmatrix}, 
					& \zeta \in \gamma_j,  \\
					\\
					\begin{pmatrix}
						1	&	\dfrac{c_{j+n}}{\zeta-\zeta_{j+n}}	\\
						0	&	1
					\end{pmatrix},
					& \zeta \in \gamma_{j+n}
				\end{cases}
$$
\end{RHP}

\begin{theorem}
\label{thm:RHP.model.disc}
Problems \ref{RHP.model.disc} and \ref{RHP.model.aug} are equivalent, i.e., there exists a unique solution to Problem \ref{RHP.model.disc} if and only if there exists a unique solution to Problem \ref{RHP.model.aug}. Their respective solutions $M$ and $M'$ are related by
$$ M(z) = \begin{cases}
					M'(z), 	&	z \notin \medcup_{j=1}^{2N} D_j\\
					\\
					M'(z) \left(I + \dfrac{V_j}{z-\zeta_j}\right), & z \in D_j
				\end{cases}
$$
\end{theorem}

For a proof see \cite[Section 6]{Zhou89}. 				
%
%

\section{The Direct Scattering Map}
\label{sec:direct}

This section is devoted to the Lipschitz continuity of the scattering data $(\rho, \{ \lambda_k,C_k\}_{k=1}^N )$
with respect to the potential $q$.  The Lipschitz continuity of the reflection coefficient $\rho$ in terms of $q$ is given in Proposition \ref{prop:direct.scatt.lip}.
We refer the reader to section 3 of Paper I or to the thesis of the second author \cite{Liu17}.
We also prove that the coefficients $\alpha$ and $\balpha$ are analytic in the lower  (resp. upper) complex plane, and the location of their zeros in a compact set of $\C$, (Propositions \ref{lemma:direct.a} and \ref{lemma:direct.a.infty}).

The new element of this study is the Lipschitz continuity of the discrete scattering data in terms of $q$  presented in Section \ref{sec:direct.lip.discrete}.
For this purpose, we make the notion of generic potential precise and prove  in Section \ref{sec:direct.generic}  that the set of potentials
 supporting at most finitely many solitons and having no spectral singularities is open and dense in $H^{2,2}(\R)$.

\subsection{Lipschitz Continuity of Scattering Data}
\label{sec:direct.scattering}

To study the Jost solutions it is convenient to set
\begin{equation}
\label{Jost.zeta.norm}
	M^\pm(x,\zeta) = \Psi^\pm(x,\zeta) e^{i x \zeta^2  \sigma_3},
	\qquad
	\lim_{x \rarr \pm \infty} M^\pm(x,\zeta) = \twomat{1}{0}{0}{1}. 
\end{equation}
We recall from Paper I that the off-diagonal components of $m$ are odd functions of $\zeta$ and the on-diagonal components are even in $\zeta$. Because of this symmetry,
the change of variables
$$N^\pm(x,\zeta^2) 
= \Twomat{m_{11}(x,\zeta)}{ \zeta^{-1} m_{12}(x,\zeta)}{\zeta m_{21}(x,\zeta) }{m_{22}(x,\zeta)}$$
yields well-defined functions $N^\pm(x,\lam)$ obeying the differential equation 
\begin{equation}
\label{direct.n.de}
\frac{d}{dx} N^\pm = -i\lam \ad(\sigma_3) { N^\pm} 
	+ \left( \twomat{0}{q(x)}{\eps \lam \overline{q(x)}}{0}+
								\frac{i\eps}{2}\twomat{-|q(x)|^2}{0}{0}{|q(x)|^2} \right) N^\pm,
\end{equation}
the respective asymptotic conditions
\begin{equation}
\label{direct.n.ac}
\lim_{x \rarr \pm \infty} N^\pm(x,\lam) = \twomat{1}{0}{0}{1},
\end{equation}
and the relation 
\begin{equation}
\label{direct.n.jc}
N^+(x,\lam) = N^-(x,\lam) 
		e^{-i\lam x \ad(\sigma_3)} T(\lambda) 
\end{equation}
{where $T(\lambda)$ is the transition matrix 
\begin{equation} \label{transition-lambda}
T(\lambda) =\begin{pmatrix}
     \alpha(\lam)  & \beta (\lam) \\
      \lam {\bbeta(\lam)}  & {\balpha(\lam)}   
     \end{pmatrix}
\end{equation}					
}
and  
\begin{equation}
\label{alpha.beta.def}
\alpha(\zeta^2) = a(\zeta), \quad \beta(\zeta^2) = \zeta^{-1}\bb(\zeta)
\end{equation} 
are well-defined functions of
$\lam = \zeta^2$ owing to the symmetries \eqref{sym.ab}.
It also follows from \eqref{sym.ab} that 
$$ \breve{\alpha}(\lam) = \overline{\alpha(\lam)}, \;\;  \breve{\beta}(\lam) = \eps \overline{\beta(\lam)}. $$

From \eqref{direct.n.jc} at $x=0$ and the fact that $\det N^\pm = 1$ we derive the Wronskian formulae
\begin{align}
\label{alpha}
\alpha(\lam)	&=	N_{11}^+(0,\lam) \overline{N_{11}^-(0,\lam)} -\eps \lam^{-1} N_{21}^+(0,\lam)\overline{N_{21}^-(0,\lam)} ,\\[10pt]
\label{beta}
\beta(\lam)		&=	\frac{\eps}{\lam}\left(\overline{ N_{11}^-(0,\lam)} \overline{N_{21}^+(0,\lam)} - 
\overline{N_{11}^+(0,\lam)} \overline{N_{21}^-(0,\lam)}\right).
\end{align}
which reduce the analysis of $\alpha$ and $\beta$ to the study of the normalized Jost functions $N^\pm$.

The analysis of $\alpha(\lam)$ and $\beta(\lam)$ presented in Paper I does not depend on any spectral assumptions other than a strictly positive lower bound on $|\alpha(\lam)|$, i.e., the absence of spectral singularities. The continuity result below follows from the analysis presented there.

\begin{lemma}
The following maps are locally Lipschitz continuous from $H^{2,2}(\R )$ into $L^2(\R )$.
$$
q \mapsto \alpha(\dotarg), 		\quad 
q \mapsto \beta(\dotarg),		\quad
q \mapsto \lw \dotarg \rw \beta'(\dotarg), \quad
q \mapsto \beta''(\dotarg), \quad
q \mapsto \lw \dotarg \rw^{-1} \alpha''(\dotarg).
$$
\end{lemma}

We omit the proof.

Let $\rho(\lam) = \beta(\lam)/\alpha(\lam)$. As a consequence of the quotient rule, we have:

\begin{proposition}
\label{prop:direct.scatt.lip}
For any $c>0$, the map $q \rarr \rho$ is Lipschitz continuous from the open subset 
$$U = 	\left\{ 
				q \in H^{2,2}(\R ): 
				\inf_{\lam \in \R }|\alpha(\lam)|> c
			\right\}
$$
of $H^{2,2}(\R )$ into $H^{2,2}(\R )$. 
\end{proposition}

In Proposition \ref{lemma:direct.a},   we prove that $\alpha$ extends to an analytic function on $\C  ^-$. 
It follows from  the relation
\begin{equation}
\label{direct:alpha.breve}
\balpha(\lam) = \overline{\alpha(\lambdabar)}.
\end{equation}
that $\balpha$ extends to an analytic function on $\C^+$.
From the equation
\eqref{direct.n.de} and the condition \eqref{direct.n.ac}, 
it is easy to see that $N_{11}^+(x,\lam)$ and $N_{21}^+(x,\lam)$ satisfy the Volterra integral equations
\begin{align}
\label{direct.n1}
N_{11}^+(x,\lam)	
	&=	1	-	 \int_x^\infty 
								q(y) N_{21}^+(y)  
					\, dy
				-	\int_x^\infty
							\frac{-i\eps}{2} |q(y)|^2 N_{11}^+(y,\lam) 
					\, dy \\
\label{direct.n2}
N_{21}^+(x,\lam)
	&=	-\int_x^\infty 
				e^{2 i\lam(x-y)}   
						\eps \lam \overline{q(y)} N_{11}^+(y,\lam) 
				\, dy
		{ -} \frac{i\eps}{2} \int_x^\infty 
				 e^{2i\lambda(x-y)} |q(y)|^2 N_{21}^+(y,\lam)
			 \, dy.			
\end{align}
Integrating by parts to remove the $\lam$ in \eqref{direct.n2} and iterating the resulting equations leads to the following system of integral equations (see Paper I, equation (3.4)):
\begin{align}
\label{direct:n11.lam}
N_{11}^+(x,\lam)	&=	 1 + \frac{\eps}{2i} \int_x^\infty q(y) \int_y^\infty e^{-2i\lam(z-y)}
q^\sharp(z) N_{11}^+(z,\lam) \, dz \, dy \\
\label{direct:n21.lam}
N_{21}^+(x,\lam) &=	-\frac{1}{2i} \eps \overline{q(x)} N_{11}^+(x,\lam) 
	-\frac{\eps}{2i} \int_x^\infty e^{-2i\lam(y-x)} q^\sharp(y) N_{11}^+(y,\lam) \, dy
\end{align}
where 
$$
q^\sharp(x) = \overline{q'(x)} - \frac{i}{2} \eps|q(x)|^2\overline{q(x)}
$$
\begin{lemma}
\label{lemma:direct.n}
There exist unique solutions of \eqref{direct:n11.lam}--\eqref{direct:n21.lam} with
$$ \sup_{\imag \lam \leq 0} \left|N_{11}^+(x,\lam) \right|
\leq\exp\left(\frac{1}{2}\|q\|_{L^1} \|q^\sharp\|_{L^1}  \right)$$
and
$$ 
\sup_{\imag \lam \leq 0} \left| N_{21}^+(x,\lam) \right| \leq 
\exp\left(\frac{1}{2}\|q\|_{L^1}\|q^\sharp \|_{L^1}\right)\left(\|q\|_{L^1}+\|q^\sharp\|_{L^1} \right).
$$
Moreover,
\begin{multline}
\label{direct.n11.c}
\left| N_{11}^+(x,\lam;q_1) - N_{11}^+(x,\lam;q_2) \right|	\\
\leq
\exp\left[
		 	C
				 \left( 	
				 	\|q_1\|_{L^1} \|q_1^\sharp\|_{L^1} + 
				 	\|q_2\|_{L^1} \|q_2^\sharp\|_{L^1}
		 		\right)
		\right] 
		\left(  
				\|q_1-q_2\|_{L^1}    \|q_1^\sharp \|_{L^1} 
+ \|q_2\|_{L^1} \|q_1^\sharp-q_2^\sharp\|_{L^1} 
		\right)
\end{multline}
where $C$ is independent of $\lam$ with 
$\imag(\lam) \leq 0$ and $x \geq 0$. 
\end{lemma}

\begin{proof}
Equation \eqref{direct:n11.lam} is of Volterra type
{  
\begin{equation}\label{volt}
  N_{11}^+ = 1 + S N_{11}^+.
 \end{equation}
   Using that}
 $\left| e^{-2i\lam(x-y)} \right| \leq 1$ for $\imag \lam \leq 0$ and $x \leq y$,
the estimate
$$ \left| (Sh)(x) \right| \leq \gamma(x) \sup_{y>x} \left| h(y) \right| $$
holds  with
$$
\gamma(x) = \int_x^\infty |q(y)| \int_y^\infty |q^\sharp(z)| \, dz \, dy.
$$
From  the standard theory of Volterra integral equations, {  we have that }
\begin{equation}
\label{direct.S.res}
\norm[C(\R ^+) \rarr C(\R ^+)]{(I-S)^{-1}} 
	\leq \exp \left(\frac{1}{2} \|q\|_{L^1} \|q^\sharp\|_{L^1} \right),
\end{equation}
uniformly in $\lam$ with $\imag \lam \leq 0$. Since 
$\|q\|_{L^1}$ and $\|q^\sharp\|_{L^1}$ are controlled by 
$\norm[H^{2,2}]{q}$, it follows from this fact and \eqref{direct:n21.lam} that 
 $\norm[C(\R ^+)]{N_{11}^+}$ and  $\norm[C(\R ^+)]{N_{21}^+}$  
 are bounded uniformly in  $\lam$ with  $\imag \lam \leq 0$
 and $q$ in a bounded subset of $H^{2,2}(\R )$.  Finally,  \eqref{direct.n11.c} follows from the resolvent estimate \eqref{direct.S.res} and the second resolvent formula.
\end{proof}

Similar estimates are obtained for $N_{11}^-$ and $N_{21}^-$ for $x \in \R ^-$.
From these estimates and the Wronskian formula \eqref{alpha}, we  conclude:

\begin{proposition}
\label{lemma:direct.a}
The function $\alpha$ is analytic on $\C  ^-$ and {  satisfies}
$$
\left| \alpha(\lam;q_1) - \alpha(\lam;q_2) \right| 
\leq \exp
				C
				\left( 
					\norm[H^{2,2}]{q_1}^2 + \norm[H^{2,2}]{q_2}^2
				\right)
			\left( \norm[H^{2,2}]{q_1} + \norm[H^{2,2}]{q_2} \right)
			\left( \norm[H^{2,2}]{q_1 - q_2}\right).
$$
where the constants are uniform in $\lam$ with $\imag \lam \leq 0$.
\end{proposition}

It {  is also } important  that  {  the } zeros of $\alpha$  lie in $\C  ^- \cap \left\{ |z| \leq R \right\}$ where $R > 0$ depends only on  $\norm[H^{2,2}(\R )]{q}$. This is the object of the next proposition.

\begin{proposition}
\label{lemma:direct.a.infty}
{  The function $\alpha$ satisfies} 
\begin{equation} \label{a.infty}
\lim_{R \rarr \infty} \sup_{|\lam| \geq R, \imag \lam \leq 0} |\alpha(\lam)-1| = 0
\end{equation}
where the convergence is uniform in $q$ in a bounded subset of $H^{2,2}(\R )$.
\end{proposition}

\begin{proof}
From the Wronskian formula \eqref{alpha}  {  for $\alpha$} 
and the uniform bounds on $N_{21}^-$ and $N_{21}^+$, estimate \eqref{a.infty} will follow from 
\begin{equation}
\label{n11.infty}
\lim_{R \rarr \infty} \sup_{|\lam| \geq R, \imag \lam \leq 0} |N_{11}^\pm(0,\lam)-1| = 0.
\end{equation}
We {   now  sketch the proof of  \eqref{n11.infty} for $N_{11}^+$; the proof for $N_{11}^-$ is similar.}

From \eqref{direct:n11.lam} and an integration by parts we see that
{
\begin{align}
\label{n11-1}
N_{11}^+(x,\lam) -1 	
	&=	\dfrac{\eps}{2i}\int_x^\infty q(y) \int_y^\infty e^{-2i\lam(z-y)} q^\sharp(z) \, dz \, dy \\
\nonumber
	&\quad +
		\frac{\eps}{4\lam}\int_x^\infty q(y) \left[ G_1(y,\lam) + G_2(y,\lam)+ G_3(y,\lam)\right] \, dy,
\end{align}
where
\begin{align*}
G_1(x,\lam)	&=	-q^\sharp(x) \left(N_{11}^+(x,\lam)-1 \right)	\\
G_2(x,\lam)	&=	-\int_x^\infty e^{-2i\lam(y-x)} \left(q^\sharp\right)'(y) \left(N_{11}^+(y,\lam)-1 \right)\, dy\\
G_3(x,\lam)	&=	-\int_x^\infty e^{-2i\lam(y-x)} q^\sharp(y) \frac{\dee N_{11}^+}{\dee x}(y,\lam) \, dy .
\end{align*}
}

Reversing the orders of integration in the first right-hand term of \eqref{n11-1} and integrating by parts
we may estimate
$$ \left| \int_x^\infty q(y) \int_y^\infty e^{-2i\lam(z-y)} q^\sharp(z) \, dz \, dy \right|
\leq \frac{1}{|\lam|} \|q^\sharp\|_{L^1}  \left( \|q\|_{L^\infty} + \|q'\|_{L^1}  \right). $$
From Lemma \ref{lemma:direct.n} we have $|G_1(x,\lam)| \lesssim 1$ where the implied constants
depend only on $\|q\|_{L^1}$ and $\|q^\sharp\|_{L^1}$. Differentiating \eqref{direct:n11.lam} to compute $\dee N_{11}^+/\dee x$ we may similarly estimate $|G_3(x,\lam)|$. To estimate $G_2(x,\lam)$, we need to show that $\norm[L^2(\R ^+)]{N_{11}^+(\dotarg,\lam) - 1}$ is bounded uniformly in $\lam$ with $\imag \lam \leq 0$ and $q$ in a bounded subset of $H^{2,2}(\R )$. This is carried out in Lemma \ref{lemma:direct.n11-1} below.
\end{proof}

To prove the $L^2$ estimate on $N_{11}^+(x,\lam) -1$, we return to the integral equation 
\eqref{direct:n11.lam} and note that the operator $S$ 
is a Hilbert-Schmidt operator on $L^2(\R ^+)$ uniformly in $\lam$ for
$\imag \lam \leq 0$.  
{  Indeed}  its integral kernel is given by
\begin{equation}
\label{S.Ker}
K(x,z) = 	\begin{cases}
						\dint_{\hspace{-1mm}x}^z q(y) e^{-2i\lam(z-y)} q^\sharp(z) \ dy, 	& x < z\\
						\\
						0,			 													& x > z
					\end{cases}
\end{equation}
{  with}
$$ \norm[L^2(\R ^+ \times \R ^+)]{K} \leq \norm[L^{2,1/2}]{q^\sharp} \|q\|_{L^1} .$$
One {  checks  }  that
$$ \ker_{L^2(\R ^+)} (I - S) \subset \ker_{C(\R ^+)}(I-S)  = \{ 0 \} $$
where the last equality follows from the existence of the resolvent $(I-S)^{-1}$ on
$C(\R ^+)$. Writing $S = S(\lam)$ to display the dependence of the operator $S$ on $\lam$, 
we can show that 
\begin{equation}
\label{direct:S.small}
\lim_{|\lam| \rarr \infty} \norm[\mathrm{HS}]{S(\lam)} = 0 
\end{equation}
uniformly in $\lam$ with $\imag \lam \leq 0$ and $q$ in a bounded subset of $H^{2,2}(\R )$. This follows from the integration by parts
$$ \int_x^z q(y) e^{-2i\lam(z-y)} \, dy = \frac{1}{2i\lam} \left[ q(z) - q(x) + \int_x^z e^{-2i\lam(z-y)} q(y) \, dy \right]$$
and a straightforward estimate of the Schmidt norm using \eqref{S.Ker}. Writing $K=K(\lam,q)$, 
we may also estimate
$$ 
\norm[L^2(\R ^+ \times \R ^+)]{K(\lam, q_1)-K(\lam,q_2)} 
	\leq \norm[L^{2,1/2}]{q_1-q_2} \|q_1\|_{L^1} + \norm[L^{2,1/2}]{q_2}\|q_1-q_2\|_{L^1}
$$
uniformly in $\lam$ with $\imag \lam \leq 0$. On the other hand, it follows from the Dominated Convergence Theorem that $\norm[L^2(\R ^+ \times \R ^+)]{K(\lam_1,q)- K(\lam_2,q)} \rarr 0$ as $\lam_1 \rarr \lam_2$ for any fixed $q \in L^1(\R) \cap L^{2,1/2}(\R)$.   Writing $S=S(\lam,q)$, we now use a `continuity-compactness argument'  
as well as  \eqref{direct:S.small} to prove:

\begin{lemma}
\label{lemma:direct.S.res}
The resolvent $(I-S(\lam,q))^{-1}$ exists as a bounded operator on $L^2(\R ^+)$
and for any $M>0$,
$$\sup_{\imag \lam \leq 0, \, \norm[H^{2,2}]{q} \leq M} 
	\norm[L^2(\R ^+) \rarr L^2(\R ^+)]{(I-S(\lam,q))^{-1}} < \infty.
$$
\end{lemma}

\begin{proof}
For any $M>0$, $R>0$, the identity map takes the set
$$ 
\left\{ \lam \in \C  : \imag \lam \leq 0, |\lam| \leq R \right\} \times 
\left\{ q \in H^{2,2}(\R ): \norm[H^{2,2}]{q} \leq M \right\}
$$
into a subset of $\C   \times (L^{2,1/2} \cap L^1)$ with compact closure.
By the second resolvent formula, the map $(\lam,q) \mapsto (I-S(\lam,q))^{-1}$ is continuous into the bounded operators on $L^2(\R ^+)$. It follows by compactness and continuity that the set
$$ \left\{ (I-S(\lam,q))^{-1}: \imag \lam \leq 0, |\lam| \leq R, \norm[H^{2,2}]{q} \leq M \right\} $$
is compact in $\calB(L^2(\R ^+))$, hence bounded. On the other hand, for sufficiently large $R$ depending on $M$, we have from \eqref{direct:S.small}
that $\sup_{|\lam| \geq R} \norm[\calB(L^2(\R ^+))]{(I-S(\lam,q))^{-1}} \leq 2$
for any $q$ with $\norm[H^{2,2}]{q} \leq M$. 
\end{proof}

We can now prove:

\begin{lemma}
\label{lemma:direct.n11-1}
If $q \in H^{2,2}(\R )$, the estimate
$$\norm[L^2(\R ^+)]{N_{11}^+(\dotarg,\lam) - 1} \lesssim 1$$
holds.
\end{lemma}

\begin{proof}
The function $\eta = N_{11}^+ -1$ obeys the integral equation
$ \eta = S(1) + S(\eta)$ where
$$
S(1) = \frac{1}{2i} 
				\int_x^\infty q(y) 
					\int_z^\infty 
						e^{-2i\lam(z-y)}q^\sharp(z) 
					\, dz 
				\, dy.
$$
We may estimate
$$\norm[L^2(\R ^+)]{S(1)} \leq  \lw x \rw^{-3/2} \norm[L^{2,2}]{q} \|q^\sharp\|_{L^1}$$
which shows that $S(1) \in L^2(\R ^+)$ uniformly in $\lam$ with $\imag \lam \leq 0$. 
The desired bound is obtained
using Lemma \ref{lemma:direct.S.res}.

\end{proof}

\subsection{Generic Properties of Spectral Data}
\label{sec:direct.generic}

Lee \cite{Lee83} showed that generic potentials $q$ in the Schwartz class have at most finitely many simple zeros of $\alpha$ and no spectral singularities. His proof is based on a general argument of Beals and Coifman \cite{BC84}. Here we {  give}  a more precise functional analytic argument inspired by analogous results in Schr\"{o}dinger scattering theory (see the manuscript of Dyatlov and Zworski \cite[Chapter 2, Theorem 2.2]{DZ17}). We will prove:

\begin{theorem}
\label{thm:generic}
The set of $U$ potentials $q$ supporting at most finitely many solitons and having 
no spectral singularities is open and dense in $H^{2,2}(\R )$. 
\end{theorem}

Our strategy {  is}  to study the dense set of $q \in C_0^\infty(\R )$ and prove that 
any such $q$ can be perturbed by an arbitrarily small amount in $H^{2,2}$-norm
to a potential with the desired properties. We then {   use} continuity of spectral quantities to show that the set is open as well as dense. These steps are carried out in Propositions \ref{prop:direct.generic} and \ref{prop:direct.open} below which together give the proof of Theorem \ref{thm:generic}.

We begin with the study of $C_0^\infty$ potentials.
The following fact is well-known and easy to prove; see
for example Chapter 2 of  Lee's thesis \cite{Lee83}. 

\begin{lemma}
Suppose that $q \in C_0^\infty(\R )$. 
Then $\alpha(\lam;q)$ is analytic in $\C  $ and has at most finitely many zeros in  
$\imag \lam \leq c$ for any $c \in \R $. 
\end{lemma}

Using this fact, a perturbation argument, and Rouch\'e's theorem, we will construct a dense set of potentials in $H^{2,2}(\R )$ for which $\alpha$ has at most finitely many simple zeros in $\C  ^-$ and no zeros on $\R $. We will then exploit {  Proposition}  \ref{lemma:direct.a} to show that this set is also open. 

To construct the dense set, we need two perturbation lemmas. The first concerns perturbation from the zero potential for which $\alpha(\lam) \equiv 1$ and $\beta(\lam) \equiv 0$. 

\begin{lemma}
{  Suppose that $\varphi \in C_0^\infty(-R,R)$, $\lam \neq 0$, and $\mu$ is a small parameter. Let $q=\mu \varphi$.
Then the associated  transition matrix has the form}
\begin{equation}
\label{T.mu}
{T(\lam,q)} = \twomat{1}{0}{0}{1} + \twomat{0}{\mu c_\varphi}{ {\color{red}}\eps  \lam \overline{\mu} \overline{c_\varphi}}{0} + \bigO{\mu^2}
\end{equation}
where
$$c_\varphi = \eps  \int_{-\infty}^\infty e^{2i\lam y} \varphi(y) \, dy$$
and the correction term depends on $\norm[H^{1,1}]{\varphi}$.
\end{lemma}

\begin{proof}
It suffices to show that 
\begin{align}
\label{direct.a.mu}
\alpha(\lam)	
	&\sim	1 + \bigO{\mu^2}, \\
\label{direct.b.mu}
\lam {\bbeta(\lam)} \
	&= -\eps \lam \overline{\mu} \int e^{-2i\lam y} \overline{\varphi(y)} \, dy
	+ \bigO{\mu^2}.
\end{align}
First, we recall from Section 3.2 of Paper I that, for $\lam \in \R $, we have 
the uniform estimate
$$ \left| \left(N_{11}^+(x,\lam), N_{21}^+(x,\lam) \right) \right| \lesssim 1 $$
where the implied constants depend only on $\norm[H^{2,2}]{q}$ (the key issue is that the large-$\lambda$ behavior is controlled despite the $\lam$-dependence of the perturbing term in \eqref{direct.n.de}; see equations (3.4) of Paper I  for the integration by parts that removes this term). 
Taking limits as $x \rarr -\infty$ in  equations \eqref{direct.n1}--\eqref{direct.n2} 
 {   for $N^+$ (and as $x\to -\infty $ in the corresponding  equations for $N^-$)} 
and using the relation \eqref{direct.n.jc}, we 
 deduce that
\eqref{direct.a.mu} and \eqref{direct.b.mu} hold.
\end{proof}

The next lemma will give a mechanism for splitting multiple poles and perturbing zeros on the real axis.

\begin{lemma}
\label{lemma:TT}
Suppose that $q_1$ and $q_2$ are compactly supported potentials with disjoints supports, and that the support of $q_2$ on the real line is to the left of the support of $q_1$. Then:
\begin{enumerate}
\item[(i)] The identity
$$T(\lam,q_1+q_2) = T(\lam,q_2) T(\lam,q_1) $$
holds.
\item[(ii)]
If ${  q_1 } \in C_0^\infty((-R,R))$ and  $q_2 = \mu \varphi$ with $\varphi \in C_0^\infty((-2R,-R))$, the formula
\begin{equation}
\label{TT}
T( {  \lam }, q_1+\mu \varphi)	
		=	 \twomat{1}{\mu c_\varphi }{\eps\lam \overline{\mu c_\varphi} }{1} T(\lam,q_1) + \bigO{\mu^2}
\end{equation}
holds.

\end{enumerate}
\end{lemma}
\begin{proof}
Consider the normalized solution $N^+(x,\lam,q)$. It is not difficult to see that
$$
N^+(x,\lam,q_1+q_2) = N^+(x,\lam,q_2) N^+(x,\lam,q_1). 
$$

We may now compute
$$
T(\lam, q_1+q_2) 	= \lim_{x \rarr -\infty} e^{i \lam x \ad(\sigma_3)}
								\left[N^+(x,\lam,q_2) N^+(x,\lam,q_1)\right]\\
						=	T(\lam,q_2) T(\lam,q_1)
$$
The second assertion is an immediate consequence of the first.
\end{proof}

Suppose that $q_1$ and $q_2$ are chosen as in Lemma \ref{lemma:TT}(ii). {  To simplify the notation, let us write  $\alpha(\lambda,\mu) $
to denote $\alpha(\lambda, q_1+\mu\varphi)$.   It follows from \eqref{TT} that
\begin{equation}
\label{balpha.p0}
\alpha(\lam,\mu)  = \alpha(\lam,0)  +  { \mu c_{\varphi} } \lam {\bbeta(\lam)}  
	+ \bigO{\mu^2}.
\end{equation}

 In the next proposition, we  will  expand the above formula near $\lambda= \lambda_0$:
}
\begin{equation}
\label{balpha.p}
\alpha(\lam,\mu)  = \alpha(\lam,0) +   { \mu c_{\varphi_0} } \lam_0 {\bbeta(\lam_0)}  +C_0(\lambda-\lambda_0)\mu
	+ \bigO{\mu^2}.
\end{equation}
where  
$$c_{\varphi_0}=\eps  \int_{-\infty}^\infty e^{2i\lam_0 y} \varphi(y) \, dy$$
and
$$
\left|C_0\right|
\leq 
\norm[L^\infty]{\dfrac{d}{d\lambda}\left(\lambda{\bbeta(\lambda)}\right)}.
$$ 
From the compact support of the potential $q$ and the asymptotic behavior $\alpha(\lambda)$ (\eqref{a.infty}) we know that $\alpha$ has finitely many zeros in $\mathbb{C}^- \cup \mathbb{R}$. We put down non-overlapping discs $D(\lam_i,r_i)$ with centers at the finitely many zeros $\lam_i$.  We will prove:

\begin{proposition}
\label{prop:direct.generic}
Suppose that $R>0$ and $q \in C_0^\infty([-R,R])$. Let $\alpha(\lam)$ be the 
$(1,1)$ entry of the transition matrix for $q$. For $\varphi \in C_0^\infty(\R )$, let $\alpha(\lam,\mu)$ be the $(1,1)$ entry for the transition matrix of $q+\mu \varphi$, so that $\alpha(\lam,0) = \alpha(\lam)$. 
\begin{itemize}
\item[(i)]		Suppose that $\lbrace \lambda_i \rbrace_{i=1}^m$ are the isolated zeros of $\alpha(\lam)$ in $\mathbb{C}^- \cup \mathbb{R}$ and $\lam_i \neq 0$ is one of the zeros of $\alpha(\lam)$ of multiplicity $n \geq 2$, i.e. 
$\alpha(\lambda)=(\lambda-\lambda_i)^n g(\lambda)$
 for some analytic function $g$ with $g(\lam_i)\neq 0$.    Then, for some $\varphi \in C_0^\infty(\R )$ and all sufficiently small $\mu \neq 0$, $\alpha(\lam,\mu)$ has $n$ simple zeros in the disc $D_{r_i}(\lambda_i)$.

\item[(ii)]  Suppose that after the perturbation in part (i),  $\Lambda_j$ is a simple zero of $\alpha(\lam,\mu)$ on the real axis, $\Lambda_j \neq 0$. Then, there is a function $\varphi \in C_0^\infty(\R )$ so that, for all real, nonzero, and sufficiently small $\mu'$, $\alpha(\lam,\mu')$ has no zeros on the real axis near $\Lambda_j$. 
\end{itemize}
In each case, we may choose $\varphi$ to have support in $(-2R,-R)\cup (R, 2R)$.
\end{proposition}

\begin{proof}
(i)  We first claim that there exists a function $\varphi\in C^\infty_0(\mathbb{R})$, $\varphi\geq 0$ such that 
$$ \widehat{\varphi}(\lam_i) = \int_{-\infty}^\infty { e^{ 2 i\lam_i x} } \varphi(x) \, dx \neq 0$$ for all $i$, $1 \leq i \leq m$.

Indeed, let $2\lambda_i=\xi_i+i\eta_i$. We  then have 
$$\widehat{\varphi}(\lam_i) = \int_{-\infty}^\infty (\cos \xi_i x + i\sin\xi_i x) e^{\eta_i x}\varphi(x) \, dx$$
 with $ e^{\eta_i x}\varphi(x)\geq 0$ for all $x$. \\
 
If we let $\xi=\text{max} \lbrace \xi_1,...\xi_i,...\xi_m\rbrace$,  $r=\pi/2\xi$ and make $|\text{supp}(\varphi)|\leq r$, then at least one of $\cos \xi_i x$ and $\sin \xi_i x$ does not change sign on $\text{supp}(\varphi)$ for all $i$. So
$ \widehat{\varphi}(\lam_i) {  \ne 0} $ for all $i$.\\

Using the Taylor expansion of $\widehat{\varphi}(\lambda)$ and $\lambda {  \bbeta(\lam)}$ we can write  (\ref{balpha.p}) as
\begin{equation}
\label{alpha-mu}
\alpha(\lam,\mu)  = (\lambda-\lambda_i)^n g(\lambda) {  +}  { \mu c_{\varphi_i} } \lam_i 
\bbeta(\lam_i) 
 +C_0(\lambda-\lambda_i)\mu
	+ \bigO{\mu^2}.
\end{equation}
If we can establish the following inequalities
 \begin{equation}
 \label{bound-1}
 | \lambda {  {\bbeta}(\lam)}|_{L^\infty}\lesssim_q 1
 \end{equation}
  and
\begin{equation}
\label{bound-2}
| (\lambda {  {\bbeta}(\lam)})'|_{L^\infty}\lesssim_q 1
\end{equation}
where $\lambda\in\overline{ D(\lambda_i, r_i) }$
 then it is clear that
\begin{align*}
|\alpha(\lam,\mu) -  (\lambda-\lambda_i)^n g(\lambda)  |&=|  { \mu c_{\varphi_i} } \lam_i  \bbeta(\lam_i)
 +C_0(\lambda-\lambda_i)\mu+ \bigO{\mu^2}  |\\
                                                        &\leq |\lambda-\lambda_i|^n |g(\lambda)|  
\end{align*}
for $\mu$ sufficiently small and $\lambda\in \partial D(\lambda_i, r_i)$. Rouch\'{e}'s Theorem shows that the number of zeros of $\alpha(\lambda, \mu)$ and $\alpha(\lambda, 0)$ agree (with multiplicities) in $D(\lambda_i, r_i )$. That is, $\alpha(\lambda, \mu)$ has exactly $ n$ zeros there.

To prove  estimates \eqref{bound-1} and \eqref{bound-2}, we use the fact that
$\lam \bbeta(\lam) = \lim_{x \rarr \infty} e^{-2i\lam x} N_{21}^-(x,\lam)$ and the analogue of \eqref{direct.n2} for $N_{21}^-$ to compute
\begin{align}
\label{lambeta}
\lambda   {\bbeta}(\lam)&= \int_{-R}^R e^{-2i\lambda y}\left(\lambda \eps \overline{q(y)}N^-_{11}(y,\lambda)  +p_2(y)N^-_{21}(y,\lambda) \right)dy.
\end{align}
From direct computation, its derivative is
\begin{align}
\label{dbeta}
\dfrac{d}{d\lambda}\left(\lambda{  {\bbeta}(\lam)}\right)
	&= \int_{-R}^R  
			e^{-2i\lambda y} (-2iy) 
				\left(
					\lambda \eps\overline{q(y)}N^-_{11}(y,\lambda)  +
					p_2(y)N^-_{21}(y,\lambda) 
				\right)	\, dy\\
\nonumber
	&\quad +\int_{-R}^R 
			e^{-2i\lambda y}
					\left( \eps \overline{q(y)}N^-_{11}(y,\lambda)  +
					p_2(y)N^-_{21}(y,\lambda) _\lambda 
				\right)	\, 	dy\\ \nonumber
	&\quad -\int_{-R}^R 
			e^{-2i\lambda y}
					\left(
						\lambda \eps \overline{q(y)}N^-_{11}(y,\lambda)_\lambda  +
						p_2(y)N^-_{21}(y,\lambda)_\lambda 
					\right) \,  dy.                                                                                       
\end{align}
Inequalities (\ref{bound-1}) and (\ref{bound-2}) follow from these expressions and Lemma \ref{lemma:direct.n}. \\

Now we want to show that the zeros  of $\alpha(\lambda, \mu)$ are simple. For $0\leq k\leq n-1$, consider the disc  around the $k^{th}$ root of unity of $\gamma_i$
\begin{equation}
\label{disc}
D_k:=D\left(| \gamma_i |^{\frac{1}{n}} e^{i (\phi+2\pi k)/n }+\lam_i ,\,  \rho |\gamma_i|^{\frac{1}{n}} \right)
\end{equation}
where
$$ \gamma_i=\frac{{ \mu c_{\varphi_i} } \lam_i 
 \bbeta(\lam_i)
 }{g(\lam_i)} , \,\,\, \phi=\arg \gamma_i+\pi$$
Notice that if 
$\rho< \pi/n$ then $D_k \cap D_\ell $ is empty for $k\neq \ell.$ 
We now expand $g(\lambda)$ at $\lambda=\lambda_i$ and get
$$\alpha(\lambda, \mu)=(\lambda-\lambda_i)^n g(\lambda_i)+\mathcal{O}(\lambda-\lam_i)^{n+1}.$$
For $\lambda\in \partial D_k$, 
$$\left\vert(\lambda-\lambda_i)^n g(\lambda_i)  + { \mu c_{\varphi_i } }  \lam_i 
 \bbeta(\lam_i)
-\alpha(\lambda, \mu) \right\vert   
{  \lesssim} C_0 |\gamma_i |^{1+\frac{1}{n}}.$$
On the other hand, if we choose $\rho > 2 C_0 |\gamma_i|^{\frac{1}{n}} $ then for $\lambda\in  \partial D_k$, 
\begin{align*}
\left\vert(\lambda-\lambda_i)^n g(\lambda_i)-\gamma_i g(\lambda_i)\right\vert
&=|\gamma_i | \rho\left(1+\bigO {\rho^2} \right)\\
&
{   \gtrsim} C_0 |\gamma_i|^{1+\frac{1}{n}}\\
&\geq | (\lambda-\lambda_i)^n g(\lambda_i)  + { \mu c_{\varphi_i} }   \lam_i  \bbeta(\lam_i)
-\alpha(\lambda, \mu) |.
\end{align*}
Since the discs $D_k$ are disjoint, Rouch\'{e}'s theorem now shows that there is exactly one zero of $\alpha(\lambda, \mu)$ in each $D_k$. This shows that all $n$ zeros are simple.\\

(ii)  After the first step of perturbation in (i),  $\alpha(\lambda, \mu)$ has simple zeros $\lbrace \Lambda_j \rbrace_{j=1}^l$. Suppose $\Lambda_j$ is a zero of  $\alpha(\lambda, \mu)$ on the real axis. We make another small perturbation of the potential as above and formulate 
\begin{align}
\label{alpha-mu'}
\alpha(\lambda, \mu')= (\lambda-\Lambda_j) h(\lambda, \mu) + { \mu' c_{\psi_j }}  \Lambda_j  \bbeta(\Lambda_j, \mu)  +C'_0(\lambda-\Lambda_j)\mu'
	+ \bigO{\mu'^2}
\end{align}
where $$(\lambda-\Lambda_j)h(\lambda, \mu)=\alpha(\lambda, \mu) $$
and we define 
\begin{equation}
\label{disc'}
D_j:=D\left( \Gamma_j +\Lambda _j ,\,  \rho'|\Gamma_j| \right)
\end{equation}
where
$$ \Gamma_j=\frac{ { \mu' c_{\psi_j} } \Lambda_j \bbeta(\Lambda_j, \mu)}{h(\Lambda_j, \mu)} .$$
Given $\Lambda_j\in \mathbb{R}$, we can make appropriate choices of small parameter $\mu'$ and $\psi\in C^\infty_0(\mathbb{R})$ such that $\Im( \Gamma_j +\Lambda_j)$ is strictly nonzero and $D_j\cap \mathbb{R}$ is empty. Since there are only finitely many zeros,  we can choose $\mu$  which works for all $j=1,2,..., l$. Repeating the argument in (i) we get the desired conclusion.\\
\end{proof}

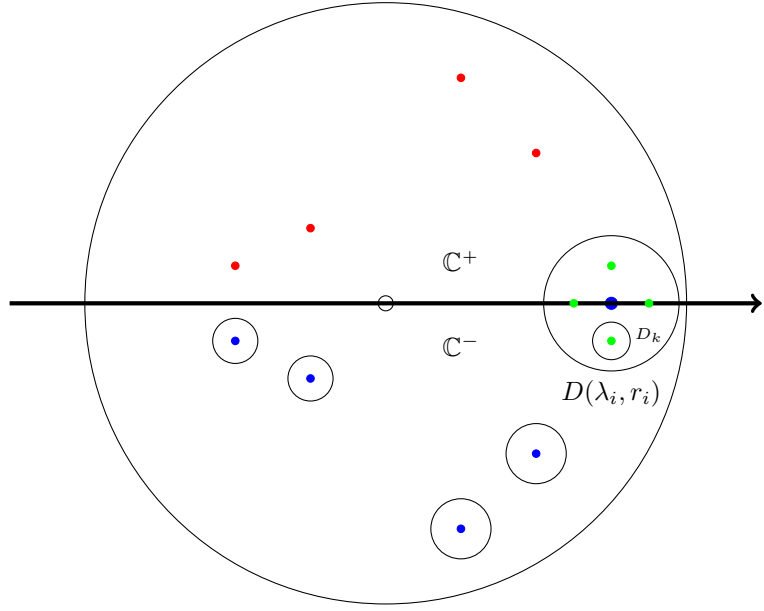
\begin{figure}
\begin{tikzpicture}
\draw [red, fill=red] (-2,0.5) circle [radius=0.05];
\draw [red, fill=red] (-1,1) circle [radius=0.05];
\draw [red, fill=red] (1,3) circle [radius=0.05];
\draw [red, fill=red] (2,2) circle [radius=0.05];
\draw [blue, fill=blue] (3, 0) circle [radius=0.08];
\draw  (0,0) circle [radius=0.1];
\draw (0,0) circle [radius=4];
\node [above] at (1,0.3) {$\mathbb{C^+}$};
\draw [->][ultra thick] (-5,0) -- (5,0);
\node [below] at (1,-0.3) {$\mathbb{C^-}$};
\draw [blue, fill=blue] (-2,-0.5) circle [radius=0.05];
\draw [blue, fill=blue] (-1,-1) circle [radius=0.05];
\draw [blue, fill=blue] (1,-3) circle [radius=0.05];
\draw [blue, fill=blue] (2,-2) circle [radius=0.05];
\draw [green, fill=green] (3,-0.5) circle [radius=0.05];
\draw [green, fill=green] (3,0.5) circle [radius=0.05];
\draw [green, fill=green] (3.5,0) circle [radius=0.05];
\draw [green, fill=green] (2.5,0) circle [radius=0.05];
\node [below] at (3.5, -0.2){{\tiny $D_k$}};
\draw (3, -0.5) circle [radius=0.25];
\draw (3, 0) circle [radius=0.9];
\node [below] at (3, -0.9) { $D(\lam_i, r_i)$};
\draw (-2, -0.5) circle [radius=0.3];
\draw (-1, -1) circle [radius=0.3];
\draw (1, -3) circle [radius=0.4];
\draw (2, -2) circle [radius=0.4];
\end{tikzpicture}
\caption{Zeros of $\balpha$ and $\alpha$ in the $\lambda$ plane}
\label{figure-1}
\end{figure}

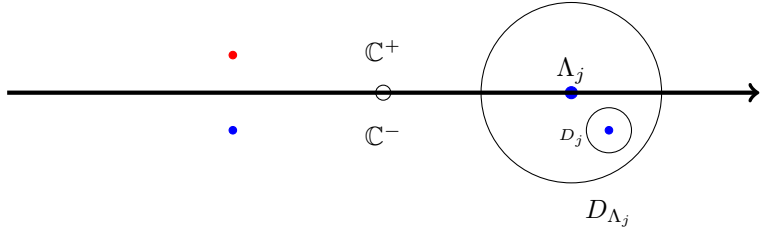
\begin{figure}
\begin{tikzpicture}
\draw [red, fill=red] (-2,0.5) circle [radius=0.05];
\draw [blue, fill=blue] (2.5, 0) circle [radius=0.08];
\draw  (0,0) circle [radius=0.1];
\node [above] at (0,0.3) {$\mathbb{C^+}$};
\draw [->][ultra thick] (-5,0) -- (5,0);
\node [below] at (0,-0.3) {$\mathbb{C^-}$};
\draw [blue, fill=blue] (-2,-0.5) circle [radius=0.05];
\node [above] at (2.5, 0) {$\Lambda_j$};
\draw (2.5, 0) circle [radius=1.2];
\draw [blue, fill=blue] (3.0, -0.5) circle [radius=0.05];
\node[left] at (2.8,- 0.6) {{\tiny $D_j$}};
\draw (3.0,- 0.5) circle [radius=0.3];
\node [below] at (3, -1.3) {$D_{\Lambda_j}$};
\end{tikzpicture}
\caption{Simple zero of $\alpha(\lambda, \mu)$ on $\mathbb{R}$}
\label{figure-2}
\end{figure}

Proposition \ref{prop:direct.generic} shows that there is a dense subset of $q \in H^{2,2}(\R )$ for which $\alpha(\lam;q)$ has at most finitely many simple zeros 
in $\C  ^-$ and no zeros on $\R $. Owing to the continuity of $\alpha$ in $q$,
the fact that {  $\alpha$  is} analytic in $\C  ^-$, and the continuity of the map $q \mapsto \alpha(\dotarg,q)$ imply that this set is also open.

\begin{proposition}
\label{prop:direct.open}
Suppose that $q_0 \in H^{2,2}(\R )$ and that $\alpha(\dotarg;q_0)$ has exactly $n$ simple zeros in $\C  ^-$ and no zeros on $\R $. There is a neighborhood $\calN$ of $q_0$ in $H^{2,2}(\R )$ so that all $q \in \calN$ have these same properties.
\end{proposition}

\begin{proof}
Since $|\alpha(\lam;q_0)|$ does not vanish on $\R $, we have $|\alpha(\lam;q_0)| \geq c$ for some $c>0$. It follows from Lipschitz continuity of $q \mapsto \alpha(\dotarg;q)-1$ as maps $H^{2,2}(\R ) \rarr H^1(\R)$  that there is an $r_0 > 0$ so that 
$|\alpha(\lam;q)| \geq c/2$ for all $q$ with $\norm[H^{2,2}]{q-q_0} < r_0$. Next, 
let $\eta_1 = \inf_{j \neq k} |\lam_j- \lam_k|$, $\eta_2 = \inf_{k} |\imag \lam_k|$, and
$\eta = \frac{1}{2} \inf(\eta_1,\eta_2)$. By {  Proposition} \ref{lemma:direct.a} and analyticity there is an $r_1>0$ so that the $n$ simple zeros of $\alpha$ remain simple and move a distance no more than $\eta$ for $q \in H^{2,2}(\R )$ with 
$\norm[H^{2,2}]{q-q_0} < r_1$. Take $\calN = B(q_0,r)$ where $r < \inf(r_1,r_2)$.
\end{proof}

\subsection{Lipschitz Continuity of Spectral Data for Generic Potentials}
\label{sec:direct.lip.discrete}

Owing to Proposition \ref{prop:direct.scatt.lip}, to complete the proof of  Theorem \ref{thm:R.lip}, it suffices to prove Lipschitz continuity of the eigenvalues $\lam_k$ and the norming constants $C_k$ on open subsets of $U_N$ with bounds uniform in bounded subsets of $U_N$. 
This is an exercise in eigenvalue perturbation theory and is carried out in Proposition \ref{prop:direct.discrete.lip}.

We order the zeros of $\ba$ by modulus and, given two zeros with the same modulus, order by increasing phase in $(0,\pi)$. We reformulate the definition of $C_k$ in terms of the normalized Jost solutions $N^\pm$ of \eqref{direct.n.de}. If $\balpha(\lam_k) = 0$, there is a constant $b_k$ with the property
\begin{equation}
\label{direct:bk.bis}
e^{-2i\lam_k x}\twovec{N_{11}^-(x,\lam_k) }{\lam N_{21}^-(x,\lam_k)  }
= b_k e^{2i\lam_k x}\twovec{\lam^{-1}N_{12}^+(x,\lam_k)}{N_{22}^+(x,\lam_k) }
\end{equation} 
{   If $\balpha'(\lambda_k) \ne 0$, one defines the norming constant as $C_k = {b_k}/{(\zeta_k \balpha'(\lambda_k))}$ where  $\lambda_k= \zeta_k^2$. The discrete scattering
data are composed of the pairs $(\lambda_k, C_k)$.}
\begin{proposition}
\label{prop:direct.discrete.lip}
Suppose that $q_0 \in U$ with $N$ simple zeros of $\balpha$ in $\C  ^+$. Let $\Lam = \{ \lam_1, \ldots, {  \lambda_N\}}$ be a listing of the zeros of 
$\balpha$ with the ordering as described above, and set 
$$ 
d_\Lam = 
	\min
		\left( 
			\min_{1 \leq j \neq k \leq N} |\lam_j(q_0) - \lam_k(q_0)|,  \,\,
			\min (\imag  \lam_k) 
		\right).
$$ There is a neighborhood $\calN$ of $q_0$ so that:
\begin{itemize}
\item[(i)]		For any $q \in \calN$, $\balpha(\lam;q)$ has exactly $n$ simple zeros in $\C  ^+$,  no zeros on $\R $, and $|\lam_j(q) - \lam_j(q_0)| \leq \frac{1}{2}d_\Lam$.
\item[(ii)]	The estimate $|\lam_j(q) - \lam_j(q_0)| \leq C \norm[H^{2,2}]{q-q_0}$
holds for $C$ uniform in $q \in \calN$. 
\item[(iii)]	The estimate $|b_j(q) - b_j(q_0)| \leq C \norm[H^{2,2}]{q-q_0}$ holds for $C$
uniform in $q \in \calN$.
{ 
\item[(iv)]	The estimate $|C_j(q) - C_j(q_0)| \leq C \norm[H^{2,2}]{q-q_0}$ holds for $C$
uniform in $q \in \calN$.
}
\end{itemize}
\end{proposition}

\begin{proof}
(i) From Proposition \ref{prop:direct.open} and \eqref{direct:alpha.breve}, we immediately conclude that there is a neighborhood $\calN$ of $q_0$ for which $q \in \calN$ has exactly $N$ simple zeros in $\C ^+$ with no singularities on the real axis. We can establish continuity of the simple zeros as a function of $q$ (and hence  estimate $|\lam_j(q)-\lam_j(q_0)| \leq \frac{1}{2} d_\Lambda$) as 
follows. The equation  $\alpha(\lam_j(q); q) = 0$, defines $\lam_j$ as an analytic function of $q$ in a neighborhood of $q$ owing to the fact that $\alpha'(\lam_j(q_0), q_0) \neq 0$ and the implicit function theorem on the Banach space $\C \times H^{2,2}(\R)$.
 $\alpha$ as a function on $ \C  ^- \times H^{2,2}(\R )$. The function $\alpha(\lam,q)$
 is  analytic  in $q$ because the functions occurring in the Wronskian formula \eqref{alpha} may be computed by convergent Volterra series which are analytic in $q$. 

(ii) The implicit function theorem also guarantees that the function $\lam_j(q)$ will be $C^1$ as a function of $q$, and hence Lipschitz continuous.

(iii) Uniqueness for the equation \eqref{direct.n.de} guarantees that at least one of $N_{12}^+(0,\lam_j)$ and $N_{22}^+(0,\lam_j)$ is nonzero at $q=q_0$. Suppose that $N_{12}^+(0,\lam_j(q_0)) \neq 0$. By shrinking the neighborhood if needed we may assume that 
$N_{12}^+(0,\lam_j(q)) >0$ strictly for all $q \in \calN$. We may then compute from \eqref{direct:bk.bis} that $b_j(q) = \lam_j(q) N_{21}^-(0,\lam_j(q))/N_{22}^+(0,\lam_j(q))$ which, as a product and quotient of Lipschitz continuous functions of $q$, is itself Lipschitz continuous in $q$.

(iv) Finally, $\alpha'(\lambda_k)$ can easily be expressed in terms of $\alpha$ through a Cauchy integral over a small circle around $\lambda_k$ due to the analyticity of 
$\alpha$ in $\C^-$, and the Lipschitz continuity 
of $b_j$ and $\alpha(\lambda_j)$  in $q$ extends to the norming constants $C_j$.
\end{proof}

Propositions \ref{prop:direct.scatt.lip} and \ref{prop:direct.discrete.lip} give the proof of:

\begin{theorem}
\label{thm:R.lip}
For each $N$, the map $\calR:U_N \rarr V_N$ is uniformly Lipschitz continuous on bounded subsets of $U_N$.
\end{theorem}
%
%

\section{The Inverse Scattering Map}
\label{sec:inverse}

In this section we study the Riemann-Hilbert problem which defines the inverse scattering map. As explained in the Introduction, to prove existence and uniqueness of solutions, the Riemann-Hilbert problem in the $\zeta$ variable is most convenient. On the other hand, to prove Lipschitz continuity of the inverse scattering map, it is most effective to work in the $\lambda$ variables.  

This section is organized as follows.  In section \ref{sec:RHP1} we  prove existence and uniqueness of solutions for the Riemann-Hilbert Problem \ref{RHP1}. In section \ref{sec:RHP2} use these results to prove existence and uniqueness for the Riemann-Hilbert Problem \ref{RHP2}. Finally, in section \ref{sec:Lip}, we prove Lipschitz continuity of scattering maps by studying the Beals-Coifman integral equations for Problem \ref{RHP2}. 

\subsection{First Riemann-Hilbert Problem}
\label{sec:RHP1}

In this subsection we study Problem \ref{RHP1} at a fixed time $t$.  For convenience we set $t=0$ in the discussion that follows (so that the phase factor $\exp\left(-it\Theta \ad(\sigma)\right)$ becomes $\exp(-ix\zeta^2 \ad(\sigma)$), and write the solution $M(x,\zeta,0)$ as $M(x,\zeta)$.
For later use we note that the maps
\begin{align}
\label{RHP1.sym1}
M(x,\zeta) & \mapsto \twomat{M_{11}(x,-\zeta)}{-M_{12}(x,-\zeta)}{-M_{21}(x,-\zeta)}{M_{22}(x,-\zeta)}\\
\label{RHP1.sym2}
M(x,\zeta) & \mapsto \sigma_\eps \overline{M(x,\zetabar)} \sigma_\eps^{-1} 
\end{align}
where 
$$ \sigma_\eps = \twomat{0}{1}{\eps}{0} $$
preserve the solution space of Problem \ref{RHP1}.

We will prove:

\begin{theorem}
\label{thm:RHP1.unique}
For given data $\left( r, \{ \zeta_i, c_i\}_{i=1}^N \right)$ with $r \in H^1(\Sigma)$ odd and satisfying \eqref{rr}, $\zeta_i \in \Omega^{++}$, and $c_i \in \C^\times$, there exists a unique solution 
of Problem \ref{RHP1} for either $\eps=+1$ or $\eps=-1$.
\end{theorem}

To prove Theorem \ref{thm:RHP1.unique}, we first exhibit an explicit invertible map from the solution space of Problem \ref{RHP1} with $\eps=-1$ to the solution space with $\eps=+1$ (Lemma \ref{lemma:bob}). 
Next, we replace the discrete data in Problem \ref{RHP1} by additional contours and jump relations, obtaining Problem \ref{RHP1c} which is equivalent owing to Theorem \ref{thm:RHP.model.disc}. Finally, we will use the symmetries of the jump matrix for Problem \ref{RHP1c} with $\eps = -1$ and a theorem of Zhou \cite[Theorem 9.3]{Zhou89} to conclude that there exists a unique solution to Problem \ref{RHP1c} with $\eps=-1$. The equivalence of Problems \ref{RHP1} and \ref{RHP1c} implies that Problem \ref{RHP1} also has a unique solution, and the isomorphism of solution spaces given by 
Lemma \ref{lemma:bob} implies there also exists a unique solution to  Problem \ref{RHP1} with $\eps=+1$.

The first step is:

\begin{lemma}
\label{lemma:bob}
The map 
\begin{equation}
\label{m-.to.m+}
M(x,\zeta) \mapsto M^\sharp(x,\zeta)=\sigma_2 M(x, i \zeta ) \sigma_2 
\end{equation}
maps the solution space of  the Riemann-Hilbert problem  for $\eps=-1$  and scattering data 
$$\left( r, \{ \zeta_k \}_{k=1}^N, \{ c_k \}_{k=1}^N \right)$$  
onto the solution space of the Riemann-Hilbert problem for $\eps=1$ and scattering data
$$\left( s, \{ \eta_k \}_{k=1}^N, \{ d_k \}_{k=1}^N \right)$$ 
where
\begin{equation}
\label{data-.to.data+}
 s(\zeta) = \overline{r(i\zetabar)}, \quad \eta_k = i\overline{\zeta_k},
\quad d_k = i\overline{c_k} 
\end{equation}
\end{lemma}

\begin{proof}
We claim that $M^\sharp$ solves the problem with $\eps=1$. 
First, it is easy to see that 
$$M^\sharp_\pm(x,\zeta)-I = \sigma_2 \left(M_\mp(x,i\zeta)-I\right) \sigma_2= \sigma_2 (C^\mp h)(i\zeta) \sigma_2$$ 
for $h \in L^2(\Sigma)$. A short computation shows that
$ \sigma_2(C^\mp h)(i\zeta) \sigma_2 =(C^\pm h^\sharp)(\zeta)$
where $h^\sharp(s) =\sigma_2 ih(is) \sigma_2 \in L^2(\Sigma)$. Hence
$M^\sharp(x,\dotarg) - I = C^\pm h^\sharp$
where $h^\sharp \in L^2(\Sigma)$.
It remains to check that the jump conditions and residue conditions are 
satisfied with the transformed data \eqref{data-.to.data+}.

First, we compute, for $\zeta \in \Sigma$,
$$
M^\sharp_+(x,\zeta) 
	=	\sigma_2 	M_-(x,i\zeta) \sigma_2	
	=	\sigma_2	M_+(x,i\zeta) v(i\zeta)^{-1} \sigma_2
	=	M^\sharp_-(x,\zeta) v^\sharp(\zeta)
$$
where
$$
v^\sharp(\zeta)
	=	\sigma_2 v(i\zeta)^{-1} \sigma_2
	=	\twomat{1-r(i\zeta) \br(i\zeta)}
						{\br(i\zeta)}{-r(i\zeta)}{1}
	=	\twomat{1-s(\zeta)\bs(\zeta)}{s(\zeta)}
						{-\bs(\zeta)}{1}		
$$
In the last step we used \eqref{r.to.br} and \eqref{data-.to.data+} to conclude 
that $\br(i\zeta) = -\overline{r(-i\zetabar)} = s(\zeta)$,
and  that $ r(i\zeta)=\overline{s(\zetabar)}=\bs(\zeta)$.

Next, we compute, for $\zeta_k \in \calZ^{++}$ and $\eta_k = i \overline{\zeta_k}$,
\begin{align*}
\Res_{z=\eta_k} M^\sharp(x,z)
	&=	\lim_{z \rarr \eta_k} (z-\eta_k) M^\sharp(x,z)\\
	&=	\lim_{z \rarr \eta_k} (z-\eta_k) \sigma_2 M(x,iz) \sigma_2\\
	&=	\lim_{w \rarr -\overline{\zeta_k}} 
					(-iw-i\overline{\zeta_k}) \sigma_2 M(x,w) \sigma_2\\
	&=	(-i)\sigma_2 
					\left(\Res_{w=-\overline{\zeta_k}} M(x,w) \right)
					\sigma_2\\
	&=	 \lim_{w \rarr -\overline{\zeta_k}} \sigma_2 M(x,w) \sigma_2 e^{ix{\overline{\zeta_k}}^2\ad(\sigma_3)} \twomat{0}{0}{i \overline{c_k}}{0}\\
	&=	\lim_{z \rarr \eta_k} M^\sharp(x,z) e^{-ix\eta_k^2 \ad(\sigma_3)}\twomat{0}{0}{i \overline{c_k}}{0}
\end{align*}
where in the second to last line we used the identity $\sigma_2 e^{i\theta \ad(\sigma_3)} \sigma_2 = e^{-i\theta \ad(\sigma_3)}$. The remaining residue conditions follow by symmetry.

We have now proved that the map $M \mapsto M^\sharp$ 
takes the solution space of Problem \ref{RHP1} with $\eps = -1$ \emph{into} the solution space of Problem \ref{RHP1} with $\eps=1$. To show that the map is onto, we note that, by a similar computation, the map 
\begin{equation}
\label{m+.to.m-}
M^\sharp(x,\zeta) \mapsto M(x,\zeta)=\sigma_2 M^\sharp(x, -i \zeta ) \sigma_2 
\end{equation}
maps the solution space of problem with $\eps=1$ to the solution space of problem with $\eps=-1$.

\end{proof}

We wish to apply a standard uniqueness result \cite[Theorem 9.3]{Zhou89} to show that Problem \ref{RHP1} with $\eps=-1$
has a unique solution.
Theorem \ref{thm:RHP.model.disc} implies that  Problem \ref{RHP1} is equivalent  to the following problem which replaces the discrete data with jumps on small circular contours centered at the points $\zeta \in \calZ$. The augmented contour ${\Sigma}'$ is given by 
\begin{equation}
\label{Sigma.prime}
\Sigma' = \Sigma \cup \left( \cup_{\zeta \in \calZ} \gamma_\zeta \right) 
\end{equation}
where the contours $\gamma_\zeta$ are oriented as shown.  Observe that, with the chosen orientation, the contour is Schwarz-reflection invariant. 

\begin{center}
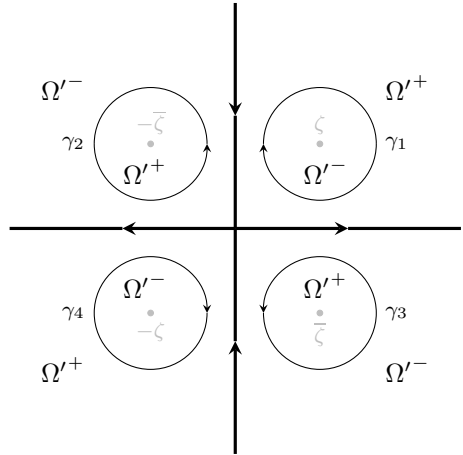
\begin{figure}[H]
\caption{The Contour $\Gamma'$}
\begin{tikzpicture}[scale=0.75]
\draw[very thick,->,>=stealth]  (0,4) 	-- (0,2);
\draw[very thick]					(0,2) 	-- (0,0);
\draw[very thick,->,>=stealth]   (0,-4) -- (0,-2);
\draw[very thick]					(0,-2) -- (0,0);
\draw[very thick,->,>=stealth]  (0,0) 	-- (-2,0);
\draw[very thick]					(-2,0) -- (-4,0);
\draw[very thick,->,>=stealth]  (0,0) 	-- (2,0);
\draw[very thick]					(2,0)	--	(4,0);
\node[right] at (2.5,2.5)		{${\Omega'}^+$};
\node[left]   at (-2.5,-2.5)		{${\Omega'}^+$};
\node[left]   at (-2.5,2.5)		{${\Omega'}^-$};
\node[right] at (2.5,-2.5)		{${\Omega'}^-$};
\draw[gray!50,fill=gray!50] (1.5,1.5) circle(0.5mm);
\draw[gray!50,fill=gray!50] (-1.5,1.5) circle(0.5mm);
\draw[gray!50,fill=gray!50] (1.5,-1.5) circle(0.5mm);
\draw[gray!50,fill=gray!50] (-1.5,-1.5) circle(0.5mm);
\node[above,gray!50] at (1.5,1.5) {\footnotesize{$\zeta$}};
\node[above,gray!50] at (-1.5,1.5) {\footnotesize{$-\overline{\zeta}$}};
\node[below,gray!50] at (1.5,-1.5) {\footnotesize{$\overline{\zeta}$}};
\node[below,gray!50] at (-1.5,-1.5) {\footnotesize{$-\zeta$}};
\draw[black,->,>=stealth] (2.5,1.5) 	arc(360:180:1);
\draw[black]					(0.5,1.5)	arc(180:0:1);	
\node[right] at (2.5,1.5) 	{\footnotesize{$\gamma_1$}};
\node[below] at (1.6,1.4) {{${\Omega'}^-$}};
\draw[black] 					(-0.5,1.5) arc(0:180:1);
\draw[black,->,>=stealth]	(-2.5,1.5) arc(180:360:1);
\node[left] at (-2.5,1.5)	{\footnotesize{$\gamma_2$}};
\node[below] at (-1.6,1.4)	{{${\Omega'}^+$}};
\draw[black,->,>=stealth]	(2.5,-1.5) arc(0:180:1);
\draw[black]					(0.5,-1.5) arc(180:360:1);
\node[right] at (2.5,-1.5) {\footnotesize{$\gamma_3$}};
\node[above] at (1.6,-1.4) 	{${\Omega'}^+$};
\draw[black,->,>=stealth]	(-2.5,-1.5) arc(180:0:1);
\draw[black]					(-2.5,-1.5) arc(180:360:1);
\node[left] at (-2.5,-1.5)	{\footnotesize{$\gamma_4$}};
\node[above] at (-1.6,-1.4) 	{${\Omega'}^-$};
\end{tikzpicture}
\label{fig:RHP1.c}
\end{figure}
\end{center}

\begin{RHP}
\label{RHP1c}
Given $r \in H^1(\Sigma)$ with $r(-\zeta) = -r(\zeta)$ and so that \eqref{rr} holds, $ \{ \zeta_i, c_i \}_{i=1}^N$,
find an analytic function $M(x,\dotarg):	\C \setminus \left( \Sigma \cup \calZ \right) \rarr SL(2,\C)$
with 
$M_\pm(x,\dotarg) - I \in \dee C(L^2)$ and $M_+(x,\zeta) = M_-(x,\zeta) e^{-ix\zeta^2 \ad(\sigma)} v'(\zeta)$ where
\begin{equation}
\label{RHP1c.jump}
v'(\zeta) = 	\begin{cases}
						\begin{pmatrix}
							1 - r(\zeta) \br(\zeta)	&	r(\zeta)	\\
							 {\color{red}- }\br(\zeta)					&	1
						\end{pmatrix},
						&  \zeta \in \Sigma\\
						\\
						\lowunitmat{\dfrac{c_i}{\zeta \pm \zeta_i}},
						&	\zeta \in \pm \gamma_i\\
						\\
							\upunitmat{\dfrac{-\eps \overline{c_i} }{\zeta \mp \overline{\zeta_i}}},
						& \zeta \in \pm \gamma_i^*
					\end{cases}
\end{equation}
\end{RHP}

The jump matrix $v'(\zeta)$ admits the factorization $v'(\zeta) = (I-w_x^-)^{-1} (I+w_x^+)$ where
$w_x^\pm = e^{-ix\zeta^2 \ad(\sigma)} w^\pm$ and $w^\pm$ are given by
\begin{equation}
\label{RHP1c.jump.w}
(w^+,w^-) =
\begin{cases}
	\left( \lowmat{-\br(\zeta)}, \upmat{r(\zeta)} \right),		&	\zeta \in \Sigma,\\
	\\
	\left( \lowmat{\frac{c_i}{\zeta \pm \zeta_i}} , \upmat{0} \right),		&	\zeta \in \pm \gamma_i, \\
	\\
	\left( \lowmat{0}, \upmat{\frac{-\eps \overline{c_i}}{\zeta \mp \overline{\zeta_i}}} \right),	&	\zeta \in \pm \gamma_i^*
\end{cases}
\end{equation}

We now consider Problem \ref{RHP1} with $\eps=-1$ and show that this problem is uniquely solvable. In this case we have
\begin{equation}
\label{br.-1}
\br(\zeta) = -\overline{r(\zetabar)}  \quad (\eps = -1).
\end{equation}
By Theorem \ref{thm:BC}, there exists a unique solution to 
Problem \ref{RHP1c} if and only if 
there exists a unique solution $\mu$ to the Beals-Coifman integral equation
\begin{equation}
\label{RHP1c.BC}\mu = I + \calC_w \mu 
\end{equation}
where, for a matrix-valued function $h \in L^2(\Sigma')$, 
\begin{equation}
\label{RHP1c.Cw}
	\calC_w h = C^+(h w_x^-) + C^-(h w_x^+) 
\end{equation}
and $C^\pm$ are Cauchy projectors for the augmented contour. The solution $\mu$ is related to the boundary values $M_\pm$ by the formulas
\begin{equation}
\label{RHP1c.M-to-mu}
\mu = M_+(I+w_x^+)^{-1} = M_-(I-w_x^-)^{-1}.
\end{equation}
We seek a solution $\mu$ with 
$\mu - I  \in L^2(\Sigma')$ and note that $\calC_w I$ belongs to $L^2(\Sigma')$ by the assumed properties of
the scattering data. Proposition 4.2 of \cite{Zhou89} shows that $I - \calC_w$ is a Fredholm operator of index $0$: thus, to show that the Beals-Coifman integral equation has a unique solution, it suffices to show that the only solution to the equation $\mu_0 = \calC_w \mu_0$ with $\mu_0 \in L^2(\Sigma')$ is the zero vector. Owing to Lemma \ref{lemma:bob}, it suffices to show that this is the case when $\eps = -1$. 

We now observe that the jump matrix \eqref{RHP1c.jump} has the following properties. For a matrix-valued function $C$ on $\Sigma'$, denote by $C^\dagger(z)$ the function $C^\dagger(z) = C(\zbar)^*$. First, 
$$
A = \left. v' \right|_{\R} = \begin{pmatrix}
									1+|r|^2 & r\\
									\overline{r} & 1
								\end{pmatrix}
$$
has the property  that $A + A^* \geq 0$. Second, we claim that for $\zeta \in \Sigma' \setminus \R$, $(v')^\dagger(\zeta) = v'(\zeta)$. For $\zeta \in \Gamma_i$ and $\zeta \in \Gamma_i^*$, this is an immediate consequence of the formulas \eqref{RHP1c.jump}. For $\zeta \in \Sigma \setminus \R = i\R$, we compute
from \eqref{br.-1} and $r(-\zeta) = -r(\zeta)$ that  $\br(\zeta) = \overline{r(\zeta)}$. Hence,
$r(\overline{\zeta})=-\overline{\br(\zeta)}$, $-\br(\overline{\zeta})=\overline{r(\zeta)}$ so that
$$
\left. v' \right|_{i\R} 
				= 	\begin{pmatrix}
						1+|r|^2	&	r\\
						\rbar		&1
					\end{pmatrix}
$$
and
$(v')^\dagger(\zeta)	= v'(\zeta)$. From \cite[Theorem 9.3]{Zhou89} we conclude:

\begin{proposition}
\label{prop:RHP1c.-} There exists a unique solution to Problem \ref{RHP1c} in the case $\eps=-1$.
\end{proposition}

Proposition \ref{prop:RHP1c.-} and Lemma \ref{lemma:bob}  together imply:

\begin{proposition}
\label{prop:RHP1c.+}  There exists  a unique solution to Problem \ref{RHP1c} in the case $\eps=+1$.
\end{proposition}

Propositions \ref{prop:RHP1c.-} and \ref{prop:RHP1c.+} together with Theorem \ref{thm:RHP.model.disc} now conclude the proof of Theorem \ref{thm:RHP1.unique}.

\begin{remark}
\label{rem:RHP1.sym}
We can now use the unique solvability of Problem \ref{RHP1} and equations \eqref{RHP1.sym1}--\eqref{RHP1.sym2} to conclude that the unique solution 
to Problem \ref{RHP1} possesses these symmetries. From the symmetries of $M$ we can deduce that the solution $\mu$ to the Beals-Coifman integral equation \eqref{BC.mu} obeys
\begin{align}
\label{BC.mu.sym1}
\mu(x,-\zeta) 		&= 	\begin{pmatrix}
									\mu_{11}(x,\zeta) 	&	-\mu_{12}(x,\zeta)	\\
									-\mu_{21}(x,\zeta)	&	\mu_{22}(x,\zeta)
								\end{pmatrix}\\[5pt]
\label{BC.mu.sym2}
\mu(x,\zetabar)	&=	\begin{pmatrix}
									\overline{\mu_{22}(x,\zeta)}	&	
									\eps \overline{\mu_{21}(x,\zeta)}	\\
									\eps \overline{\mu_{12}(x,\zeta)}	&
									\overline{\mu_{11}(x,\zeta)}
								\end{pmatrix}
\end{align}

\end{remark}

\subsection{Second Riemann-Hilbert Problem}
\label{sec:RHP2}

We now turn to the analysis of Riemann-Hilbert Problem \ref{RHP2}, again setting $t=0$ for convenience. In the introduction, we assumed that 
the given data $\rho$ belonged to the space $H^{2,2}(\R)$. Here we will begin with a slightly weaker assumption (which will be important at a key point in the proof of Theorem \ref{thm:Lip.I}--see Lemma \ref{lemma:Kx.small} and its proof). Let
\begin{equation}
\label{Y}
Y= \left\{ \rho \in L^{2,5/4}(\R): \rho' \in L^{2,3/4}(\R), \inf_{\lam \in \R} \left(1-\lam|\rho(\lam)|^2\right) > 0 \right\}.
\end{equation}   
We will connect Problems \ref{RHP1} and \ref{RHP2} under these assumptions by
setting 
\begin{equation}
\label{r.stringent}
r(\zeta) = \zeta\rho(\zeta^2).
\end{equation} 
A short computation shows that, if $r$ is given by\eqref{r.stringent} for $\rho \in Y$, then $r \in H^{1,1}(\Sigma)$.   On the other hand, for $\rho \in Y$, the jump matrix $J(\lam)$ in Problem \ref{RHP2}(iii) satisfies $J(\lam)-I \in L^2(\R) \cap L^\infty(\R)$. As we will see, the function $\rho$ becomes the scattering data in Problem \ref{RHP2}. 

We will prove:
\begin{theorem}
\label{thm:RHP2.unique}
Suppose that $\rho \in Y$ and that $\{\lam_j, C_j\}_{j=1}^N$ is a sequence from $\C^+ \times C^\times$. There exists a unique solution to Problem \ref{RHP2} with data 
$\calD = \left( \rho, \{\lam_j, C_j\}_{j=1}^N \right)$ for either sign of $\eps$.
\end{theorem}

Theorem \ref{thm:RHP2.unique} is a direct consequence of Propositions \ref{prop:RHP2.exist} and \ref{prop:RHP2.unique} below.

We begin by showing how to construct a solution of Problem \ref{RHP2} from the unique solution of Problem \ref{RHP1} for $r$ given by \eqref{r.stringent} and $\{\zeta_j,c_j\}$ given by the inverse map of \eqref{zeta->lambda}. From Remark \ref{rem:RHP1.sym}, it follows that the unique solution $M$ of Problem \ref{RHP1} takes the form
$$M(x,\zeta)  = 	\begin{pmatrix}
							M_{11}(x,\zeta)								&	 M_{12}(x,\zeta)	\\[5pt]
							\eps \overline{M_{12}(x,\zetabar)}	&	\overline{M_{11}(x,\zetabar)}
						\end{pmatrix}
$$
where
$$ M_{11}(x,-\zeta) = M_{11}(x,\zeta), \quad M_{12}(x,-\zeta) = -M_{12}(x,\zeta). $$
Let $\varphi_\zeta$ be the automorphism
$$ \varphi_\zeta 
			\begin{pmatrix} 
					a  & b \\ c & d 
			\end{pmatrix} = 
			\begin{pmatrix} 
				a & \zeta^{-1}b \\ \zeta c & d 
			\end{pmatrix} 
$$
and define a $2 \times 2$ matrix-valued function $N(x,\lam)$ by
\begin{equation}
\label{M-to-N}
N(x,\zeta^2) = \varphi_\zeta \left( M(x,\zeta) \right)
\end{equation}
which is well-defined because the right-hand side is an even function of $\zeta$. If $r$ is an odd function on $\Sigma$ then
$r(\zeta)/\zeta$ induces an even function on $\R$ by the formula
\begin{equation}
\label{r-to-rho}
\rho(\zeta^2) = \zeta^{-1} r(\zeta)
\end{equation}
which motivates the assumption \eqref{r.stringent}. 
The map $\zeta \mapsto \zeta^2$ maps $\Sigma$ onto $\R$ and $\calZ^{++}$ onto the set $\Lam = \{ \lam_i \}_{i=1}^N$ where $\lam_i = \zeta_i^2$. The set $\calZ$ is mapped onto the set of the $\lam_i$ and their complex conjugates. A computation shows that the row-vector valued function 
\begin{equation}
\label{N-to-n}
n(x,\lam) = \left( N_{11}(x,\lam), N_{12}(x,\lam) \right) 
\end{equation}
obeys Problem \ref{RHP2}, where $\rho$ and $\lam_i$ have already been defined and $C_j = 2c_j$.  The factor of $2$ relating $c_j$ and $C_j$ comes from the quadratic change of variables.  Corresponding to the assumption \eqref{rr} on $r$, we assume that \eqref{rho.con} holds for a positive constant $c$.

\begin{center}
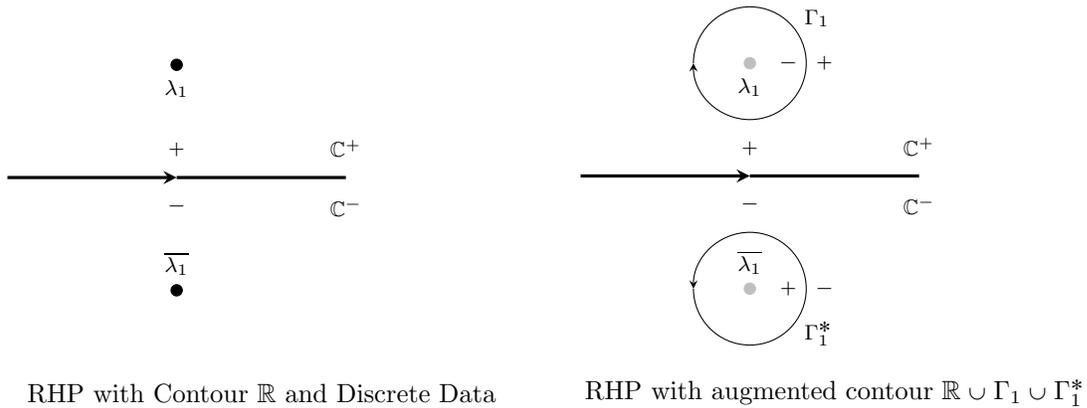
\begin{figure}[H]
\caption{Contours for Problem \ref{RHP2} and Problem \ref{RHP2c}
\label{fig:RHP2}
}

\bigskip

\begin{minipage}{0.45\textwidth}
\begin{tikzpicture}[scale=0.75]
\draw[very thick,->,>=stealth] (-3,0) -- (0,0);
\draw[very thick]				(0,0) -- (3,0);
\node[above] at (0,0.2) 					{\footnotesize{$+$}};
\node[below] at (0,-0.2)					{\footnotesize{$-$}};
\draw[black,fill=black] 		(0,2)			circle(0.1cm);
\draw[black,fill=black]		(0,-2)			circle(0.1cm);
\node[below] at				(0,1.9)		{\footnotesize{$\lam_1$}};
\node[above] at 				(0,-1.9)		{\footnotesize{$\overline{\lam_1}$}};
\node[above] at 				(3,0.2)		{\footnotesize{$\C^+$}};
\node[below] at				(3,-0.2)		{\footnotesize{$\C^-$}};
\end{tikzpicture}
\end{minipage}
\qquad
\begin{minipage}{0.45\textwidth}
\begin{tikzpicture}[scale=0.75]
\draw[very thick,->,>=stealth] (-3,0) -- (0,0);
\draw[very thick]						(0,0) -- (3,0);
\node[above] at (0,0.2) 					{\footnotesize{$+$}};
\node[below] at (0,-0.2)					{\footnotesize{$-$}};
\draw[gray!50,fill=gray!50] 		(0,2)		circle(0.1cm);
\draw[gray!50,fill=gray!50]		(0,-2)		circle(0.1cm);
\node[below] at				(0,1.9)		{\footnotesize{$\lam_1$}};
\node[above] at 				(0,-1.9)		{\footnotesize{$\overline{\lam_1}$}};
\draw[->,>=stealth]			(1,2)			arc(360:180:1cm);
\draw								(-1,2)			arc(180:0:1cm);
\node[right] at					(.8,2.8)		{\footnotesize{$\Gamma_1$}};
\node[left] at 					(1,2)			{\footnotesize{$-$}};
\node[right] at 				(1,2)			{\footnotesize{$+$}};
\draw[->,>=stealth]			(1,-2)			arc(0:180:1cm);
\draw								(-1,-2)		arc(180:360:1cm);
\node[right] at					(.8,-2.8)		{\footnotesize{$\Gamma_1^*$}};
\node [left] at					(1,-2)			{\footnotesize{$+$}};
\node [right] at 				(1,-2)			{\footnotesize{$-$}};
\node[above] at 				(3,0.2)		{\footnotesize{$\C^+$}};
\node[below] at				(3,-0.2)		{\footnotesize{$\C^-$}};
\end{tikzpicture}
\end{minipage}

\bigskip

\begin{minipage}{0.45\textwidth}
\begin{center}
RHP with Contour $\R$ and Discrete Data
\end{center}
\end{minipage}
\qquad
\begin{minipage}{0.45\textwidth}
\begin{center}
RHP with augmented contour
$\R \cup \Gamma_1 \cup \Gamma_1^*$
\end{center}
\end{minipage}
\end{figure}
\end{center}

We also record the equivalent Riemann-Hilbert problem in which the discrete data are replaced by new contours and associated jump conditions. We denote by $\Gamma_i$ and $\Gamma_i^*$ small circular contours centered respectively at $\lam_i$ and $\overline{\lam_i}$ (see Figure \ref{fig:RHP2}). Finally, we set $\Lambda' = \bigcup_{i=1}^N \left( \Gamma_i \cup \Gamma_i^* \right)$. 

\begin{RHP}
\label{RHP2c}
Given $\rho$ satisfying \eqref{r.stringent} and  \eqref{rho.con} and  a sequence $\{ \lam_k,C_k \}_{k=1}^N$ from $\C^+ \times \C^\times$, 
find a row vector-valued function $n(x,\dotarg)$ on  $\C \setminus \left(\R \cup \Lam' \right)$ with the following properties:
\begin{itemize}
\item[(i)]		$n(x,z)$ has continuous boundary values on $\R$ with 
$$n_\pm(x,\dotarg) - (1,0) \in \dee C(L^2),$$ 
and 
$n_+(x,\lam) = n_-(x,\lam) e^{-i\lam x \ad(\sigma)} V(\lam)$ where
$$
V(\lam) = 	\begin{pmatrix}
						1-\rho(\lam) \brho(\lam)		&	\rho(\lam)	\\
						-\lam \brho(\lam)						&  1
					\end{pmatrix}
$$
and $\brho(\lam) = \eps \overline{\rho(\lam)}$.
\item[(ii)]	$n(x,z)$ has continuous boundary values on each $\Gamma_i$, $\Gamma_i^*$ and 
for $\lam \in \Gamma_i$ (resp.\ $\lam \in \Gamma_i^*$) the jump relation
$ n_+(x,\lam) = n_-(x,\lam) V(\lam)$ holds, where
$$ \left. V(\lam)\right|_{\Gamma_k} =
\begin{pmatrix}
1										&	0	\\
\frac{\lam_k C_k}{\lam - \lam_k}	&	1
\end{pmatrix},
\quad
\text{resp.\ }
\quad
\left. V(\lam) \right|_{\Gamma_k^*} =
\begin{pmatrix}
1	&	\frac{-\eps \overline{C_k}}{\lam - \overline{\lam_k}}	\\
0	&	1
\end{pmatrix}
$$
\end{itemize}
\end{RHP}

\begin{remark}
To formulate the Beals-Coifman integral equations for Problem \ref{RHP2}, we use the factorization 
$V(\lam)=(I-W^-)^{-1}(I+W^+)$ where 
$$
(W^-,W^+)=
\begin{cases}
\left( \begin{pmatrix}
			0 & \rho \\ 0 & 0
		\end{pmatrix},
		\begin{pmatrix}
			0 & 0 \\ -\lam \brho & 0
		\end{pmatrix}
\right), & \lam \in \R\\
\\
\left( \begin{pmatrix}	
			1 & 0 \\ 0 & 1\\
		\end{pmatrix},
		\begin{pmatrix}
			0	&	0	\\
			\frac{\lam_k C_k}{\lam-\lam_k}	&	0
		\end{pmatrix}
\right), & \lam \in \Gamma_k\\
\\
\left(	\begin{pmatrix}
			0	&	\frac{-\eps \overline{C_k}}{\lam-\overline{\lam_k}}	\\
			0	&	0
		\end{pmatrix},
		\begin{pmatrix}
			1	&	0	\\	0	&	1
		\end{pmatrix}
\right), & \lam \in \Gamma_k^*
\end{cases}
$$
\end{remark}

\begin{lemma}
\label{lemma:RHP2.sol}
Any solution of Problem \ref{RHP1} with $r$ given by \eqref{r.stringent} and $\rho$ satisfying \eqref{rho.con}  induces a solution of Problem \ref{RHP2} via the formulas
\begin{equation}
\label{M-to-n}
(n_1(x,z^2),n_2(x,z^2) = \left(M_{11}(x,z), z^{-1} M_{12}(x,z) \right).
\end{equation}
\end{lemma}

\begin{proof} 
It is easy to see that the jump relations and residue condition in Problem \ref{RHP1} satisfied by $M$ imply that the jump relations and residue conditions for Problem \ref{RHP2} are satisfied by $n$ as defined by \eqref{M-to-n}.  It remains to show that $n$ satisfies the condition $n_\pm(x,\dotarg) - (1,0) \in \dee C(L^2)$.

To this end we show that a solution $\left(\mu_{11}(x,\zeta),\mu_{12}(x,\zeta) \right)$ of the Beals-Coifman integral equations for Problem \ref{RHP1c}  (see Appendix \ref{app:RHP1c.BC})
induces a solution $\left(\nu_{11}(x,\lam),\nu_{12}(x,\lam) \right)$ of the Beals-Coifman integral equations for Problem \ref{RHP2} (see Appendix \ref{app:RHP2.BC})   via the formulas
\begin{equation}
\label{mu-to-nu}
\nu_{11}(x,\zeta^2) = \mu_{11}(x,\zeta), \quad \nu_{12}(x,\zeta^2) = \zeta^{-1} \mu_{12}(x,\zeta). 
\end{equation}
We then show that $\nu_{11}(x,\dotarg)-1$ and $\nu_{12}(x,\dotarg)$ belonging to $L^2(\R \cup \Lambda')$
Given these properties of $\nu$, one can express
$$ n^\pm - I = C^\pm \left( \nu (W_x^+ + W_x^-) \right), $$
(where $\nu = (\nu_{11}, \nu_{12})$). One can then use,the facts that $W_x^\pm \in L^\infty\R \cup (\Lambda') \cap L^2(\R \cup \Lambda')$ and that $\nu_{11}(x,\dotarg) -1$ and $\nu_{12}(x,\dotarg) \in L^2(\Lambda')$ 
to deduce that $n^\pm - (1,0) \in \dee C(L^2)$ as required.

To show that the functions $\nu_{11}$ and $\nu_{12}$ defined by \eqref{mu-to-nu} have the required properties, we show that the integral equations \eqref{RHP1.11.AI} -- \eqref{RHP1.12.AI.disc} for $(\mu_{11},\mu_{12})$ imply the integral equations \eqref{RHP2.11}--\eqref{RHP2.12.disc} (recall that $\mu_{11}(x,\dotarg)$ is analytic in the discs enclosed by $\gamma_j^*$, while $\mu_{12}(x,\dotarg)$ is analytic in the discs enclosed by $\gamma_j$, so that the Cauchy integrals in \eqref{RHP1c.11.Sig}--\eqref{RHP1c.12.Gam*} are actually equal to their discrete counterparts in \eqref{RHP1.11.AI}--\eqref{RHP1.12.AI.disc}). In the discrete sums, we use the formulae
\begin{align*}
\frac{1}{\zeta - \zeta_j} - \frac{1}{\zeta+\zeta_j} = \frac{2\zeta_j}{\zeta^2 - \zeta_j^2},
\qquad
\frac{1}{\zeta - \zeta_j} + \frac{1}{\zeta+\zeta_j} = \frac{2\zeta}{\zeta^2 - \zeta_j^2}.
\end{align*}
In the terms involving Cauchy projectors for $\Sigma$, 
we use formula 
\begin{equation}
\label{Cpm-Sigma-to-R}
(C_\Sigma^\pm f)(\zeta) = \left(C^\pm_\R g)\right(\zeta^2) + 
			\zeta \left( C^\pm_\R h \right)(\zeta^2).
\end{equation}
valid for $f \in H^{1,1}(\Sigma)$, where
$$
g(u) = \frac{1}{2} \left(  f(\sqrt{u}) + f(-\sqrt{u}) \right), \quad 
h(u) =\frac{\left(  f(\sqrt{u})  -  f(-\sqrt{u} )  \right)}{2\sqrt{u}}
$$ 
(see  \S 2.4 of Paper I for discussion and references). We also exploit the fact that $\mu_{11}(x,\dotarg) - 1$ and $\mu_{12}(x,\dotarg)$ belong to $H^1(\Sigma)$ while $r$ and $\br$ belong to $H^{1,1}(\Sigma)$,  together with the facts that $\mu_{12}(x,\zeta) \br_x(\zeta)$ is an even function of $\zeta \in \Sigma$ while
$\mu_{11}(x,\zeta) r_x(\zeta)$ is an odd function of $\zeta$ (see \eqref{RHP1.AI.sym1}--\eqref{RHP1.AI.sym2} and recall that $r$ and $\br$ are odd functions).
\end{proof}

From Theorem \ref{thm:RHP1.unique} and Lemma \ref{lemma:RHP2.sol}, we have:

\begin{proposition}
\label{prop:RHP2.exist}
Let $\left( \rho, \{ \lam_k, C_k \}_{k=1}^N \right) \in Y \times \left(\C^+ \times \C^\times\right)^N$.
Then, there exists a solution to Problem \ref{RHP2}.
\end{proposition}

Given the existence of a solution to Problem \ref{RHP2}, it is easy to prove that the solution is unique.

\begin{proposition}
\label{prop:RHP2.unique}
Suppose that $n$ and $n^\sharp$ are two rwo vector-valued solutions to Problem \ref{RHP2} where $r$ is given by \eqref{r.stringent} and $\rho$ satisfies \eqref{rho.con}. Then $n=n^\sharp$.
\end{proposition}

\begin{proof}
Extend $n$ and $n^\sharp$ to matrix-valued functions $N$ and $N^\sharp$ by setting
$$
N(x,\lam)
=
\begin{pmatrix}
	n_{1}(x,\lam)		&	n_{2}(x,\lam)	\\[5pt]
	\eps \lambar \overline{n_{2}(x,\lambar)} & \overline{n_{1}(x,\lambar)}
\end{pmatrix}
$$
and similarly for $n^\sharp$. 
A short computation shows that $N$ and $N^\sharp$ obey the jump relations (i) and residue relations (ii). We claim that
(1) $\det N(x,\lam) = \det N^\sharp(x,\lam) = 1$, (2) $B(x,\lam)$ is holomorphic near each $\lam_k$ and $\newtext{\overline{\lam_k}}$, and (3) the function $B(x,\lam)=N(x,\lam) N^\sharp(x,\lam)^{-1}$ has no jump across $\R$,
and . If so then $B(x,\lam)$ is an entire matrix-valued function with $\lim_{|\lam| \rarr \infty} B(x,\lam)=I$, and hence $B(x,\lam)=I$, proving that $N=N^\sharp$. Thus it suffices to establish these claims.

To establish the first claim, we can use the formulation in Problem \ref{RHP2c} to 
study the determinant 
of the solution to the contour problem.
Since $\det V(\lam) = 1$, it follows that the determinant
of $N$ has no jumps across the contours so that the determinant is holomorphic. Next, we claim that
\begin{equation}
\label{N.det}
\lim_{|z| \rarr \infty} N^*(x,z) = 
\begin{pmatrix}
1	&	0	\\
c^*(x)	&	1
\end{pmatrix},
\end{equation}
where $N^*$ is either $N$ or $N^\sharp$
for a bounded function $c^*(x)$ (we establish this claim at the end of the proof).  If so, it follows from Liouville's theorem that  $\det N(x,\lam)=\det N^\sharp(x,\lam)=1$ so that the function $B(x,\lam)$ is well-defined.  

Second, a short computation using the residue conditions shows that $B(x,\lam)$ has no singularity at the points $\lam_k$ and $\overline{\lam_k}$ .

Third, the continuity of $B(x,\lam)$ across the real axis is an immediate consequence of the fact that $N$ and $N^\sharp$ obey the same jump relation there. 

This proves the required uniqueness, modulo the claim \eqref{N.det}.
Observe that
\begin{align*}
N(x,z)	&=	I+ \frac{1}{2\pi i} 
						\int_\R \frac{N_-(x,\lam)(V(\lam,x)-I)}{\lam-z}	\, d\lam\\
			&\quad
					+	\sum_j	\frac{1}{z-\lam_j} 
											\twomat{\alpha_j}{ 0}{\beta_j }{0 }
					+  \sum_j 	\frac{1}{z-\overline{\lam_j}} 
											\twomat{0}{\eps\overline{\beta_j}}
														{0}{\overline{\alpha_j}}  	
\end{align*}
for $\alpha_j$ and $\beta_j$ independent of $z$ (but possibly depending on $x$), so that in particular
\begin{align*}
N_{11}(x,z) 	&= 1 + \bigO{z^{-1}}, \\
z N_{12}(x,z) 	&= -\int_\R N_-(x,\lam)(V(\lam,x)-I) \, d\lambda
 							- \sum_j  \eps \overline{\beta_j} + o(1).
\end{align*}
Since $N_{12}(x,z) = \eps z \overline{N_{21}(x,\zbar)}$ and $N_{22}(x,z) = \overline{N_{11}(x,\zbar)}$, this proves the required asymptotics \eqref{N.det}.
\end{proof}

\subsection{Lipschitz Continuity of the Inverse Scattering Map}
\label{sec:Lip}

In this section we analyze the Beals-Coifman integral equations \eqref{RHP2c.11.R}--\eqref{RHP2c.12.Lam*} in order to prove the following continuity result about the inverse scattering map. In most of this section, we will assume that $\rho \in H^{2,2}(\R)$, and at a key point we exploit the fact that bounded subsets of $H^{2,2}(\R)$ are precompact in the space $Y$
considered in Section \ref{sec:RHP2}.   Recall Definition \ref{def:V.bounded} which defines a bounded subset of $V_N$.
\begin{theorem}
\label{thm:Lip.I}
The inverse scattering map
\begin{align*}
\calI:  V_{N}  										&\rarr 			U_{N}\\
\left(\rho,\{\lam_j,C_j\}_{j=1}^N \right)	&\mapsto		q
\end{align*}
is Lipschitz continuous as a map from $V_{N}$ to $H^{2,2}(\R)$, and uniformly Lipschitz on bounded subsets of $V_N$. Moreover,  $\calI: V_N \rarr U_N$, $\calI \circ \calR$ is the identity on $U_{N}$ and $\calR$ is one-to-one as a map from $U_{N}$ onto $V_{N}$. 
\end{theorem}

The complete proof of this theorem follows in outline the proof of Theorem 1.9 in section 6 of Paper I.  That is,
\begin{itemize}
\item[(1)]	We prove that the ``right'' Riemann-Hilbert problem, Problem \ref{RHP2} and the reconstruction formula \eqref{q.zeta} yield a Lipschitz continuous map from
$V_{N}$ to $H^{2,2}(a,\infty)$ for any $a \in \R$.
\item[(2)]	We prove that an analogous ``left'' Riemann-Hilbert problem and its reconstruction formula yield a Lipschitz continuous map from $V_{N}$ to $H^{2,2}(-\infty,a)$ for any $a \in \R$.
\item[(3)]	We show that the left- and right-hand reconstructions of $q$ from the same scattering data coincide on $(-a,a)$ for any $a \in \R$. 
\item[(4)]	We show that $\calI \circ \calR$ is the identity on $U_{N}$ 
\end{itemize}

Here we will show how the analysis of Paper I, section 6 must be modified to accommodate solitons in order to carry out step (1). The modifications to step (2) are analogous and the argument for step (3) is identical to that given in Paper I, Lemma 6.16. The argument for step (4) rests (as is standard) on the uniqueness of Beals-Coifman solutions, compare Proposition 6.17 in Paper I.  Thus, in what follows, we will discuss the details of step (1) only. In Appendix \ref{app:left}, we derive the left-hand Riemann-Hilbert problem. 

We will prove the Lipschitz continuity using equations \eqref{RHP2.11}--\eqref{RHP2.12.disc} and the reconstruction formula \eqref{RHP2.nu.q.recon}. In what follows, we will first
make a reduction of the integral equations. Next, we will prove continuity of the solution and its $x$ and $\lam$ derivatives in $\calD$. We will use this continuity and the reconstruction formula \ref{RHP2.nu.q.recon} to prove Lipschitz continuity of the inverse
scattering map.

First, we reformulate \eqref{RHP2.11}--\eqref{RHP2.12.disc}. 
Let 
\begin{equation}
\label{X}
X = \left(\C \oplus L^2(\R) \right) \oplus \C^N 
\end{equation}
and regard
\begin{equation}
\label{nuflat}
\nu^\flat = 
\left( 
	\nu_{11}(x,\lam), \{ \nu_{11}(x,\overline{\lam_j})\}_{j=1}^N 
\right) 
\end{equation}
as an element of $X$ for each $x$. 
Iterating equations \eqref{RHP2.11}--\eqref{RHP2.12.disc} we find 
\begin{equation}
\label{nuflat.int}
\nu^\flat = \mathbf{e} + \calK_x \nu^\flat
\end{equation}
where $\mathbf{e}= (1,\{ 1,\dots, 1\}) \in X$ and $\calK_x : X \rarr X$ is defined as follows. For $(v,h) \in X$
we have 
\begin{equation}
\label{calK}
\calK_x 
\begin{pmatrix} v \\ h \end{pmatrix} = 
\begin{pmatrix}\calK_{00} & \calK_{01} \\ \calK_{10} & \calK_{11} \end{pmatrix} \begin{pmatrix} v \\ h \end{pmatrix}
\end{equation}
where
\begin{align}
\label{K00}
(\calK_{00} v)(\lam)
	&=	\left(S_x v\right)(\lam)
		+ \sum_j 	\frac{C_{j,x} \lam_j}{\lam-\lam_j} \,
								C_\R
									\left(
										\rho_x(\dotarg) v(\dotarg)
									\right)(\lam_j)	\\
\label{K01}
(\calK_{01} h)(\lam) 
	&=	- \sum_j h_j \, \overline{C_{j,x}} \,
								C^-
									\left[ 
										\frac{(\dotarg)\overline{\rho_x(\dotarg)}}
											{\dotarg - \lam_j} 
									\right](\lam)
		+ \sum_{j,k}	h_k\, \frac{\eps \lam _j C_{j,x} \overline{C_{k,x}}}
											{(\lam-\lam_j)(\lam_j - \overline{\lam_k})}	\\
\label{K10}
(\calK_{10} v)_i
	&=	-\eps C_\R \left[ (\dotarg) \overline{\rho_x(\dotarg)} 
									C^+\left( \rho_x(\diamond) v(\diamond) \right)
										(\dotarg)
							\right](\overline{\lam_i})
			+	\sum_j	\frac{C_{j,x} \lam_j}{\overline{\lam_i}-\lam_j}
									C_\R \left(\rho_x(\dotarg) v(\dotarg)\right)(\lam_j) \\
\label{K11}
(\calK_{11} h)_i
	&=
		- 	\sum_{j}  h_j \overline{C_{j,x}} 
									C_\R 
										\left[ \frac{(\dotarg) 
												\overline{\rho_x(\dotarg)}}
												{\dotarg - \overline{\lam_j}}
										\right] (\overline{\lam_i})
		+	\sum_{j,k}	\frac{\eps \lam_j C_{j,x} \overline{C_{k,x}}}
											{\overline{\lam_i}-\lam_j} h_k	
\end{align}
In \eqref{K00}
\begin{equation}
\label{Sx}
(S_x h)(\lam) = -\eps C^-
	\left[ 
				(\dotarg) \overline{\rho_x(\dotarg)} 
				\, C^+\left(\rho_x (\diamond) h(\diamond)\right) 
	\right] (\lam),
\end{equation}
\begin{equation}
\label{rhoCx}
\rho_x(\lam) = e^{-2i\lam x} \rho(\lam), \quad C_{j,x} = C_j e^{2i\lam_j x}, 
\end{equation}
and all of the remaining operators are of finite rank. We will usually suppress the dependence of $\calK_x$ on the scattering data $\calD = \left(\rho, \{\lam_j, C_j\}_{j=1}^N\right)$ but sometimes write $\calK_x(\calD)$ when needed to make the dependence explicit.

Setting $\nu^\sharp = \nu^\flat - \mathbf{e}$ and 
\begin{equation}
\label{Xsharp}
X^\sharp = L^2(\R)\oplus \C^N
\end{equation}
we obtain the equation
\begin{equation}
\label{nusharp.int}
\nu^\sharp = \calK_x \mathbf{e} + \calK_x \nu^\sharp
\end{equation}
We will use this equation to study Lipschitz continuity of $\nu$ in $\calD$, and use the reconstruction formula \eqref{RHP2.nu.q.recon} to prove continuity of $q$ in $\calD$. To see what estimates are needed, we write
\begin{equation}
\label{nuflat.to.nu0}
\nu^\flat = \left(1+\nu_0, \left\{ \nu_j(x) \right\}_{j=1}^N \right)
\end{equation}
and see that the reconstruction formula is
$q(x) = q_1(x) + q_2(x) + q_3(x)$ where
\begin{gather*}
q_1(x)	
	=	\frac{1}{\pi} \int_\R e^{-2i\lam x} \rho(\lam) d \lam, \quad 
q_2(x)	
	=	\frac{1}{\pi} \int_\R e^{-2i\lam x} \rho(\lam) \nu_0(x,\lam) \, d\lam,\\
q_3(x)
	=	\sum_{j=1}^N 2i\eps \overline{C_{j,x}} \nu_j(x)
\end{gather*}

First, the map $\rho \mapsto q_1$ is Lipschitz from $H^{2,2}(\R)$ to itself owing to mapping properties of the Fourier transform. 

Second, to show that $\calD \mapsto q_2$ 
is Lipschitz continuous from $V_N$ to $H^{2,2}([a,\infty))$, we need estimates on $\nu_0(x,\lam)$ and its derivatives in the space $L^2((a,\infty) \times \R, dx \, d\lam)$.
The formulas
\begin{align*}
x^2 q_2(x)	
	&=	\frac{1}{4\pi} \int_\R e^{-2i\lam x} (\rho {\nu_0})_{\lam\lam} (x,\lam) \, d\lam,\\
q_2''(x)
	&=	\frac{1}{\pi} 
				\int_\R e^{-2i\lam x} \left( -4\lam^2 \rho(\lam) \nu_0(x,\lam)\right)  \, d\lam \\
	&\quad + \frac{1}{\pi} 
				\int_\R \bigl( -2i\lam\rho(\lam) (\nu_0)_x(x,\lam)+ \, \rho(\lam) (\nu_0)_{xx}(x,\lam)\bigr) \, d\lam,
\end{align*}
imply that, in order to show that $\calD \mapsto q_2$ is Lipschitz continuous from $V_N$ to $H^{2,2}(\R)$, it suffices to show that
\begin{equation}
\label{nu.todo}
\nu_0, \quad (\nu_0)_x, \quad (\nu_0)_{xx}, \quad (\nu_0)_\lam,
\text{ and }\langle \lam \rangle^{-1} (\nu_0)_{\lam\lam}
\end{equation}
are Lipschitz continuous from $V_N$ to $L^2((a,\infty) \times \R)$. 

Third, the function $q_3(x)$ is a sum of decaying exponential functions of $x$ as $x \rarr \infty$ multiplied by the functions $\nu_j(x)$.  Thus, to show that $x^2 q_3(x)$ and $(q_3)_{xx}(x)$ belong to $L^2(a,\infty)$, it suffices to show that the functions $\nu_j(x)$ and their second derivatives in $x$ are bounded. 

In order to prove these statements, it sufffices to prove that:
\begin{enumerate}
\item[(1)]	The resolvent $(I - \calK_x)^{-1}$ is a bounded operator from $X^\sharp$ to itself uniformly in $x \in (a,\infty)$ so that it may be ``lifted'' to a bounded operator from $L^2((a,\infty),X^\sharp)$ to itself
\item[(2)]	The inhomogeneous term $\calK_x \mathbf{e}$ together with appropriate derivatives in $x$ and $\lam$ belong to the space $L^2((a,\infty),X^\sharp)$, and
\item[(3)]	The integral equation \eqref{nuflat.int} can be differentiated in $x$ and $\lam$ to obtain estimates on $x$ and $\lam$ derivatives of $\nu_0$ and $x$ derivatives of $\nu_j$.
\end{enumerate}

We will prove the following propositions in three subsections.  In what follows, we denote by $W$ a bounded subset of $V_N$ in the sense of Definition \ref{def:V.bounded}, and by $\calD$ a generic element of $W$.

\begin{proposition}
\label{prop:nu.L2}
The unique solution $\nu^\sharp$ of \eqref{nusharp.int}
satisfies
$$\norm[L^2([a,\infty),X^\sharp) \cap L^\infty([a,\infty),X^\sharp)]{\nu^\sharp} \lesssim 1$$ 
with estimates uniform in $\calD \in W$.
\end{proposition}

\begin{proposition}
\label{prop:nu.lam}
The estimates
$$
\norm[L^2([a,\infty) \times \R)]{{\nu_0}_\lam} +
\norm[L^2([a,\infty) \times \R)]
		{\langle \diamond \rangle^{-1}
			{\nu_0}_{\lam\lam}(\dotarg,\diamond)}
\lesssim 1
$$
and 
$$
\norm[L^\infty([a,\infty) \times \R)]{{\nu_0}_\lam} +
\norm[L^\infty([a,\infty) \times \R)]
		{\langle \diamond \rangle^{-1}
			{\nu_0}_{\lam\lam}(\dotarg,\diamond)}
\lesssim 1
$$
holds uniformly for $\calD \in W$.  Moreover, the maps 
$\calD \rarr \nu_\lam$ and $\calD \rarr \langle \dotarg \rangle^{-1}\nu_{\lam\lam}$ are Lipschitz continuous.
\end{proposition}

\begin{proposition}
\label{prop:nu.x}
Denote by $\nu^\sharp(x,\dotarg)$ the unique solution to \eqref{nusharp.int}. Then $\nu^\sharp$, $(\nu^\sharp)_x$,
and $(\nu^\sharp)_{xx}$ all belong to $L^2([a,\infty),X^\sharp)$
with norms bounded uniformly in $\rho$ in a fixed bounded subset of $H^{2,2}(\R)$ with $(1-\lam|\rho(\lam)|^2 \geq c>0$
for a fixed $c$, and $\{ \lam_j, C_j\}$ in a fixed compact subset of $(\C^+ \times \C^\times)^N$. 
\end{proposition}

\begin{proof}[Proof of Theorem \ref{thm:Lip.I}, given Propositions \ref{prop:nu.L2}--\ref{prop:nu.x}]
Recall from the discussion following the statement of Theorem \ref{thm:Lip.I} that we give details only for Step (1) of the four steps since step (2) is similar to step (1) while steps (3) and (4) follow closely the argument of Paper I. 

To carry out step (1), we note that Propositions \ref{prop:nu.L2}--\ref{prop:nu.x} immediately imply that the functions \eqref{nu.todo} are Lipschitz in the scattering data and that the functions $\nu_j^*$ are bounded Lipschitz functions of the scattering data. This gives the necessary control to prove that $q_2$ and $q_3$ in \eqref{RHP2.nu.q.recon} have the required continuity properties as maps from $W$ to $H^{2,2}(\R)$.
\end{proof}

\subsubsection{Resolvent Estimates, $L^2$ Estimates on $\nu^\sharp$}

We begin the study of \eqref{nusharp.int}. It is easy to see that $\calK_x \mathbf{e} \in X^\sharp$ if $\rho \in H^{2,2}(\R)$.  Thus, to solve \eqref{nusharp.int}  for $\nu^\sharp$, we need to show that $(I-\calK_x)^{-1}$ exists. First, we observe:

\begin{lemma}
\label{lemma:Kx.large}
The operator $I-\calK_x : X^\sharp \rarr X^\sharp$ is Fredholm and
$$\lim_{x \rarr \infty} \sup_{\calD \in W} \norm[X^\sharp \rarr X^\sharp]{\calK_x} = 0$$ where $\calD$ in a bounded subset $W$ of $V_N$.
\end{lemma}

\begin{proof}
The operator $\calK_x$ is a finite-rank perturbation of the operator
$S_x:L^2(\R) \rarr L^2(\R)$ given by \eqref{Sx}. The operator $\calS_x$
was shown to be compact in Paper I, Lemma 6.7. In the same Lemma we showed that $\norm[L^2 \rarr L^2]{S_x} = 0$ uniformly for $\rho$ in a bounded subset of $H^{2,2}(\R)$. Coefficients of all remaining terms have exponential decay at least $\bigO{e^{-d_\Lambda x/2}}$ as $x \rarr \infty$ owing to factors $C_{j,x}$ or their complex conjugates. This decay rate is uniform in bounded subsets of $V_N$ since $d_\Lambda$ has a fixed lower bound on such sets. 
\end{proof}

We seek an estimate on $\norm[X^\sharp \rarr X^\sharp]{(I-\calK_x)^{-1}}$ uniform in $\calD$ in a bounded subset of $V_N$ and $x \in [a,\infty)$ for any fixed $a$. It follows from Lemma \ref{lemma:Kx.large} that, given a bounded subset $W$ of $V_N$, there is an $R>0$ depending on $W$ so that 
$$ \sup_{\calD \in W, x \geq R} \norm[X^\sharp \rarr X^\sharp]{\calK_x}  < 1/2 $$
so that
\begin{equation}
\label{RKx.est}
\sup_{\calD \in W, x \geq R} \norm[X^\sharp \rarr X^\sharp]{(I-\calK_x)^{-1}} \leq 2. 
\end{equation}
Thus, to obtain a uniform resolvent bound, it suffices to estimate the resolvent
for $\calD \in W$ and $a \leq x \leq R$ for a fixed $a$. To do so we exploit the following  facts:
\begin{enumerate}
\item[(1)]  For data $\calD \in Y \times (\C^+ \times \C^\times)^N$, the operator $\calK_x$ is bounded from $X$ to itself
\item[(2)] the resolvent exists for $\calD = \left(\rho,\{\lam_j,C_j\}\right) \in Y \times (\C^+ \times \C^\times)^N$ and the map
$$ \left(\rho,\{\lam_j,C_j\}_{j=1}^N\right) \mapsto (I-\calK_x)^{-1} $$
is continous 
\item[(3)] If $W$ is a bounded subset of $V_N$, then $W$ is compactly embedded in
$Y \times (\C^+ \times \C^\times)^N$
\end{enumerate}
Together, these three facts imply that the image of $W \times [a,r]$ in $\calB(X^\sharp \rarr X^\sharp)$ is a precompact set, hence bounded, for any $-\infty< a < r < \infty$.

We will prove:
\begin{lemma}
\label{lemma:Kx.small}
For any bounded subset $W$ of $V_N$ and any $-\infty < a < r < \infty$,
$$\sup_{(\calD,x) \in W \times [a,r]} \norm[X^\sharp \rarr X^\sharp]{(I-\calK_x)^{-1}} \lesssim 1 $$
\end{lemma}

\begin{proof}
We check assertions (1)--(3) above. 

(1) For $\rho \in Y$ it is immediate that $\rho, \lam \rho \in L^2(\R)$. To see that 
$\rho, \lam \rho \in L^\infty(\R)$ we use the conditions on $\rho$ to show that
$(\rho^2)'$ and $(\lam^2 \rho^2)'$ are integrable. These estimates suffice to prove the claimed operator bounds from
\eqref{K00}--\eqref{K11}.

(2) Existence follows from Lemma \ref{lemma:Kx.large}, the uniqueness Theorem \ref{thm:RHP2.unique}, and Fredholm theory.

(3) The compactness now follows from the compact embedding of $Y$ in $H^{2,2}(\R)$ and the fact that the $\lam_j, C_j$ that occur in a bounded subset of $V_n$ run through compact subsets of $\C^+$ and $\C^\times$.
\end{proof}

For $\calD = \left( \rho, \{\lam_j, C_j \}_{j=1}^N \right)$, define
\smallskip
$$ \norm[V_N]{\calD} = \norm[H^{2,2}(\R)]{\rho} + \sup_j |\lam_j| + \sup_j |C_j|. $$
Although $V_N$ is not complete in this norm (it does not control lower bounds on $C_j$ and $|\imag \lam_j|$) it will suffice for the continuity estimates we need.

\begin{lemma}
\label{lemma:res}
Let $W$ be a bounded subset of $V_N$.
The resolvent $(I-\calK_x)^{-1}$ satisfies the uniform estimate
$$ \sup_{(x,\calD) \in [a,\infty) \times W} \norm[X^\sharp \rarr X^\sharp]{(I-\calK_x)^{-1}} \lesssim 1 $$
and the Lipschitz estimate
$$ \norm[X^\sharp \rarr X^\sharp]{(I-K_x(\calD_1))^{-1} - (I-\calK_x(\calD_2))^{-1}}
\lesssim \norm[V_N]{\calD_1 - \calD_2}
$$
with uniform Lipschitz constant for $\calD_1, \calD_2 \in W$.
\end{lemma}

\begin{proof}
The uniform estimate is an immediate consequence of Lemmas \ref{lemma:Kx.large} and \ref{lemma:Kx.small}. The Lipschitz estimate follows from the second resolvent identity and the estimate
$$ \norm[X^\sharp \rarr X^\sharp]{K_x(\calD_1) - K_x(\calD_2)}
\lesssim \norm[V_N]{\calD_1 - D_2}$$
which is easily proved from the formulas \eqref{K00}--\eqref{K11}.
\end{proof}

Owing to the uniform estimate, we can lift $(I-\calK_x)^{-1}$ to an operator
which we denote by $(I-\calK)^{-1}$ on $L^2((a,\infty);X^\sharp)$ by the formula
$$ \left[(I-\calK)^{-1}f\right](x,\lam) = \left[(I-\calK_x)^{-1}f(x,\dotarg)\right](\lam). $$
As an immediate consequence of Lemma \ref{lemma:res}, we have:

\begin{proposition}
\label{prop.res}
Let $W$ be a bounded subset of $V_N$.
The resolvent $(I-\calK)^{-1}$ satisfies the estimates
$$ 
\norm[\calB(L^2([a,\infty);X^\sharp)]{(I-{\calK})^{-1}}
\lesssim 1
$$
with constants uniform in  $\calD \in W$ and
$$
\norm[\calB(L^2([a,\infty);X^\sharp)]
		{(I-{\calK}(\calD_1))^{-1}- (I-{\calK}(\calD_2))^{-1}} 
\lesssim
\norm
	[V_N]	
	{\calD_1 - \calD_2}
$$
with constants uniform in  $\calD_1$, $\calD_2 \in W$. 
\end{proposition}

To solve \eqref{nusharp.int}, we also need to control the inhomogeneous term $\calK_x e$. Let
\begin{equation}
\label{bff}
\bff	=	\calK_x \bfe 
		= 	\left(
				f^\sharp(x,\lam),
				\{ f^\sharp(x,\overline{\lam_j})\}_{j=1}^N 
			\right). 
\end{equation}
From \eqref{calK} we see that
$$
f^\sharp =	f_1^\sharp + f_2^\sharp + f_3^\sharp + f_4^\sharp
$$
where
\begin{align*}
f_1^\sharp(x,\lam)
	&=	
	\sum_k \frac{C_{k,x} \lam_k }{\lam-\lam_k}
			\left(
					\sum_j 
						\frac	{\eps\overline{C}_{j,x} }
								{\overline{\lambda}_j-\lambda_k} 
			\right)
\\
f_2^\sharp(x,\lam)
	&=	 \sum_k \frac{1}{\lam - \lam_k}
			\left(
				{\dint_{\mathbb{R}}\frac{\rho_x(s)}{s-\lambda_k}
				\dfrac{ds}{2\pi i} C_{k,x} \lam_k }
			\right)
\\
f_3^\sharp(x,\lam)
	&= \sum_k\overline{C}_{k,x} 
			C^-\left[
					\frac{ (\dotarg)\overline{\rho_x(\dotarg)}  }
						{(\dotarg)-\overline{\lambda}_k} 
				\right](\lambda)	\\
f_4^\sharp(x,\lam)
	&=	- \eps C^-
				\left\lbrace
						C^+
							\left[
								\rho_x(\diamond) 
							\right]  
							(\dotarg) 
							\overline{\rho_x(\dotarg)} 
				\right\rbrace(\lam)
\end{align*}
We can get $f^\sharp (x,\overline{\lambda}_j)$, $1 \leq j \leq N$, by substituting $\overline{\lambda}_j$  for $\lambda\in\mathbb{R}$  and changing the corresponding Cauchy projection $C^-$ to a Cauchy integral over the real line.

\begin{lemma}
\label{lemma:f1}
For $f^\sharp$ given by \eqref{bff} and indices $i=1,2,3,4$ and $1 \leq j \leq N$,
$$
\left| f^\sharp_i(x,\overline{\lam}_j) \right|+
\left|\dee f^\sharp_i(x,\overline{\lam}_j)/\dee x\right| +
\left|\dee^2 f^\sharp_i(x,\overline{\lam}_j)/\dee x^2\right|
\lesssim \left(1+ \norm[H^{2,2}(\R)]{\rho}\right)^2
$$
uniformly for
$\{ \lam_j, C_j\}$ in a fixed compact subset of $(\C^+ \times C^\times)^N$ and $x \geq a$.
\end{lemma}

\begin{lemma}
\label{lemma:f2}
For $f^\sharp$ given by \eqref{bff} and indices $i=1,2,3 $ and $1 \leq j \leq N$ 
\begin{equation}
\label{f2.1}
\sup_{0 \leq k \leq 2}
	\norm[L^2([a,\infty))]
			{\dee^k f^\sharp_i(\dotarg,\overline{\lam}_j)/\dee x^k} 
\lesssim \left(1+ \norm[H^{2,2}(\R)]{\rho}\right)^2
\end{equation}
uniformly for 
$\{ \lam_j, C_j\}$ in a fixed compact subset of $(\C^+ \times \C^\times)^N$.  Moreover
\begin{equation}
\label{f2.2}
\sup_{0 \leq k \leq 2}
	\norm[L^2([a,\infty)\times \R)]
			{\dee^k f^\sharp_i(\dotarg,\diamond)/\dee x^k}
	\lesssim  \left(1+ \norm[H^{2,2}(\R)]{\rho}\right)^2\\
\end{equation}
and
\begin{equation}
\label{f2.3}
\norm[L^2([a,\infty)\times \R)]
			{\dee f_i^\sharp(\dotarg,\diamond)/\dee \lam} +
\norm[L^2([a,\infty)\times \R)]
			{\langle \diamond \rangle^{-1} 
			\dee^2 f_i^\sharp/\dee \lam^2 (\diamond,\dotarg)}\\
	\lesssim  \left(1+ \norm[H^{2,2}(\R)]{\rho}\right)^2
\end{equation}
with the same uniformity.
For $f_4$ we have the following estimates:
\begin{subequations}
\label{f4.est}
\begin{align}
\label{f4}
\norm[L^2_\lam]{{f^\sharp_4}(x,\dotarg)}	
	&\lesssim (1+|x|)^{-1} \norm[H^{2,0}]{\rho} \norm[L^2]{\rho}, \\
\label{f4.x}
\norm[L^2((-a,\infty) \times \R)]{(f^\sharp_4)_x}
	&\leq		\norm[L^{2,1}]{\rho} \norm[L^2]{\rho}\\
\label{f4.xx}
\norm[L^2((-a,\infty)\times \R)]{(f^\sharp_4)_{xx}}
	&\leq	\norm[L^{2,2}]{\rho} \norm[L^{2,1}]{\rho}\\
\label{f4.l}
\norm[L^2_\lam]{(f^\sharp_4)_{\lam}(x,\dotarg)}
	&\lesssim	(1+|x|)^{-1}\norm[H^2]{\rho}\norm[H^{2,2}]{\rho}\\
\label{f4.ll}
\norm[L^2_\lam]{ \langle \dotarg \rangle^{-1} (f^\sharp_4)_{\lam\lam}(x,\dotarg)}		
	&\lesssim	(1+|x|)^{-1} \norm[H^{2,2}]{\rho}^2
\end{align}
\end{subequations}
\end{lemma}

\begin{remark}
\label{rem:f.infty}
These estimates and Sobolev embedding imply that
$$
\norm[L^\infty([a,\infty);X^\sharp)]{f^\sharp}
	\lesssim	 \left( 1 + \norm[H^{2,2}(\R)]{\rho} \right)^2
$$
\end{remark}

\begin{remark}
\label{rem:f.bilinear}
Since $\textbf{f}^\sharp$ and its derivatives are bilinear in the data $\left(	\rho,\{ \lam_k \}, \{ C_k \}	\right)$, the estimates used to prove the lemma above can easily be adapted to show that $\textbf{f}^\sharp$ and its derivatives are Lipschitz continuous as a function of $\left(	\rho,\{ \lam_k \}, \{ C_k \}	\right)\in V_N $. 
\end{remark}
\begin{proof}
We can establish the inequalities (\ref{f4})-(\ref{f4.ll}) using the conclusion from Paper I, Lemma 6.5 and (6.14). The other part of the lemma is trivial.
\end{proof}

Lemma \ref{lemma:res}, Proposition \ref{prop.res}, Lemmas \ref{lemma:f1}, and \ref{lemma:f2} together with Remark \ref{rem:f.infty} immediately imply Proposition \ref{prop:nu.L2}.

\subsubsection{$\lam$-Derivatives}

Next,  we consider $\lam$-derivatives of $\nu^\sharp$.  Write
$$
\nu^\sharp  
	= \left(\nu_0, \{ \widetilde{\nu}^*_j \}_{j=1}^N \right)
	=	\left(\nu_0, \widetilde{\nu}^*\right).
$$
Note that only $\nu_0$ depends on $\lam$, and obeys 
the integral equation
\begin{equation}
\label{int.nu1}
\nu_0	
	=	f_1 	+ \calK_{00}(\nu_0) 
				+ \calK_{01}(\widetilde{\nu}^*).
\end{equation}

We will prove Proposition \ref{prop:nu.lam} by differentiating \eqref{int.nu1}
and using resolvent bounds to estimate the derivatives.

The following estimates allow us to control the third term
$\calK_{01} (\widetilde{\nu}^*)$ in \eqref{int.nu1} and its derivatives in 
$\lam$.

\begin{lemma}
\label{lemma:K01}
Suppose that $\nu^\sharp \in L^2([a,\infty),X^\sharp)$. Then
\begin{align}
\label{K01.lam}
\norm[L^2([a,\infty) \times \R)]
		{\left[\calK_{01} (\widetilde{\nu}^*)\right]_\lam}
&\lesssim \norm[H^{2,2}(\R)]{\rho} \norm[L^2(\R,\C^N)]{\widetilde{\nu}^*}	\\
\label{K01.lamlam}
\norm[L^2([a,\infty) \times \R)]
		{\left[\calK_{01} (\widetilde{\nu}^*)\right]_{\lam\lam}}
&\lesssim	\norm[H^{2,2}(\R)]{\rho} \norm[L^2(\R,\C^N)]{\widetilde{\nu}^*}  
\end{align}
hold, with estimates uniform 
in  $\{ \lam_j, C_j\}$ in a compact subset of $(\C^+ \times \C^\times)^N$.
Similar estimates hold, with the same uniformity, if the $L^2([a,\infty) \times \R)$-norm is replaced by the $L^\infty([a,\infty),L^2(\R))$-norm.
\end{lemma}

\begin{proof}
These estimates follow from the explicit formula
\begin{align}
\label{01}
\calK_{01}\left(\widetilde{\nu}^*\right) 
	&=	-\sum_k 
				\widetilde{\nu}_k^*(x)
				\overline{{C}_{k,x}}
				C_{-}
				\left [ 
					\frac{(\dotarg)\overline{\rho_x(\dotarg)}} 
						{(\dotarg)-\overline{\lambda}_k}
				\right](\lambda) 
				\\
\nonumber
	&\quad +
		\sum_j
			\widetilde{\nu}_j^*(x)
			\eps\overline{C_{j,x}}
			\left(
				\sum_k
					\frac{\lambda_k C_{k,x} }
						{(\lambda-\lambda_k)
						(\lambda_k-\overline{\lambda}_j)} 
			\right)
\end{align}
noting that the functions $C_{j,x}$ are bounded and exponentially decaying in $x$.
\end{proof}

The operator $\calK_{00}=S_x + S'$ (cf. \eqref{K00}) where
$S_x$ was studied in Paper I and denoted $S$ there. We recall the following estimates from Paper I. Proofs may be constructed from computations in the  proof of Paper I, Lemma 6.13.

\begin{lemma}
\label{lemma:S0}
Suppose that $\rho \in H^{2,2}(\R)$. The following estimates hold for 
$h \in L^2([a,\infty) \times \R) \cap L^2([a,\infty),L^1(\R))$
\begin{align}
\label{S_0.lam}
\norm[L^2([a,\infty) \times \R)]{\frac{\dee {S_x}}{\dee \lam} \left[h\right]}	
	&\lesssim	
		\norm[H^{2,2}]{\rho}^2 \norm[L^\infty([a,\infty),L^2(\R))]{h} \\
\label{S_0.lamlam}
\norm[L^2([a,\infty) \times \R)]
	{\langle \lam \rangle^{-1}\frac{\dee^2 {S_x} }{\dee \lam^2} \left[h\right]}
		&\lesssim	
			\norm[L^1]{\widehat{\rho}'}
			\norm[L^{2,2}]{\widehat{\rho}}
			\norm[L^2([a,\infty),L^1(\R))]{\widehat{h}}		
\end{align}
where $\widehat{h}$ denotes the partial Fourier transform of $h(x,\lam)$ in the second variable, and the implied constants depend only on $a$.
\end{lemma}

\begin{remark}
\label{rem:S_0.lamlam}
Using the inequality
$$\norm[L^1(\R)]{\widehat{\psi}\,}
	\leq	\norm[L^2(\R)]{\psi} + \norm[L^2(\R)]{\psi_\lam}$$
we may rewrite the estimate \eqref{S_0.lamlam} as 
\begin{equation}
\label{S_0.lamlam.bis}
\norm[L^2([a,\infty) \times \R)]
	{\langle \lam \rangle^{-1}\frac{\dee^2 {S_x} }{\dee \lam^2} \left[h\right]}	
		\lesssim\norm[H^{2,2}(\R)]{\rho}^2 
				\left( 
					\norm[L^2([a,\infty) \times \R)]{h} +
					\norm[L^2([a,\infty) \times \R)]{h_\lam}	
				\right)		
\end{equation}
\end{remark}

Next, we study the operator $S'$.

\begin{lemma}
\label{lemma:S1}
The following estimates hold for the operator $S'$.
\begin{align}
\label{S1.lam}
\norm[L^2([a,\infty) \times \R)]
		{(S' h)_\lam}
	&\lesssim	
		\norm[L^2(\R)]{\rho} 
		\norm[L^2([a,\infty) \times \R)]{h}	\\
\label{S1.lamlam}
\norm[L^2([a,\infty) \times \R)]
		{(S' h)_{\lam\lam}}
	&\lesssim	
		\norm[L^2(\R)]{\rho} 
		\norm[L^2([a,\infty) \times \R)]{h}\\
\label{S1.x}
\norm[L^2([a,\infty) \times \R)]{(S')_x h}
	&\lesssim  \norm[H^{2,2}(\R)]{\rho} 
		\left( 
			\norm[L^2([a,\infty) \times \R)]{h}
		\right)\\
\label{S1.xx}
\norm[L^2([a,\infty) \times \R)]{(S')_{xx} (h)}
	&\lesssim	\norm[H^{2,2}(\R)]{h}
		\left(	
			\norm[L^2([a,\infty) \times \R)]{h}
		\right)
\end{align}
where the implied constants depend only on $a$ and $\{ \lam_j , C_j \}$.
\end{lemma}

\begin{remark}
In the estimates above, the derivative is taken with respect to $\lam$
on the composition $S' h$
for \eqref{S1.lam} and \eqref{S1.lamlam}, but the derivative of the \emph{operator only} with respect to the parameter $x$ is taken in \eqref{S1.x} and \eqref{S1.xx}.
\end{remark}

\begin{proof}
These estimates are easy consequences of the explicit formula
$$
(S'h)(x,\lam) = 
\sum_{j=1}^N \frac{C_{j,x} \lam_j}{\lam-\lam_j}
	C_\R \left(\rho_x(\dotarg) h(x,\dotarg)\right)(\lam_j)
$$
recalling \eqref{rhoCx}.
\end{proof}


\begin{proof}[Proof of Proposition \ref{prop:nu.lam}]
 We may use \eqref{int.nu1} together with Lemmas \ref{lemma:K01}, \ref{lemma:S0}, and \ref{lemma:S1} to estimate
\begin{align*}
\norm[L^2([a,\infty) \times \R)]{{\nu_0}_\lam} 
	&	\lesssim 
	\norm[L^2([a,\infty) \times \R)]{(f^\sharp)_\lam} +
	\norm[L^2([a,\infty) \times \R)]{(\calK_{00} \nu_0)_\lam} +
	\norm[L^2([a,\infty) \times \R)]{(\calK_{01} \nu_0)_\lam}\\
	& 	\lesssim
	\left(1+\norm[H^{2,2}(\R)]{\rho}\right)^2 
		\left(
				1+ \norm[L^\infty([a,\infty);L^2(\R))]{\nu_0}
				+ \norm[L^2([a,\infty);X^\sharp)]{\nu^\sharp} 		
		\right)
\end{align*}
which shows that $\norm[L^2([a,\infty) \times \R)]{{\nu_0}_\lam}$ is bounded uniformly for $\calD \in W$.

Differentiating \eqref{int.nu1} a second time yields the 
equation
$$
\langle \lam \rangle^{-1} (\nu_0)_{\lam\lam}
	=	\langle \lam \rangle^{-1} (f)_{\lam\lam} +
		\langle \lam \rangle^{-1} 
			\left( \calK_{00}(\nu_0) \right)_{\lam\lam}  +
		\langle \lam \rangle^{-1} 
			\left( \calK_{01} (\widetilde{\nu}^*) \right)_{\lam\lam} .
$$
Using estimates \eqref{f2.3} and \eqref{f4.ll} in the first right-hand term, \eqref{S_0.lamlam.bis}, \eqref{S1.lamlam} and Remark \ref{rem:S_0.lamlam} in the second right-hand term,
and \eqref{K01.lamlam}
in the third term, we conclude that
$$
\norm[L^2([a,\infty) \times \R)]
		{\langle \diamond \rangle^{-1}
			{\nu_0}_{\lam\lam}(\dotarg,\diamond)}\\
\lesssim	
	\left( 1 + \norm[H^{2,2}]{\rho} \right)^2 
	\left( 
			1 	+  \norm[L^2([a,\infty),X^\sharp)]{\nu^\sharp} 
				+	\norm[L^2([a,\infty) \times \R)]{{\nu_0}_\lam} 
	\right)
$$
which shows that 
$\norm[L^2([a,\infty) \times \R)]
		{\langle \diamond \rangle^{-1}
			{\nu_0}_{\lam\lam}(\dotarg,\diamond)}$
is also bounded uniformly for $\calD \in W$. Thus, 
we obtain Proposition \ref{prop:nu.lam}.
\end{proof}

\subsubsection{$x$-Derivatives}

We now turn to estimates on ${\nu_0}_x$ and ${\nu_0}_{xx}$.
We will differentiate \eqref{nusharp.int} with respect to $x$ and therefore will need the following estimates on the operators $\calK_x$ and $\calK_{xx}$.

\begin{lemma}
\label{lemma:K-estimates}
Fix $a \in \R$ and suppose that $h = (\widetilde{h}, h^*) \in X^\sharp$.
\begin{align}
\label{Kx.est}
\sup_{x \geq a}
	\norm[\calB(X^\sharp)]{\calK_x}
	&	\lesssim		\left( 1 + \norm[H^{2,2}(\R)]{\rho}\right)^2\\
\label{Kxx.est}
\sup_{x \geq a}
	\norm[X^\sharp]{\calK_{xx} h}
	& \lesssim	
	\left(
		1+ \norm[H^{2,2}(\R)]{\rho}
	\right)^2
	\norm[X^\sharp]{h} + \norm[L^2(\R)]{\widetilde{h}_\lam}			
\end{align}
where constants are uniform for $\{\lam_j,C_j \}_{j=1}^N$ in a compact subset of 
$(\C^+ \times \C^\times)^N$.
\end{lemma}

\begin{proof}
In Paper I, Lemma 6.7, we proved
that $S_x$, $\left(S_x\right)_x$, and $\left(S_x\right)_{xx}$ are all bounded operators on $L^2(\R)$. The boundedness of $S'$ and its first two derivatives in $x$, uniform in $x \in [a,\infty)$ is a 
trivial consequence of the formulas, showing that $\calK_{00}$ is twice-differentiable in $x$ with uniform estimates. 

The remaining operators are of finite rank. The differentiability of $\calK_{11}$ and $\calK_{01}$ is an immediate consequence of the formulas. Finally, to treat $\calK_{01}$ it suffices to study mapping properties of Cauchy projection $C^+$.  For fixed $x$ and some $h\in L^2(\mathbb{R})$, we have
\begin{align}
\label{cauchy+}
C_{+}
	\left[ h(\dotarg)\rho(\cdot) e^{-2ix(\cdot)}\right](s)
=\frac{1}{\pi}
	\int_{0}^{\infty}
		\left[\widehat{h}(\dotarg)*\widehat{\rho}(\dotarg+x)\right](\xi)
		e^{2i\xi s} \, d\xi
\end{align}
Computing the first and second derivatives in $x$, we get
\begin{align}
\label{cauchy+_x}
C_{+}
	\left[ 
		h(\dotarg)\rho(\cdot) e^{-2ix(\cdot)}
	\right]_{x}(s)
&=	\frac{1}{\pi}
			\int_{0}^{\infty}
				\left[\widehat{h}(\dotarg)*\widehat{\rho}'(\dotarg+x)\right](\xi)
			e^{2i\xi s} \, d\xi	\\
\label{cauchy+_xx}
C_{+}
	\left[ 
		h(\dotarg)\rho(\cdot) e^{-2ix(\cdot)}
	\right]_{xx}(s)
&=	\frac{1}{\pi}
		\int_{0}^{\infty}
			\left[
				\widehat{h}(\dotarg)*\widehat{\rho}''(\dotarg+x)
			\right](\xi)
			e^{2i\xi s} \, d\xi
\end{align}
Using Plancherel's identity we have 
$$
\left\Vert 
	C_{+}
		\left[ 
			h(\dotarg)\rho(\cdot) e^{-2ix(\cdot)}
		\right]_{x}     
\right\Vert_{L^2}
\lesssim 
\Vert	\widehat{h}\Vert_{L^2} 
\Vert\widehat{\rho}'\Vert_{L^1}
$$
and
$$
\left\Vert 
	C_{+}
		\left[ 
			h(\dotarg)\rho(\cdot) e^{-2ix(\cdot)}
		\right]_{xx}     
\right\Vert_{L^2}
\lesssim 
	\Vert
		\widehat{h}
	\Vert_{L^1} 
	\Vert\widehat{\rho}''\Vert_{L^2}.
$$
The estimate \eqref{Kx.est} is now easily deduced.
Since
$$  
\Vert\widehat{h}\Vert_{L^1}
\leq 
\norm[L^2]{\langle\xi\rangle \widehat{h}}
\leq 
\left\Vert \frac{\dee h}{\dee \lambda}\right\Vert_{L^2}+
\Vert {h}\Vert_{L^2}
$$
the estimate \eqref{Kxx.est} now follows.
\end{proof}

\begin{remark}
Since all estimates in the proof of Proposition \ref{lemma:K-estimates} are bilinear in the scattering data $\left(	\rho,\{ \lam_k, C_k \}\right)$, it follows that 
$\left( \rho,\{ \lam_k, C_k \}\right) \mapsto \calK_x$ and 
$\left( \rho,\{ \lam_k, C_k \}\right)\mapsto\calK_{xx}$ 
are locally Lipschitz maps from 
$H^{2,2}(\R) \times (\C^{+} \times \C^\times)^N$ to the bounded operators on $L^2_\lambda\oplus \mathbb{C}^n $.
\end{remark}

\begin{proof}[Proof of Proposition \ref{prop:nu.x}]
Differentiating the integral equation \eqref{nusharp.int} we see that
\begin{align*}
\nu_x^\sharp		&= (f)_x + ({\calK})_x \nu^\sharp + {\calK} (\nu_x^\sharp)	\\
\nu_{xx}^\sharp 	&= (f)_{xx} + ({\calK})_{xx} \nu^\sharp + ({\calK})_x (\nu^\sharp)_x + {\calK} (\nu^\sharp_{xx})
\end{align*}
so that we can conclude from  that $\nu_x^\sharp$, $\nu_{xx}^\sharp \in L^2([a,\infty),X^\sharp)$ provided we show that
\begin{align*}
h_1	&=	(f)_x + ({\calK})_x\nu^\sharp\\
h_2	&=	(f)_{xx} + ({\calK})_{xx}\nu^\sharp + ({\calK})_x (\nu^\sharp)_x + {\calK}(\nu_{xx}^\sharp)
\end{align*}
belong to $L^2([a,\infty),X^\sharp)$.
It follows from \eqref{f2.1}, \eqref{f2.2}, and \eqref{f4.x} that 
$f_x \in L^2([a,\infty),X^\sharp)$, while $(\calK)_x \nu^\sharp \in L^2([a,\infty),X^\sharp)$ by \eqref{Kx.est}. Hence 
$h_1 \in L^2([a,\infty),X^\sharp)$, so that $(\nu^\sharp)_x \in L^2([a,\infty),X^\sharp)$. 

To see that $h_2 \in L^2([a,\infty),X^\sharp)$, we use \eqref{f2.1}, \eqref{f2.2} and \eqref{f4.xx} to conclude 
that $f_{xx} \in L^2([a,\infty),X^\sharp)$; \eqref{Kxx.est}
and the fact that $\nu^\sharp$ and ${\nu_0}_\lam$ both belong to $L^2([a,\infty),X^\sharp)$ to show that $(\calK)_{xx}(\nu^\sharp)$ belongs to $L^2([a,\infty),X^\sharp)$; and \eqref{Kx.est} and our previous result to show that 
$(\calK)_x \nu_x^\sharp \in L^2([a,\infty),X^\sharp)$. 
Hence $h_2 \in L^2([a,\infty),X^\sharp)$ and so 
$\nu^\sharp_{xx} \in  \in L^2([a,\infty),X^\sharp)$.
\end{proof}
%
%
%
%
%
 
\section{Deformation to a mixed \texorpdfstring{$\bar{\partial}$}{DBAR}-Riemann-Hilbert Problem}
\label{sec:deform}

This section is devoted to the two first transformations in the reduction of the original RHP \ref{RHP2} to a model one that can be solved explicitly
and provides the precise behavior of the solution of the DNLS equation for long time up, to small terms of order $O(|t|^{-3/4})$.
We present the analysis for $t>0$ and $t<0$ simultaneously by introducing the parameter $\eta= \sgn(t)$.

The first step  is the conjugation of the solution of RHP \ref{RHP2} by a scalar function denoted $\delta$ defined in \eqref{T}, which is itself solution of a scalar
RHP (Section \ref{sec:conj}. This operation is standard in this type of problems and its effect is described in details in \cite{DZ03}, see also \cite{BJM16} and Paper 2.
The second step (Section \ref{sec:extensions}) is the deformation of contours from the real axis to the contour $\Sigma^{(2)}$ shown in Figure \ref{fig:n2def}. 
Our presentation follows 
Paper 2 with the addition of the  treatment of the discrete data associated to the  residue conditions \cite{BJM16}.

%
%

\subsection{Conjugation}
\label{sec:conj}
The long-time asymptotic analysis of RHP~\ref{RHP2} is determined by the growth and decay of the exponential function $e^{2it\theta}$ 
appearing in both the jump relation (Problem \ref{RHP2}(iii))  and the residue conditions (Problem \ref{RHP2}(iv)).
Let 
\begin{equation}\label{parameters}
	\xi = -\frac{x}{4t} 
		\qquad \qquad
	\sgnt = \sgn t
\end{equation}
From \eqref{phase.lambda} it's clear that for $|t| \gg 1$, $|e^{2it \theta}| \ll 1$ whenever $\sgnt \Re(\lambda - \xi) < -c < 0$ and $|e^{2it \theta}| \gg 1$ when $\sgnt \Re(\lambda - \xi) > c > 0$. The RHP~\ref{RHP2} is derived from the scattering operator in such a way that it has identity asymptotics as $x \to +\infty$ with $t$ fixed. Our interest is in studying the behavior of solutions when $|t| \to \infty$ with $x/t$ fixed in some interval. It is necessary to renormalize the RHP 
\newtext{so}
that it is well behaved as $t \to \infty$ in the sector in interest. 

\begin{figure}[htb]
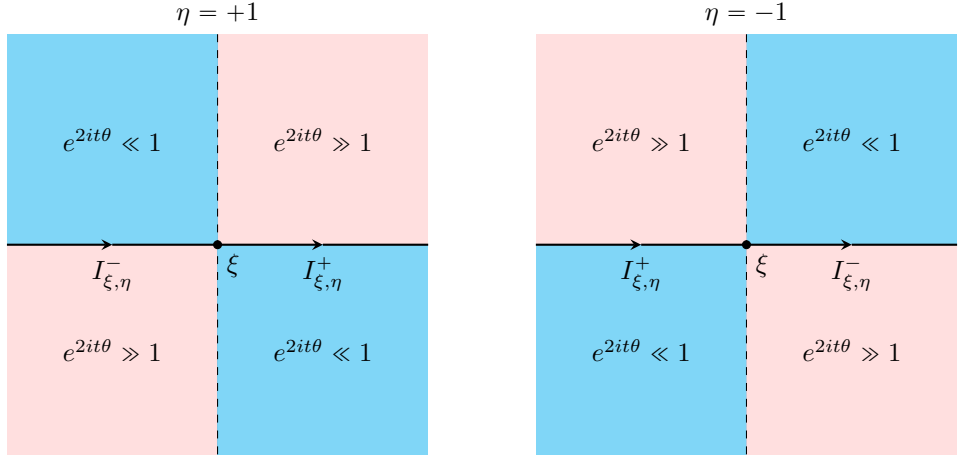

\centering
\hspace*{\stretch{1}}
\FigPhaseA
\hspace*{\stretch{1}}
\FigPhaseB
\hspace*{\stretch{1}}
\caption{The regions of growth and decay of the exponential factor $e^{2it\theta}$ in the $\lambda$-plane for either sign of $\sgnt = \sgn t$. 
\label{fig:theta signs}
}
\end{figure}

Define the real intervals 
\begin{equation}\label{Tint}
	\begin{aligned}
	\negint &= \{  \lambda \in \C \,:\, \Im \lambda = 0, \ -\infty < \sgnt \Re \lambda \leq \sgnt \xi \}, \\
	\posint &= \{  \lambda \in \C \,:\, \Im \lambda = 0, \  \sgnt \xi < \sgnt \Re \lambda \leq \infty  \}
	\end{aligned}
\end{equation}
oriented left-to-right. 
Note that if the sign of $t$, and thus $\sgnt$, is changed with $\xi$ held fixed, the effect is simply to exchange intervals 
$I_{\xi,\eta}^\pm$.
See Figure~\ref{fig:theta signs}.

{\color{blue} Let}
\begin{equation}\label{T}
	\begin{gathered}
	\delta(\lambda) = \delta(\lambda, \xi, \sgnt)  := 
	\exp \lp i \int_{\negint}  \frac{\kappa(z) }{z-\lambda} dz \rp,
	\\
	\kappa(z) = -\frac{1}{2\pi} \log(1 -\eps z |\rho(z)|^2).
	\end{gathered}
\end{equation}

\begin{lemma}\label{lem:T}
The function $\delta(\lambda)$ defined by \eqref{T} has the following properties:
\begin{enumerate}[i.]
	\item $\delta$ is meromorphic in $\C \setminus \negint$. 
	\item For $\lambda \in \C \setminus \negint$, $ \delta(\lambda) \bar{\delta(\bar{\lambda)}} = 1$,
	moreover,
	\[
		e^{-\| \kappa \|_\infty/2} \leq 
		\left| \delta(\lam) \right|
		\leq e^{\| \kappa \|_\infty/2}
	\]
	\item For $\lambda \in \negint$, {  its}  boundary values  $\delta_\pm$ {  as $\lambda$ approaches the real axis from above and below}, satisfy 
	\begin{equation}\label{Tjump}
		\delta_+(\lambda) / \delta_-(\lambda) = 1 - \eps \lambda |\rho(\lambda)|^2, 
		\quad \lambda \in \negint.
	\end{equation}
	
	\item As $ |\lambda| \to \infty$ with $|\arg(\eta \lambda)| \neq \pi$, 
	\begin{equation}\label{Texpand}
		{  \delta(\lambda)} = 1 - 
		\frac{1}{2\pi i \lam} \int_{\negint} \log( 1 -\eps z |\rho(z)|^2 ) dz 
		+ \bigo{ \lambda^{-2} }.
	\end{equation}
	\item As $\lambda \to \xi$ along any ray $\xi + e^{i \phi} \R_+$ with 
	$| \arg (\sgnt(\lambda-\xi) )| < \pi$
	\begin{gather} \label{Tbound}
		\left| \delta(\lambda, \xi,\sgnt) - \delta_0(\xi,\sgnt) (\sgnt(\lambda-\xi))^{i \sgnt \kappa(\xi)} \right| 
		\lesssim_{\rho,\phi}  |\lambda - \xi | \log |\lambda - \xi|.
	\intertext{The implied constant depends on $\rho$ through its $H^{2,2}(\R)$-norm and is independent of $\xi$. Here $\delta_0(\xi, \sgnt) = \exp(i\beta(\xi,\xi)$ is the complex unit}
		\nonumber
		\label{delta0 arg}
		\beta(z,\xi) = - \sgnt \kappa(\xi) \log(\sgnt(z-\xi+1))
		+ \int_{\negint} \frac{ \kappa(s) - \chi(s) \kappa(\xi)}{s-z} ds
		,
	\end{gather}
	and $\chi(s)$ is the characteristic function of the interval $\sgnt \xi - 1 < \sgnt s < \sgnt \xi$. 
In all of the above formulas, we choose the principal branch of power and logarithm functions.
\end{enumerate}
\end{lemma}

\begin{proof}
Parts $i$.--$iv$. are elementary consequences of the definition \eqref{T} and the Sokhotski-Plemelj formula. 
For part $iv$. one geometrically expands 
the factor $(z-\lambda)^{-1}$ for large $\lambda$, and uses the fact that $\| \kappa \|_{L^1(\R)} \lesssim \| \rho \|_{H^{2,2}(\R)}$ to bound the remainder in the integral term for $\lambda$ bounded away from the contour of integration. 
The proof of part $v.$ can be found in Appendix A of \cite{LPS16}.
\end{proof}
We now define a new unknown function $\nk{1}$ using our partial transmission coefficient 
\begin{equation}\label{n1}
\nk{1}(\lambda) = \nn(\lambda) \delta(\lambda)^{-\sigma_3}.
\end{equation}
{  We claim that $\nk{1}$ satisfies the following RHP.}

\begin{RHP}\label{rhp.n1}
Find a row vector\newtext{-}valued function $\nk{1}: \C \setminus (\R  \cup \poles) \to \C^2$ with the following properties
\begin{enumerate}[1.]
	\item $\nk{1}(\lambda) = (1,0) + \bigo{\lambda^{-1}}$ as $\lambda \to \infty$. 
	\item For $\lambda \in \R$, the boundary values
	$\nk[\pm]{1}(\lambda)$ satisfy the jump relation 
	$\nk[+]{1}(\lambda) = \nk[-]{1}(\lambda) \vk{1}(\lambda)$ where
	\begin{equation}\label{V1}
		\vk{1} (\lambda) = \begin{cases}
			\triu{ \rho(\lambda) \delta(\lambda)^{2} e^{2it \theta} }
			\tril{ -\eps \lambda \bar{\rho}(\lambda) \delta(\lambda)^{-2} e^{-2it \theta} }	
			& z \in \posint \\[1.5em]
			\tril{ \frac{ -\eps \lambda \bar{\rho}(\lambda) \delta_-(\lambda)^{-2}}
			{1 -\eps \lambda |\rho(\lambda)|^2} e^{-2it \theta} }
			\triu{ \frac{ \rho(\lambda) \delta_+(\lambda)^{2}}
			{1 -\eps \lambda |\rho(\lambda)|^2} e^{2it \theta} }
			& z \in \negint
		\end{cases}
	\end{equation}
	\item $\nk{1}(\lambda)$ has simple poles at each point in $\poles$, for each $\lambda_k \in \poles^+$, 
		\begin{subequations}\label{n1 residue}
			\begin{gather} 
				\begin{aligned}
				\res_{\lambda = \lambda_k} \nk{1}(\lambda) 
				&= \lim_{\lambda \to \lambda_k} \nk{1}(\lambda) \vk{1}(\lambda_k) \\
				\res_{\lambda = \bar{\lambda_k}} \nk{1}(\lambda) 
				&= \lim_{\lambda \to \bar{\lambda_k}} \nk{1}(\lambda) \vk{1}(\bar{\lambda_k})
				\end{aligned}
			\shortintertext{where}
		\label{n1 residue matrices}		
				\begin{aligned}
					\vk{1}(\lambda_k) &= 
						\tril[0]{ \lambda_k C_k \delta(\lambda_k)^{-2} e^{-2it \theta} } \\
					\vk{1}(\bar{\lambda}_k) &= 					
                    	\triu[0]{ \eps \bar{C}_k \delta(\bar{\lambda}_k)^{2} e^{2it \theta} } 
				\end{aligned}
			\end{gather}
		\end{subequations}		
\end{enumerate}
\end{RHP}

\begin{proposition}\label{prop:M1}
Suppose that $\nn$ satisfies RHP~\ref{RHP2}, then the function $\nk{1}$ defined by \eqref{n1} satisfies the RHP~\ref{rhp.n1}.
\end{proposition}

\begin{proof}
{  The fact that}  $\nk{1}$ is analytic in $\C \setminus (\R \cup \poles)$, and approaches $(1,0)$ as $\lambda \to \infty$ follows directly from its definition, Lemma~\ref{lem:T} and the analytic properties of $\nn$ given in RHP~\ref{RHP2}. 
The relation 
$\vk{1}(\lambda) = \delta_-^{\sig}(\lambda) v(\lambda) \delta_+^{\sig}(\lambda)$, the standard factorizations of $v$
\begin{align*}
	v&= \triu{\rho e^{2it \theta} } \tril{-\eps \lambda \rho^* e^{-2it \theta} } \cr
	&=\tril{ \frac{-\eps \lambda \rho^* e^{-2it \theta} }{1-\eps \lambda |\rho|^2} }
	 \diag{1-\eps \lambda |\rho|^2}{(1-\eps \lambda |\rho|^2)^{-1} }
	 \triu{ \frac{\rho e^{2it \theta} }{1-\eps \lambda |\rho|^2} }
\end{align*}

and the jump relation \eqref{Tjump} satisfied by $\delta(z)$ on $\negint$ allow us to write $\vk{1}$ as 
\begin{equation*}
	\vk{1}(\lambda)= \begin{cases}
	\triu{ \rho(z) \delta^2(\lambda) e^{2it \theta} } 
	\tril{-\eps \lambda \rho(\lambda) \delta^{-2}(\lambda)  e^{-2it \theta} }  
	&  \lambda \in \posint  \\
	\tril{ \frac{ -\eps \lambda \bar{\rho}(z) \delta_-^{-2}(\lambda) }{ 1-\eps \lambda \rho \bar{\rho}} e^{-2it \theta} } 
	\triu{ \frac{ \rho(z) \delta_+^2(\lambda) }{1-\eps \lambda \rho \bar{\rho} } e^{2it \theta} }
	& \lambda \in \negint.,
	\end{cases}
\end{equation*}
Concerning the residues, since $\delta(\lambda)$ is analytic near each $\lambda_k$ we have
\[
	\res_{\lam_k} \nk{1} 
	= \lim_{\lam \to \lam_k} \nn(\lam) v(\lam_k) \delta(\lam_k)^{-\sig}
	= \lim_{\lam \to \lam_k} \nk{1}(\lam) \delta(\lam_k)^{\sig} v(\lam_k) \delta(\lam_k)^{-\sig}
\]
which gives the first equation in \eqref{n1 residue matrices}. The residue at $\bar{\lam}_k$ is treated similarly. 

\end{proof}

%
%

\subsection{\texorpdfstring{$\dbar$}{DBAR}-extensions of jump factorization}
\label{sec:extensions}

The next step  is to introduce a transformation which uses the factorization \eqref{V1} to deform the jump matrix $\vk{1}$  replacing it with new jumps along contours in the complex plane which are near the identity. The phase function \eqref{phase.lambda} has a single (quadratic) critical point at $\xi = -x/4t$. Let
\begin{equation}\label{Sigma_k}
	\begin{gathered}
	\Sk{2} = \Sigma_1 \cup \Sigma_2 \cup \Sigma_3 \cup \Sigma_4 \\
	\Sigma_k = \xi + e^{\frac{i\pi}{4}(2+(2k-3)\sgnt) }\, \R_+ , \quad k = 1,2,3,4,
	\end{gathered}
\end{equation}
with each ray oriented with increasing (resp. decreasing) real part for $\sgnt = +1$ (resp. $\sgnt = -1$). The function $e^{2it \theta}$ is exponentially increasing along $\Sigma_1$ and $\Sigma_3$ and decreasing along $\Sigma_2$ and $\Sigma_4$, while the reverse is true of $e^{-2it \theta}$.
Let  $\Omega_k,\, k=1,\dots,6$, denote the six connected components of $\C \setminus \lp \R \bigcup_{k=1}^4 \Sigma_k \rp$, starting with the sector $\Omega_1$ between $\posint$ and $\Sigma_1$ and numbered consecutively continuing counterclockwise (resp. clockwise) if $\sgnt = +1$ (resp. $\sgnt = -1$), see Figure~\ref{fig:n2def}.

In order to deform the contour $\R$ to the contour $\Sk{2}$, we introduce a new unknown $\nk{2}$ obtained from $\nk{1}$ as
\begin{align}
\label{n2 def}
	\nk{2}(\lam) = \nk{1} (\lam)\mathcal{R}^{(2)}(\lam).
\end{align}
The condition that the new unknown $\nk{2}$ have no jump on the real axis determines the boundary value of $\mathcal{R}^{(2)}$ in each of the sectors $\Omega_k$, $k=1,3,4,6$ meeting the real axis through the factorization of $\vk{1}$ in \eqref{V1}. 
These factorizations involve the reflection coefficient $\rho$ which does not extend analytically to the complex plane. 
To extend $\mathcal{R}^{(2)}$ off the real axis, we use the method of 
\cite{BJM16,CJ16,DM08} which introduces non-analytic extensions. 
The new unknown $\nk{2}$ will satisfy  a mixed $\dbar$-RHP.
The only condition on the extension is that we have some mild control on $\dbar \mathcal{R}^{(2)}$ sufficient to ensure that the $\dbar$ contribution to the long-time asymptotics of $q(x,t)$ is negligible. 
This is the content of Lemma~\ref{lem:extensions} below. 
We have considerable freedom in choosing the extension. We use this freedom to ensure that: $1)$ the new jumps on $\Sk{2}$ match a well known model RHP; and $2)$ in a small neighborhood of each pole in $\Lambda$,  $\mathcal{R}^{(2)}(\lam) = \sdiag{1}{1}$ so that the residues are unaffected by the transformation. 
We choose $\mathcal{R}^{(2)}$ as shown in Figure~\ref{fig:n2def}, where the functions $R_1,R_3,R_4,R_6$ satisfy
%
%
%
\newcommand{\temp}{\frac{ -\eps \xi \rho^*(\xi)}{1 - \eps \xi |\rho(\xi)|^2} 
		\delta_0(\xi,\sgnt)^{-2} (\sgnt(\lam-\xi))^{-2i \sgnt \kappa(\xi)} 
		(1-\indicator(z))}
\begin{subequations}\label{R_k}
\begin{align}
	\label{R1}
	R_1(\lam) &= \begin{dcases}
		 \spacing{-\eps \lambda \rho^*(\lam) \delta(\lam)^{-2} }{\temp}
			& \lam \in \posint \\
		-\eps \xi \rho^*(\xi) \delta_0(\xi, \sgnt)^{-2} (\sgnt(\lam-\xi))^{-2i \sgnt \kappa(\xi)}
		(1-\indicator(\lam)) 		
			& \lam \in  \Sigma_1
	\end{dcases} \bigskip \\
	\label{R3}
	R_3(\lam) &= \begin{dcases}
		\spacing{\frac{ \rho(\lam)}{1 - \eps \lambda |\rho(z)|^2} \delta_+(\lam)^{2}}{\temp} 
			& \lam \in \negint \\
		\frac{ \rho(\xi)}{1 - \eps \lam |\rho(\xi)|^2} 
		\delta_0(\xi, \sgnt)^{2} (\sgnt(\lam-\xi))^{2i \sgnt \kappa(\xi)} 
		(1-\indicator(\lam)) 
			& \lam \in  \Sigma_2
	\end{dcases} \bigskip \\
	\label{R4}
	R_4(\lam) &= \begin{dcases}
		\spacing{\frac{ -\eps \lambda \rho^*(\lam)}{1 -\eps \lam |\rho(\lam)|^2} \delta_-(\lam)^{-2} }
		{\temp}
			& \lam \in \negint \\
		\frac{ -\eps \xi \rho^*(\xi)}{1 - \eps \xi |\rho(\xi)|^2} 
		\delta_0(\xi, \sgnt)^{-2} (\sgnt(\lam-\xi))^{-2i \sgnt \kappa(\xi)} (1-\indicator(\lam)) 
			& \lam \in  \Sigma_3
	\end{dcases} \bigskip \\
	\label{R6}
	R_6(\lam) &= \begin{dcases}
		\spacing{\rho(\lam) \delta(\lam)^{2} }{\temp}
			& \lam \in \posint \\
		\rho(\xi) \delta_0(\xi, \sgnt)^{2} (\sgnt(\lam-\xi))^{2i \sgnt\kappa(\xi)}(1-\indicator(\lam)) 
			& \lam \in  \Sigma_4
	\end{dcases} 
\end{align}
\end{subequations}
Here $\indicator$ is a $C_0^\infty(\C,[0,1])$ cutoff function supported on a neighborhood of each point of discrete spectra such that 
\begin{equation}\label{chi prop}
	\indicator(\lam) = \begin{cases}
		1 & \dist(\lam, \poles) < \poledist/3 \\
		0 & \dist(\lam, \poles) > 2\poledist/3
	\end{cases}
\end{equation} 
where $\poledist$, defined by \eqref{distances}, is sufficiently small to ensure that the disks of support intersect neither each other nor the real axis.
\begin{figure}[t]
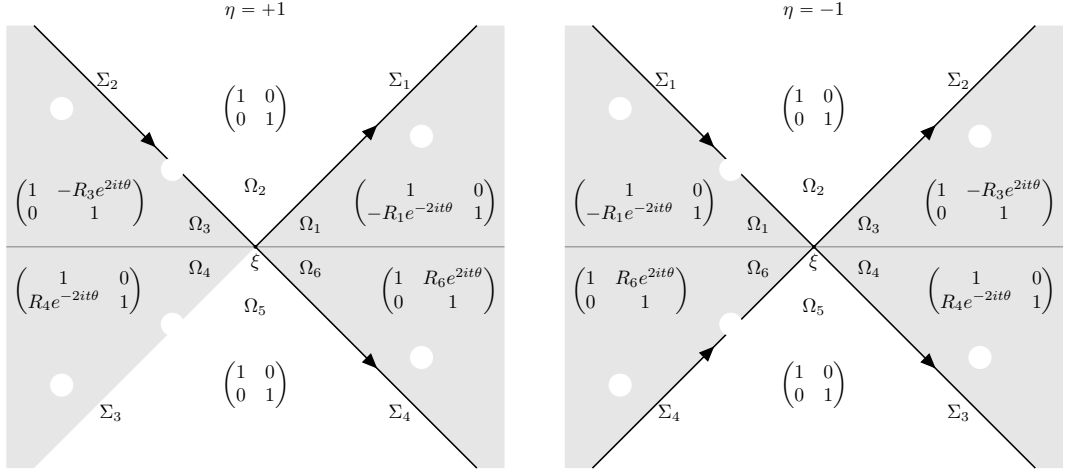

\hspace*{\stretch{1}}
\posDBARcontours
\hspace*{\stretch{1}}
\negDBARcontours
\hspace*{\stretch{1}}
\caption{Depicted here are the contour $\Sk{2} = \bigcup_{k=1}^4 \Sigma_{k}$ and regions $\Omega_k$ $k=1,\dots,6$ defining the transformation $\nk{2} = \nk{1} \mathcal{R}^{(2)}$. 
The labeling of the regions depends on $\sgnt$.
The non-analytic matrix $\mathcal{R}^{(2)}$ is given in each region $\Omega_k$. 
The support of  the $\bar{\partial}$-derivatives, $\Wk{2} = \dbar \mathcal{R}^{(2)}$ is shaded in gray.
\label{fig:n2def}
}
\end{figure}

The following lemma and its proof are almost identical to 
\cite[Proposition 2.1]{DM08}
or 
\cite[Lemma 4.1]{LPS16} .
The Lemma establishes the existence of functions $R_k$ and  estimates that are useful to control the contribution of the solution of the $\dbar$-problem (Section~\ref{sec:dbar}) to the large time behavior of $q(x,t)$  To state it, we introduce the factors  
\begin{align*}
	& p_1(\lam) =   {-\eps \bar{\rho}(\lam)}
	&& p_3(\lam) = \frac{\rho(\lam)}{1-\eps \lam |\rho(\lam)|^2} \\
	& p_6(\lam) = \rho(\lam)
	&& p_4(\lam) =   { \frac{-\eps \bar{\rho}(\lam)}{1-\eps \lam |\rho(\lam)|^2} }
\end{align*}

\begin{lemma}\label{lem:extensions}
Suppose that $\rho \in H^{2,2}(\R)$ and that $c:= \inf_{\lam \in \R} (1 - \eps \lam | \rho(\lam)|^2) > 0$ strictly. Then there exist functions $R_k$ on $\Omega_k$, $k=1,3,4,6$ satisfying \eqref{R_k}, such that
\begin{gather*}
	\left| \dbar R_k \right| \lesssim 
	\begin{cases}
		{ \left| \dbar \indicator(\lam) \right| }+ |   {\lam^{m_k}\,} p_k'(\Re \lam) | + \log |\lam - \xi|^{-1} 
		& z \in \Omega_k, \quad |z-\xi| \leq 1 \\
		{ \left| \dbar \indicator(\lam) \right| }+ |   {\lam^{m_k}\,} p_k'(\Re \lam) | + |\lam - \xi|^{-1} 
		& z \in \Omega_k, \quad |z-\xi| > 1 
	\end{cases} 
\shortintertext{and}
	\dbar R_k(\lambda) = 0, \qquad 
	  { \text{if $\lam = 0$ or  }}
	\dist(\lam, \Lambda) \leq \poledist/3
\end{gather*}
where   {$m_k = 1$ for $k=1,4$ and $m_k=0$ for $k=2,6$ and} the implied constants are uniform in $\xi \in \R$ and $\rho$ in a fixed bounded subset of $H^{2,2}(\R)$ with $1- \eps \lam |\rho(\lam)|^2 \geq c > 0$ for a fixed constant $c$.
\end{lemma}

This lemma has the following immediate corollary:
\begin{corollary}\label{cor:R2.bd}
Let $\lam - \xi = u + i v$ with $u,v \in \R$. Then under the assumptions of Lemma~\ref{lem:extensions}  for $k=1,3,4,6$ we have
\begin{gather*}
	\left| \dbar \mathcal{R}^{(2)}(\lam;\xi) \right| \lesssim 
	\begin{cases}
		\lp   { \left| \dbar \indicator(\lam) \right| 
		+ \left|   {\lam^{m_k}\,} p_k'(\Re \lam) \right| } + \log \frac{1}{|z-\xi|} \rp 
		e^{-8t |u| |v|}
		& \lam \in \Omega_k,\ |\lam - \xi| \leq 1 \\
		\lp   { \left| \dbar \indicator(\lam) \right| 
		+ \left|   {\lam^{m_k}\,} p_k'(\Re \lam) \right| } + \frac{1}{\sqrt{1+|\lam - \xi|^2} } \rp 
		e^{-8t |u| |v|}
		& \lam \in \Omega_k,\ |\lam - \xi| > 1,
	\end{cases}
\shortintertext{and}
	\begin{aligned}
	\dbar \mathcal{R}^{(2)}(\lam,\xi) & \equiv 0 
	\qquad \text{if } \lam \in \Omega_2 \cup \Omega_5
	\text{ or } \dist(\lam, \Lambda) \leq \poledist/3, \\
	  { \lim_{\lam \to 0} \dbar \mathcal{R}^{(2)}(\lam,\xi) }&  {= 0 
	\qquad \text{when } \lam \in \Omega_1 \cup \Omega_4.} 
	\end{aligned}
\end{gather*}
  {Here $m_k$ is as defined in Lemma~\ref{lem:extensions}, and} all the implied constants are uniform in $\xi \in \R$ and $|t|>1$.
\end{corollary}

\begin{proof}[Proof of Lemma~\ref{lem:extensions}]
We give the construction for $R_1$. Define $f_1(\lam)$ on $\Omega_1$ by
\[
	f_1(\lam) =   {\xi} p_1(\xi) \delta_0(\xi, \sgnt)^{-2} ( \eta (z-\xi) )^{-2i \eta \kappa(\xi) \sig} \delta(\lam)^{2}
\]
and let
\[
	R_1(\lam) = \lb f_1(\lam) + \lp   {\lam} p_1(\Re \lam)  - f_1(\lam) \rp \sgnt \cos (2\phi) \rb 
	\delta(\lam)^{-2}(1- \indicator(\lam)), 
\]
where $\phi = \arg(\lam - \xi)$. It is easy to see that $R_1$ as constructed has the boundary values in \eqref{R1} and that $\dbar R_1(\lam) = 0$ for $\dist(\lam, \Lambda) < \poledist/3$. Writing $\lam - \xi = r e^{i \phi}$ we have
\[
	\dbar = \frac{1}{2} \lp \pd{}{x} + i \pd{}{y} \rp 
	= \frac{e^{i\phi}}{2} \lp \pd{}{r} + \frac{i}{r} \pd{}{\phi} \rp,
\]
and
\begin{multline*}
	\dbar R_1(\lam) = 
	-\lb f_1(\lam) +   {\sgnt} \lp   {\lam} p_1(\Re \lam)  
	- f_1(\lam) \rp \cos (2\phi) \rb \delta(\lam)^{-2} \dbar\indicator(\lam) \\
	\eta \lb \frac{  {\lam}}{2} p_1'(\lam) \cos(2\phi) 
	- \frac{i e^{i\phi}}{|z-\xi|} (p_1(\Re \lam) 
	- f_1(\lam)) \sin(2\phi) \rb \delta^{-2}(\lam) (1-\indicator(\lam))
\end{multline*}

  {Clearly, $\dbar R_1(0) = 0$;} it follows from Lemma~\ref{lem:T}(ii. and v.) that
\[
\left| \dbar R_{  {1} } \right| \lesssim_\rho 
	\begin{cases}
		  { | \dbar \indicator(\lam) | } + |   {\lam} p_k'(\Re \lam) | 
		+ \log |\lam - \xi|^{-1}, 
		&  |z-\xi| \leq 1 \\
		  { | \dbar \indicator(\lam) | } + |   {\lam} p_k'(\Re \lam) | 
		+ |\lam - \xi|^{-1}, 
		&  |z-\xi| > 1 
	\end{cases} 
\]
where the implied constants depend on $\inf_\R (1-\eps \lam |\rho(\lam)|^2)$, $\oldnorm{\rho}_{H^{2,2}(\R)} $, and $\poles$. The constructions of $R_3,R_4$ and $R_6$ are similar.
\end{proof}

The new unknown $\nk{2}$ satisfies a mixed $\dbar$-RHP. We compute the new jumps on $\Sk{2}$ using the formula
\[
	\vk{2} = \nk[-]{1}^{-1}  \nk[+]{1}
	= \lp \mathcal{R}^{(2)}_-\rp^{-1} \vk{1} \mathcal{R}^{(2)}_+
\]
where the $\pm$ subscripts refer to the left/right side of the contour with respect to the orientation. Away from $\Sk{2}$, remembering that $\nk{1}$ is analytic in $\C \setminus (\R \cup \poles)$, we have
\[
	\dbar \nk{2} = \nk{1} \dbar \mathcal{R}^{(2)} = \nk{2} \lp \mathcal{R}^{(2)}\rp^{-1} \dbar \mathcal{R}^{(2)}
	= \nk{2} \dbar \mathcal{R}^{(2)}.
\]
where the last step follows from the nilpotency of $\dbar \mathcal{R}^{(2)}$.
 
%
%
\begin{DBAR}\label{rhp.n2}
Find a row vector valued function 
$\nk{2}: \C \setminus (\Sk{2} \cup \poles) \to \C^2$ with the following properties
\begin{enumerate}[1.]
\item $\nk{2}(\lam)$ has continuous first partial derivatives in $\C \setminus (\Sk{2} \cup \poles)$ and continuous boundary values $\nk[\pm]{2}(\lambda)$ on $\Sk{2}$. 
\item $\nk{2}(\lam) = (1,0) + \bigo{\lam^{-1}}$ as $ \lam \to \infty$.
\item For $\lam \in \Sk{2}$, the boundary values satisfy the jump relation 
	$\nk[+]{2}(\lam) = \nk[-]{2}(\lam) \vk{2}(\lam)$, where
\begin{equation}\label{V2}
	\begin{gathered}
		\vk{2}(\lam) = I + (1-\indicator(\lambda)) h(\lam), \\
		h(\lam) = \begin{cases}
			\tril[0]{ -\eps \xi \bar{\rho}(\xi) \delta_0(\xi,\sgnt)^{-2} 
			(\sgnt(\lam-\xi))^{-2i \sgnt \kappa(\xi)} e^{2it \theta} } 
		 		& z \in \Sigma_1 \smallskip \\
    		\triu[0]{ \frac{\rho(\xi) \delta_0(\xi,\sgnt)^{2} }{1+|r(\xi)|^2}  
    		(\sgnt(z-\xi))^{2i \sgnt \kappa(\xi)} e^{-2it \theta} } 
    			& z \in \Sigma_2 \smallskip\\
    		\tril[0]{ \frac{-\eps \xi \bar{\rho}(\xi) \delta_0^{-2}(\xi,\sgnt) } {1+|r(\xi)|^2} 
    		(\sgnt(z-\xi))^{-2i \sgnt \kappa(\xi)} e^{2it \theta} } 
    			& z \in \Sigma_3 \smallskip \\
    		\triu[0]{ \rho(\xi) \delta_0(\xi,\sgnt)^{2}  
    		(\sgnt(z-\xi))^{2i  \sgnt \kappa(\xi)} e^{-2it \theta} } 
				& z \in \Sigma_4 
		\end{cases}
	\end{gathered}
\end{equation}

\item For $\lam \in \C$ we have
\[
	\dbar \nk{2} (\lam) = \nk{2}(\lam) \dbar \mathcal{R}^{(2)}(\lam)
\]	
where
\begin{equation}\label{W2}
	\dbar \mathcal{R}^{(2)}(\lam) = \begin{cases}
		\tril[0]{-\dbar R_1(\lam) e^{-2it\theta} } & \lam \in \Omega_1 \smallskip \\
		\triu[0]{-\dbar R_3(\lam) e^{2it\theta} } & \lam \in \Omega_3 \smallskip \\
		\tril[0]{\dbar R_4(\lam) e^{-2it\theta} } & \lam \in \Omega_4 \smallskip \\
		\triu[0]{\dbar R_6(\lam) e^{2it\theta} } & \lam \in \Omega_6 \smallskip \\
		\qquad \quad \mathbf{0} & \textrm{elsewhere}
	\end{cases}
\end{equation}	

\item $\nk{2}(\lam)$ has simple poles at each point in $\poles$. For each $\lam_k \in \poles_+$ 
	\begin{equation}\label{n2 residue}
			\begin{aligned}
				\res_{\lam = \lam_k} \nk{2}(\lam) 
				&= \lim_{\lam \to \lam_k} \nk{2}(\lam) \vk{1}(\lam_k) \\
				\res_{\lam = \bar{\lam}_k} \nk{1}(\lam) 
				&= \lim_{\lam \to \bar{\lam}_k} \nk{2}(\lam) \vk{1}(\bar{\lam}_k)
			\end{aligned}
	\end{equation}
	where $\vk{1}$ is  given in \eqref{n1 residue matrices}.
\end{enumerate}
\end{DBAR}
%
%

%
%

\section{ Decomposition into a RH model problem and a pure \texorpdfstring{$\dbar$}{DBAR}-problem}
\label{sec:models}
The next step in our analysis is to construct the solution $\Nrhp$ of a \textit{matrix-valued} Riemann-Hilbert problem such that the transformation
\begin{equation}\label{n3 def}
	\nk{3}(\lam) = \nk{2}(\lam) (\Nrhp(\lam))^{-1}
\end{equation}
results in a pure $\dbar$-problem, \ie, the new unknown $\nk{3}$ is continuous; it has no jumps or poles. We arrive at the problem for $\Nrhp$ by essentially ignoring the $\dbar$ component of Problem~\ref{rhp.n2}.

%
%
\begin{RHP}\label{rhp.Nrhp}
Find a $2\times2$ matrix valued function 
$\Nrhp: \C \setminus (\Sk{2} \cup \poles) \to SL_2(\C)$ with the following properties
\begin{enumerate}[1.]
\item $\Nrhp$ satisfies the symmetry relation 
	\[
		\Nrhp(\lam) = \begin{pmatrix}
			\overline{\Nrhp[22](\lambar)} & \eps \lam^{-1} \overline{\Nrhp[21](\lambar)} \\
			\eps \lam \overline{\Nrhp[12](\lambar)} &\overline{\Nrhp[11](\lambar)}
		\end{pmatrix} 
	\]
\item $\Nrhp(\lam) = \tril{\alpha} + \bigo{\lam^{-1}}$ as $ \lam \to \infty$, for 
a constant $\alpha$ determined by the symmetry condition above.
\item For $\lam \in \Sk{2}$, the boundary values satisfy the jump relation 
	$\Nrhp[+](\lam) = \Nrhp[-](\lam) \vk{2}(\lam)$, where $\vk{2}$ is given by \eqref{V2}.
\item $\Nrhp(\lam)$ has simple poles at each point in $\poles$. 
	For each $\lam_k \in \poles_+$ 
	\begin{equation}\label{Nrhp residue}
			\begin{aligned}
				\res_{\lam = \lam_k} \Nrhp(\lam) 
				&= \lim_{\lam \to \lam_k} \Nrhp(\lam) 
				\tril[0]{\lam_k C_k \delta^{-2}(\lam_k) e^{-2it\theta} } \\
				\res_{\lam = \lambar_k} \Nrhp(\lam) 
				&= \lim_{\lam \to \lambar_k} \Nrhp(\lam) 
				\triu[0]{\eps \overline{C_k} \delta^{2}(\lambar_k) e^{2it\theta} }
			\end{aligned}
	\end{equation}
\end{enumerate}
\end{RHP}


%
%

\begin{lemma}\label{lem:n2.to.n3} 
Suppose that $\nk{2}$ solves $\dbar$-RHP~\ref{rhp.n2}. Given a solution $\Nrhp$ of RHP~\ref{rhp.Nrhp},
the function $\nk{3}$ defined by \eqref{n3 def} satisfies $\dbar$-~Problem~\ref{dbar.n3} below.
\end{lemma}

\begin{DBAR}
\label{dbar.n3}
Given $x,t \in \R$ and $\rho \in H^{2,2}(\R)$ with $\inf_{\R} (1 - \eps \lam|\rho(\lam)|^2) > 0$, find a continuous, row vector-valued function $\nk{3}(\lam)$ with the following properties 
\begin{enumerate}[1.]
	\item $\nk{3}(\lam) \to \begin{pmatrix} 1 & 0 \end{pmatrix}$
		as $|\lam | \to \infty$.
	\item $\dbar \nk{3}(\lam) = \nk{3}(\lam) \Wk{3}(\lam)$, where
	\begin{equation}\label{W3def}
		\Wk{3}(\lam) = \Nrhp (\lam)  \dbar \mathcal{R}^{(2)}(\lam) (\Nrhp)^{-1}(\lam).
	\end{equation}	
\end{enumerate}
\end{DBAR}

\begin{proof}[Proof of Lemma~\ref{lem:n2.to.n3}]
Given solutions $\nk{2}$ and $\Nrhp$ of $\dbar$-RHP~\ref{rhp.n2} and RHP~\ref{rhp.Nrhp} respectively, the normalization condition for $\nk{3}$ is immediate. As $\Nrhp$ is holomorphic in $\C \backslash \Sk{2}$, the $\dbar$-derivative of $\nk{3}$ satisfies
\begin{equation}\label{W3}
	\dbar \nk{3} = \dbar \nk{2} \Nrhp^{-1} = \lb \nk{2} 
\bar{\partial}\mathcal{R}^{(2)} 
	\rb \Nrhp^{-1} 
	= \nk{3} \lb \Nrhp 
 \bar{\partial}\mathcal{R}^{(2)} 
	\Nrhp^{-1} \rb.
\end{equation}

The computation 
\begin{equation}
	\begin{aligned}[t]
	\nk[+]{3}(\lam) 
		&= \nk[-]{3}(\lam) \Nrhp[-](\lam) \vk{2}(\lam) \Nrhp[+](\lam)^{-1}  
		 \\
		&= \nk[-]{3}(\lam) \Nrhp[-](\lam) \vk{2}(\lam) \lb \vk{2}(\lam)^{-1} \Nrhp[-](\lam)^{-1} \rb \\
		&= \nk[-]{3}(\lam). 
	\end{aligned}
	\tag*{$\lam \in \Sk{2}$} 
\end{equation}
shows that $\nk{3}$ has no jumps and is everywhere continuous. Another direct calculation shows that 
$\nk{3}$ has removable singularities at each pole in $\poles$: for instance if $p \in \poles$ and $\vk{1}(p)$ is the nilpotent residue matrix in \eqref{n1 residue} then 
using \eqref{n2 residue} and \eqref{Nrhp residue} we have the local Laurent expansion of at $p$ 
\begin{align*}
	\nk{2}(\lam) = a(p) \lb \frac{ \vk{1}(p)}{\lam - p} + I \rb + \bigo{(\lam -  {  p} )}\\ 
	\Nrhp(\lam) = A(p) \lb \frac{ \vk{1}(p)}{\lam - p} + I \rb + \bigo{(\lam -  {  p} )} 
\end{align*}
where $a(p)$ and $A(p)$ are the constant row vector and matrix in their respective expansions. As $\Nrhp \in SL_2(\C)$, 
\[
	(\Nrhp)^{-1} = \sigma_2 (\Nrhp)^\intercal \sigma_2
	= \lb \frac{ -\vk{1}(p)}{\lam - p} + I \rb \sigma_2 A(p)^\intercal \sigma_2 + \bigo{(\lam -   {  p} )}.  
\]
It follows that 
\begin{multline*}
	\nk{3}(\lam) = \nk{2}(\lam) \Nrhp(\lam)^{-1} \\
	= \left\{ a(p) \lb \frac{ \vk{1}(p)}{\lam - p} + I \rb + \bigo{(\lam -  {  p} )} \right\}  
	\left\{ \lb \frac{ -\vk{1}(p)}{\lam - p} + I \rb 
		\sigma_2 A(p)^\intercal \sigma_2 + \bigo{(\lam -  {  p}   )}  \right\} \\
	= \bigo{1}	.
\end{multline*}
where the last equality follows from the fact that $\vk{1}(p)^2 = 0$.

\end{proof}

The remainder of this section is dedicated to proving the following proposition
\begin{proposition}\label{prop:Nrhp.est}
Given $\rho \in H^{2,2}(\R)$ with $c:= \inf_{\lam \in \R} (1 - \eps \lam | \rho(\lam)|^2) >0$ strictly, then there exists $T> 0$ such that for $|t| > T$, there exists a unique solution $\Nrhp(\lam)$ of RHP~\ref{rhp.Nrhp} satisfying 
\[
	\oldnorm{\Nrhp(\lam)}_{L^\infty(\C \backslash B_\poles)} \lesssim 1
\] 
where $B_\poles$ is any open neighborhood of $\poles$ and the implied constants are uniform in $x$ and $|t|>T$; they depend on $B_\poles$ and $\rho$.
\end{proposition}

To prove the existence of $\Nrhp$, we will first construct two explicit models: one which exactly solves the pure soliton problem obtained by ignoring the jump conditions, and a second which uses parabolic cylinder functions to build a matrix whose jumps exactly match those of $\nk{2}$ in a neighborhood of the critical point $\xi$. Using our models we prove that $\Nrhp$ exists and extract its behavior for large $t$.

\subsection{The outer model: the soliton component}
\label{sec:outer model}

The matrix $\Nrhp$ is meromorphic away from the contour $\Sk{2}$ on which its boundary values satisfy the jump relation $\Nrhp[+](\lam) = \Nrhp[-](\lam) \vk{2}(\lam)$. It is clear from \eqref{V2} that 
\begin{equation}\label{V2 bound outside} 
	\left| \vk{2}(\lam) - I \right| \lesssim 
	e^{-2\sqrt{2} t |\lam-\xi|^2},
\end{equation}
where the implied constant depends upon $\poledist$ and $c:= \inf_{\lam \in \R}(1-\eps \lam |\rho(\lam)|^2)$. It follows that outside a fixed neighborhood of $\xi$ we introduce only exponentially small error (in  $t$) by completely ignoring the jump condition on $\Nrhp$. This results in the following outer model problem

\begin{problem}\label{outermodel} 
For any fixed $\xi \in \R$, let $\Nsol: \C \to SL_2(\C)$ be a meromorphic function such that
\begin{itemize}
	\item $\Nsol$ satisfies the symmetry relation 
	\begin{equation}\label{Nsol symmetry}
		\Nsol(\lam) = \begin{pmatrix}
			\overline{\Nsol[22](\lambar)} & \eps \lam^{-1} \overline{\Nsol[21](\lambar)} \\
			\eps \lam \overline{\Nsol[12](\lambar)} & \overline{\Nsol[11](\lambar)}
		\end{pmatrix} 
	\end{equation}
	\item $\Nsol(\lam) = \tril{\alpha(x,t)} + \bigo{\lam^{-1}}$ as $\lam \to \infty$, where $\alpha$ is determined via the symmetry condition.
	\item $\Nsol$ has a simple pole at each point in $\poles$ satisfying the residue relations in \eqref{Nrhp residue} with $\Nsol$ replacing $\Nrhp$.
	
\end{itemize}
\end{problem}

The essential fact we need concerning $\Nsol$ is as follows.

\begin{proposition}
\label{outer.soliton}
A unique solution $\Nsol$ of Problem~\ref{outermodel} exists. 
Moreover, the solution $\Nsol$ is precisely the matrix solution of RHP~\ref{RHP2} corresponding to the reflectionless scattering data $\mathcal{D}_\xi = \{ (\lam_k, 
{    \widetilde{C_k}}) \}_{k=1}^N$ generated by an exact $N$-soliton solution $\qsol(x,t;\mathcal{D}_\xi) $ of \eqref{DNLS2} where $\{\lam_k\}_{k=1}^N$ are the points generated by our original initial data \eqref{data} and the modified connection coefficients are given by 
\[
 {  \widetilde{C_k}	} = C_k \exp\lp \frac{i}{\pi} \int_{\negint}
	\log(1 - \eps z |\rho(z)|^2) \frac{dz}{z - \lam_k} \rp.
\]
That is, as $\lam \to \infty$, $\Nsol$ admits the expansion
\begin{equation}\label{Nsol.asymp}
	\Nsol(\lam) = \tril{ -\frac{\eps \overline{\qsol}(x,t;\mathcal{D}_\xi)}{2i} } + \frac{ \Nsol[1] }{\lam} + \bigo{\frac{1}{\lam^2}},
\quad \text{where} \quad
	2i (\Nsol[1] )_{12} = \qsol(x,t;\mathcal{D}_\xi).
\end{equation}
\end{proposition}
\begin{proof}
Comparing Problem~\ref{outermodel} to RHP~\ref{RHP2}, it is obvious that $\Nsol$ solves RHP~\ref{RHP2} for scattering data $\mathcal{D}_\xi$. 
{  Recalling that $\delta(\lambda)$ is given in \eqref{T},} this is exactly the scattering data generated by the $N$-soliton solution $\qsol(x,t;\mathcal{D}_\xi)$ of \eqref{DNLS2}. 
As was shown in {   Section \ref{sec:inverse} (Theorem  \ref{thm:RHP1.unique}), } a unique solution of RHP~\ref{RHP2} always exists for any admissible scattering data, which includes the data considered here. The expansion for large $\lam$ follows from the fact that $\Nsol$ is meromorphic, and the given off-diagonal coefficients of the leading and first moment terms follow from \eqref{q.lam} and the symmetry condition in Problem~\ref{outermodel}.

\end{proof}

Since $\Nsol$ is the solution of RHP~\ref{RHP2} for reflectionless scattering data, Lemma \ref{lem:sol.bound}  of Appendix~\ref{app:solitons} provides the following useful facts. 

\begin{lemma}\label{lem:outer.bound}
Given $\rho \in H^{2,2}(\R)$ and $\{ (\lam_k, c_k) \}_{k=1}^N \subset \C^+ \times \C^\times$ for RHP~\ref{RHP2}, the solution $\Nsol$ of Problem~\ref{outermodel} satisfies
\[
	\oldnorm{\Nsol}_{\infty} = \bigo{1}
\]
where the implied constant is independent of $(x,t) \in \R^2$ and depends on $\rho$ through its $H^{2,2}(\R)$ norm. 
\end{lemma}

\subsection{Local model at the {  stationary phase } point}
\label{sec:local model}
In any neighborhood of the critical point $\lam = \xi$ the bound \eqref{V2 bound outside} does not give a uniformly small estimate of the jump $\vk{2}$ for large times. It follows that our outer model, which replaced $\vk{2}$ with identity, is not a good approximation of $\Nrhp$ in a neighborhood of $\xi$. We require a new model $\NPC$ which is an accurate approximation inside a small--but fixed with respect to $|t|$--neighborhood of $\lam = \xi$. Let 
\begin{equation}\label{xi disk}
	\Uxi = \{ \lam \in \C \,:\, |\lam - \xi| \leq \poledist/3 \} 
\end{equation}
where the radius $\poledist/3$ is chosen such that $(1-\indicator(\lam))\equiv 1$ for $\lam \in \Uxi$; this has the effect of making the jump matrix $\vk{2}$, \eqref{V2}, constant along $\Sigma_k \cap \Uxi$, $k=1,\dots,4$.

Define the time-scaled local coordinate
\begin{gather}\label{zeta}
	\zeta(\lam) = |8t|^{1/2} (\lam - \xi).
\end{gather}
Under this change of variables we have the identifications
\[
	e^{2it \theta}  = {  e^{- i\eta \zeta^2/2}} e^{4it\xi^2},
	\qquad
	(\eta(\lam - \xi))^{2i\eta \kappa(\xi)} = (\eta \zeta)^{2i \eta \kappa(\xi)} e^{-i \eta \kappa(\xi) \log |8 t|}.
\]
Also set,
\begin{equation}
\begin{gathered}\label{PCconst}
	r_\xi = \rho(\xi) \delta_0(\xi,\sgnt)^2 e^{-i \sgnt \kappa(\xi) \log|8t|} e^{4it\xi^2} \\
	s_\xi = -\eps \xi \bar{\rho}(\xi) \delta_0(\xi,\sgnt)^{-2} e^{i \sgnt \kappa(\xi) \log|8t|} e^{-4it\xi^2},
\end{gathered}
\end{equation}
so that $1+r_\xi s_\xi = 1 - \eps \xi | \rho(\xi)|^2$.
\begin{figure}[htbp]
	\centering
	\hspace*{\stretch{1}}
	\FigPCjumps 
	\hspace*{\stretch{1}}
	\FigPCjumpsB
	\hspace*{\stretch{1}}
	\caption{
	The system of contours for the local model problem near $\lam = \xi$. 
	The model jumps are $\vk{\textsc{PC}} = 
	(\sgnt \zeta)^{i \sgnt \kappa(\xi)\sig} e^{-i \sgnt \zeta^2 \sig/4} V 
	(\sgnt \zeta)^{-i \sgnt \kappa(\xi)\sig} e^{i \sgnt \zeta^2 \sig/4}$ 
	where $V$ is given above in terms of the local variable $\zeta$ defined by \eqref{zeta}.
	\label{fig:PCjumps}
	}
\end{figure}

Using the notation just introduced, and extending the constant jump of $\vk{2}\Big|_{\lam \in \Uxi}$ to infinity along each of the four rays $\Sigma_k,\ k=1,\dots,4$, (see Figure~\ref{fig:PCjumps}) our local model $\NPC$ satisfies  

\begin{RHP}\label{rhp.localmodel}
	Find a $2\times2$ matrix-valued function $\NPC(\lam) = \NPC(\lam;\xi,\sgnt)$, analytic in $\C \setminus \Sk{2}$ with the following properties:
	\begin{enumerate}[1.]
		\item $\NPC(\lam;\xi,\sgnt) = \sdiag{1}{1} + \bigo{\lam^{-1}}$ as $|\lam| \to \infty$.
		\item $\NPC(\lam;\xi,\sgnt)$ has continuous boundary values 
		$\NPC[\pm](\lam; \xi,\sgnt)$ on $\Sk{2}$ which satisfy the jump relation 
		$\NPC[+] = \NPC[-] \vk{\textsc{pc}}$, where
		\begin{equation} \label{NPC jump}
			\vk{\textsc{pc}}(\lam) = 
			\begin{cases}
				\tril{ s_\xi(\eta \zeta)^{-2i \eta \kappa(\xi)} e^{i \eta \zeta^2/2} } 
					& z \in \Sigma_1, \smallskip \\
				\triu{\frac{r_\xi}{1+ r_\xi s_\xi} 
				(\eta \zeta)^{2i \eta \kappa(\xi)} e^{-i \eta \zeta^2/2} } 
					& z \in \Sigma_2, \smallskip \\
				\tril{\frac{s_\xi}{1+ r_\xi s_\xi} 
				(\eta \zeta)^{-2i \eta \kappa(\xi)} e^{i \eta \zeta^2/2} } 
					& z \in \Sigma_3, \smallskip \\
				\triu{  r_\xi (\eta \zeta)^{2i \eta \kappa(\xi)}  e^{-i \eta \zeta^2/2} } 
					& z \in \Sigma_4. 
			\end{cases} 
		\end{equation}
	\end{enumerate}
\end{RHP}
\begin{remark}
RHP~\ref{rhp.localmodel} does not possess the symmetry condition shared by RHP~\ref{rhp.Nrhp} and Problem~\ref{outermodel}. This is because it is a local model and will only be used for bounded values of $\lam$. The normalization is chosen such that the residual error $\error$ defined by \eqref{error def} has a near identity jump on the shared boundary between the local and outer models. 
\end{remark}

This type of model problem is typical in integrable systems whenever there is a phase function, here $\theta$, which has a quadratic critical point along the real line 
The solution in each of these cases is found by a further reduction of RHP~\ref{rhp.localmodel} to a problem with constant jumps (at the price of nontrivial behavior at infinity) whose solution {  satisfies a} 
differential equation,
 which can be solved using parabolic cylinder functions, $D_a(z)$, whose properties are tabulated in \cite[Chapter 12]{DLMF}. The precise details of the construction for DNLS, which differ only slightly from the construction for KdV or NLS can be found in \cite{LPS16}; here we give only the necessary details.

\begin{proposition}\label{prop:PCmodel.est}
Fix $\xi$ and let $\kappa = \kappa(\xi)$ be as given in \eqref{T}. 
Then for any choice of constants $r_\xi, s_\xi$ in \eqref{PCconst} such that $1 + r_\xi s_\xi = e^{-2\pi \kappa} \neq 0$, the solution $\NPC(\lam;\xi,\sgnt)$ of RHP~\ref{rhp.localmodel} is given by  
\begin{gather}\label{local model soln}
	\begin{aligned}
		\NPC(\lam;\xi, +) &= F(\zeta(\lam); s_\xi, r_\xi)  \\
		\NPC(\lam;\xi, -) &= \sigma_2 F(-\zeta(\lam); r_\xi, s_\xi) \sigma_2 
	\end{aligned}
\shortintertext{where,}	
\nonumber	
	F(\zeta; s,r) := \Phi_{s,r}(\zeta) \mathcal{P}_{s,r}(\zeta) 
	\zeta^{-i \kappa \sig} e^{i\zeta^2\sig/4} \\
\nonumber
	\mathcal{P}_{s,r}(\zeta) = 
	\left\{ \begin{array}{c@{\quad}l@{\hspace{2em} }c@{\quad}l}
		\tril{s_\xi} & \arg \zeta \in \lp 0, \frac{\pi}{4} \rp, &
		\triu{\frac{r_\xi}{1+r_\xi s_\xi}} & \arg \zeta \in \lp \frac{3\pi}{4},\pi \rp, \smallskip \\ 
		\triu{-r_\xi} & \arg \zeta \in \lp -\frac{\pi}{4}, 0 \rp,  &
		\tril{\frac{-s_\xi}{1+r_\xi s_\xi}} & \arg \zeta \in \lp -\pi, -\frac{3\pi}{4} \rp, \\
		\multicolumn{4}{c}{\diag{1}{1} \quad |\arg \zeta| \in \lp \frac{\pi}{4},\frac{3\pi}{4} \rp.} 
	\end{array}\right.
\\
\nonumber	
\label{SimpliPhi+}
 	\Phi_{s,r} (\zeta)= 
	\begin{pmatrix}
 		{e^{-\frac{3\pi}{4}\kappa} D_{i\kappa}(\zeta e^{-3i\pi/4})} &
		{\dfrac{e^{\frac{\pi}{4}(\kappa-i)}}{\beta_{21}(s,r) }
			(-i\kappa) D_{-i\kappa-1}(\zeta e^{-\pi i/4})} \\
		{\dfrac{e^{-\frac{3\pi}{4}(\kappa+i)}}{\beta_{12}(s,r)}
			i\kappa D_{i\kappa -1}(\zeta e^{-3i\pi/4})} &
		{e^{\pi\kappa/4}D_{-i\kappa}(\zeta e^{-i\pi/4})}	
	\end{pmatrix}
\shortintertext{for $\Im(\zeta)>0$, and for $\Im(\zeta)<0$}
\nonumber
\label{SimpliPhi-}
 	\Phi_{s,r} (\zeta)= 
	\begin{pmatrix}
		{e^{\pi\kappa/4}D_{i\kappa}(\zeta e^{\pi i/4})} &
		{-\dfrac{i\kappa}{\beta_{21}(s,r)} e^{-\frac{3\pi}{4}(\kappa-i)}
				D_{-i\kappa-1}(\zeta e^{3i\pi/4})} \\
		{\dfrac{(i\kappa)}{\beta_{12}(s,r)} e^{\frac{\pi}{4}(\kappa+i)}
				D_{i\kappa-1}(\zeta e^{\pi i/4})} &
		{e^{-3\pi \kappa/4}D_{-i\kappa}(\zeta e^{3i\pi/4})}
	\end{pmatrix}.
\shortintertext{where,}
\label{betas}
	\beta_{12}(s,r) = \frac{ \sqrt{2\pi}e^{-\pi \kappa/2} e^{i \pi/4} }{s \Gamma(-i\kappa)}
	\qquad
	\beta_{21}(s,r) = \frac{\kappa}{\beta_{12}(s,r)} 
		= \frac{ -\sqrt{2\pi}e^{-\pi \kappa/2} e^{-i \pi/4} }{r \Gamma(i\kappa)}.
\shortintertext{As $\zeta \to \infty$}
	F(\zeta; s,r) = I + \frac{1}{\zeta} \offdiag{ -i\beta_{12}(s,r)}{i\beta_{21}(s,r)}
	+ \bigo{\zeta^{-2}}.
\end{gather}
\end{proposition}

The essential property of $\NPC$ that we will need later is the asymptotic expansion for large $\zeta$. 
Using \eqref{local model soln},  we have 

\begin{equation} \label{local model expand}
	\NPC(\lam;\xi,\sgnt) = 
	 I + \frac{|8t|^{-1/2}}{\lam - \xi} A(\xi,\sgnt)
	 	+ \bigo{t^{-1}}, \\
	\qquad \lam \in \partial \Uxi,
\end{equation}
where
\begin{equation}\label{Axi}
	A(\xi,\sgnt) = \offdiag{ -iA_{12}(\xi, \sgnt) }{iA_{21}(\xi, \sgnt)}
\end{equation}
{  with 
\begin{align*}
&A_{12} (\xi, +) = \beta_{12} (s_\xi, r_\xi ), \   \  A_{21} (\xi, +) = \beta_{21} (s_\xi, r_\xi ) \nonumber \\
&A_{12} (\xi, -) = -\beta_{21}(r_\xi, s_\xi),  \  A_{21} (\xi, -) = -\beta_{12}(r_\xi, s_\xi)
\end{align*}
satisfies}

\begin{equation}\label{mod beta}
    	|A_{12}(\xi,\sgnt)|^2 =  \frac{\kappa(\xi)}{\xi} 
		\qquad \qquad 
		A_{21}(\xi, \sgnt) = \eps \xi \overline{A_{12}(\xi,\sgnt)}
\end{equation}
\begin{subequations}\label{beta arg}
\begin{align}
	&\begin{multlined}[.9\textwidth]
	\arg A_{12}(\xi, +) = 
	\frac{\pi}{4} + \arg \Gamma(i \kappa(\xi)) - \arg ( -\eps \xi \overline{\rho(\xi)} ) \\
	+ \frac{1}{\pi} \int_{-\infty}^\xi \log|\xi - \lam|\, \mathrm{d}_\lam \log(1-\eps \lam |\rho(\lam)|^2) 
	- \kappa(\xi) \log|8t| + 4t \xi^2 
	\end{multlined} \\
	&\begin{multlined}[.9\textwidth]
	\arg A_{12}(\xi, -) = 
	\frac{\pi}{4} - \arg \Gamma(i \kappa(\xi)) - \arg (-\eps \xi \overline{\rho(\xi)} ) \\
	+ \frac{1}{\pi} \int_{\xi}^\infty \log|\xi - \lam|\, \mathrm{d}_\lam \log(1-\eps \lam |\rho(\lam)|^2) 
	+ \kappa(\xi) \log|8t| + 4t \xi^2 
	\end{multlined} 
\end{align}	
\end{subequations}

The first line of \eqref{local model expand} and \eqref{mod beta} are proved in \cite{LPS16}; the second  of \eqref{local model expand} follows easily from the first and the \eqref{local model soln}. The second line of \eqref{mod beta} is a consequence of the fact that $\beta_{12} \beta_{21} = \kappa$. Equations \eqref{beta arg} follow simply from \eqref{betas} and \eqref{PCconst}
where we use \eqref{delta0 arg} and integration by parts to express the integral terms.

We will also need the values of the model problem at $z=0$, which from \eqref{zeta} gives $\zeta(0) = -|8t|^{1/2} \xi$. Note that, though \eqref{local model soln} is piecewise defined across the real axis, $\NPC$ does not have a jump across the real axis (cf. \eqref{NPC jump}); in the formulas below we have  chosen the components of $\Phi_{s,r}$ for which right multiplication by $\mathcal{P}_{s.r}(\zeta)$ has no effect. {  The first column  $\NPC[1]$  of $\NPC$ at $\lambda=0$ is given by:}
\begin{equation}\label{NPC1.at.0}
	\NPC[1](0;\xi,\sgnt) = 
	e^{\frac{\pi \kappa(\xi)}{4}} 
	{e^{2it\xi^2 - \frac{i}{2} \sgnt \kappa(\xi)  \log |8t\xi^2|} }
	\diag{1}{ -e^{\frac{i\sgnt \pi}{4}} \sgn(\xi) }
	\begin{bmatrix}
	 D_{i\sgnt \kappa(\xi)} \lp e^{\frac{i\sgnt \pi}{4}} |8t\xi^2|^{1/2} \rp \\
	 iA_{21}(\xi,\sgnt) D_{i\sgnt \kappa(\xi)-1} 
	   \lp e^{\frac{i\sgnt \pi}{4}} |8t\xi^2|^{1/2} \rp
	\end{bmatrix}
\end{equation}
\begin{lemma}\label{lem:NPC21}
Let $c_1, c_2, c_3$ be strictly positive constants, and suppose that $\rho \in H^{2,2}(\R)$ with $\oldnorm{\rho}_{H^{2,2}(\R)} \leq c_1$, $\inf_{\lam \in \R} (1 - \eps \lam |\rho(\lam)|^2 ) \geq c_2$, and $|\xi| < c_3$. Then as $|t| \to \infty$, 
\[
	|\Nsol[21](0;\xi,\sgnt)|  \lesssim t^{-1/2},
\]
where the implied constant is independent of $\xi$ and $\rho$.
\end{lemma}

\begin{proof}
	From \eqref{NPC1.at.0} and \eqref{mod beta} we have, setting  $p := e^{\frac{i\sgnt \pi}{4}} |8t\xi^2|^{1/2} $,
\begin{align*}
	|\NPC[21](0; \xi,\sgnt)|	
	= 
	e^{\frac{\pi \kappa(\xi)}{4} } 
	\left| 
	  A_{21}(\xi,\sgnt) 
	  D_{i\eta \kappa(\xi) - 1} \lp  p \rp 
	\right| 
	= 
	\left| \frac{\kappa(\xi)}{8t\xi} \right|^{1/2} 
	\left| 
	  e^{\frac{\pi \kappa(\xi)}{4} } p 
	  D_{i\eta \kappa(\xi) - 1} \lp   p \rp.
	\right|,
\end{align*}
Since $\kappa(\xi)/\xi \to \tfrac{\eps}{2\pi} |\rho(0)|^2$ as $\xi \to 0$, 
 {  it is sufficient to } show that the last factor is bounded in $p \geq 0$. For finite $p$ this is trivial, and for large $p$, the asymptotic expansion of $D_{\nu}(z)$ \cite[Eq. 12.9.1]{DLMF} gives 
\[
	\left| 
	  e^{\frac{\pi \kappa(\xi)}{4} } p 
	  D_{i\eta \kappa(\xi) - 1} \lp  p \rp
	\right|
	= \left| e^{-i p^2/4} p^{i \sgnt \kappa(\xi)} \left[1 + \bigo{p^{-2}} \right] \right|
	= 1 + \bigo{p^{-2}}.
	\qedhere
\]
\end{proof}

We also need the following boundedness property:
\begin{lemma}[{see \cite[Appendix D]{LPS16}}]
\label{lem:PCmodel bound}
Let $c_1$ and $c_2$ be strictly positive constants, and suppose that $\rho \in H^{2,2}(\R)$ with $\oldnorm{\rho}_{H^{2,2}(\R)} \leq c_1$ and $\inf_{\lam \in \R} (1 - \eps \lam |\rho(\lam)|^2 ) \geq c_2$.
Then, 	
\begin{gather*}
	\oldnorm{ \NPC(\, \cdot\, ; \xi,\sgnt) }_\infty \lesssim 1 \\
	\oldnorm{ \NPC(\, \cdot\, ; \xi,\sgnt)^{-1} }_\infty \lesssim 1,
\end{gather*}
where the implied constants are uniform in $\xi$ and $|t|>1$ and depends only on $c_1$ and $c_2$.
\end{lemma}

\subsection{Existence theory for the RH model problem}
\label{sec:RHP.exist}
In this section, we prove that the solution $\Nrhp$ of our model problem, RHP~\ref{rhp.Nrhp} exists, by constructing it from the outer and local models introduced previously. The models are not an exact solution, and some residual error persists. We will show that for large times, the error solves a small norm Riemann-Hilbert problem which we can expand asymptotically.

Write the solution $\Nrhp$ of RHP~\ref{rhp.Nrhp} in the form 
\begin{equation}\label{error def}
	\Nrhp(\lam) = 
	\begin{cases}
		\error(\lam) \Nsol(\lam) & |\lam-\xi| \notin \Uxi \\
		\error(\lam) \Nsol(\lam) \NPC(\lam) & |\lam-\xi| \in \Uxi \\
	\end{cases}
\end{equation}
where  {  $\Uxi$ is defined in \eqref{xi disk},} $\Nsol$, the solution of Problem~\ref{outermodel}, and $\NPC$, the solution of RHP~\ref{rhp.localmodel}, are both  bounded functions of $(x,t)$ {  having determinant equal to $1$}. This relation implicitly defines a transformation to a new unknown $\error$ which satisfies a new RH problem. In order to state it let 
\begin{equation}\label{error contour}
	\Ske = \partial \Uxi \cup (\Sigma_2 \setminus \Uxi)
\end{equation}
where the circle $\partial \Uxi$ is oriented counter clockwise. 

\begin{RHP}
\label{rhp.E}
Find a $2 \times 2$ matrix value function $\error$ analytic in $\C \backslash \Ske$ with the following properties
\begin{enumerate}[1.]
	\item For $z \in \C \setminus \calU_\xi$, $\error$ satisfies the symmetry relation 
	\[
		\error(\lam) = \begin{pmatrix}
			\overline{\error_{22}(\overline{\lam})} & \eps \lam^{-1}\overline{\error_{21}(\overline{\lam})}\\
			\eps \lam \overline{\error_{12}(\overline{\lam})} & \overline{\error_{11}(\overline{\lam})}
		\end{pmatrix} 
	\]	
	\item $\error (\lam) = \tril{\newtext{\overline{q}_\error}} + \bigo{\lam^{-1}}$ as $|\lam| \to \infty$, for a constant $\newtext{\overline{q}_\error}$ determined by the symmetry condition above. 
	\item For $\lam \in \Ske$, the boundary values $\error_\pm$ satisfy the jump relation $\error_+(\lam) = \error_-(\lam) \vke(\lam)$ where
\begin{equation}\label{error jump}
	\vke(\lam) = 
	\begin{cases}
		\Nsol(\lam) \vk{2}(\lam) \Nsol(\lam)^{-1} & \lam \in \Sigma_2 \setminus \Uxi \\
		\Nsol(\lam) \NPC(\lam)^{-1}  \Nsol(\lam)^{-1} &  z \in \partial \Uxi
	\end{cases}
\end{equation}

\end{enumerate}
\end{RHP}

The jump matrix $\vke$ is uniformly near identity for large times; it follows from \eqref{V2 bound outside}, \eqref{local model expand} and Lemma~\ref{lem:outer.bound}, that 
\begin{gather}\label{error jump bound0}
	\left| \vke(\lam) - I \right| \lesssim 
	\begin{cases}
		t^{-1/2} & \lam \in \partial \Uxi \\
		e^{-2\sqrt{2} t |\lam-\xi|^2} & \lam \in \Sigma_2 \setminus \Uxi
	\end{cases}
\shortintertext{and} 
	\label{error jump bound}
	\oldnorm{\vke - I}_{L^{2,k}(\R) \cap L^\infty(\R)} \lesssim t^{-1/2}, \quad k \in {   \N}.
\end{gather}
There is a well known existence and uniqueness theorem for RHPs with near identity jump matrices \cite{}. 
Let $C_\error$ denote the Cauchy integral operator
\begin{equation}
	C_\error f  = C^-(f(\vke-I)),
\end{equation}
where $C^-$ is the usual Cauchy projection operator on $\Ske$: 
\[
	C^- f(\lam) = \lim_{\lam \to \Ske[-]} \frac{1}{2\pi i} \int_{\Ske} \frac{ f(z)}{z-\lam} dz.
\]
The essential fact needed for the small-norm theory is that $C_\error$ is a small norm operator,
\begin{equation}\label{error.op.bound}
	\oldnorm{ C_\error }_{L^2(\Ske) \rarr L^2(\Ske)}
	= \bigo{ \| \vke - I \|_\infty } = \bigo{t^{-1/2}}.
\end{equation}

\begin{lemma}
\label{lem:Err}
Suppose that $\rho \in H^{2,2}(\R)$ and $c:= \inf_{\lam \in \R} (1 - \lam | \rho(\lam)|^2) > 0$ strictly. 
Then, for sufficiently large times $|t|>0$, there exists a unique solution $\error(\lam;x,t)$ of RHP~\ref{rhp.E} with the property that 
\[
	\oldnorm{\error - {\begin{pmatrix} 1 & 0 \\\overline{q}_\error & 1 \end{pmatrix}}  }_{L^\infty(\C) } \lesssim t^{-1/2}.
\]
Moreover, as $\lam \to \infty$
\[
	\error(\lam) =  {\begin{pmatrix} 1 & 0 \\ \overline{q}_\error & 1 \end{pmatrix}}
	  + \lam^{-1} \error_{1} + \bigo{ \lam^{-2} }
\]

where ${{q}_\error} := \eps \lp \error_{1} \rp_{12}$ and
\begin{equation}\label{E1.12}
	2i \lp \error_{1} \rp_{12} = 
{ 	\frac{1}{|2t|^{1/2} }  }
	\lb A_{12}(\xi, \sgnt) \Nsol[11](\xi)^2 
		+ A_{21}(\xi, \sgnt) \Nsol[12](\xi)^2 
	\rb + \bigo{t^{-1}}, 
\end{equation}
Here, $\Nsol$ is the solution of the Problem~\ref{outermodel} described in Lemma~\ref{outer.soliton}
while $A_{12}$ and $A_{21}$ are given by \eqref{Axi}-\eqref{beta arg}. 
\end{lemma}
 
\begin{proof}
Due to the nonstandard normalization we will construct the solution $\error$ row-by-row.
We begin by considering the first row, which we denote $\vect{e}_1 = \begin{pmatrix} \error_{11} & \error_{12} \end{pmatrix}$, which is canonically normalized.

By standard results in the theory of Cauchy integral operators \cite{DZ03}, $\vect{e}_1$ must satisfy
\begin{equation}\label{row1}
	\vect{e}_1(\lam) = (1 \ 0 )
	+ \frac{1}{2\pi i} \int_{\Ske} \frac{ ((1 \ 0)  + \vect{\mu}_1(z) )(\vke(z)-I) }{z-\lam} dz
\end{equation}
where $\vect{\mu}_1 \in L^2(\Ske)$ is the unique row vector solution of 
\begin{equation}\label{mu}
	(\one - C_\error) \vect \mu_1 = C_\error (1 \ 0)
\end{equation}
The existence and uniqueness of $\vect \mu_1$ follows immediately from \eqref{error.op.bound}
which establishes the existence of $(\one- C_\error)^{-1}$, and allows one to construct $\vect \mu_1$ by Neumann series, moreover, we have 
\begin{equation}\label{mu bound}
	\oldnorm{\vect \mu_1 }_{ {  L^2(\Ske) }} \lesssim  
	\frac{ \oldnorm{ C_\error }_{L^2(\Ske) \rarr L^2(\Ske)} }{1- \oldnorm{ C_\error }_{L^2(\Ske) \rarr L^2(\Ske)} }
	\lesssim t^{-1/2}.
\end{equation}
Fix a small constant $d$ and suppose that $\inf_{z \in \Ske} |\lam - z| > d$, then 
\[
	\left| \vect{e}_1 - (1\ 0) \right| \leq  
	\frac{d^{-1}}{2\pi} \lp
		\oldnorm{\vke - I}_{{  L^1}} + \oldnorm{\vect \mu_1}_{{  L^2}} \oldnorm{\vke -I}_{{  L^2}} \rp \lesssim t^{-1/2}.
\]
To get $L^\infty$ control for $\lam$ approaching $\Ske$ we observe that the jumps on the contours $\Ske$ are locally analytic, and so can be freely deformed locally by a bounded invertible transformation $\vect{e}_1 \mapsto \widetilde{\vect{e}}_1$. The previous argument then goes through to show that $\left| \widetilde{ \vect{e}}_1 - (1\ 0) \right|$ is bounded on $\Ske$ which then gives a similar bound on $\vect{e}_1$ as the transformation itself is bounded. 

To build the second row $\vect{e}_2 = \begin{pmatrix} \error_{21} & \error_{22} \end{pmatrix}$, we begin by using the symmetry condition to compute $\newtext{\overline{q}_\error}$. Since $\error_{21}(\lam) = \eps \lam \overline{\error_{12}(\overline{\lam})}$ for all large $\lam$, we use \eqref{row1} and take the limit as $\lam \to \infty$ to find
\[
	\newtext{q_\error} = \frac{\eps}{2 \pi i } \lp \int_{\Ske} \lb (1\ 0) + \vect{e}_1(z) \rb \lb \vke{z} - I \rb_2 dz \rp,
\]
where the subscript $2$ on the second factor of the integrand denotes the second column of the matrix. 
Finally, using \eqref{error jump bound} and \eqref{mu bound} we have the bound
\begin{equation}\label{alpha bound}
	|\newtext{\overline{q}_\error}| \lesssim t^{-1/2}.
\end{equation}

Now that $\overline{q}_\error$ is well defined, we construct the second row as 
\begin{equation}\label{row2}
	\vect{e}_2(\lam) = ({\overline{q}_\error} \ 1 )
	+ \frac{1}{2\pi i} \int_{\Ske} \frac{ (({\overline{q}_\error} \ 1)  + \vect{\mu}_2(z) )(\vke(z)-I) }{z-\lam} dz
\end{equation}
where
\begin{equation}\label{mu2}
	(\one - C_\error) \vect \mu_2 = C_\error ({\overline{q}_\error} \ 1).
\end{equation}
Then repeating the arguments above we have that $\oldnorm{ \vect{\mu}_2}_{L^2{\Ske} } \lesssim  t^{-1/2}$ and $\oldnorm{\vect{e}_2 - ({\overline{q}_\error} \ 1)}_{L^\infty(\C)} \lesssim t^{-1/2}$.

Define the matrices $\error = \twovec{ \vect{e}_1 }{\vect{e}_2 }$ and $\vect{\mu} = \twovec{ \vect{\mu}_1 }{\vect{ \mu}_2 }$. Then, for large $\lam$ write $\error_0 = 
{\begin{pmatrix} 1 &0 \\ \overline{q} &1
\end{pmatrix}}  $
and
\begin{gather*}
	\error(\lam) = \error_0 + \lam^{-1} \error_{1} + \lam^{-2} \mathcal{S}(\lam) \\
	\error_{1} = \frac{-1}{2\pi i} \int\limits_{\Ske} (\error_0 + \vect \mu(z))(\vke(z) - I) dz, 
	\qquad
	\mathcal{S}(\lam) = \frac{\lam}{2\pi i} 
		\int\limits_{\Ske} \frac{(\error_0 + \vect \mu(z)) z (\vke(z) - I)}{(z-\lam)} dz .
\end{gather*}
Using \eqref{error jump bound} and \eqref{mu bound}, as $\lam \to \infty$ 
the residual $\mathcal{S}$ satisfies 
\[
	|\mathcal{S} | \lesssim \oldnorm{V_E-I }_{L^{2,2}(\R)} \lesssim t^{-1/2},
\]
while using \eqref{alpha bound} we have
\begin{equation}\label{E1}
	\begin{aligned}
		\error_{1} 
		&= -\frac{1}{2\pi i} \oint_{\partial \Uxi} (\vke(z) -I) dz + \bigo{t^{-1}} \\
		&= |8t|^{-1/2} \Nsol(\xi) {A(\xi,\sgnt) } \Nsol(\xi)^{-1} + \bigo{t^{-1}}.
	\end{aligned}
\end{equation}
The last line in \eqref{E1} follows from a residue calculation using \eqref{local model expand}. A direct calculation using the fact that $\det \Nsol = 1$ then gives \eqref{E1.12}.
\end{proof}

Combining Lemmas~\ref{lem:outer.bound}, \ref{lem:PCmodel bound}, and \ref{lem:Err}, it follows from \eqref{error def} that

\begin{proposition}\label{prop:Nrhp.bound}
Let $c_1$ and $c_2$ be strictly positive constants, and suppose that $\rho \in H^{2,2}(\R)$ with $\oldnorm{\rho}_{H^{2,2}(\R)} \leq c_1$ and $\inf_{\lam \in \R} (1 - \eps \lam |\rho(\lam)|^2 ) \geq c_2$.
Then, 	
\begin{gather*}
	\oldnorm{ \Nrhp }_\infty \lesssim 1 \qquad
	\oldnorm{ (\Nrhp)^{-1} }_\infty \lesssim 1,
\end{gather*}
where the implied constants are uniform in $\xi$ and $t>1$ and depend only on $c_1$ and $c_2$.
\end{proposition}

When estimating  the gauge factor for  the solution  u of the DNLS equation (Section  \ref{subsec:sol-u}),  we need the following result that provides the large-time 
behavior of the error term $\error$ at $z=0$ : 
\begin{proposition}
\label{prop:E}
Suppose that $\rho \in H^{2,2}(\R)$ and $c:= \inf_{\lam \in \R} (1 - \eps \lam |\rho(\lam)|^2) > 0$ strictly. Then, as $|t| \to \infty$ the unique solution of RHP~\ref{rhp.E} described by Lemma~\ref{lem:Err} satisfies
\begin{align}
\label{e11.at0.out}	
	\error_{11}(0) 
	&= 1
	- \frac{i\eps}{|8t|^{1/2}} 
	\lb
		\Nsol[12](\xi) \overline{A_{12}(\xi,\sgnt) \Nsol[11](\xi)}
		+ \overline{\Nsol[12](\xi)} A_{12}(\xi,\sgnt) \Nsol[11](\xi) 
	\rb \\
\nonumber
	&\quad
	+ \bigo{t^{-1}}
\shortintertext{
whenever $0 \neq \Uxi$; when $0 \in \Uxi$ it satisfies
}
\label{e11.at0.in}	
	\error_{11}(0) 
	&= 1
	- \frac{i\eps}{|8t|^{1/2}} \times \\
\nonumber
	&\quad
	\lb
		\overline{A_{12}(\xi,\sgnt) }
		\lp 
			\Nsol[12](\xi) \overline{\Nsol[11](\xi)} 
		  - \Nsol[12](0)  \overline{\Nsol[11](0)}
		\rp
		+ A_{12}(\xi,\sgnt) \overline{\Nsol[12](\xi)} \Nsol[11](\xi) 	
	\rb \\
\nonumber
	&\quad
	+ \bigo{t^{-1}} 
\shortintertext{and}
\label{e12.at0.in}	
	\error_{12}(0) 
	&= \frac{i\eps}{|8t|^{1/2}} 
	\lb
		\eps A_{12}(\xi,\sgnt) \frac{ \Nsol[11](\xi)^2 - \Nsol[11](0)^2}{\xi}
		+ \overline{(A_{12}(\xi,\sgnt)} \lp \Nsol[12](\xi)^2 - \Nsol[12](0)^2 \rp
	\rb \\
\nonumber
	&\quad 
	+ \bigo{t^{-1}}
\end{align}
\end{proposition}

\begin{proof}
	Write $\vect{e}_1 = (\error_{11} \ \error_{12})$ for the first row of $\error$. Then starting from \eqref{row1} we use \eqref{error jump} and \eqref{local model expand} together with the bounds \eqref{error jump bound0}, \eqref{error jump bound} and \eqref{mu bound} to write
\[
	\vect{e}_1(0) = (1\ 0) - \frac{(1\ 0)}{|8t|^{1/2} 2\pi i} \int_{\partial \Uxi} 
	\frac{ \Nsol(z) A(\xi,\sgnt) \Nsol(z)^{-1}}{z(z-\xi)} dz + \bigo{t^{-1}}.
\]	
A residue calculation, using the symmetry condition \eqref{Nsol symmetry} and \eqref{mod beta} to simplify the result, completes the proof. 
\end{proof}

\subsection{The remaining \texorpdfstring{$\dbar$}{DBAR} problem}
\label{sec:dbar}

We are ready to consider the pure $\dbar$-Problem~\ref{dbar.n3} for $\nk{3}$. 
The next proposition describes  its large-time asymptotics.

\begin{proposition}
\label{prop:N3.est}
Given $\rho \in H^{2,2}(\R)$ and $c \coloneqq \inf_{\lam \in \R} \left(1- \lam |\rho( \lam)|^2 \right) >0$ strictly.
Then, for sufficiently large time $t>0$, there exists a unique solution $\nk{3}(\lam;x,t)$ for $\dbar$-Problem~\ref{dbar.n3} with the property that 
\begin{equation}
\label{N3.exp}
\nk{3}(\lam;x,t) = I + \frac{1}{\lam}  \nk[1]{3}(x,t) + \littleo[\xi,t]{\frac{1}{\lam}}
\end{equation}
for $\lam=i y$ with $y \to +\infty$ where

\begin{equation}
\label{N31.est}
\left| \nk[1]{3}(x,t) \right| \lesssim t^{-3/4}
\end{equation}
 where the implied constant in \eqref{N31.est} is independent of $\xi$ and $t$ and uniform for 
$\rho$ in a bounded subset of $H^{2,2}(\R)$ with $\inf_{\lam \in \R} (1-\lam|\rho(\lam)|^2) \geq c>0$ for a fixed $c>0$.
\end{proposition}

\begin{proposition}\label{prop:n3 at 0}
Given the same assumptions as Proposition~\ref{prop:N3.est}  and for sufficiently large times $t>0$, the unique solution $\nk{3}(\lam;x,t)$ of $\dbar$-Problem~\ref{dbar.n3} satisfies
\begin{equation}
	\nk[11]{3}(0;x,t) = 1 + \bigo{t^{-3/4}}
\end{equation}
where the implied constant is independent of $\xi$ and $t$ and is uniform for $\rho$ in a bounded subset of $H^{2,2}(\R)$ with $\inf_{\lam \in \R} (1-\lam|\rho(\lam)|^2) \geq c>0$ for a fixed $c>0$. 	
\end{proposition}

\begin{remark}
The remainder estimate in \eqref{N3.exp} need not be (and is not) uniform in $\xi$ and $t$; what matters for the proof of Theorem \ref{thm:long-time} is that the implied constant in the estimate \eqref{N31.est} for $\nk[1]{3}(x,t)$ is independent of $\xi$ and $t$.
\end{remark}

To prove Proposition~\ref{prop:N3.est} we recast $\dbar$-Problem~\ref{dbar.n3} as a Fredholm-type integral equation using the solid Cauchy transform
\[
	(Pf)(\lam) = \frac{1}{\pi} \int_\C \frac{1}{\lam - z} f(z) \, dm(z) 
\]
where $dm$ denotes Lebesgue measure on $\C$.  The following lemma is standard.

\begin{lemma}
A continuous, bounded row vector-valued function $\nk{3}(\lam;x,t)$ solves $\dbar$~Problem~\ref{dbar.n3} if and only if
\begin{equation}
\label{DNLS.dbar.int}
\nk{3}(\lam;x,t) = (1,0) + \frac{1}{\pi} \int_\C \frac{1}{\lam - z} \nk{3}(z;x,t) \Wk{3}(z;x,t) \, dm(z).
\end{equation}
\end{lemma}

\begin{proof}[Proof of Proposition \ref{prop:N3.est}, given Lemmas \ref{lemma:KW}--\ref{lemma:N31.est}]
As in \cite{BJM16}  and \cite{DM08}, we   first show that, for large times, the integral operator $K_W$ 
defined by
\begin{equation*} 
	\left( K_W f \right)(\lam) = \frac{1}{\pi} \int_\C \frac{1}{\lam-z} f(z) \Wk{3}(z) \, dm(z)
\end{equation*}
(suppressing the parameters $x$ and $t$) 
obeys the estimate
\begin{equation} \label{dbar.int.est1}
	\norm[L^\infty \to L^\infty] {K_W}\lesssim t^{-1/4} 
\end{equation}
where the implied constants depend only on $\norm[H^{2,2}]{\rho}$ and 
$c : = \inf_{\lam \in \R} \left( 1- \eps \lam|\rho(\lam)|^2 \right)$ 
and, in particular, are independent of $\xi$ and $t$.   
This is the object of Lemma \ref{lemma:KW}.
In particular, this shows that the solution formula
\begin{equation}
\label{N3.sol}
\nk{3} = (I-K_W)^{-1} (1,0) 
\end{equation}
makes sense and defines an $L^\infty$ solution of \eqref{DNLS.dbar.int} bounded uniformly in $ \xi \in \R$ and $\rho$ in a bounded subset of $H^{2,2}(\R)$ with $c>0$.  

We then prove that the solution $\nk{3}(\lam;x,t)$ has a large-$\lam$ asymptotic expansion of the form \eqref{N3.exp} where $\lam \to \infty$ along the \emph{positive imaginary axis} (Lemma \ref{lemma:N3.exp}). Note that, for such $\lam$, we can bound $|\lam-z|$ below by a constant times $|\lam|+|z|$.
The remainder need not be bounded uniformly in $\xi$. Finally, we  prove  estimate \eqref{N31.est}.
\end{proof}

\begin{proof}[Proof of Proposition~\ref{prop:n3 at 0}, given Lemmas \ref{lemma:KW}--\ref{lemma:N31.est}]
From \eqref{DNLS.dbar.int} we have 
\[
	\nk[11]{3}(0) = 1 - \frac{1}{\pi} \int_\C  
	\frac{\nk[11]{3}(z)\Wk[11]{3}(z)+ \nk[12]{3}(z)\Wk[21]{3}(z)}{z} dm(z).
\] 
Computing $\Wk[11]{3}$ {and $\Wk[21]{3}$} using \eqref{W3def} and the symmetry \eqref{Nsol symmetry}, recalling that $\dbar \mathcal{R}^{(2)}$ has zeros on its diagonal, gives
\begin{align*}
	&\frac{\Wk[11]{3}(z)}{z}  = 
	\Nsol[12](z) \overline{\Nsol[11](\bar z)} \frac{ \dbar \mathcal{R}_{21}^{(2)}(z)}{z}
   +\eps \Nsol[11](z) \overline{\Nsol[12](\bar z)} \dbar \mathcal{R}_{12}^{(2)}(z). \\
&{
	\frac{\Wk[21]{3}(z)}{z} =
	\overline{\Nsol[11](\bar z)}^2 \frac{\dbar \mathcal{R}_{21}^{(2)}(z)}{z}
   -z \overline{\Nsol[12](\bar z)}^2 \dbar \mathcal{R}_{12}^{(2)}(z) 
}
\end{align*}
{Equations \eqref{dbar.int.est1} and \eqref{N3.sol} imply that $| \nk{3}(z) | \lesssim 1$. Using} Lemma~\ref{lem:outer.bound} then gives
\[
	{\left| \nk{3}(0) - 1 \right| \lesssim }
		\int_\C \left| \frac{ \dbar \mathcal{R}_{21}^{(2)}(z)}{z} \right| 
		+ \left|\dbar \mathcal{R}_{12}^{(2)}(z) \right| dm(z)  = \bigo{t^{-3/4}}.
\]
Where the last equality uses Corollary~\ref{cor:R2.bd} to control the size of each term in the integrand, allowing identical estimates as those used to bound $\int |\Wk{3}(z)| dm(z)$ in 
{  Proposition ~\ref{prop:N3.est}} to establish the result. 
\end{proof}

Estimates \eqref{N3.exp}, \eqref{N31.est}, and \eqref{dbar.int.est1} rest on the  bounds stated in three  lemmas.

\begin{lemma} \label{lemma:KW}
Suppose that $\rho \in H^{2,2}(\R)$ and $c := \inf_{\lam \in \R} \left(1- \eps \lam |\rho(\lam)|^2 \right) >0$ strictly.
Then, the estimate \eqref{dbar.int.est1} holds, where the implied constants depend on $\norm[H^{2,2}]{\rho}$ and $c$.
\end{lemma}

\begin{proof}
To prove \eqref{dbar.int.est1}, first note that
\begin{align}
 	\norm[\infty]{K_W f} & \leq \norm[\infty]{f} \int_\C \frac{1}{|\lam-z|}|\Wk{3}(z)| \, dm(z) 
\end{align}
where, using Proposition~\ref{prop:Nrhp.bound}, 
\[
	|\Wk{3}(z)| \leq 
	\norm[\infty]{\Nrhp} \norm[\infty]{(\Nrhp)^{-1}} \left| \dbar \mathcal{R}^{(2)}\right|
	\lesssim \left| \dbar \mathcal{R}^{(2)}\right|.
\]
We will prove the estimate for $z \in \Omega_1$ since estimates for $\Omega_3$, $\Omega_4$, and $\Omega_6$ are similar.  
Setting $\lam =\alpha+i\beta$ and $z-\xi=u+iv$, the region $\Omega_1$ corresponds to 
\begin{equation}\label{omega.1}
	\Omega_1= \left\{ (\xi+u,v): v \geq 0, \, v \leq u < \infty\right\}.
\end{equation}
We then have from Corollary~\ref{cor:R2.bd} that
\[
	\int_{\Omega_1}  \frac{1}{|\lam-z|} |\Wk{3}(z)| \, dm(z)  \lesssim  I_1 + I_2 + I_3 + I_4
\]
where
\begin{align*}
	I_1 &= 
		\int_0^\infty \int_v^\infty \frac{1}{|\lam-z|} |p_1'(u)| e^{-8tuv} \, du \, dv \\[5pt]
	I_2	&= 
		\int_0^1 \int_v^1 \frac{1}{|\lam-z|} \left| \log(u^2 + v^2) \right| e^{-8tuv} \,du\,dv\\[5pt]
	I_3	&= 
		\int_0^\infty \int_v^\infty \frac{1}{|\lam-z|} \frac{1}{1+ |z-\xi|} e^{-8tuv} \, du \, dv\\[5pt] 
	I_4 &= 
		\int_0^\infty \int_v^\infty \frac{1}{|\lam-z|} |\indicator(z)| e^{-8tuv} \, du \, dv. 
\end{align*}
We recall from \cite[proof of Proposition C.1]{BJM16} the bound
\begin{equation*} \label{BJM16.bd1}
	\norm[L^2(v,\infty)]{\frac{1}{\lam - z}} \leq \frac{\pi^{1/2}}{|v-\beta|^{1/2}}
\end{equation*}
where $z= \xi + u+iv$ and $\lam =\alpha + i \beta$. Using this bound and Schwarz's inequality on the $u$-integration we may bound $I_1$  by constants times
\[
	(1+\norm[2]{p_1'})  \int_0^\infty \frac{1}{|v-\beta|^{1/2}} e^{-tv^2} \, dv \lesssim t^{-1/4}
\]
(see for example \cite[proof of Proposition C.1]{BJM16} for the estimate)
For $I_2$, we remark that $ |\log(u^2+v^2)| \lesssim 1+ |\log(u^2)|$ and that $1+ |\log(u^2)|$ is square-integrable on $[0,1]$. We can then argue as before to conclude that
$ I_2 \lesssim t^{-1/4}$. Similarly, the inequality 
\[ 
	\frac{1}{1+|z-\xi|} \leq \frac{1}{1+u}
\] 
and the finite support of $\indicator$ shows that we can bound $I_3$ and $I_4$ in a similar way.

It now follows that
\[ 
	\int_{\Omega_1} \frac{1}{|\lam-z|} |\Wk{3}(z)| \, dm(z) \lesssim t^{-1/4} 
\]	
which, together with similar estimates for the integrations over $\Omega_3$, $\Omega_4$, and $\Omega_6$,  proves \eqref{dbar.int.est1}.
\end{proof}

\begin{lemma}
\label{lemma:N3.exp}
For $\lam=iy$, as $y \to +\infty$, the expansion \eqref{N3.exp} holds with 
\begin{equation}
\label{N3.1}
\nk[1]{3}(x,t) = \frac{1}{\pi} \int_{\C} \nk{3}(z;x,t) \Wk{3}(z;x,t) \, dm(z) . 
\end{equation}
\end{lemma}

\begin{proof}
We write \eqref{DNLS.dbar.int} as 
\[
	\nk{3}(\lam;x,t) = (1,0) + \frac{1}{\lam} \ \nk[1]{3}(x,t) 
	+ \frac{1}{\pi \lam} \int_{\C} \frac{z}{\lam-z} \nk{3}(z;x,t) \Wk{3}(z;x,t) \, dm(z)
\]
where $\nk[1]{3}$ is given by \eqref{N3.1}. If $\lam =iy$ and $z \in \Omega_1 \cup \Omega_3 \cup \Omega_4 \cup \Omega_6$, it is easy to see that $|\lam|/|\lam-z|$ is bounded above by a fixed constant independent of $\lam$, while $|\nk{3}(z;x,t)| \lesssim 1$ by the remarks following \eqref{N3.sol}. If we can show that $\int_\C |\Wk{3}(z;x,t)| \, dm(z)$ is finite, it will follow from the Dominated Convergence Theorem that 
\[
	\lim_{y \to \infty} \int_\C \frac{z}{iy-z} \nk{3}(z;x,t) \Wk{3}(z;x,t) \, dm(z) = 0 
\] 
which implies the required asymptotic estimate. We will estimate the integral $\displaystyle \int_{\Omega_1} |\Wk{3}(z)| \, dm(z)$ since the other estimates are similar.  
Using Corollary~\ref{cor:R2.bd} and \eqref{omega.1}, we may then estimate
\[
\int_{\Omega_1} |W^{(3)}(z;x,t)| \, dm(z)	\lesssim  I_1+I_2+I_3+I_4
\]
where, 
\begin{align*}
I_1	&=	\int_0^\infty \, \int_v^\infty \left| p_1'(\xi+u) \right| e^{-8tuv} \, du \, dv\\
I_2	&=	\int_0^1 \int_v^1 \left| \log(u^2 + v^2) \right| e^{-8tuv} \, du \, dv \\
I_3	&=	\int_0^\infty \, \int_v^\infty \frac{1}{\sqrt{1+u^2+v^2}} e^{-8tuv} \, du \, dv \\
I_4 	&= 		\int_0^\infty \int_v^\infty  |\indicator(z)| e^{-8tuv} \, du \, dv. 
\end{align*}
To estimate $I_1$, we use the Schwarz inequality on the $u$-integration to obtain
\[
	I_1 \leq 	\norm[2]{p_1'} \frac{1}{4\sqrt{t}} \int_0^\infty \frac{1}{\sqrt{v}}e^{-  8 tv^2} \, dv
		=		\norm[2]{p_1'} \frac{\Gamma(1/4)}{ 8^{5/4}t^{3/4}}.
\]
Since $\log(u^2+v^2) \leq \log(2u^2)$ for $v \leq u \leq 1$, we may similarly bound
\[
	I_2  \leq 	\norm[L^2(0,1)]{\log(2u^2)} \frac{\Gamma(1/4)}{8^{5/4}t^{3/4}}.
\]
To estimate $I_3$, we note that $1+u^2+v^2 \geq  1+u^2$ and $(1+u^2)^{-1/2} \in L^2(\R^+)$, so we may 
conclude that
\[
	I_3 \leq \norm[2]{(1+u^2)^{-1/2}} \frac{\Gamma(1/4)}{ 8^{5/4}t^{3/4}}.
\]
Finally, as $\indicator$ is finitely supported, by similar bounds
\[
	I_4 \leq C_\poles \frac{\Gamma(1/4)}{ 8^{5/4}t^{3/4}}.
\]
where the constant $C_\poles$ depends only on the discrete spectrum $\poles$. 
These estimates together show that
\begin{equation} \label{W.L1.est}
	\int_{\Omega_1} |\Wk{3}(z;x,t)| \, dm(z)	\lesssim t^{-3/4}
\end{equation}
and that the implied constant depends only on $\norm[{H^{2,2}}]{\rho}$ and $\poles$.  
In particular, the integral \eqref{W.L1.est} is bounded uniformly as $t \to \infty$.
\end{proof}

The estimate \eqref{W.L1.est} is also strong enough to prove 
\eqref{N31.est}:

\begin{lemma} \label{lemma:N31.est}
The estimate \eqref{N31.est} 
holds with constants uniform in $\rho$ in a bounded subset of $H^{2,2}(\R)$ and $\inf_{\lam \in \R} \left(1-\eps \lam |\rho(\lam)|^2 \right) > 0$ strictly.
\end{lemma}

\begin{proof}
From the representation formula \eqref{N3.1}, Lemma \ref{lemma:KW}, and the remarks following, we have
\[ 
	\left|\nk[1]{3}(x,t) \right| \lesssim \int_\C |\Wk{3}(z;x,t)| \, dm(z). 
\]
In the proof of Lemma \ref{lemma:N3.exp}, we bounded this integral by $t^{-3/4}$ modulo constants with the required uniformities.
\end{proof}


%
%

\section{Large-time asymptotics for solutions of DNLS} 
\label{sec:largetime}

We now gather  the estimates on the RHPs considered previously to reconstruct asymptotic formulae for  the solutions $q(x,t)$ and $u(x,t)$ 
of \eqref{DNLS2} and \eqref{DNLS1} to
 prove our main results, namely Theorems \ref{thm:long-time} and \ref{thm:long-time-gauge}  that provide a precise description of their  long-time behavior.

We start  with a  generalized version of the separation of solitons. For any real interval $I$, let
\begin{equation}
	\poles(I) = 
	\{ \lam \in \poles \,:\, \Re \lam \in I \},
	\qquad
	N(I) = \left| \poles(I) \right|
\end{equation}
denote the set of discrete spectra which lie within the vertical strip extending over $I$ and the cardinality of this set respectively; 
also let 
\begin{equation}
	\nu_0(I)  = \nu_0 = \min_{ \lam \in \poles \setminus \poles(I)} 
	\dist(\Re \lam, I) .
\end{equation}

\begin{proposition}
\label{prop:sol separation}
Let $\qsol(x,t;\mathcal{D})$ denote the $N$-soliton solution of \eqref{DNLS2} encoded in RHP~\ref{RHP2} given reflectionless scattering data $\mathcal{D} = \{(\lam_k,C_k) \}_{k=1}^N$. Fix real constants $x_1,x_2$ and $v_1 <v_2$. Let $I = [-v_2/4, -v_1/4]$. 
Then as $|t| \to \infty$ inside the cone 
\[
	x_1 + v_1 t \leq x \leq x_2 + v_2 t  
\]
we have
\[
	\left| \qsol(x,t;\mathcal{D}) - \qsol(x,t;\mathcal{D}_I) \right| 
	= \bigo{ e^{-4\poledist \nu_0 t} },
\]
where $\qsol(x,t;\mathcal{D}_I)$ is the reduced $N(I)$-soliton solution of \eqref{DNLS2} associated with scattering data 
$\mathcal{D}_I = \{(\lam_k, \widehat{C}_k) \,:\, \lam_k \in \poles(I) \}$ and
\[
	\widehat C_k = C_k 
	\smashoperator[l]{ \prod_{
	\substack{\lam_j \in \poles \setminus \poles(I) \\ \Re \lam_j \in  \negint}
	}}
	\lp \frac{ \lam_k - \lam_j}{ \lam_k - \overline{\lam_j} } \rp^2 ~.
\]
\end{proposition}
The proof of this proposition is given at the end of Appendix \ref{app:solitons}.

\subsection{Large-time asymptotic for solution \texorpdfstring{$q$}{q}  of  (1.3) }
 \label{subsec:sol-q}

The solution $q$ is recovered through the reconstruction formula
\eqref{q.lam}
\begin{equation} \label{recons-q}
q(x,t)  = \lim_{z\to \infty}  2iz n_{12}(x,t,z),
\end{equation}
where $z \rarr \infty$ in any direction not tangent to the contour $\R$ and
$\nn$  satisfies RHP~\ref{RHP2}. 
Inverting the sequence of transformations \eqref{n1}, \eqref{n2 def}, \eqref{n3 def} we construct the solution $\nn$  as 
\begin{align} \label{n.soln}
	\nn(\lam) 
	= \nk{3}(\lam) \Nrhp(\lam) \mathcal{R}^{(2)}(\lam)^{-1}  \delta(\lam)^{\sig}.
\end{align}

\begin{proof}[Proof of Theorem~\ref{thm:long-time}]
It follows from \eqref{Texpand} and Corollary~\ref{cor:R2.bd} 
that as $\lam \to \infty$ non-tangentially to the real axis, 
\begin{align*}
	 \delta(\lambda))^{\sig} &= I -\sigma_3  \lam^{-1} \delta_1(\xi,\eta) + \bigo{ \lam^{-2} },
	\qquad
	\delta_1(\xi,\eta) =  i \int_{\negint} \kappa(z) dz \\
	\mathcal{R}^{(2)} &= I + \bigo{ e^{-c|t|} }. 
\end{align*}
From 	\eqref{error def}, we have 
\begin{equation*}
\Nrhp(\lam) = \error(\lam)\Nsol(\lam) 
\end{equation*}
and  the large-$\lambda$ behavior of  $\error(\lam)$ and $\Nsol(\lam)$ are given in Proposition~\ref{outer.soliton}, Lemma~\ref{lem:Err} as 
$$\mathcal{E}(\lam) = \mathcal{E}_0 + \lam^{-1} \mathcal{E}_1 + \bigo{ \lam^{-2} }, \quad \mathcal{E}_0 = \tril{ \bar q_\mathcal{E} }
$$
and
$$
\Nsol(\lam) = \mathcal{N}_0 + \lam^{-1} \mathcal{N}_1 + \bigo{ \lam^{-2} }, \quad \Nsol[0] = \tril{  \eps \overline{ (\Nsol[1])_{12} } }.
$$	
Inserting these expansions into \eqref{n.soln} and making use of Proposition~\ref{prop:N3.est} yields the following expansion formula for $\nn$:
\[
	\nn(\lam) = (1, 0) 
	\left[ \Big(\mathcal{E}_0 + \frac{\mathcal{E}_1}{\lambda} + \bigo{ \lam^{-2} }\Big) \Big(\mathcal{N}_0 + \frac{\mathcal{N}_1}{\lam} + \bigo{ \lam^{-2} }\Big)
\Big(I -\sigma_3  \frac{ \delta_1}{\lambda} + \bigo{ \lam^{-2} }\Big)\right].
\].

The reconstruction formula \eqref{q.lam}  gives
\begin{gather*}
	q(x,y) = 2i (\Nsol[1])_{12} + 2i (\error_1)_{12} + \bigo{t^{-3/4}}.
\end{gather*}
We complete the proof by using \eqref{Nsol.asymp} and 
Proposition ~\ref{prop:sol separation} to write $2i (\Nsol[1])_{12} =  \qsol(x,t; \mathcal{D}_\xi) = \qsol(x,t; \mathcal{D}_I) + \bigo{e^{-4\poledist \nu_0 t} }$ and \eqref{E1.12} to identify $2i (\error_1)_{12}$ with the correction factor $f(x,t)$ in Theorem~\ref{thm:long-time}. 
\end{proof}

\subsection{Large-time asymptotic for solution \texorpdfstring{$u$}{u}  of  (1.1) }
 \label{subsec:sol-u}

To prove Theorem~\ref{thm:long-time-gauge} we construct the solution $u(x,t)$ of the DNLS equation \eqref{DNLS1} with initial data $u_0$ by means of the inverse gauge transformation
\begin{equation}\label{u.gauge}
	u(x,t) = q(x,t) \exp \left( i \epsilon \int_{-\infty}^x |q(y,t)|^2 dy \right).
\end{equation}
As we have the large-time behavior of $q(x,t)$ in hand, to compute the large-time behavior of $u(x,t)$ it will suffice to evaluate the large-time asymptotics of the expression
\begin{equation}\label{u.gauge.fact}
	\exp \left( i \epsilon \int_{-\infty}^x |q(y,t)|^2 dy \right).
\end{equation}
\begin{proposition}\label{prop:u.gauge.expansion}
	Suppose that $q_0 \in H^{2,2}(\R)$ and that $q(x,t)$ solves \eqref{DNLS2} with initial data $q_0$.
	Let $\{ \rho, \{(\lam_k, c_k)\}_{k=1}^N \}$ be the scattering data associated to $q_0$. {Fix $\xi = -x/(4t)$ and $M>0$}. 
Fix constants $v_1,v_2,x_1,x_2 \in \R$ and define {$\mathcal{S}$}, $I$ and $\mathcal{D}_I$ as described in Theorem~\ref{thm:long-time}. Then as $|t| \to \infty$ {with $(x,t) \in \mathcal{S}(v_1,v_2,x_1,x_2)$:} \\

\noindent
We have for {$\xi \geq M|t|^{1/8}$} 
\begin{multline}\label{u.gauge.out}
	\exp \lp i \eps \int_{-\infty}^x |q(y,t)|^2 dy \rp = 
	\lb
	1
	+ \frac{i\eps}{|2t|^{1/2}} 
	\lb 2 \real \Nsol[12](\xi) \overline{A_{12}(\xi,\sgnt) \Nsol[11](\xi)} \rb
	+ \bigo{t^{-3/4}}
	\rb \\
	\times \exp 
	\lp 
		i \eps \int_{-\infty}^x |\qsol(y,t;\mathcal{D}_\xi)|^2 dy  
		- \frac{i}{\pi} \int_{\posint} 
		  \frac{ \log(1 - \eps \lam | \rho(\lam)|^2)}{\lam} d\lam 
	\rp,
\end{multline}
while for {$\xi \leq M|t|^{-1/8}$}
\begin{multline}\label{u.gauge.in}
	\exp \lp i \eps \int_{-\infty}^x |q(y,t)|^2 dy \rp = 
	\\
	F(\xi,t,\sgnt)
	\Bigg[
	  1  
	  + \frac{i\eps}{|2t|^{1/2}} 
	  \bigg\{
		\Nsol[12](\xi) \overline{A_{12}(\xi,\sgnt)} \overline{ \Nsol[11](\xi)}
		+ \overline{\Nsol[12](\xi)} A_{12}(\xi,\sgnt) \Nsol[11](\xi)  
	 \\
	 + \overline{A_{12}(\xi,\sgnt)} 
	 \lp 1 - G(\xi,t,\sgnt) \rp 
	 e^{4i\sum_{k=1}^N \arg \lam_k } \int_x^\infty \usol(y,t;\mathcal{D}_\xi) dy
	\bigg\} 
	+\bigo{t^{-3/4}} \Bigg] 
	\\
	\times \exp
	\lp
	    i \eps \int_{-\infty}^x |\qsol(y,t;\mathcal{D}_\xi)|^2 dy  
		- \frac{i}{\pi} \int_{\posint} 
		  \frac{ \log(1 - \eps \lam | \rho(\lam)|^2)}{\lam} d\lam
	\rp
\end{multline}
where
\begin{equation}\label{FandG}
\begin{aligned}
	F(\xi,t,\sgnt) &=  \lb e^{p^2/4} p^{-i \sgnt \kappa(\xi)}
	   D_{i\eta \kappa(\xi)} \lp p \rp \rb^{-2} \\
	G(\xi,t,\sgnt) &= 
	\frac
	{ 
	  p
	  D_{i\eta \kappa(\xi) - 1} \lp p \rp
	}
	{D_{i\eta \kappa(\xi)} \lp p \rp},
\end{aligned}
\qquad 
p := e^{\frac{i\sgnt \pi}{4}} |8t\xi^2|^{1/2} 
\end{equation}
and $\usol$ is defined by \eqref{usol.1}.
\end{proposition}

We will prove this proposition with the help of the following weak Plancherel-like result:

\begin{lemma}\label{lem.Plancherel}
	Suppose that $q(x,t)$ is the solution of \eqref{DNLS2} for initial data $q_0 \in H^{2,2}(\R)$ and let $\{ \rho, \{(\lam_k, c_k)\}_{k=1}^N \}$ be the scattering data associated to $q_0$. Then the identity
\begin{equation} \label{Plancherel}
	\exp \lb i \eps \int_\R |q(y,t)|^2 \rb = 
		\exp 
		\lb 
			-4i \lp \sum_{k=1}^N \arg \lam_k \rp 
			- \frac{i}{\pi} \int_\R 
			\frac{ \log(1 - \eps \lam | \rho(\lam)|^2) } {\lam} d\lam 
		\rb
\end{equation}
holds.	
\end{lemma}

\begin{proof}
 Notice that both sides of equality \eqref{Plancherel} are time-independent quantities.
We begin by considering the following  identity for  the transmission coefficient  $\alpha(\lambda)$, analytic in the lower half-plane
\begin{equation}\label{alpha.trace}
	\alpha(\lam) = \prod_{k=1}^N \lp \frac{ \lam - \overline{\lam_k}}{\lam - \lam_k} \rp 
	\exp \lp \int_\R \frac{ \log(1 - \eps z | \rho(z)|^2) }{z - \lam} \frac{d z}{2\pi i} \rp, 
	\quad \Im \lam < 0,
\end{equation}
which we derive in 
Appendix~\ref{trace-f}.
 On the other hand, we can express $\alpha$ in terms of the normalized Jost function matrices $N^{\pm}(x,t,\lam)$ defined by\footnote{
In \eqref{direct.n.ac}, $q(x)$ should be replaced by the time evolved solution of $q(x,t)$ \eqref{DNLS2} with initial data $q_0(x)$.} 
\eqref{direct.n.de}-\eqref{direct.n.ac}; 
combining the relations \eqref{direct.n.jc}-\eqref{transition-lambda} and taking the limit as $x \to -\infty$ using \eqref{direct.n.ac} gives 
\begin{equation}\label{alpha.jost}
	\alpha(\lam) = \lim_{x \to -\infty} \NN_{11}^+(x,t,\lambda). 
\end{equation}	

Consider \eqref{direct.n.de}-\eqref{direct.n.ac} for $\lam \approx 0$
\begin{equation}
	\begin{gathered}
	\od{\NN^+}{x} = 
		\begin{pmatrix}
			-\tfrac{i\eps}{2} |q(x,t)|^2 & q(x,t) \\ 
			0 & \tfrac{i\eps}{2} |q(x,t)|^2
		\end{pmatrix} \NN^+
		+ \lambda \left[ \begin{pmatrix} -i & 0 \\ \eps \overline{q(x,t)} & i \end{pmatrix} N^+    +    N^+ \begin{pmatrix} i & 0 \\ 0 & -i \end{pmatrix} \right]
	\\
	\lim_{x\to \infty} \NN^+(x,t,\lam) = \begin{pmatrix} 1  &0\\ 0 &1 \end{pmatrix}.
	\end{gathered} .
\end{equation}
As $\lam = 0$ is a regular point of this system of equations, one can easily show that 
\begin{equation}\label{N1.near.0}
\NN^+(x,t,\lam) = 
	e^{ \frac{i \eps \sig}{2} \int_x^\infty |q(y,t)|^2 dy }
	\begin{pmatrix} 
	 1 & -\int_{x}^\infty q(y,t) e^{-i\eps \int_{y}^\infty |q(w,t)|^2 dw} dy \\ 
	 0 & 1 
	\end{pmatrix} + \bigo{\lam} 
\end{equation}
and in particular,
\begin{equation}\label{N1.at.0}
	 \NN^+(x,t,0) = \NN_-(x,t,0) = \NN(x,t,0),
\end{equation}
where we have used \eqref{BC.right} to replace the column vector Jost function 
with the first column of the Beals-Coifman solution $\NN_-(x,t,0)$ of RHP~\ref{RHP2}. We can drop the minus-boundary value because the jump relation in Problem \ref{RHP2}(iii) gives ${\NN_{11}}_+(x,t,0) = {\NN_{11}}_-(x,t,0)$ so that $\NN_{11}(x,t,\lam)$ is continuous at the origin.
Evaluating $\alpha(0)^2$ two ways: by combing \eqref{N1.at.0} with \eqref{alpha.jost}; and evaluating \eqref{alpha.trace} at $\lambda = 0$, gives the result.
\end{proof}

\begin{proof}[Proof of Proposition~\ref{prop:u.gauge.expansion}]
The gauge factor \eqref{u.gauge.fact} can be expressed exclusively in terms of spectral information: 
\begin{multline}\label{u.gauge.spectral}
	\exp \lp i \eps \int_{-\infty}^x |q(y,t)|^2 dy \rp 
	=\exp \lp -i \eps \int_{x}^\infty |q(y,t)|^2 dy \rp
	   \exp \lp i \eps \int_{-\infty}^\infty |q(y,t)|^2 dy \rp \\
	=\nn_{11}(x,t,0)^{-2} \exp 
	\lp 
		-4i \sum_{k=1}^N ( \arg \lam_k ) 
		- \frac{i}{\pi} \int_\R \frac{ \log(1- \eps \lam | \rho(\lam)|^2)}{\lam} d\lam 
	\rp
\end{multline}
where we have used Lemma~\ref{lem.Plancherel} and \eqref{N1.near.0}-\eqref{N1.at.0} to arrive at the second line above. It remains to find an asymptotic expansion for $n_{11}(0;x,t)$. Starting from \eqref{n.soln} we observe, see  Figure~\ref{fig:n2def} and equation~\eqref{R_k}, that 
$\mathcal{R}^{(2)}(0) = \begin{pmatrix} 1 & * \\ 0 & 1 \end{pmatrix}$ 
is upper triangular;
similarly, the symmetry in condition 1 of RHP~\ref{rhp.Nrhp} guarantees that $\Nrhp[21](0) = 0$.
So (suppressing $x,t$ dependence) we have
\begin{gather}
	\nonumber
	\nn_{11}(0) = \lb \nk{3}(0) \Nrhp(0) 
	(\mathcal{R}^{(2)}(0))^{-1} \rb_{11} \delta(0)  
	= \nk[11]{3}(0) \Nrhp[11](0) \delta(0), 
\shortintertext{which using \eqref{T} and Proposition~\ref{prop:n3 at 0} gives}
	\label{n11.at.0.a}
	\nn_{11}(0) = \Nrhp[11](0) \exp \lp 
		-\frac{i}{2\pi} \int_{\negint} 
		\frac{ \log( 1 - \eps \lam |\rho(\lam)^2|)}{\lam} d\lam \rp + \bigo{t^{-3/4}}.
\end{gather}
The value of $\Nrhp[11](0)$ depends on the location of the point $\xi$ in the spectral plane.

 If $|\xi| > \poledist/3$ then (cf. \eqref{xi disk}) $0 \not\in \Uxi$, so it follows from \eqref{error def} that  
\begin{equation}\label{nrhp.at.0.a}
	\Nrhp[11](0) 
		=  \error_{11}(0) \Nsol[{11}](0) + \error_{12}(0) \Nsol[_{21}](0) 
		= \error_{11}(0) 
		\exp \lp \frac{i\eps}{2} \int_x^\infty |\qsol(y,t,\mathcal{D}_\xi)|^2 dy \rp
\end{equation}
where in the last line we've used the fact that as $\Nsol$ is also a solution of RHP~\ref{RHP2} corresponding to the reduced scattering data for the soliton potential $\qsol(x,t; \mathcal{D}_\xi)$ (cf. Proposition~\ref{outer.soliton}) and so it must satisfy \eqref{N1.at.0} for $q = \qsol(x,t;\mathcal{D}_\xi)$. 
Plugging \eqref{n11.at.0.a} and \eqref{nrhp.at.0.a} into \eqref{u.gauge.spectral} gives
\begin{multline}
	\exp \lp i \eps \int_{-\infty}^x |q(y,t)|^2 dy \rp =
	\error_{11}(0)^{-2} \\
	\times
	\exp 
	\lp 
	  -i\eps \int_x^\infty |\qsol(y,t,\mathcal{D}_\xi)|^2 dy 
	  -4i \sum_{k=1}^N \arg \lam_k
	  -\frac{i}{\pi} \int_{\posint} \frac{ \log(1 - \eps \lam | \rho(\lam)|^2)}{\lam} d\lam
	\rp.
\end{multline}
This implies  \eqref{u.gauge.out}, using \eqref{e11.at0.out} and the 
equality that
\[
	\exp \lp -i\eps \int_x^\infty |\qsol(y,t,\mathcal{D}_\xi)|^2 dy - 4i \sum_{j=1}^N   \arg \lam_j \rp
	= \exp \lp i\eps \int_{-\infty}^x |\qsol(y,t,\mathcal{D}_\xi)|^2 dy 
	\rp,
\]
which expresses the soliton component of the gauge transform.

If $|\xi| < \poledist/3$ then $0 \in \Uxi$. We expand \eqref{error def}, using Lemma~\ref{lem:NPC21} and \eqref{e12.at0.in} to drop terms of order $t^{-1}$: 
\begin{equation}\label{Nrhp.at0.in.a}
\begin{aligned}
	\Nrhp[11](0) 
		&=  \lb \error(0) \Nsol(0) \NPC(0) \rb_{11} \\
		&= \error_{11}(0) \Nsol[11](0) \NPC[11](0) 
		 + \error_{11}(0) \Nsol[12](0) \NPC[21](0)
		 + \bigo{t^{-1}} \\
		 &=  \error_{11}(0) \Nsol[11](0) \Big(  \NPC[11](0) + \frac{\Nsol[12](0)}{\Nsol[11](0)} 
		 \NPC[21](0) \Big)
		 + \bigo{t^{-1}}. 
\end{aligned}
\end{equation}
As before, we use \eqref{N1.near.0} applied to $\Nsol$ and  $q = \qsol(x,t;\mathcal{D}_\xi)$, to write (suppressing $x,t$ dependence) 
\begin{equation}\label{working}
\begin{aligned}
	\frac{\Nsol[12](0)}{\Nsol[11](0)} = \Nsol[12](0) \overline{\Nsol[11](0)}
	&= -\int_{x}^\infty \qsol(y;\mathcal{D}_\xi) 
	e^{-i\eps \int_{y}^\infty |\qsol(w;\mathcal{D}_\xi)|^2 dw} dy \\
	&= - e^{-i\eps \| \qsol(\cdot,\mathcal{D}_\xi) \|^2_{L^2(\R)} } 
	\int_{x}^\infty \qsol(y;\mathcal{D}_\xi) 
	e^{i\eps \int^{y}_{-\infty} |\qsol(w;\mathcal{D}_\xi)|^2 dw} dy \\
	&= -\exp \lp 4i\sum_{k=1}^N \arg \lam_k \rp \int_x^\infty \usol(y,\mathcal{D}_\xi) dy
\end{aligned}
\end{equation}
where we have used \eqref{u.gauge} to define
\begin{equation}\label{usol.1}
	u_{\mathrm{sol}}(x,t;\mathcal{D}_\xi)  =
	\qsol(x,t;\mathcal{D}_\xi) 
	\exp \lp
		  i\eps \int_{-\infty}^x |\qsol(y,t;\mathcal{D}_\xi)|^2 dy
	\rp
\end{equation}
and the weak Plancherel identity \eqref{Plancherel} for reflectionless potential $\qsol$.
Inserting \eqref{working} into \eqref{Nrhp.at0.in.a}
\begin{multline}\label{nrhp.at.0.b}
	\Nrhp[11](0) =
	\error_{11}(0) \NPC[11](0) 
	\exp \lp \frac{i\eps}{2} \int_x^\infty |\qsol(y,t,\mathcal{D}_\xi)|^2 dy \rp 
	\\
	\times \lb	1 
	  -\frac{\NPC[21](0)}{\NPC[11](0)} 
	  e^{4i \sum_{k=1}^N \arg \lam_k} 
	  \int_{x}^\infty u_{\mathrm{sol}}(y,t) dy
	  + \bigo{t^{-1}}
	\rb .
\end{multline}
Finally, substituting \eqref{n11.at.0.a} and \eqref{nrhp.at.0.b} into \eqref{u.gauge.spectral} gives
\begin{multline}
	\exp \lp i \eps \int_{-\infty}^x |q(y,t)|^2 dy \rp = \\
	\error_{11}(0)^{-2} \NPC[11](0)^{-2} 
	\exp \lp
	  i\eps \int_{-\infty}^x |\qsol(y,t,\mathcal{D}_\xi)|^2 dy 
	  -\frac{i}{\pi} \int_{\posint} \frac{ \log(1 - \eps \lam | \rho(\lam)|^2)}{\lam} d\lam
	\rp  \\
	\times \lb	1 
		  +2\frac{\NPC[21](0)}{\NPC[11](0)} 
		  e^{4i \sum_{k=1}^N \arg \lam_k} 
		  \int_{x}^\infty u_{\mathrm{sol}}(y,t) dy
		  + \bigo{t^{-3/4}}
		\rb .
\end{multline}
Introducing the notation $p := e^{\frac{i\sgnt \pi}{4}} |8t\xi^2|^{1/2} $, the quantity $ \NPC[11](0)$, given in \eqref{NPC1.at.0}, is rewritten as
\[  
	\NPC[11](0) = e^{p^2/4} p^{-i\eta \kappa(\xi)} 
	D_{i \sgnt \kappa(\xi)}(p)
\]
while  expanding $\error_{11}(0)^{-2}$ using \eqref{e11.at0.in} we get
 \begin{multline}
\error_{11}(0)^{-2} = 
	  1  
	  + \frac{i\eps}{|2t|^{1/2}} 
	  \bigg\{
		\Nsol[12](\xi) \overline{A_{12}(\xi,\sgnt)} \overline{ \Nsol[11](\xi)}
		+ \overline{\Nsol[12](\xi)} A_{12}(\xi,\sgnt) \Nsol[11](\xi)  
	 \\
	 + \overline{A_{12}(\xi,\sgnt)} 
	 e^{4i\sum_{k=1}^N \arg \lam_k } \int_x^\infty \usol(y,t;\mathcal{D}_\xi) dy
	\bigg\} 
	+\bigo{t^{-1}} 
\end{multline}
where we have expressed $ \Nsol[12](0) \overline{\Nsol[11](0)}  $ using \eqref{working}.
We  evaluate  $ 2 \NPC[21](0)/\NPC[11](0)$ using \eqref{NPC1.at.0} again as 
\begin{align*}
2 \frac{\NPC[21](0)}{\NPC[11](0)} &= - 2 i A_{21} e^{-i\eta \pi/4} \sgn(\xi) 
\frac{  D_{i\eta \kappa(\xi) - 1} \lp p \rp} 	{D_{i\eta \kappa(\xi)} \lp p \rp} = - 2 i\eps |\xi|  e^{i\eta \pi/4} \overline{A_{12}} \frac{  D_{i\eta \kappa(\xi) - 1} \lp p \rp} 	{D_{i\eta \kappa(\xi)} \lp p \rp} 
\nonumber \\
&=  -   \overline{A_{12}} i\eps |2 t|^{-1/2}  p \frac{  D_{i\eta \kappa(\xi) - 1} \lp p \rp} 	{D_{i\eta \kappa(\xi)} \lp p \rp}.
\end{align*}
This concludes the proof of \eqref{u.gauge.in} 

Finally, we observe \cite[Eq.~12.9.1]{DLMF} that inserting the expansion $D_\nu(p) = e^{-p^2/4} p^{\nu} \lb 1 + \bigo{p^{-2}} \rb$ into \eqref{FandG} gives
\[
	F(p) = 1 + \bigo{\frac{1}{|t|\xi^2}}
	\qquad 
	G(p) = 1 + \bigo{\frac{1}{|t|\xi^2}}
\]	
so that the inner expansion \eqref{u.gauge.in} for $|\xi| \leq \poledist/3$ agrees with the outer expansion \eqref{u.gauge.out} for $|\xi| > M |t|^{-1/8}$.

\end{proof}


\appendix							

%
%

\section{Beals-Coifman Integral Equations for Problem \ref{RHP1} and Problem \ref{RHP2}}
\label{app:BC}

In this appendix we state the Beals-Coifman integral equations for Problem \ref{RHP1c} and deduce the algebraic-integral equations for Problem \ref{RHP1}. 
We also develop integral equations for Problem \ref{RHP2}.
A full derivation of the Beals-Coifman integral equations from the Beals-Coifman solutions in the direct problem is given in the thesis of the second author (Liu 2017).

\subsection{Integral Equations for Problems \ref{RHP1} and \ref{RHP1c}}
\label{app:RHP1c.BC}

Recall the augmented contour \eqref{Sigma.prime} and Figure \ref{fig:RHP1.c}. 
Here we expand the system of equations \eqref{RHP1c.BC}-\eqref{RHP1c.Cw} where $(w^+,w^-)$ is defined by \eqref{RHP1c.jump.w}. 
We only give the equations for $\mu_{11}$ and $\mu_{12}$ since the others can be obtained by symmetry.  
In what follows, if $\gamma$ is a component of $\Sigma'$, then $\calC_\gamma$ denotes the Cauchy integral for the contour, while $C^\pm$ denote the Cauchy projectors for the contour under consideration. We set
$$ 
r_x(s) = r(s) e^{-2ix s^2}, \quad
\br_x(s) = \br(s) e^{2ixs^2}, \quad
c_{j,x} 	=  c_j e^{2ixs^2}.
$$
First, for $\zeta \in \Sigma$:
\begin{align}
\label{RHP1c.11.Sig}
\mu_{11}(x,\zeta)
	&=	1	+
			 C^-\left(-\mu_{12}(x,\dotarg) \br_x(\dotarg)  \right)(\zeta)\\
\nonumber
	&\quad	 +	\sum_{j,\pm}
						 \calC_{\pm \gamma_j}
						 	\left(	
						 		\frac	{\mu_{12}(x,\dotarg) c_{j,x}}
						 				{\dotarg - {\pm  \zeta_j} }
						 	\right)(\zeta)\\[5pt]
\label{RHP1c.12.Sig}
\mu_{12}(x,\zeta)
	&=	C^+\left(\mu_{11}(x,\dotarg) r_x(\dotarg) \right)(\zeta)\\
\nonumber
	&\quad	+	\sum_{j,\pm}
						\calC_{\pm \gamma_j*}
							\left(
								\frac{-\eps \overline{c_{j,x}} 
										\mu_{11}(x,\dotarg)}
										{\dotarg - {\pm  \overline{\zeta_j}} }
							\right)(\zeta).
\intertext{For $\zeta \in \pm \gamma_i$:}
\label{RHP1c.11.gam}
\mu_{11}(x,\zeta)
	&=	1	+
			 C_{\Sigma}
			 	\left(-\mu_{12}(x,\dotarg) \br_x(\dotarg)  \right)(\zeta)\\
\nonumber
	&\quad	+	\sum_{(j,\pm) \ne (i,\pm) }
						 \calC_{\pm \gamma_j}
						 	\left(	
						 		\frac	{\mu_{12}(x,\dotarg) c_{j,x}}
						 				{\dotarg - {\pm  \zeta_j} }
						 	\right)(\zeta)\\
\nonumber
	&\quad	+	C^-
						\left(	
						 		\frac	{\mu_{12}(x,\dotarg) c_{i,x}}
						 				{\dotarg - {\pm  \zeta_j} }
						 	\right)(\zeta)\\[5pt]
\label{RHP1c.12.gam}
\mu_{12}(x,\zeta)
	&=	\calC_\Sigma 
				\left(\mu_{11}(x,\dotarg) r_x(\dotarg) \right)(\zeta)\\
\nonumber
	&\quad	+	\sum_{j,\pm}
						\calC_{\pm \gamma_j*}
							\left(
								\frac{-\eps \overline{c_{j,x}} 
										\mu_{11}(x,\dotarg)}
										{\dotarg - {\pm  \overline{\zeta_j}} }
							\right)(\zeta).
\intertext{Finally, for $\zeta \in \pm \gamma_i^*$:}
\label{RHP1c.11.Gam*}
\mu_{11}(x,\zeta)
	&=	1	+
			 C_\Sigma \left(-\mu_{12}(x,\dotarg) \br_x(\dotarg)  \right)(\zeta)\\
\nonumber
	&\quad	 +	\sum_{j,\pm}
						 \calC_{\pm \gamma_j}
						 	\left(	
						 		\frac	{\mu_{12}(x,\dotarg) c_{j,x}}
						 				{\dotarg - {\pm  \zeta_j} }
						 	\right)(\zeta)\\[5pt]
\label{RHP1c.12.Gam*}
\mu_{12}(x,\zeta)
	&=	C^+\left(\mu_{11}(x,\dotarg) r_x(\dotarg) \right)(\zeta)\\
\nonumber
	&\quad	+	\sum_{(j,\pm) \neq (i, \pm)}
						\calC_{\pm \gamma_j*}
							\left(
								\frac{-\eps \overline{c_{j,x}} 
										\mu_{11}(x,\dotarg)}
										{\dotarg - {\pm  \overline{\zeta_j}} }
							\right)(\zeta)\\
\nonumber
	&\quad	+  C^+	\left(
								\frac{-\eps \overline{c_{i,x}} 
										\mu_{11}*(x,\dotarg)}
										{\dotarg - {\pm  \overline{\zeta_j}} }
							\right)
\end{align}
In the restricted summations over $j$, we fix an index $i$ and one sign for the contour $\gamma_i$, and then sum over all 
$(\pm,j)$ for which either $i \neq j$ or $i=j$ but the signs do not coincide.

From these equations it is clear that $\mu_{11}(x,\zeta)$ extends analytically to the interiors of the discs bounded by $\gamma_i^*$ and $\mu_{12}(x,\zeta)$ extends analytically to 
the interiors of the discs bounded by $\gamma_i$. For this reason we can evaluate the Cauchy integrals around these contours using Cauchy's integral formula and obtain the following system of algebraic-integral equations for the functions  values 
$$
\left. \mu(x,\zeta) \right|_{\zeta \in \Sigma},
\quad
\left\{ 
	\mu_{11}(x,\pm \overline{\zeta_i}), 
	\mu_{12}(x,\pm \zeta_j) 
\right\}.
$$
\begin{align}
\label{RHP1.11.AI}
\mu_{11}(x,\zeta)
	&=	1	+	
				C^-\left(-\mu_{12}(x,\dotarg)\br_x(\dotarg) \right)(\zeta)
			+	\sum_{j,\pm} 
					\frac	{c_{j,x} \mu_{12}(x,\pm \zeta_j)}
							{\zeta \mp \zeta_j}\\[5pt]
\label{RHP1.12.AI}
\mu_{12}(x,\zeta)
	&=	C^+\left(\mu_{11}(x,\dotarg) r_x(\dotarg)\right)(\zeta)
			{\color{red}-}	\sum_{j, \pm} 
						\frac	{-\eps \overline{c_{j,x}}
									\mu_{11}(x,\pm \overline{\zeta_j})}
								{\zeta \mp \overline{\zeta_j}}\\[5pt]
\label{RHP1.11.AI.disc}
\mu_{11}(x,\pm \overline{\zeta_i})
	&=	1 
			+ 	\calC_\Sigma
					\left(-\mu_{12}(x,\dotarg) \br_x(\dotarg)\right)
					(\pm \overline{\zeta_j})
			+  \sum_{j, \pm } 
					\frac	{ \mu_{12}(x,\pm \zeta_j) c_{j,x} }
							{\pm \overline{\zeta_i} \mp \zeta_j}\\[5pt]
\label{RHP1.12.AI.disc}
\mu_{12}(x,\pm \zeta_i)
	&=	\calC_\Sigma
				\left( \mu_{11}(x,\dotarg) r_x(\dotarg) \right)(\pm \zeta_j)
			{\color{red}-}	\sum_{j, \pm}
					\frac  {-\eps \overline{c_{j,x}} 
								\mu_{11}(x,\pm \overline{\zeta_j})}
							{\zeta_i \mp \overline{\zeta_j}}
\end{align}
Theorems \ref{thm:BC} and \ref{thm:RHP1.unique} guarantee that the system of integral equations \eqref{RHP1c.11.Sig}--\eqref{RHP1c.12.Gam*}, and hence, also, the algebraic-integral equations \eqref{RHP1.11.AI}--\eqref{RHP1.12.AI.disc}, have a unique solution. 
It follows from the equations and the uniqueness property that
\begin{gather}
\label{RHP1.AI.sym1}
\mu_{11}(x,-\zeta)	=	\mu_{11}(x,\zeta),	
\quad \mu_{11}(x,-\overline{\zeta_j}) = \mu_{11}(x,\overline{\zeta_j})\\
\label{RHP1.AI.sym2}
\mu_{12}(x,-\zeta)	=	-\mu_{12}(x,\zeta),	\quad \mu_{12}(x,-\zeta_j) = -\mu_{12}(x,\zeta_j)
\end{gather}

Finally, in terms of the function $\mu$, the reconstruction formula \eqref{q.zeta} is given by
\begin{equation}
\label{RHP1c.mu.q.recon}
q(x)	=	-\frac{1}{\pi} \int_\Sigma \br_x(s) \mu_{11}(x,s) \, ds + 
			\sum_{j,\pm} 2i\eps \overline{c_j} \mu_{11}(x,\pm \overline{\zeta_j})
\end{equation}

\subsection{Integral Equations for Problem \ref{RHP2} and \ref{RHP2c}}
\label{app:RHP2.BC}

Here we develop the integral equations for Problem \ref{RHP2}
with the augmented contour $\Gamma = \R \cup \left(\cup_i \Gamma_i \cup \Gamma_i^* \right)$ described in Problem \ref{RHP2c}.   A full derivation of the Beals-Coifman integral equations from the Beals-Coifman solutions in the direct problem is given in the thesis of the second author (Liu 2017).  As before it is easy to obtain from these equations the algebraic-integral equations for Problem \ref{RHP2}.  For $\lam \in \R$:
\begin{align}
\label{RHP2c.11.R}
\nu_{11}(x,\lam)
	&=	1	+	
				C^-	\left(
							-\eps (\dotarg) \overline{\rho_x (\dotarg)}
							\nu_{12}(x,\dotarg)
						\right)
			+  \sum_j \calC_{\Gamma_j}
						\left(
							\frac
								{C_{j,x} \lam_j \nu_{12}(x,\dotarg) } 
								{(\dotarg - \lam_j)}
						\right)\\[5pt]
\label{RHP2c.12.R}
\nu_{12}(x,\lam)
	&=	1	+
				C^+	\left(
								\rho_x(\dotarg) \nu_{11}(x,\dotarg)
						\right)
				+	\sum_j \calC_{\Gamma_j^*}
						\left(
							\frac
								{ -\eps \overline{C_{j,x}}
									\nu_{11}(x,\dotarg) 
								}
								{(\dotarg - \overline{\lambda_j}) }
						\right)
\intertext{For $\lam \in \Lam_i$:}
\label{RHP2c.11.Lam}
\nu_{11}(x,\lam)
	&=	1	+	
				\calC_\R	
						\left(
							-\eps (\dotarg) \overline{\rho_x (\dotarg)}
							\nu_{12}(x,\dotarg)
						\right)
			+  \sum_{j \neq i} \calC_{\Gamma_j}
						\left(
							\frac
								{\nu_{12}(x,\dotarg) C_{j,x} \lam_j}
								{(\dotarg - \lam_j)}
						\right)\\
\nonumber
		&\quad	+	
				C^-	
						\left(
							\frac
								{\nu_{12}(x,\dotarg) C_{i,x} \lam_i}
								{(\dotarg - \lam_i)}
						\right)\\[5pt]
\label{RHP2c.12.Lam}
\nu_{12}(x,\lam)
	&=	1	+
				C_\R	\left(
								\rho_x(\dotarg) \nu_{11}(x,\dotarg)
						\right)
				+	\sum_j \calC_{\Gamma_j^*}
						\left(
							\frac
								{ -\eps \overline{C_{j,x}}
									\nu_{11}(x,\dotarg) 
								}
								{(\dotarg - \overline{\lambda_j}) }
						\right)
\intertext{For $\lam \in \Gamma_i^*$:}
\label{RHP2c.11.Lam*}
\nu_{11}(x,\lam)
	&=	1	+	
				\calC_\R	
						\left(
							-\eps (\dotarg) \overline{\rho_x (\dotarg)}
							\nu_{12}(x,\dotarg)
						\right)
			+  \sum_j \calC_{\Gamma_j}
						\left(
							\frac
								{\nu_{12}(x,\dotarg) C_{j,x} \lam_j}
								{(\dotarg - \lam_j)}
						\right)\\[5pt]
\label{RHP2c.12.Lam*}
\nu_{12}(x,\lam)
	&=	1	+
				\calC_\R	
						\left(
								\rho_x(\dotarg) \nu_{11}(x,\dotarg)
						\right)
				+	\sum_{j \neq i} \calC_{\Gamma_j^*}
						\left(
							\frac
								{ -\eps \overline{C_{j,x}}
									\nu_{11}(x,\dotarg)
								}
								{(\dotarg - \overline{\lambda_j}) }
						\right)\\
\nonumber
	&\quad	+	C^+
						\left(
							\frac
								{ -\eps \overline{C_{j,x}}
									\nu_{11}(x,\dotarg)
								}
								{(\dotarg - \overline{\lambda_j}) }
						\right)
\end{align}
From these equations it is clear that $\nu_{11}(x,\lam)$ is analytic for $\lam$ in the interior of $\Gamma_i^*$ and
$\nu_{12}(x,\lam)$ is analytic in the interior of $\Gamma_i$. 
We can then use Cauchy's theorem to derive the algebraic-integral equations for Problem \ref{RHP2}.
\begin{align}
\label{RHP2.11}
\nu_{11}(x,\lam)
	&=	1	+	C^-\left(
							-\eps (\dotarg) 
							\overline{\rho_x (\dotarg)}
							\nu_{12}(x,\dotarg)
						\right)
				+	\sum_j
						\frac{C_{j,x}\lam_j \nu_{12}(x,\lam_j)}{\lam -\lam_j}\\[5pt]
\label{RHP2.12}
\nu_{12}(x,\lam)
	&=	C^+	\left(
						\rho_x(\dotarg) \nu_{11}(x,\dotarg)
					\right)
				{\color{red}-}	\sum_j
						\frac{
								-\eps \overline{C_{j,x}}
								\nu_{11}(x,\overline{\lambda_j})
								}
								{ \lam - \overline{\lam_j} }\\[5pt]
\label{RHP2.11.disc}
\nu_{11}(x,\overline{\lam_i})
	&=	1	+	\calC_\R
						\left(
							-\eps (\dotarg) 
							\overline{\rho_x (\dotarg)}
							\nu_{12}(x,\dotarg)
						\right)(\overline{\lam_i})
				+	\sum_j
						\frac	{C_{j,x}\lam_j\nu_{12}(x,\lam_j)}
								{\overline{\lam_i} -\lam_j}\\[5pt]
\label{RHP2.12.disc}
\nu_{12}(x,\lam_i)
	&=	\calC_\R
					\left(
						\rho_x(\dotarg) \nu_{11}(x,\dotarg)
					\right)(\lam_i)
			{\color{red}-}	\sum_j
						\frac{
								-\eps \overline{C_{j,x}}
								\nu_{11}(x,\overline{\lambda_j})
								}
								{ \lam_i - \overline{\lam_j} }
\end{align}

Finally, in terms of the function $\nu$, the reconstruction formula \eqref{q.lam} becomes
\begin{equation}
\label{RHP2.nu.q.recon}
q(x)	=	{\color{red}-}\frac{1}{\pi} \int_\R \rho_x(s) \nu_{11}(x,s) \, ds
			+	\sum_j 2i \eps \overline{C_{j,x}} \nu_{11}(x,\overline{\lam_j})
\end{equation}

\subsection{Left and Right Transmission Coefficients}  \label{trace-f}

We establish 
relations between the transmission coefficients  $\balpha$ and $\alpha$ and the scattering data $ \{ \rho, \{\lambda_k\}_{k=1}^N$ \}. 

\begin{lemma}
The following relations 
\begin{equation}
\label{transmission+}
\balpha(\lambda)=\prod_{k=1}^N \frac{\lambda-\lambda_k}{\lambda-\overline{\lambda}_k}  \exp \left(  
-\int_{-\infty}^{+\infty}\frac{\log(1-\eps\xi |\rho(\xi)|^2)}{\xi-\lambda}\frac{d\xi}{2\pi i}
\right)
\end{equation}
\begin{equation}
\label{transmission-}
\alpha(\lambda)=\prod_{k=1}^N \frac{\lambda-\overline{\lambda}_k}{\lambda-{\lambda}_k}   \exp \left(  
\int_{-\infty}^{+\infty}\frac{\log(1-\eps\xi |\rho(\xi)|^2)}{\xi-\lambda}\frac{d\xi}{2\pi i}\right)
\end{equation}
hold.
\end{lemma}

\begin{proof}
The functions $\balpha(\lambda)$ and $\alpha(\lambda)$ have  simple zeros $\lbrace \lambda_k: \Im(\lambda_k)>0 \rbrace_{k=1}^N$ and $\lbrace \overline{\lambda}_k : \Im(\overline{\lambda}_k) <0 \rbrace_{k=1}^N$ respectively. Defining
\begin{equation}
\label{trace}
\bgamma(\lambda)=\prod_{k=1}^N\frac{\lambda-\overline{\lambda}_k}{\lambda-\lambda_k}\balpha(\lambda), \,\,  ~~
\gamma(\lambda)=\prod_{k=1}^N\frac{\lambda-\lambda_k}{\lambda-\overline{\lambda}_k}\alpha(\lambda),
\end{equation}
$\bgamma(\lambda)$  is analytic in the upper half plane where it has no zeros,  while $\gamma$ is analytic in the lower half plane where it has no zeros. Also $\bgamma$ and $\gamma$ $\rightarrow$ 1 as $|\lambda|\rightarrow\infty$ in the respective half planes.

Therefore we have 
$$\log \bgamma(\lambda)=\int_{-\infty}^{+\infty}\frac{\log\bgamma(\xi)}{\xi-\lambda}\frac{d\xi}{2\pi i}\, ,
 \quad\quad 
 \int_{-\infty}^{+\infty}\frac{\log\gamma(\xi)}{\xi-\lambda}\frac{d\xi}{2\pi i}=0 \quad\Im(\lambda)>0$$
and
$$\log \gamma(\lambda)=-\int_{-\infty}^{+\infty}\frac{\log\gamma(\xi)}{\xi-\lambda}\frac{d\xi}{2\pi i}\, ,
 \quad\quad 
 \int_{-\infty}^{+\infty}\frac{\log\bgamma(\xi)}{\xi-\lambda}\frac{d\xi}{2\pi i}=0 \quad\Im(\lambda)<0. $$
Using \eqref{trace}, as well as the identities $\balpha(\xi)\alpha(\xi)=\bgamma(\xi)\gamma(\xi) = \left(1-\xi|\rho(\xi)|^2\right)^{-1}$, we deduce
\begin{equation}
\label{trace-balpha}
\log\balpha(\lambda)=\sum_{k=1}^N\log\left(\frac{\lambda-\lambda_k}{\lambda-\overline{\lambda}_k}   \right)-\int_{-\infty}^{+\infty}\frac{\log(1-\eps\xi |\rho(\xi)|^2)}{\xi-\lambda}\frac{d\xi}{2\pi i},\quad \Im(\lambda)>0,
\end{equation}
\begin{equation}
\label{trace-alpha}
\log\alpha(\lambda)=\sum_{k=1}^N\log\left(\frac{\lambda-\overline{\lambda}_k}{\lambda-{\lambda}_k}   \right)+\int_{-\infty}^{+\infty}\frac{\log(1-\eps\xi|\rho(\xi)|^2)}{\xi-\lambda}\frac{d\xi}{2\pi i},\quad \Im(\lambda)>0.
\end{equation}
from which the identities \eqref{transmission+} and \eqref{transmission-} are obtained.

\end{proof}

\section{The Left Riemann-Hilbert Problem}
\label{app:left}

Riemann-Hilbert problem~\ref{RHP2} was constructed to be right normalized, i.e., it's solution satisfies $\lim_{x \to +\infty} n(x,\lam) = (1,0)$. The right normalized problem gives good estimates for the inverse scattering map for $x \geq a$.
To reconstruct the potential $q$ for $x<a$, we use a new left RHP that gives good estimates for the inverse scattering map for $x < a$. 
This RHP yields solutions normalized as $x\to -\infty$ by the condition $\lim_{x \rarr -\infty} {n^\ell}(x,z) = (1,0)$, and a stable reconstruction of $q$ on any interval $(-\infty, a)$. 
 The  associated jump matrix ${V^\ell_{x}}$ is 
$$ {V^\ell_{x}}(\lam) = e^{-ix\lam \ad \sigma_3} 
		\Twomat{1}{{\rho^{\ell}}(\lam)}{-\eps \lam \overline{{\rho^{\ell}} (\lam)   }}{1-\eps \lam|{\rho^{\ell}}(\lam)|^2},$$

where
\begin{equation}
\label{left-rho}
	{\rho^{\ell}}(\zeta^2) = \zeta^{-1}{  \bb(\zeta)/\ba(\zeta)}
\end{equation}
\begin{equation}
\label{left-residue}
	{V_x^\ell}(\lam_k) = 
	\twomat{0}{{C^{\ell}_k} e^{-2i\lambda_k x} }{0}{0},
	\quad 
	{V_x^\ell}(\overline{\lam_k}) = 
	\twomat{0}{0}
	{\eps \overline{{C^{\ell}_k}}\overline{\lambda}_k e^{2i\overline{\lambda}_k x}}
	{0}
\end{equation}
The construction of (\ref{left-rho}) can be found in Paper I Section 6.2.

We now derive ${C^{\ell}_k}$ in (\ref{left-residue}) from the set of scattering data  ${ \{ \rho,  \{ \lambda_k , C_k \}_{k=1}^N \} }$ of the right RHP. Recall that for the right RHP, {for which we omit the superscript $r$ elsewhere},
\begin{equation}\label{BC.right} 
{
	n^r(x,\lam) = n(x,\lam) 
	=\begin{cases}
		\left( \dfrac{n_{11}^-(x,\lambda)}{\balpha(\lambda)},  n_{12}^+(x,\lambda) \right)
		& \lam \in \C^+ \\
		\left( n_{11}^+(x,\lambda), \dfrac{n_{12}^-(x,\lambda)}{\alpha(\lambda)}  \right)
		& \lam \in \C^-		
	\end{cases}
}
\end{equation}
If $\balpha(\lambda_k)=0$, then
\begin{equation}
\label{lam-k}
n^-_{11}(x,\lambda_k)=B_k\lam_k n^+_{12}(x,\lambda_k)e^{2i\lam_k x}
\end{equation}
\begin{equation}
\label{lam-k-bar}
n^-_{12}(x,\lambdabar_k)=\eps\overline{B_k} n^+_{11}(x,\lambda_k)e^{-2i\lambdabar_k x}
\end{equation}
and we have the norming constant
\begin{equation}
	C_k=\dfrac{B_k}{\balpha'(\lambda_k)}
\end{equation}
For the left RHP, we have
\begin{equation}\label{BC.left} 
{
	n^\ell(x,\lam)
	=\begin{cases}
		\left(n_{11}^-(x,\lambda), \dfrac{n_{12}^+(x,\lambda)}{\balpha(\lambda)} \right)
		& \lam \in \C^+ \\
		\left( \dfrac{n_{11}^+(x,\lambda)}{\alpha(\lambda)}, n_{12}^-(x,\lambda) \right)
		& \lam \in \C^-
	\end{cases}
}
\end{equation}
Thus
\begin{align*}
	\text{Res}_{\lambda=\lambda_k} {n^\ell} (x,\lambda) 
	&=\dfrac{1}{\balpha'(\lambda_k)}\left(0, n^+_{12}(x,\lambda_k) \right)\\
    &=\dfrac{e^{-2i\lam_k x}}{B_k\lambda_k \balpha'(\lambda_k) } 
      \left(0, n^-_{11}(x,\lambda_k) \right)
\end{align*}
and
\begin{align*}
	\text{Res}_{\lambda=\lambdabar_k} {n^\ell}(x,\lambda) 
	&=\dfrac{1}{\alpha'(\lambdabar_k)}\left( n^+_{11}(x,\lambdabar_k), 0 \right)\\
    &=\dfrac{e^{2i\lambdabar_k x}}{\eps \overline{B_k} \alpha'(\lambdabar_k) } 
    \left(n^-_{12}(x,\lambdabar_k), 0 \right)
\end{align*}
We now define 
\begin{equation}
\label{C-left}
{C^{\ell}_k}=\dfrac{1}{B_k \balpha'(\lambda_k)}=\dfrac{1}{C_k \left(\balpha'(\lambda_k )\right)^2}
\end{equation}

Now we arrive at the following left Riemann-Hilbert problem:
\begin{RHP}
\label{RHP.lambda-left}
Fix $x \in \R $ and {$\{ \rho^{\ell}, \{ \lambda_k , C^\ell_k \}_{k=1}^N \}$}, such that ${\rho^\ell} \in H^{2,2}(\R )$, $1-{\eps}s |{\rho^{\ell}}(s)|^2 \geq c>0$ strictly, and { $\{ \lambda_k , C^\ell_k \}_{k=1}^N \subset \C^+ \!\times \C^\times$}.
Find a vector-valued function {$n^\ell(x,\dotarg): \C \to \C^2$} with the following properties:
\begin{enumerate}
\item[(i)] ${n^\ell(x, \lam)}$ is an analytic function of ${\lam}$ for ${\lam} \in \mathbb{C}\setminus \Lambda' $  where
\begin{equation*}
	\Lambda'= \mathbb{R} \cup 
	\lbrace  
		\Gamma_1, ..., \Gamma_n, \Gamma_1^*, ..., \Gamma_n^* 
	\rbrace 
\end{equation*}
\item[(ii)]  ${n^\ell}(x,\lam)= (1 ,0) + \bigo{\lam^{-1}}$ as $\lam \rightarrow\infty$.
\item[(iii)] For each $\lambda \in \Lambda'$, ${n^\ell(x,z)}$ has continuous 
boundary values
${n^\ell_{\pm}}(x,\lambda)$ as $z \to \lambda \in \Lambda'$ {from the left or right of $\Lambda'$ respectively.} 
Moreover, the jump relation 
\[
	{n^\ell_+}(x,\lambda) = 
	{n^\ell_-}(x,\lambda) {V^\ell_{x}}(\lambda)
\]
holds, where for $\lambda\in\mathbb{R}$
\[
	{V^\ell_{x}}(\lambda)=
		e^{-i\lambda x \ad \sigma_3} 
		\Twomat{1}{{\rho^{\ell}}(\lam)}
		{- \lam\eps \overline{{\rho^{\ell}} (\lam)   }}
		{1-\eps \lam|{\rho^{\ell}}(\lam)|^2}
\]
and for $\lam \in \Gamma_k \cup \Gamma_k^*$ 
\[
	{V^\ell_x}(\lam_k) 
	=\begin{cases}
	   \triu{\dfrac{{C^\ell_k}  e^{-2i\lam_k x}}{\lam_k(\lam-\lam_k)}}
	   & \lam \in \Gamma_k, \\[.05em]
	   \tril{\dfrac{-\eps\overline{{C^\ell_k}} \, e^{2ix \overline{\lam_k}} }
	   {\lam-\overline{\lam_k}}}
	   & \lam \in \Gamma_i^*
	\end{cases}
\]
\end{enumerate}
\end{RHP}

We recover ${q(x)}$ from the formula
\begin{equation}
\label{bq.recon}
{q(x)} = \lim_{z \rarr \infty} 2iz {n^\ell_{12}}(x,z)
\end{equation}
where the limit is taken in a direction not tangential to $\R $.

\section{Solutions of RHP~\ref{RHP2} for reflectionless scattering data}
\label{app:solitons}

The bright (non-algebraic) soliton solutions of \eqref{DNLS2} can be characterized as the potentials $q(x,t)$ for which the associated scattering data are reflectionless: $\lp  \rho \equiv 0, \{ (\lam_k ,C_k) \}_{k=1}^N \rp$, and $(\lam_k , C_k) \in \C^+ \times \C^\times$ for each $k=1,\ldots,N$. 
If $N=1$, with scattering data $(\lam = u+iv, C)$, the single soliton solution of \eqref{DNLS2} is given by 
\begin{gather}\label{1sol}
	q(x,t) = 
	\varphi(x-x_0 + 4 u t) 
	\exp i \left\{  { 4 }|\lam|^2 t -2{u}(x+4ut) - \frac{\eps}{4}
	\int_{-\infty}^{x-x_0+4u t} \varphi(\eta)^2 d\eta
	- \alpha_0
	\right\}
\shortintertext{where}
\nonumber
	\varphi(y) = 
	\sqrt{\frac{8 v^2}{|\lam| \cosh(4 v y) - \eps u}}
\\
\nonumber
	x_0 = \frac{1}{4 v} \log \frac{ |\lam| |C|^2 }{4 v^2}
	\qquad
	\alpha_0 = \arg(\lambda) + \arg(C) + \pi/2
\end{gather}
which describes a solitary wave with amplitude envelope $\varphi$ traveling at speed $c = -4 \Re \lambda$. For $N > 1$ the solution formulae become ungainly, but we expect, generically, that for $|t| \gg 1$, the solution will resemble $N$ independent 1-solitons each traveling at its unique speed $-4 \Re \lam_k$\footnote{The non-generic case occurs when $\Re \lam_j = \Re \lam_k$ for one or more pairs $j \neq k$. In this case the solution possesses localized, quasi-periodic traveling waves known as breather solitons.}. For this reason, these solutions are called $N$-solitons of \eqref{DNLS2}. We will write $\Nsol$ for the solution of RHP~\ref{RHP2} when $\rho \equiv 0$ to emphasize its relation to soliton solutions of \eqref{DNLS2}.

When $\rho \equiv 0$, RHP~\ref{RHP2} reduces to a question of meromorphic function theory, and the singular integral equations for its solutions take the form of a system of linear equations.

\begin{lemma}
\label{lem:sol.bound}
There exists a unique solution $\Nsol$ to RHP~\ref{RHP2} with reflectionless scattering data  $\rho \equiv 0$, $\{ \lam_k, C_k \}_{k=1}^N \subset \C^+ \times \C^\times$ whenever $\lam_j \neq \lam_k$, for $j \neq k$. Moreover, the solution satisfies 
\begin{equation} \label{bound-nsol}
	\| \Nsol \|_{L^\infty(\C \setminus \mathcal{B}_\poles) }
	\lesssim 1
\end{equation}
where $\mathcal{B}_\poles$ is any open neighborhood of the poles $\Lambda = \{ \lam_k, \overline{\lam_k} \}_{k=1}^N$, and the implied constant depends only on $\mathcal{B}_\poles$ and the  scattering data; it is independent of $x,t$.
\end{lemma}

We will prove Lemma~\ref{lem:sol.bound} with the help of an auxilary transformation.

Recall that  $\xi = -x/4t$ denote the critical point of the phase function $\theta$ and that from Figure~\ref{fig:theta signs}, that the growth and decay properties of the exponentials $e^{\pm 2it \theta}$, which appear in the residue  relations of RHPs~\ref{RHP2}
change as one passes through either the real axis or the line $\Re \zeta = \xi$.
The partition of $\{ 0,1, \dots, N\} = \negpoles \cup \pospoles$, where 
\begin{equation}\label{index sets}
	\begin{aligned}
		\negpoles &= \{ k \in \{0,1, \dots, N \} \,:\, \eta (\Re z_k -\xi) < 0   \}, \\
		\pospoles &= \{ k \in \{0,1, \dots, N \} \,:\, \eta (\Re z_k -\xi) \geq 0   \}.
	\end{aligned}
\end{equation}
splits the  residues relations in Problem \ref{RHP2}(iii) into two sets;
As $|t| \to \infty$ with $x = -4\xi t$, those $k \in \pospoles$, correspond to poles $\lambda_k$ whose residues are exponentially decaying, while those $k \in \negpoles$, correspond to poles $\lambda_k$ whose residues are exponentially growing. 
We use this partition to define the transformation:
\begin{equation}\label{Blaschke}
	\NsolAlt(\lam) = \Nsol(\lam) \mathfrak{B}(\lam)^{-\sig},
	\qquad
	\mathfrak{B}(\lam) = \prod_{k \in \negpoles} \lp \frac{ \lam - \overline{\lam_k}}{\lam - \lam_k} \rp.
\end{equation}
We will show that the coefficients appearing in the residues relations for the  new RHP satisfied by  $\NsolAlt$ are bounded uniformly in $(x,t)$ for all index $k$. (see \eqref{gamma.bounds}).
The transformation  \eqref{Blaschke} results in the following problem for $\NsolAlt$:
\begin{problem}\label{outmodel2}
Find an analytic function $\NsolAlt: (\C \setminus \poles) \to SL_2(\C)$ such that
\begin{enumerate}[1.]
	\item $\NsolAlt$ satisfies the symmetry relation 
	\[
		\NsolAlt(\lam) = 
		\lam^{-\sig/2} \sigma_\eps^{-1} \overline{\NsolAlt(\overline{\lam})} 
		\sigma_\eps \lam^{\sig/2}
	\]
	\item $\NsolAlt(\lam) = \tril{\alpha(x,t)} + \bigo{\lam^{-1}}$ as $\lam \to \infty$.
	\item $\NsolAlt$	 has simple poles at each point in $\poles$. For each $\lam_k \in \poles_+$ 
	\begin{gather}\label{out2 residue}
			\begin{aligned}
				\res_{\lam = \lam_k} \NsolAlt(\lam) 
				&= \lim_{\lam \to \lam_k} \NsolAlt(\lam) \vk{\Delta}(\lam_k) \\
				\res_{\lam = \overline{\lam_k}} \NsolAlt(\lam) 
				&= \lim_{\lam \to \overline{\lam_k}} \NsolAlt(\lam) \vk{\Delta}(\overline{\lam_k})
			\end{aligned}
\shortintertext{where}
\label{out2 residue matrices}
		\begin{aligned}
			\vk{\Delta}(\lambda_k) &= 
			\begin{dcases}
				\tril[0]{ \lam_k \gamma_k(x,t) } 
					& k \in \pospoles 	\\		
				\triu[0]{ \lam_k^{-1} \gamma_k(x,t) } 
					& k \in \negpoles 	
			\end{dcases} 
			\\					
			\vk{\Delta}(\overline{\lambda_k}) &= 
				\begin{dcases}
					\triu[0]{ \eps \overline{\gamma_k}(x,t) }
						& k \in \Delta_\xi^{+} \\
					\tril[0]{ \eps \overline{\gamma_k}(x,t) }  
						& k \in \Delta_\xi^{-} .
				\end{dcases} 
			\\
			\text{with \hfill} \gamma_k &= 
			\begin{cases}
				C_k \mathfrak{B}(\lam_k)^{-2} e^{-2it\theta(\lam_k)}
				 & k \in \pospoles \\[,5em]
				C_k^{-1} (1/\mathfrak{B})'(\lam_k)^{-2} e^{2it\theta(\overline{\lam_k})} 
				 & k \in \negpoles .
			\end{cases}
		\end{aligned}
	\end{gather}
\end{enumerate}
\end{problem}

\begin{proof}[Calculation of residues in Problem~\ref{outmodel2}.]
The conditions in Problem~\ref{outmodel2} are a direct consequence of RHP~\ref{RHP2} (with $\rho \equiv 0$) and \eqref{Blaschke}. 
We omit most of the details, except the new residue conditions for $k \in \negpoles$. 

Consider $\lam_k \in \poles^+$, for $k \in \negpoles$. 
Denote $\Nsol[1]$ and $\Nsol[2]$ the first and second column of $\Nsol$, with the same convention for $\NsolAlt$.

Since $B(\lam)$ has zeros at each $\overline{\lam_k}$ and poles at each $\lam_k$,
$\NsolAlt[1](\lam) = \Nsol[1](\lam) B(\lam)^{-1}$ has a removable singularity at $\lam_k$, but acquires a pole at $\overline{\lam_k}$. For $\NsolAlt[2](\lam) = \Nsol[2](\lam) B(\lam)$ the situation is reversed.
We have
\begin{align*}
		\lim_{\lam \to \lam_k} \NsolAlt[1](\lam) 
		&= \res_{\lam = \lam_k} \Nsol[1](z)  \cdot (1/B)'(\lam_k) 
		= \lam_k C_k e^{-2it\theta_k} \Nsol[2](\lam_k)  (1/B)'(\lam_k),  \\
		\res_{\lam_k} \NsolAlt[2](\lam) &= 
		\res_{\lam=\lam_k} \Nsol[2](\lam) B(\lam) 
		= \Nsol[2](\lam_k) \lb (1/B)'(\lam_k) \rb^{-1}.  
\end{align*}
Using the calculation of $\NsolAlt[1](\lam_k)$ to replace $\Nsol[2](\lam_k)$ gives
\[
		\res_{\lambda_k} \NsolAlt[2](\lam) 
		= \lam_k^{-1} C_k^{-1} \lb (1/B)'(\lam_k) \rb^{-2} 
		e^{2it \theta} \NsolAlt[1](\lam_k)
		= \lam_k^{-1} \gamma_k(x,t) \NsolAlt[2](\lam_k), 
\]
which verifies the first formula in \eqref{n1 residue matrices}. The computation of the residue at $\overline{\lam_k}$ for $k \in \negpoles$ is similar,
aided by the symmetry $\overline{B(\overline{\lam})} = B(\lam)^{-1}$.
\end{proof}

The normalization and pole conditions imply that $\NsolAlt$ is  meromorphic in $\C$ with prescribed poles,  and using the symmetry condition 
(i) of RHP \ref{outmodel2}, $\NsolAlt$  must take the form
\begin{multline}\label{Nsol.expand} 
\NsolAlt(\lam) = 
\tril{\sum_{k=1}^N B_k} \\
+ \sum_{k \in \pospoles} 
	\frac{ \stwomat{A_k}{0}{\lam_k B_k}{0} }{\lam-\lam_k}
	+
   	\frac{ \stwomat{0}{\eps \overline{B_k}}{0}{\overline{A_k}} }{\lam-\overline{\lam_k}} 
+ \sum_{k \in \negpoles} 
	\frac{\stwomat{A_k}{0}{\overline{\lam_k} B_k}{0} }{\lam- \overline{\lam_k}} 
	+
   	\frac{ \stwomat{0}{\eps \overline{B_k}}{0}{\overline{A_k}} }{\lam-\lam_k},
\end{multline}
for unknown coefficients $A_k, B_k$, $k=1,\dots, N$ which are determined by the residue conditions \eqref{out2 residue} 
which  yield a system of  $2N$ linear equations:
\begin{subequations}\label{sol.lin.sys}
\begin{align}
	A_j &= \lam_j \gamma_j(x,t) 
	\lp
	\sum_{k \in \pospoles} \frac{ \eps \overline{B_k} }{ \lam_j - \overline{\lam_k}}
	+
	\sum_{k \in \negpoles} \frac{ \eps \overline{B_k} }{ \lam_j - \lam_k} 
	\rp
	&& j \in \pospoles, \\
	A_j &= \eps \overline{\gamma_j}(x,t) 
	\lp
	\sum_{k \in \pospoles} \frac{ \eps \overline{B_k} }{ \overline{\lam_j} - \overline{\lam_k}}
	+
	\sum_{k \in \negpoles} \frac{ \eps \overline{B_k} }{ \overline{\lam_j} - \lam_k} 
	\rp
	&& j \in \negpoles, \\
	\eps \overline{B_j} &= \eps  \overline{\gamma_j}(x,t) 
	\lp 1 +
	\sum_{k \in \pospoles} \frac{ A_k }{ \overline{\lam_j} - \lam_k}
	+
	\sum_{k \in \negpoles} \frac{ A_k }{ \overline{\lam_j} - \overline{\lam_k}} 
	\rp
	&& j \in \pospoles, \\
	\eps \overline{B_j} &= \lam_j^{-1} \gamma_j(x,t) 
	\lp 1 +
	\sum_{k \in \pospoles} \frac{ A_k }{ \lam_j - \lam_k}
	+
	\sum_{k \in \negpoles} \frac{ A_k }{ \lam_j - \overline{\lam_k}} 
	\rp
	&& j \in \negpoles.
\end{align}
\end{subequations}

\begin{proof}[\hypertarget{proof:sol.bound}Proof of Lemma~\ref{lem:sol.bound}]
\label{proof:sol.bound}
It is sufficient to  establish the boundness estimate \eqref{bound-nsol}  for $\NsolAlt$, as the invertible transformation between $\Nsol$ and $\NsolAlt$ depends only on $\poles$. 
Furthermore, to establish the boundedness of $\NsolAlt$ away from its poles, we only need to show that the unknowns $A_k, \overline{B_k}$ in \eqref{sol.lin.sys} are bounded in $(x,t)$. 
 
The system of equations \eqref{sol.lin.sys} can be expressed as a matrix equation
\begin{equation}\label{matrix.sol.sys}
	(\vect{I} - \vect{\Gamma} \vect{M}_\poles)\vect{\alpha} = \vect{\Gamma} \twovec{ \vect{0}_N }{ \one_N }
\end{equation}
where $\vect{\alpha} = \mathrm{col}(A_1, \dots, A_N, \eps \overline{B_1}, \dots, \eps \overline{B_N})$ is the vector of unknowns; $\vect{\Gamma}$ is a diagonal matrix whose entries are the appropriate residue coefficients, $\gamma_k$ or $\overline{\gamma_k}$, for the corresponding entry in $\vect{\alpha}$; $\vect{M}_\poles$ is a complex $2N\times 2N$ matrix 
whose entries depend only on $\poles$ and $\eps$; finally, $\vect{0}_N, \one_N$ are the vectors of $N$ zeros and ones respectively. 
We make the following observations: 
1) The only $(x,t)$ dependence in the coefficients of \eqref{matrix.sol.sys} appears in $\vect{\Gamma}$; 
2) For any choice of $\vect{\Gamma}$ with nonzero entries, any solution of \eqref{matrix.sol.sys} corresponds exactly to a solution of Problem~\ref{outmodel2}. As Problem~\ref{outmodel2} is a reflectionless reduction of RHP~\ref{RHP2}, the general existence and uniqueness result for solutions of RHP~\ref{RHP2} established in Section \ref{sec:RHP.exist} implies that 
$(\vect{I} - \vect{\Gamma} \vect{M}_\poles)$ is invertible for any such $\Gamma$;
3) If we allow any $\gamma_k \to 0$ in \eqref{matrix.sol.sys} the corresponding coefficients $A_k, B_k$ in $\vect{\alpha}$ must also vanish, and \eqref{matrix.sol.sys} reduces to an equivalent system for the coefficients of an expansion \eqref{Nsol.expand} with $2(N-1)$ poles at $\Lambda \setminus \{ \lam_k \cup \overline{\lam_k} \}$.

The advantage of Problem~\ref{outmodel2} is that, in light of \eqref{phase.lambda}, \eqref{index sets} and \eqref{out2 residue matrices}, the coefficients $\gamma_k(x,t)$ are uniformly bounded as $(x,t)$ vary over $\R^2$:
\begin{equation}\label{gamma.bounds}
	|\gamma_k(x,t)|  \leq 
	K_\poles
	\left.
	\begin{cases}  
		e^{\mp 8t \Im \lam_k (\Re \lam_k - \xi ) }
		& |\xi| \leq \Xi_0,\ k \in \posnegpoles \\
		e^{\mp 2x \Im \lam_k \lp 1 - \xi^{-1} \Re \lam_k \rp }
		& |\xi| \geq \Xi_0, \ k \in \posnegpoles
	\end{cases}
	\right\}
	\leq K_\poles,
\end{equation}
where $\Xi_0>0$ is any fixed constant and $K_\poles$ is a fixed constant depending only on the scattering data. Finally, using the observations above, since $(\vect{I} - \vect{\Gamma}\vect{M}_\poles)^{-1}$ exist and depends continuously on the entries in $\vect{\Gamma}$, it follows that $(\vect{I} - \vect{\Gamma}\vect{M}_\poles)^{-1}$ is a bounded for all $(x,t)$. The result follows immediately.  
\end{proof}

\begin{proof}[\hypertarget{proof:sol separation} Proof of Proposition ~\ref{prop:sol separation}]

	Order the discrete spectra such that $\lam_j \in \poles(I)$, $j=1,\dots N(I)$, and $\lam_j \not\in \poles(I)$, $j> N(I)$. Also introduce the permutation matrix
\[
	P = 
	\begin{pmatrix}
		I_{N(I)} & 0 & 0 & 0 \\
		0 & 0 & I_{N(I)} & 0 \\
		0 & I_{N(I)'} & 0 & 0 \\
		0 & 0 & 0 & I_{N(I)'} 
	\end{pmatrix}
	\qquad
	N(I)' := N - N(I).
\]	
For any matrix $A$ and vector $\vect{v}$ let 
\[
	A^{\pi} = P A P^\intercal, \qquad \vect{v}^{\pi} = P \vect{v}.
\]
Conjugating \eqref{matrix.sol.sys} by $P$ we arrive at the equivalent system of equations
\begin{equation}\label{permute.matrix.sys}
	(I - \vect{\Gamma}^\pi \vect{M}_\poles^\pi) \vect{\alpha}^\pi = \vect{\Gamma}^\pi
	\begin{pmatrix} 
		\vect{0}_{N(I)} &  \vect{1}_{N(I)} & 
		\vect{0}_{N(I)'} &  \vect{1}_{N(I)'} 
	\end{pmatrix}^\intercal.
\end{equation}
This systems has a natural block structure; writing
\[
	\vect{\Gamma}^\pi = \diag{ \vect{\Gamma}_0 }{ \vect{\Gamma}_1 },
	\qquad
	\vect{M}^\pi_\poles = 
	\begin{pmatrix*}[l]
		\vect{M}_{\poles(I)} & \vect{M}_{12} \\
		\vect{M}_{21} & \vect{M}_{22}
	\end{pmatrix*},
	\quad \text{and} \quad
	\vect{\alpha}^\pi = 
	\begin{pmatrix} 
		\vect{\alpha}_{N(I)} \\ \vect{\alpha}_{N(I)'}
	\end{pmatrix}		
\]
\eqref{permute.matrix.sys} becomes
\[
	\begin{pmatrix}
		I_{N(I)} -\vect \Gamma_0 \vect{M}_{\poles(I)} &
		- \vect\Gamma_0  \vect{M}_{12} \\
		- \vect\Gamma_1  \vect{M}_{21} & 
		I_{N(I)'} -\vect \Gamma_1 \vect{M}_{22}
	\end{pmatrix}	
	\twovec{\vect{\alpha}_{N(I)} }{ \vect{\alpha}_{N(I)'} }
	=
	\begin{pmatrix} 
		\vect \Gamma_0 
		\begin{psmallmatrix} \vect 0_{N(I)} \\ \vect 1_{N(I)} \end{psmallmatrix} 
		\\[1.05em]
		\vect \Gamma_1  
		\begin{psmallmatrix} \vect 0_{N(I)'} \\ \vect 1_{N(I)'} \end{psmallmatrix}
	\end{pmatrix}
\]
where the first $2N(I)$ rows determine the unknowns $\alpha_{N(I)} = \mathrm{col}( A_1,\dots, A_{N(I)},\eps \overline{B_1},\dots, \eps \overline{B_{N(I)}})$ corresponding to the discrete spectra in $\poles(I)$. Observe that the upper left block of the system is precisely the left hand side of \eqref{matrix.sol.sys} corresponding for the $N(I)$

As $|t| \to \infty$ inside the cone specified in the theorem, it follows from \eqref{gamma.bounds} that 
\begin{align}\label{Gamma.bound}
	\| \vect \Gamma_0 \| = \bigo{1}
	\quad \text{and} \quad
	\| \vect \Gamma_1 \| = \bigo{e^{-4\poledist \nu_0 t}},
\end{align}
so we may regard terms with $\vect{\Gamma}_1$ as perturbative. By elementary row operations this system is equivalent to
\begin{gather*}
	(\vect M_0 + \vect M_1) 
	\twovec{\vect{\alpha}_{N(I)} }{ \vect{\alpha}_{N(I)'} }
	=
	\vect \beta_0 + \vect \beta_1
\shortintertext{where}
	\vect M_0 = \diag{ I_{N(I)} - \vect \Gamma_0 \vect M_{\poles(I)} }{ I_{N(I)'} }
\qquad
	\vect \beta_0 =
	  \begin{pmatrix} 
		\vect \Gamma_0 
		\begin{psmallmatrix} \vect 0_{N(I)} \\ \vect 1_{N(I)} \end{psmallmatrix} 
		\\[1.05em]
		\vect 0_{2N(I)'}
	  \end{pmatrix}
\\
    \vect M_1 = 
	 \begin{pmatrix}
		- \vect \Gamma_0 \vect M_{12} \vect \Gamma_1 \vect M_{21} &
		- \vect \Gamma_0 \vect M_{12} \vect \Gamma_1 \vect M_{22} \\
		- \vect \Gamma_1 \vect M_{21} &
		- \vect \Gamma_1 \vect M_{22} 
	 \end{pmatrix} 
\qquad
	\vect \beta_1 =
	  \begin{pmatrix} 
		\vect \Gamma_0 \vect M_{12} \vect \Gamma_1 
		\begin{psmallmatrix} \vect 0_{N(I)'} \\ \vect 1_{N(I)'} \end{psmallmatrix}
		\\[1.05em]
		\vect \Gamma_1  
		\begin{psmallmatrix} \vect 0_{N(I)'} \\ \vect 1_{N(I)'} \end{psmallmatrix}
	\end{pmatrix}
\end{gather*}

Let 
\[
	\widehat{\mathfrak{B}} = \prod_{\substack{\lam_j \in \poles(I) \\ j \in \negpoles}} 
	\lp \frac{ \lam - \overline{\lam_j}}{\lam - \lam_j} \rp
\]
be the Blaschke product corresponding to the reduced poles in $\poles(I)$ akin to the full product $\mathfrak{B}$ defined in \eqref{Blaschke}. Then with $\widehat{C}_k$ as defined in the statement of Proposition \ref{prop:sol separation}, for each $\lam_k \in \poles(I)$ the corresponding connection coefficients satisfy
\begin{align*}
	&C_k \mathfrak{B}(\lam_k)^{-2} = \widehat{C_k} \widehat{\mathfrak{B}}(\lam_k)^{-2} 
	&& k \in \pospoles \\
	&C_k^{-1} (1/\mathfrak{B})'(\lam_k)^{-2} = \widehat{C_k}^{-1} (1/\widehat{\mathfrak{B}})'(\lam_k)^{-2} 
	&& k \in \negpoles .
\end{align*}  
It is clear then that the limiting system of equations (in which $\vect M_1$ and $\vect \beta_1$ are set to zero) is exactly equivalent to \eqref{matrix.sol.sys} for the reduced scattering data given by 
the discrete spectra $\poles(I)$ and connection coefficients $\widehat{C}_k$. Since we know from the proof of Lemma~\ref{lem:sol.bound} that $I_{N(I)} - \vect \Gamma_0 \vect M_{\poles(I)}$ has a bounded inverse for all $(x,t)$ if follows from \eqref{Gamma.bound} that
\[
	\vect{\alpha}_{N(I)} = 
	(I_{N(I)} - \vect \Gamma_0 \vect M_{\poles(I)})^{-1} \vect \beta_0 + \bigo{ e^{-4\poledist \nu_0 t} }
	\qquad
	\vect{\alpha}_{N(I)'} =  \bigo{ e^{-4\poledist \nu_0 t} }.
\]
The result of the proposition follows immediately. 
\end{proof}

%
%
%
%

\section{Soliton-free initial Data of Large $L^2$-norm}
\label{app:empty}

In this appendix we prove the following proposition which establishes the existence of $q \in H^{2,2}(\R)$ with arbitrary $L^2$-norm which have no spectral singularities and no solitons. 
\begin{proposition}
\label{prop:empty}
	The scattering map $\mathcal{R}(q)$ described by Definition~\ref{def:R} can be explicitly computed for the family of potentials 
	\begin{equation}\label{q.fam}
		q(x) = \nu \sech(x)^{1-2i\mu} e^{i \lp S_0 - \eps \nu^2 \tanh(x) - 2\delta x \rp},
		\qquad
		\| q \|_{L^2(\R)}^2 = 2 \nu^2.
	\end{equation}
	where $\nu > 0$, and $\mu, \delta, S_0 \in \R$. 
	The condition 
	\begin{equation}\label{cond.empty}
		\eps \delta < \mu^2/\nu^2
	\end{equation}
	is sufficient to guarantee the discrete spectrum is empty, i.e., $q \in U_0$ .
\end{proposition}

We will prove this proposition by showing that the linear system for DNLS:
\begin{equation}
\label{LS.bis}
\psi_x = 
\begin{pmatrix}
-i\zeta^2 - \dfrac{i\eps|q|^2}{2}	&	\zeta q	\\
\zeta \eps \qbar						&	i\zeta^2 + \dfrac{i\eps}{2}|q|^2
\end{pmatrix}
\psi
\end{equation}
can be reduced to solving the hypergeometric differential equation for any potentials $q$ in the family \eqref{q.fam}. This reduction involves several changes of dependent variable as well as the map 
\begin{equation}\label{svar.def}
	s = \frac{1}{2} \lp 1 + \tanh(x) \rp
\end{equation}
which compactifies $\R$ to $(0,1)$. Thus $x \rarr \pm \infty$ asymptotics are equivalent to Taylor expansions at $s=0$ or $s=1$ modulo a leading singular term. 
The identity
\begin{equation}
\label{exs}
e^x = s^{1/2}(1-s)^{-1/2}
\end{equation} 
implies that $e^{i\zeta^2 x} \sim s^{i\zeta^2/2}$ near
$s=0$, while $e^{i\zeta^2 x} \sim (1-s)^{-i\zeta^2/2}$
near $s=1$. 
We normalize the solution as $s \darr 0$ (i.e., $x \rarr -\infty$) and use  transformation formulas for hypergeometric functions to compute the scattering data by finding asymptotics as $s \uarr 1$ (i.e., $x \rarr  +\infty$).

\subsection{Reduction to a Hypergeometric Equation}
We begin by writing the potential $q$ in \eqref{LS.bis} in the form
\begin{equation}
\label{q.form}
	q(x) = A(x) e^{iS(x)}.
\end{equation} 
We make a change of dependent variable to remove the oscillatory factor $e^{iS}$ from $q$. If we set $\psi = e^{iS(x) \sigma_3/2} \varphi$, then
\begin{equation}
\label{LS.phi}
   \varphi_x	=	
   \begin{pmatrix}
      -i\zeta^2 - \dfrac{i\eps}{2} A^2 - \dfrac{i}{2}S_x
      &	\zeta A				\\[5pt]
      \zeta \eps A								
      &	i\zeta^2 + \dfrac{i\eps}{2} A^2 + \dfrac{i}{2}S_x
   \end{pmatrix}
   \varphi.
\end{equation}
Introducing the change of independent variable \eqref{svar.def},  noting that
$d/dx = 2s(1-s) d/ds$, we obtain
\begin{equation}
\label{LS.phi.t}
   2s(1-s) \varphi_s = 
   \begin{pmatrix}
      -i\zeta^2 - \dfrac{i\eps}{2} A^2 - \dfrac{i}{2}S_x
      &	\zeta A				\\[5pt]
      \zeta \eps A										
      &	i\zeta^2 + \dfrac{i\eps}{2} A^2 + \dfrac{i}{2}S_x
   \end{pmatrix}
   \varphi
\end{equation}

Next, we reduce the system to a single second order ODE, by introducing the change of dependent variables
\begin{equation}\label{WW.def}
   W(s) = g(s) \diag{1}{ \dfrac{\zeta A}{2s(1-s)} } \varphi(s),
\end{equation}
where the scalar factor $g$ removes the $(1,1)$-entry of the coefficient matrix on the right hand side of \eqref{LS.phi.t} by choosing it to 
satisfy
\begin{equation}
\label{g}
   2s(1-s) g'(s) = \left(i\zeta^2 + \dfrac{i\eps}{2} A^2 + \dfrac{i}{2}S_x \right) g(s),
\end{equation}
while the diagonal matrix factor in \eqref{WW.def} normalizes the $(1,2)$-entry of the coefficient matrix to $1$. 
A short computation shows that $W$ obeys the differential equation
\begin{equation}
\label{LS.W.t}
   W_s = 
   \begin{pmatrix}
      0	  & 
      1	  \\[5pt]
      \dfrac{\eps \zeta^2 A^2}{4s^2(1-s)^2}   &		
      i\dfrac{2\zeta^2+ \eps A^2 + S_x}{2s(1-s)} 
         + \left( \dfrac{A_s}{A} + \dfrac{2s-1}{s(1-s)} \right)
   \end{pmatrix}
   W
\end{equation}
If we now set
\[
   W = \begin{pmatrix} w \\ w_s \end{pmatrix} ,
\]
This system reduces to the single second-order equation
\begin{equation}
\label{LS.w}
   s(1-s) w'' 
   - \left[ 
      i \left( \zeta^2 + \frac{\eps A^2}{2} + \frac{S_x}{2}  \right) 
     + \frac{s(1-s)A_s}{A} + 2s-1 
   \right] w' 
	- \frac{\eps \zeta^2 A^2}{4s(1-s)}w =0.
\end{equation}

Now, observe that for the family of potentials \eqref{q.fam} we have
\begin{gather}
\label{A}
	A^2 = \nu^2 \sech^2(x) = 4\nu^2 s(1-s) \\
\label{S_x}
	S_x = -\eps \nu^2 \sech^2(x) + 2 \mu \tanh(x) - 2\delta
	    = -\eps A^2 + 4 \mu s - 2(\mu+\delta)
\end{gather}
which reduces \eqref{LS.w} to
\begin{equation}
\label{LS.w.2}	
	s(1-s) w'' 
	+ \lb -i\zeta^2 + i(\mu+\delta) + \frac{1}{2} - (1+2i\mu) s \rb w' 
	- \eps \nu^2 \zeta^2 w = 0.
\end{equation}
which is the hypergeometric equation
\begin{gather}\label{hyper}
	s(1-s)w'' + \left[\gamma - (1+ \alpha + \beta)s \right]w' - \alpha \beta w = 0 \\
\shortintertext{with} 
\nonumber
	\alpha + \beta	=	2i\mu, \qquad
	\alpha \beta	=	\eps \nu^2 \zeta^2, \qquad
	\gamma			=	-i\zeta^2 + i (\mu+ \delta) + \frac{1}{2}.
\shortintertext{Hence}
\label{alpha.beta}
	\alpha			=	i\mu + i\nu \sqrt{-\eps} R(\zeta), \qquad
	\beta			=	i\mu -  i\nu \sqrt{-\eps} R(\zeta),
\shortintertext{where}
\label{Rzeta}
	R(\zeta) = \sqrt{\zeta^2 - \eps\left(\dfrac{\mu}{\nu}\right)^2}. 
\end{gather}
is taken to be finitely branched on the line segment connecting its branch points with $R(\zeta) \sim \zeta$ for $|\zeta| \gg 1$. In particular, this implies that $R$ maps $\C^{++}$ into itself for either sign of $\eps$. We take $\sqrt{\eps} = 1$ or $i$ when $\eps = 1$ or $-1$ respectively.

\subsection{The Jost Solution}
Denote by $F(\alpha,\beta,\gamma;s)$ the hypergeometric function, analytic in the disk $|s|<1$, with 
\begin{equation} 
\label{hyper.series}
	F(\alpha,\beta,\gamma;s) = 
	1 	+ \frac{\alpha\beta}{\gamma}s 
		+ \frac{1}{2!}\frac{\alpha(\alpha+1) \beta(\beta+1)}{\gamma(\gamma+1)}s^2 + \ldots
	\qquad
	|s| < 1	 
\end{equation}
(see \cite[15.2.1]{DLMF}).
A basis for the solution space of \eqref{LS.w.2} with singularity at $s=0$ is given by  \cite[15.10.2-15.10.3]{DLMF}
\begin{equation}
\label{w0}
\begin{split}
w_1(s)	&=	F(\alpha,\beta,\gamma;s)\\
w_2(s)	&=	s^{1-\gamma} F(\alpha-\gamma+1,\beta-\gamma+1,2-\gamma;s),
\end{split}
\end{equation}
while a basis for the solution space of \eqref{LS.w.2} with singularity at $s=1$ is
given by \cite[15.10.4-15.10.5]{DLMF}
\begin{equation}
\label{w1}
\begin{split}
w_3(s)	&=	F(\alpha, \beta, \alpha+\beta+1-\gamma;1-s)\\
w_4(s)	&=	(1-s)^{\gamma-\alpha-\beta}F(\gamma-\alpha,\gamma-\beta,\gamma-\alpha-\beta+1;1-s)
\end{split}
\end{equation}
We'll need the connection formula \cite[15.10.21]{DLMF}
\begin{equation}
\label{w1c}
w_1(s) = 
\frac	{\Gamma(\gamma)\Gamma(\gamma-\alpha-\beta) }
		{\Gamma(\gamma-\alpha) \Gamma( \gamma-\beta)} 
		w_3(s) 
+ 
\frac	{\Gamma(\gamma) \Gamma(\alpha+\beta-\gamma)}
		{\Gamma(\alpha) \Gamma(\beta)} 
		w_4(s) 
\end{equation}
to compute scattering data for the potential $q$. 

We now seek the correctly normalized solution
\begin{equation}
\label{m1-.form}
m_1^-(x,\zeta) = e^{iS(x)\sigma_3/2} \dfrac{ e^{i\zeta^2 x} }{g(t)}
	\begin{pmatrix}
	1		&		0	\\[3pt]
	0		&		\dfrac{\sqrt{s(1-s)}}{\zeta \nu}
	\end{pmatrix}
	\begin{pmatrix}
	w(s)		\\[3pt]
	w'(s)
	\end{pmatrix}
\end{equation}
by setting $w(s)=c_1 w_1(s) + c_2 w_2(s)$ and using the asymptotics
\begin{equation}
\label{m1-.left}
m_1^-(x,\zeta) =
\begin{pmatrix}
1\\0
\end{pmatrix}
+ o(1)
\end{equation}
for $x \rarr -\infty$, i.e., $s \darr 0$, to choose $c_1$ and $c_2$.

One may solve \eqref{g} for $g$, and use \eqref{q.fam} and \eqref{exs}
to compute
\begin{equation}\label{three.form}
\begin{aligned}
	e^{iS(x)/2}	&=	
	2^{-i\mu} e^{iS_0/2-i\eps \nu^2(s-1/2)} s^{-i(\mu+\delta)/2}(1-s)^{-i(\mu-\delta)/2}\\
	g(s) &=	s^{i\zeta^2/2-i(\mu+\delta)/2}(1-s)^{-i\zeta^2/2-i(\mu-\delta)/2}\\
	e^{i\zeta^2x}	&=	s^{i\zeta^2/2}(1-s)^{-i\zeta^2/2}
\end{aligned}
\end{equation}
so that 
\begin{equation} \label{m11-m21.pre}
	\begin{aligned}
		e^{iS(x)/2}  e^{i\zeta^2 x} / g(t) 
		  &= 2^{-i\mu} e^{iS_0/2-i\eps\nu^2(s-1/2)}\\
		e^{-iS(x)/2} e^{i\zeta^2 x} / g(t) 
		 &= 2^{i\mu} e^{-iS_0/2+i\eps\nu^2(s-1/2)}s^{i(\mu+\delta)}(1-s)^{i(\mu-\delta)}.
	\end{aligned}
\end{equation}
Using \eqref{m11-m21.pre} in \eqref{m1-.form} we obtain
\begin{align}
\label{-}
	m_1^-(x,\zeta)	=	
	\begin{pmatrix}
	   2^{-i\mu} e^{iS_0/2-i\eps \nu^2(s-1/2)} w(s)	\\
       \frac{2^{i\mu}}{\nu \zeta} e^{-iS_0/2+i\eps\nu^2(s-1/2)} 
	     s^{i(\mu+\delta)+1/2}(1-s)^{i(\mu-\delta)+1/2} \frac{dw}{ds}
	\end{pmatrix}.
\end{align}
Note that $\real(\gamma)=1/2$.
Setting $w(s) = c_1w_1(s) + c_2w_2(s)$ and using the convergent series expansions
\[
    w_1(s) \sim 1 + \bigO{s}, \quad 
    w_2(s) \sim s^{1-\gamma}\left(1+\bigO{s} \right), 
\]
we find that $c_1= 2^{i\mu} e^{-iS_0/2-i\eps \nu^2/2}$, $c_2=0$ so that
\begin{equation}
\label{m1-.sol}
m_1^-(x,\zeta)	=	
\begin{pmatrix}
e^{-i\eps \nu^2 s} w_1(s)\\[4pt]
\frac{2^{2i\mu}}{\nu \zeta} e^{-iS_0+i\eps \nu^2(s-1)}s^{\gamma+i\zeta^2}(1-s)^{1-\gamma+\alpha+\beta-i\zeta^2} w_1'(s)
\end{pmatrix}
\end{equation}
Repeating this calculation for the Jost functions $m_1^+$ and $m_2^+$ defined by \eqref{Jost.zeta.norm} yields
\begin{equation}\label{m+.sol}
   \begin{aligned}
		m_1^+(x,\zeta) &= 
		\begin{pmatrix}
		   e^{-i \eps \nu^2(s-1)} w_3(s) \\[4pt]
		   \frac{2^{2i\mu}}{\zeta \nu} e^{-i S_0 + i \eps \nu^2 s} 
		     s^{1/2 + i(\mu+\delta)} (1-s)^{1/2 + i(\mu-\delta)} w_3'(s)
		\end{pmatrix} \\
		m_2^+(x,\zeta) &= (\alpha + \beta - \gamma)^{-1} 
		\begin{pmatrix}
		   \zeta \nu 2^{2i\mu} e^{i S_0 - i \eps \nu^2 s}
		     s^{-i\zeta^2} (1-s)^{i\zeta^2} w_4(s) \\[4pt] 		     
		   e^{i\eps \nu^2(s-1)} 
		     s^{-i\zeta^2 + i(\mu+\delta) +1/2} (1-s)^{ i\zeta^2 + i(\mu-\delta)} w_4'(s)
		\end{pmatrix}
   \end{aligned}
\end{equation}

\subsection{Scattering Data}

The scattering coefficients $\ba(\zeta)$ and $-b(\zeta)$ defined by \eqref{trans} satisfy
\begin{equation}
\label{m1-.right}
   m_1^-(x,\zeta) =  \ba(\zeta) m_1^+(x,\zeta) - b(\zeta) e^{2i\zeta^2 x} m_2^+(x,\zeta),
   \qquad
   \zeta \in \Sigma.
\end{equation}
Using the connection formula \eqref{w1c}, the Jost functions \eqref{m1-.sol}-\eqref{m+.sol}, and the last line of \eqref{three.form} we may read off that 
\begin{align}
\label{ex.ba}
   \ba(\zeta)	&=	
       e^{-i\eps \nu^2} 		
	      \frac	{\Gamma(\gamma)\Gamma(\gamma-\alpha-\beta)}
	            {\Gamma(\gamma-\alpha)\Gamma(\gamma-\beta)}, 	
				\\
\label{ex.b}
   b(\zeta)		&=	
      - \frac{2^{2i\mu}}{\nu \zeta} e^{-iS_0} 
	  \frac{\Gamma(\gamma)\Gamma(1+\alpha+\beta-\gamma)}
	       {\Gamma(\alpha)\Gamma(\beta)}.
\end{align}
From the formulas
\begin{align*}
	\gamma	= -i\zeta^2+ i(\mu+\delta)+\frac{1}{2}
	\qquad
	\gamma-\alpha-\beta =	-i\zeta^2 -i(\mu-\delta) + \frac{1}{2}
\end{align*}
it is easy to see that for $\zeta \in \Sigma$, 
$-i\zeta^2$ is purely imaginary so that $\real(\gamma)=1/2$, $\real(\gamma-\alpha-\beta)=1/2$ and $\ba(\zeta)$ has no spectral singularities. 
Moreover, $\real(\gamma) > 1/2$ and $\real(\gamma-\alpha-\beta) > 1/2$ for $\zeta \in \C^{++}$ so that $\ba(\zeta)$ is holomorphic\footnote{The parameters $\alpha$ and $\beta$ appearing in the argument of $\ba(\zeta)$ and $b(\zeta)$ inherit a branch cut from $R(\zeta)$; their boundary values satisfy $\alpha^\pm = \beta_{\mp}$. Since both $\ba(\zeta)$ and $b(\zeta)$ are invariant under the map $\alpha \leftrightarrow \beta$, both scattering coefficients are analytic across the branch of $R$. 
} in $\C^{++}$ as required.   
It's also clear from \eqref{ex.b} and the above observations that $b(\zeta)$ extends meromorphically to $\zeta \in \C^{++}$; $b(\zeta)$ will have isolated simple poles at those values of $\zeta$ where
\begin{equation}\label{b.pole}
	\gamma - \alpha - \beta = m, \quad m = 1,2,3,\dots
\end{equation}
and neither $\alpha$ nor $\beta$ is a nonpositive integer.

The discrete spectrum of \eqref{LS.bis} are precisely the quartets $\pm \zeta_n, \pm \bar \zeta_n$, $\zeta_n \in \C^{++}$ such that $\ba(\zeta_n) = 0$. From \eqref{ex.ba}, $\ba$ will have a zero whenever $\gamma - \alpha = 1-n$, $\gamma-\beta = 1-n$. We have
\begin{align}
\label{gamma-alpha}
   \gamma-\alpha &=	
      -i\zeta^2 + i\delta - i\nu \sqrt{-\eps}R(\zeta) + \frac{1}{2}\\
\label{gamma-beta}
   \gamma-\beta &=	
      -i\zeta^2 + i\delta + i\nu \sqrt{-\eps}R(\zeta) + \frac{1}{2}
\end{align}
Notice that $\Re \gamma - \alpha = \Im \zeta^2 + \nu \Im( \sqrt{-\eps} R(\zeta)) > 0$ for either sign of $\eps$ as $R(\zeta) \in \C^{++}$ when $\zeta \in \C^{++}$. Thus zeros of $\ba(\zeta)$ must satisfy $\gamma-\beta = 1-n$ or equivalently
\begin{equation}
\label{cond 0}
   \zeta^2 -  \delta - \nu \sqrt{-\eps} R(\zeta) + i(n-1/2) = 0, \qquad \zeta \in \C^{++} \quad n=1,2,3,\ldots .
\end{equation}

Finally, we point out that $b(\zeta)$ is analytic at any zero of $\ba(\zeta)$ since if 
$\zeta_n$ satisfies both $\gamma - \beta = 1-n$ and \eqref{b.pole} then $\alpha = 1 - n - m$ is a negative integer and the pole of $b(\zeta)$ is removable. Then because the scattering relation \eqref{m1-.right} extends analytically to $\zeta_n$ we can compute the connection coefficient $c_n$ (see \eqref{czeta}-\eqref{czeta.bdef} ) as
\begin{equation}\label{ex.czeta}
	c_n = \frac{ b(\zeta_n)}{\ba'(\zeta_n)}.
\end{equation}

\begin{proof}[Proof of Proposition~\ref{prop:empty}]

From the arguments already given in this section the scattering map is easily constructed for data in the family \eqref{q.fam}. Using the formulae \eqref{sym.ab},\eqref{ab->rho}, and \eqref{zeta->lambda} we have
\[
	\rho(\lam) = 
	\eps \zeta^{-1} \frac{\overline{ b\lp \zetabar \rp}}{\overline{ \ba \lp \zetabar \rp}},
	\qquad
	\lam_n = \zeta_n^2,
	\qquad
	C_n = 2 c_n.
\]
It remains to show that \eqref{cond.empty} guarantees that the discrete spectrum will be empty.

We want to find values of $\mu$ and $\nu$ such that \eqref{cond 0} cannot be satisfied for any $\zeta \in \C^{++}$ and all $n = 1,2,3,\dots$.
Introduce the parameters
\[ 
	\lam \coloneqq \zeta^2, \quad z \coloneqq \zeta^2 - \delta, \quad L_n \coloneqq n-1/2
\]
and note that $\imag z>0$ and $L_n > 0$. The condition \eqref{cond 0}
becomes
\begin{equation}
\label{cond 1}
	z + iL_n = \nu \sqrt{-\eps}  \sqrt{z+\delta - \eps(\mu/\nu)^2}. 
\end{equation}
where we used \eqref{Rzeta}.
Writing $z=x+iy$ with $y>0$, squaring both sides of \eqref{cond 1},
and taking real and imaginary parts we find
\begin{subequations}
\label{cond 2}
\begin{align}
\label{cond 2a}
x^2 - (y+L_n)^2 + \eps \nu^2 x		&=	-\eps \nu^2 \delta + \mu^2,	\\
\label{cond 2b}
2x \left(y+L_n \right)					&=	-\eps \nu^2 y.
\end{align}
\end{subequations}
Solving \eqref{cond 2b} for $L_n$ we have
\begin{equation}
\label{L def}
L_n 	=	-y \frac{\eps \nu^2 +  2x}{2x}.
\end{equation}
Since $L_n$ and $y$ are both strictly positive, it follows that  $-(\eps \nu^2 + 2x)/(2x)$ is strictly positive as well. Thus, any solution $z=x+iy$ lies in the vertical half-strip
\begin{equation}
\label{strip}
S_\epsilon=	\left\{  (x,y) \in \R^2: y >0, \,\, -\frac{\nu^2}{2} < \eps x < 0 \right\}. 
\end{equation}
On the other hand, using \eqref{L def} in \eqref{cond 2a}, we see that any solutions $z$ of \eqref{cond 1} lie along the curve 
\begin{equation}
\label{cond 3}
\calC \coloneqq \left\{ (x,y) \in \R \times (0,\infty): x^2 Q(x) = \nu^4 y^2/4 \right\}
\end{equation}
where
\begin{equation}
\label{Q def}
	Q(x) = 	\left( \eps x + \frac{\nu^2}{2} \right)^2 - 
			\left[
			   \nu^2 \left( \frac{\nu^2}{4} - \eps \delta \right) + \mu^2
			\right]
\end{equation}
To find parameter values of $(\mu,\delta)$ for which $\ba$ has no zeros, it suffices to find $(\mu,\delta)$ so that $\calC$ does not intersect the strip $S_\eps$. From \eqref{Q def}, it is clear that $\calC$ will have empty intersection with $S_\eps$ provided $Q(x)<0$ for $x$ with $0 < -\eps x < \nu^2/4$.  For either sign of $\eps$ this is guaranteed if $Q(0) < 0$. Since $Q(0)= \eps \nu^2 \delta - \mu^2$ it follows that \eqref{cond.empty} guarantees the discrete spectrum is empty.
\end{proof}					

%
%
%
%

%
%
\endgroup
\end{document}